\documentclass{book}

\usepackage{amsmath,amsthm,amssymb,url}
\usepackage{bm,ulem,enumerate,mathrsfs}

\usepackage{CJKutf8}

\makeindex
\usepackage{makeidx}

\numberwithin{equation}{section}

\theoremstyle{plain}
\newtheorem{lemma}[equation]{Lemma}
\newtheorem{proposition}[equation]{Proposition}
\newtheorem{theorem}[equation]{Theorem}
\newtheorem{corollary}[equation]{Corollary}

\theoremstyle{definition}
\newtheorem{remark}[equation]{Remark}
\newtheorem{definition}[equation]{Definition}
\newtheorem{example}[equation]{Example}

\newtheorem*{definition*}{Definition}

\begin{document}

\title{Continuous representations of semisimple Lie groups concerning homogeneous holomorphic vector bundles over elliptic adjoint orbits}
\author{Nobutaka Boumuki\thanks{This work was supported by JSPS KAKENHI Grant Number JP 17K05229.}}

\maketitle

\begin{address}
   {Nobutaka Boumuki\endgraf Division of Mathematical Sciences\endgraf Faculty of Science and Technology\endgraf Oita University\endgraf 700 Dannoharu, Oita-shi, Oita 870-1192, JAPAN}{boumuki@oita-u.ac.jp}
\end{address}

\chapter*{Preface}\label{ch-0}
   Our interest lies in continuous representations of real semisimple Lie groups concerning homogeneous holomorphic vector bundles over elliptic adjoint orbits, especially the representation $\varrho:G\to GL(\mathcal{V}_{G/L})$ below.\par

   Let $G_\mathbb{C}$ be a connected complex semisimple Lie group, let $G$ be a connected closed subgroup of $G_\mathbb{C}$ whose Lie algebra $\operatorname{Lie}(G)=\frak{g}$ is a real form of $\frak{g}_\mathbb{C}$, and let $T$ be a non-zero, element of $\frak{g}$ such that (1) the linear transformation $\operatorname{ad}T:\frak{g}\to\frak{g}$, $X\mapsto[T,X]$, is semisimple and (2) all the eigenvalues of $\operatorname{ad}T$ are purely imaginary. 
   Consider the adjoin orbit $\operatorname{Ad}G(T)=G/L$ of $G$ through $T$, where $L:=\{g\in G \,|\, \operatorname{Ad}g(T)=T\}$. 
   These $T$ and $G/L$ are called an {\it elliptic element} and an {\it elliptic adjoint orbit} (or an {\it elliptic orbit} for short), respectively.
   It is shown that $G/L$ can be embedded into a complex flag manifold $G_\mathbb{C}/Q^-$ (which is also called a K\"{a}hler C-space or a generalized flag manifold) via $\iota:G/L\to G_\mathbb{C}/Q^-$, $gL\mapsto gQ^-$, and furthermore the image $\iota(G/L)$ is a domain in $G_\mathbb{C}/Q^-$. 
   Identifying $G/L$ with $\iota(G/L)$ we induce a $G$-invariant complex structure $J$ on $G/L$ from $G_\mathbb{C}/Q^-$.
   Then, the elliptic orbit $G/L$ is a homogeneous complex manifold of $G$.
\begin{center}
\unitlength=1mm
\begin{picture}(50,25)
\put(5,0){$G/L$}
\put(15,1){\vector(1,0){22}}
\put(40,0){$G_\mathbb{C}/Q^-$}
\put(25,3){$\iota$}
\put(9,20){\vector(0,-1){15}}
\put(44,20){\vector(0,-1){15}}
\put(0,23){$\iota^\sharp(G_\mathbb{C}\times_\rho{\sf V})$}
\put(37,23){$G_\mathbb{C}\times_\rho{\sf V}$}
\end{picture}
\end{center}
   Take a finite-dimensional complex vector space ${\sf V}$ and a holomorphic homomorphism $\rho:Q^-\to GL({\sf V})$, $q\mapsto\rho(q)$, where $GL({\sf V})$ is the general linear group on ${\sf V}$.
   Denote by $G_\mathbb{C}\times_\rho{\sf V}$ the homogeneous holomorphic vector bundle over the complex flag manifold $G_\mathbb{C}/Q^-$ associated with $\rho$, and by $\iota^\sharp(G_\mathbb{C}\times_\rho{\sf V})$ its restriction to the domain $G/L\subset G_\mathbb{C}/Q^-$.
   In this setting, one may assume that
\[
   \mathcal{V}_{G/L}
   :=\left\{\begin{array}{@{}c|c@{}}
   \psi:GQ^-\to{\sf V} 
   & \begin{array}{@{}l@{}} \mbox{(i) $\psi$ is holomorphic},\\ \mbox{(ii) $\psi(xq)=\rho(q)^{-1}\bigl(\psi(x)\bigr)$ for all $(x,q)\in GQ^-\times Q^-$}\end{array}
   \end{array}\right\}
\]
is the complex vector space of holomorphic cross-sections of the bundle $\iota^\sharp(G_\mathbb{C}\times_\rho{\sf V})$; and can define a continuous representation $\varrho$ of $G$ on $\mathcal{V}_{G/L}$ by 
\[
   \mbox{$\bigl(\varrho(g)\psi\bigr)(x):=\psi(g^{-1}x)$ for $g\in G$, $\psi\in\mathcal{V}_{G/L}$, and $x\in GQ^-$}.
\]
   Here the topology for $\mathcal{V}_{G/L}$ is the topology of uniform convergence on compact sets.

\section*{Notation}
   Throughout this note we utilize the following notation, where $G$ is a Lie group and $\frak{g}$ is a Lie algebra:
\begin{enumerate}[(n1)] 
\item
   $\operatorname{Lie}(G)$ : the Lie algebra of $G$, i.e., the real Lie algebra of left invariant vector fields on $G$,
\item
   $G_0$ : the identity component of $G$, 
\item 
   $L_g$ (resp.\ $R_g$) : the left (resp.\ right) translation of $G$ by an element $g\in G$,  
\item
   $\operatorname{ad}$, $\operatorname{Ad}$ : the adjoint representations of a Lie algebra and a Lie group, respectively,
\item
   $Z(G)$ : the center of $G$, 
\item 
   $C_G(A):=\{g\in G \,|\, \mbox{$gag^{-1}=a$ for all $a\in A$}\}$ for a subset $A\subset G$,
\item 
   $C_G(\frak{a}):=\{g\in G \,|\, \mbox{$\operatorname{Ad}g(X)=X$ for all $X\in\frak{a}$}\}$ for a subset $\frak{a}\subset\operatorname{Lie}(G)$,   
\item
   $C_G(X):=\{g\in G \,|\, \operatorname{Ad}g(X)=X\}$ for an element $X\in\operatorname{Lie}(G)$,
\item 
   $N_G(\frak{m}):=\{g\in G \,|\, \operatorname{Ad}g(\frak{m})\subset\frak{m}\}$ for a vector subspace $\frak{m}\subset\operatorname{Lie}(G)$,
\item 
   $\frak{c}_\frak{g}(X):=\{Z\in\frak{g} \,|\, \operatorname{ad}X(Z)=0\}$ for an element $X\in\frak{g}$, which is the kernel of the linear mapping $\operatorname{ad}X:\frak{g}\to\frak{g}$, 
\item 
   $\operatorname{ad}X(\frak{g})$ : the image of the linear mapping $\operatorname{ad}X:\frak{g}\to\frak{g}$ above, 
\item 
   $B_\frak{g}$ : the Killing form of $\frak{g}$,
\item
   $\mathbb{N}$, $\mathbb{Z}$, $\mathbb{Q}$, $\mathbb{R}$, $\mathbb{C}$ : the sets of natural numbers, integers, rational numbers, real numbers, and complex numbers, respectively, where $\mathbb{N}$ does not contain the zero,
\item
   $\mathbb{Z}_{\geq 0}$ : the set of non-negative integers, 
\item
   $\mathbb{R}^+$ : the set of positive real numbers, 
\item 
   $\mathbb{K}=\mathbb{R}$ or $\mathbb{C}$,     
\item
   $GL({\sf V})$ : the general linear group on a vector space ${\sf V}$ over $\mathbb{K}$,
\item 
   $\frak{m}\oplus\frak{n}$ : the direct sum of vector spaces $\frak{m}$ and $\frak{n}$,
\item
   $A\amalg B$ : the disjoint union of sets $A$ and $B$, 
\item
   $\mathcal{C}^\infty(M)$ : the associative algebra of real-valued smooth functions on a smooth manifold $M$, 
\item 
   $T_pM$ : the tangent vector space of a smooth manifold $M$ at a point $p\in M$,
\item
   $\frak{X}(M)$ : the real Lie algebra of smooth vector fields on a smooth manifold $M$, which is also a $\mathcal{C}^\infty(M)$-module,  
\item
   $f|_A$ : the restriction of a mapping $f$ to a set $A$,
\item
   $\operatorname{id}_A$ or $\operatorname{id}$ : the identity mapping of a set $A$, 
\item
   $c_A$ : the characteristic function of a set $A$,
\item
   $\overline{A}{}^X$ or $\overline{A}$ : the closure of a subset $A$ in a topological space $X$.  
\end{enumerate}
   In addition, for a Lie group $G$ we usually denote its Lie algebra by the corresponding Fraktur small letter $\frak{g}$.

\section*{Remark}
   We say that a Lie group $G$ is {\it semisimple}, {\it nilpotent}, or {\it parabolic}, respectively, whenever the Lie algebra $\frak{g}$ is.

\tableofcontents

\chapter{Homogeneous spaces}\label{ch-1}
   In this chapter we review fundamental facts about homogeneous spaces. 
   We deal with (real) homogeneous spaces in Section \ref{sec-1.1} and complex homogeneous spaces in Section \ref{sec-1.2}.
   Finally in Section \ref{sec-1.3} we show that homogeneous spaces are principal fiber bundles.

\section{Real case}\label{sec-1.1}
   Let $G$ be a (real) Lie group which satisfies the second countability axiom, and let $H$ be a closed subgroup of $G$. 
   Consider the left quotient space $G/H=\{gH \,|\, g\in G\}$ of $G$ by $H$ and define a surjective mapping $\pi:G\to G/H$ (which is called the {\it projection}\index{projection@projection\dotfill} of $G$ onto $G/H$) as follows:
\begin{equation}\label{eq-1.1.1}
   \mbox{$\pi(g):=gH$ for $g\in G$}.
\end{equation}
   Provide $G/H$ with the quotient topology relative to this $\pi$. 
   Then, $G/H$ is called a {\it homogeneous space}\index{homogeneous space@homogeneous space\dotfill}, and one has 
\begin{theorem}\label{thm-1.1.2}
   There exists a real analytic structure $\mathcal{S}=\{(U_\alpha,\psi_\alpha)\}_{\alpha\in A}$ on the homogeneous space $G/H$ so that 
\begin{enumerate}
\item[{\rm (1)}]
   $\pi:G\to G/H$, $g\mapsto gH$, is a surjective, open, real analytic mapping,
\item[{\rm (2)}]
   $\mu:G\times G/H\to G/H$, $(g_1,g_2H)\mapsto g_1g_2H$, is a real analytic mapping.
\end{enumerate} 
   Moreover, for each $\alpha\in A$ there exists a real analytic mapping $\sigma_\alpha:U_\alpha\to G$ such that $\pi\bigl(\sigma_\alpha(x)\bigr)=x$ for all $x\in U_\alpha$.
\end{theorem}   

   The main purpose of this section is to demonstrate Theorem \ref{thm-1.1.2}. 
\begin{remark}\label{rem-1.1.3}
   The condition (2) in Theorem \ref{thm-1.1.2} implies that $\pi:G\to G/H$ is real analytic, since $\pi(g)=\mu\bigl(g,\pi(e)\bigr)$ for all $g\in G$. 
   Here, $e$ is the unit element of $G$. 
\end{remark} 

\begin{remark}[Uniqueness]\label{rem-1.1.4}
\begin{enumerate}
\item[]
\item[(i)]
   Suppose $G/H$ to admit another real analytic structure $\mathcal{S}'$ so that $\mu:G\times G/H\to G/H$, $(g_1,g_2H)\mapsto g_1g_2H$, is real analytic, where the topology for $G/H$ is the quotient one relative to $\pi$. 
   Then, $(G/H,\mathcal{S})$ is $G$-equivariant real analytic diffeomorphic to $(G/H,\mathcal{S}')$ via the identity mapping of $G/H$.
   cf.\ Subsection \ref{subsec-1.1.4}.
\item[(ii)]
   Let $r\in\mathbb{N}\cup\{0,\infty,\omega\}$. 
   Suppose that $G$ acts transitively on a differentiable manifold $M$ of class $C^r$ as a differentiable transformation group of class $C^r$, $G\times M\ni(g,x)\mapsto g\cdot x\in M$.
   Let us denote by $H'$ the isotropy subgroup of $G$ at an $x_0\in M$.
   Then, $G/H'\ni gH'\mapsto g\cdot x_0\in M$ is a $G$-equivariant diffeomorphism of class $C^r$ of $G/H'=(G/H',\mathcal{S})$ onto $M$.\footnote{e.g.\ \begin{CJK}{UTF8}{min}定理 3.7.11 in 杉浦\end{CJK} \cite[p.131]{Su2}.} 
\end{enumerate}
\end{remark}   

\subsection{Topological properties of $G/H$}\label{subsec-1.1.1}
   Recall that the topology for $G/H$ is the quotient topology relative to $\pi$.

\begin{lemma}\label{lem-1.1.5}
\begin{enumerate}
\item[]
\item[{\rm (1)}]
   $\pi:G\to G/H$, $g\mapsto gH$, is a surjective, open, continuous mapping.
\item[{\rm (2)}]
   For any open subset $U\subset G/H$, there exists an open subset $O\subset G$ such that $\pi(O)=U$.
\end{enumerate}
\end{lemma}
\begin{proof}
   (1).
   It suffices to confirm that $\pi:G\to G/H$, $g\mapsto gH$, is an open mapping. 
   For any open subset $O'$ of $G$, we deduce 
\[
   \pi^{-1}\bigl(\pi(O')\bigr)
   =O'H
   =\bigcup_{h\in H}R_h(O')
\]
by a direct computation, where $R_h$ stands for the right translation of the Lie group $G$ by $h$. 
   Since each $R_h(O')$ is open in $G$, the union $\bigcup_{h\in H}R_h(O')=\pi^{-1}\bigl(\pi(O')\bigr)$ is also open in $G$. 
   Therefore $\pi(O')$ is an open subset of $G/H$.\par

   (2). 
   Since $\pi:G\to G/H$ is continuous, $O:=\pi^{-1}(U)$ is an open subset of $G$. 
   Furthermore, one has $\pi(O)=U$ because $\pi:G\to G/H$ is surjective.
\end{proof}

\begin{lemma}\label{lem-1.1.6}
\begin{enumerate}
\item[]
\item[{\rm (1)}]
   $G/H$ is a Hausdorff space. 
\item[{\rm (2)}]
   $\mu:G\times G/H\to G/H$, $(g_1,g_2H)\mapsto g_1g_2H$, is a continuous mapping.   
\item[{\rm (3)}] 
   $G/H$ satisfies the second countability axiom. 
\end{enumerate}
\end{lemma}
\begin{proof}
   (1) follows by $H$ being a closed subset of $G$ and Lemma \ref{lem-1.1.5}-(1).\par

   (2). 
   Take any $(g_1,g_2H)\in G\times G/H$ and any open neighborhood $U$ of $\mu(g_1,g_2H)=\pi(g_1g_2)\in G/H$. 
   Lemma \ref{lem-1.1.5}-(1) implies that $\pi^{-1}(U)$ is an open neighborhood of $g_1g_2\in G$, so that there exist open subsets $O_1$, $O_2\subset G$ satisfying $g_1\in O_1$, $g_2\in O_2$ and $O_1O_2\subset\pi^{-1}(U)$.
   Then, $O_1\times\pi(O_2)$ is an open neighborhood of $(g_1,g_2H)\in G\times G/H$, and it follows that $\mu\bigl(O_1\times\pi(O_2)\bigr)\subset U$.\par      

   (3). 
   Since $G$ satisfies the second countability axiom, there exists a countable open base $\{O_n\}_{n\in\mathbb{N}}$ for the topological space $G$. 
   Lemma \ref{lem-1.1.5} implies that $\{\pi(O_n)\}_{n\in\mathbb{N}}$ is a countable open base for the topological space $G/H$.
\end{proof}

   Lemma \ref{lem-1.1.6}-(2) leads to
\begin{corollary}\label{cor-1.1.7}
   Fix a $g\in G$ and define a transformation $\tau_g$ of $G/H$ by 
\[
   \mbox{$\tau_g(aH):=gaH$ for $aH\in G/H$}.
\]
   Then for each $g\in G$, $\tau_g$ is a homeomorphic transformation of $G/H$, and $\tau_g\circ\pi=\pi\circ L_g$ on $G$. 
   Here $L_g$ stands for the left translation of the Lie group $G$ by $g$.
\end{corollary}

\subsection{Local cross-sections}\label{subsec-1.1.2}
   Choose a real vector subspace $\frak{m}\subset\frak{g}$ such that 
\[
   \frak{g}=\frak{m}\oplus\frak{h},
\] 
and define a real analytic mapping $\varphi:\frak{m}\times\frak{h}\to G$ by $\varphi(X,Y):=\exp X\exp Y$ for $(X,Y)\in\frak{m}\times\frak{h}$. 
   Then,  
     
\begin{lemma}\label{lem-1.1.8}
   There exist two open neighborhoods $V_1$ of $0\in\frak{m}$ and $B_1$ of $0\in\frak{h}$ such that 
\begin{enumerate}
\item[{\rm (1)}]
   $\varphi:(X,Y)\mapsto\exp X\exp Y$ is a real analytic diffeomorphism of $V_1\times B_1$ onto an open neighborhood of $e\in G$,
\item[{\rm (2)}]
   $\exp B_1$ is an open neighborhood of $e\in H$.
\end{enumerate} 
\end{lemma}
\begin{proof}
   It turns out that $\varphi(0,0)=e$ and the differential $(d\varphi)_{(0,0)}$ of $\varphi$ at $(0,0)$ is a real linear isomorphism of the tangent vector space $T_{(0,0)}(\frak{m}\times\frak{h})$ onto $T_eG$. 
   Thus the inverse mapping theorem assures the existence of open neighborhoods $V_1$ of $0\in\frak{m}$ and $B_1$ of $0\in\frak{h}$ satisfying (1); besides, one may assume that the (2) holds for this $B_1$ by substituting a sufficiently small open neighborhood $B_1'$ of $0\in\frak{h}$ for $B_1$ (if necessary).
\end{proof}

   Let $V_1$, $B_1$ have the properties in Lemma \ref{lem-1.1.8}.
   In this setting, we assert 
\begin{proposition}\label{prop-1.1.9}
   There exists an open neighborhood $V$ of $0\in\frak{m}$ so that 
\begin{enumerate}
\item[{\rm (1)}]
   $0\in V\subset V_1$, 
\item[{\rm (2)}]
   $N:=\exp V$ is a regular submanifold of $G$,
\item[{\rm (3)}]
   $\exp:V\to N$ is a real analytic diffeomorphism,
\item[{\rm (4)}]
   $\pi(N)$ is an open subset of $G/H$,
\item[{\rm (5)}]
   $\pi:N\to\pi(N)$ is homeomorphic.
\end{enumerate}   
\end{proposition}
\begin{proof}
   Taking Lemma \ref{lem-1.1.8}-(2) and the topology for $H$ into account, we see that there exists an open neighborhood $O$ of $e\in G$ satisfying 
\begin{equation}\label{eq-a}\tag{a}
   \exp B_1=(O\cap H).
\end{equation} 
   Since the mapping $G\times G\ni(g_1,g_2)\mapsto g_1^{-1}g_2\in G$ is continuous, one can choose a compact subset $C\subset V_1$ containing an open neighborhood of $0\in\frak{m}$ and satisfying 
\begin{equation}\label{eq-b}\tag{b}
   \exp(-C)\exp C\subset O.
\end{equation} 
   In this setting, $\pi:\exp C\to\pi(\exp C)$ is a homeomorphism because if $\pi(\exp X_1)=\pi(\exp X_2)$ with $X_1,X_2\in C$, then it follows from \eqref{eq-b}, \eqref{eq-a} that $\exp(-X_2)\exp X_1\in(O\cap H)=\exp B_1$. 
   This and Lemma \ref{lem-1.1.8}-(1) yield $X_1=X_2$; consequently $\pi:\exp C\to\pi(\exp C)$ is injective, and so it is homeomorphic due to Lemma \ref{lem-1.1.6}-(1).\par
 
   Now, let $V$ be an open neighborhood of $0\in\frak{m}$ such that $V\subset C$, and let $N:=\exp V$. 
   Then, it turns out that 
\begin{enumerate}
\item[(i)]
   $0\in V\subset C\subset V_1$,
\item[(ii)]
   $V\times B_1$ is an open neighborhood of $(0,0)\in V_1\times B_1$ ($\because$ (i)), 
\item[(iii)]
   $\varphi(V\times B_1)$ is an open neighborhood of $e\in G$ ($\because$ (ii), Lemma \ref{lem-1.1.8}-(1)), 
\item[(iv)]
   $\varphi:V\times B_1\to\varphi(V\times B_1)$, $(X,Y)\mapsto\exp X\exp Y$, is a real analytic diffeomorphism ($\because$ (ii), Lemma \ref{lem-1.1.8}-(1)),
\item[(v)]
   $V\times\{0\}$ is a regular submanifold of $V\times B_1$,
\item[(vi)]
   $N=\exp V=\varphi(V\times\{0\})$ is a regular submanifold of $\varphi(V\times B_1)$ ($\because$ (iv), (v)),
\item[(vii)]
   $\varphi:V\times\{0\}\to N$, $(X,0)\mapsto\exp X$, is a real analytic diffeomorphism ($\because$ (iv), (v), (vi)).
\end{enumerate}
   Therefore (1), (2) and (3) hold for the $V$.
   From $\exp B_1\subset H$ we obtain  
\[
   \pi(\varphi(V\times B_1)\bigr)=\pi(N\exp B_1)=\pi(N),
\]
which assures (4) because the subset $\varphi(V\times B_1)\subset G$ is open and the projection $\pi:G\to G/H$ is an open mapping. 
   The last (5) comes from $N\subset\exp C$ and $\pi:\exp C\to\pi(\exp C)$ being homeomorphic.   
\end{proof}

   Let $V$ have the properties in Proposition \ref{prop-1.1.9}, and let $N:=\exp V$. 
   Proposition \ref{prop-1.1.9}-(4) and Lemma \ref{lem-1.1.5}-(1) imply that $\pi^{-1}\bigl(\pi(N)\bigr)=(\exp V)H$ is an open neighborhood of $e\in G$.
   For any $g\in(\exp V)H$ there exists a unique $(X,h)\in V\times H$ satisfying 
\[
   g=(\exp X)h
\]
because $\pi(g)=\pi(\exp X)\in\pi(N)$ and Proposition \ref{prop-1.1.9}-(5), (3) yield $(\exp|_V)^{-1}\bigl((\pi|_N)^{-1}(\pi(g))\bigr)=X$; therefore $X$ is uniquely determined by $g$, and so is $h$.
   Then, one can define a mapping $\chi:(\exp V)H\to V$ as follows: 
\begin{equation}\label{eq-1.1.10}
   \chi(g):=X 
\end{equation}
for $g=(\exp X)h\in(\exp V)H$ with $(X,h)\in V\times H$.
\begin{lemma}\label{lem-1.1.11}
   The above $\chi:(\exp V)H\to V$, $g\mapsto\chi(g)$, is a real analytic mapping such that 
\begin{enumerate}
\item[{\rm (1)}]
   $\chi=\chi\circ R_h$ for all $h\in H$,
\item[{\rm (2)}]
   $\pi(g)=\pi\bigl(\exp\chi(g)\bigr)$ for all $g\in (\exp V)H$.
\end{enumerate}
   Here we refer to Proposition {\rm \ref{prop-1.1.9}} for $V$.
\end{lemma}
\begin{proof}
   From the definition \eqref{eq-1.1.10} of $\chi$ it is immediate that (1) and (2) hold for $\chi$. 
   Let us prove that $\chi:(\exp V)H\to V$ is real analytic. 
   In view of Lemma \ref{lem-1.1.8} we see that 
\begin{enumerate}
\item[(i)]
   $W:=\exp V\exp B_1$ is an open neighborhood of $e\in G$, 
\item[(ii)]
   $\varphi:V\times B_1\to W$, $(X,Y)\mapsto\exp X\exp Y$, is a real analytic diffeomorphism, 
\item[(iii)]
   $W\subset(\exp V)H$.
\end{enumerate}   
   Take any $g=(\exp X)h\in(\exp V)H$ with $(X,h)\in V\times H$.
   It is natural that $R_{h^{-1}}(g)\in \exp V\subset W$, and hence there exists an open neighborhood $O$ of $g\in(\exp V)H$ such that 
\[
   R_{h^{-1}}(O)\subset W,
\] 
where we recall that $(\exp V)H$ is an open subset of $G$.
   Considering a real analytic mapping $\operatorname{proj}:V\times B_1\to V$, $(X,Y)\mapsto X$, we conclude that $\operatorname{proj}\circ\varphi^{-1}:W\to V$ is a real analytic mapping. 
   Therefore 
\[
   \mbox{$\operatorname{proj}\circ\varphi^{-1}\circ R_{h^{-1}}:O\to V$ is a real analytic mapping}
\]
because the right translation $R_{h^{-1}}:G\to G$ is real analytic. 
   This enables us to conclude that $\chi:O\to V$ is real analytic, because (iii), $\operatorname{proj}\circ\varphi^{-1}=\chi$ on $W$ and $\chi=\chi\circ R_{h^{-1}}$ on $(\exp V)H$ imply that $\chi=\operatorname{proj}\circ\varphi^{-1}\circ R_{h^{-1}}$ on $O$.
\end{proof}

\subsection{Proof of Theorem 1.1.2}\label{subsec-1.1.3}
   From now on, let us demonstrate Theorem \ref{thm-1.1.2}.
\begin{proof}[Proof of Theorem $\ref{thm-1.1.2}$]
   Take an open neighborhood $V$ of $0\in\frak{m}$ having the properties in Proposition \ref{prop-1.1.9}, and put $N:=\exp V$.
   Proposition \ref{prop-1.1.9}-(4), (5), (3) enables us to define an open neighborhood $U$ of $\pi(e)\in G/H$ by 
\[
   U:=\pi(N),
\]
and moreover, define two homeomorphisms $\sigma:U\to N$ and  $\psi:U\to V$ by 
\begin{equation}\label{eq-a}\tag{a}
\begin{array}{ll}
   \sigma:=(\pi|_N)^{-1}, &
   \psi:=(\exp|_V)^{-1}\circ\sigma,
\end{array}
\end{equation}
respectively. 
\begin{center}
\unitlength=1mm
\begin{picture}(64,29)
\put(1,24){$\frak{m}\supset V$}
\put(17,28){$\exp$}
\put(14,27){$\vector(1,0){15}$}
\put(29,24){$\vector(-1,0){15}$}
\put(15,20){$(\exp|_V)^{-1}$}
\put(32,24){$N=\exp V\subset G$}
\put(32,1){$U=\pi(N)\subset G/H$}
\put(38,12){$\pi$}
\put(41,22){$\vector(0,-1){17}$}
\put(44,5){$\vector(0,1){17}$}
\put(45,12){$\sigma=(\pi|_N)^{-1}$}
\put(30,3){$\vector(-1,1){18}$}
\put(17,11){$\psi$}
\end{picture}
\end{center}
   Let us fix a real basis $\{X_i\}_{i=1}^n$ of the vector space $\frak{m}$, identify $\frak{m}$ with $\mathbb{R}^n$, and set  
\begin{equation}\label{eq-b}\tag{b}
\begin{array}{lll}
   U_g:=\tau_g(U), &
   \mbox{$\psi_g(x):=\psi\bigl(\tau_g^{-1}(x)\bigr)$ for $x\in U_g$} &
   (g\in U).
\end{array}
\end{equation}
   Then, Lemma \ref{lem-1.1.6}-(1) and Corollary \ref{cor-1.1.7} imply that 
\begin{enumerate}
\item[1.]
   $G/H$ is an $n$-dimensional topological manifold, 
\item[2.]
   each pair $(U_g,\psi_g)$ is a coordinate neighborhood of $G/H$ with $\pi(g)\in U_g$ ($g\in G$),
\item[3.]
   $\mathcal{S}:=\{(U_g,\psi_g)\}_{g\in G}$ is an atlas of $G/H$.
\end{enumerate}
   Our first aim is to show that 
\begin{equation}\label{eq-1}\tag*{\textcircled{1}}
   \mbox{the above $\mathcal{S}=\{(U_g,\psi_g)\}_{g\in G}$ defines a real analytic structure in $G/H$}.
\end{equation} 
   Suppose that $U_{g_1}\cap U_{g_2}\neq\emptyset$ ($g_1,g_2\in G$).
   For any $X\in\psi_{g_2}(U_{g_1}\cap U_{g_2})\subset V$, it follows from $\tau_{g_2}\circ\pi=\pi\circ L_{g_2}$, $N=\exp V$, \eqref{eq-a} and \eqref{eq-b} that $\pi(g_2\exp X)=\psi_{g_2}^{-1}(X)\in U_{g_1}\cap U_{g_2}$, so that $\pi(g_1^{-1}g_2\exp X)\in \tau_{g_1^{-1}}(U_{g_1}\cap U_{g_2})\subset U=\pi(N)$; and furthermore, $g_1^{-1}g_2\exp X\in\pi^{-1}(U)=(\exp V)H$ and Lemma \ref{lem-1.1.11}-(2) yield 
\[
\begin{split}
   (\psi_{g_1}\circ\psi_{g_2}^{-1})(X)
  &=\psi\bigl(\pi(g_1^{-1}g_2\exp X)\bigr)
   =\psi\bigl(\pi\bigl(\exp\chi(g_1^{-1}g_2\exp X)\bigr)\bigr)\\
  &=\bigl((\exp|_V)^{-1}\circ(\pi|_N)^{-1}\bigr)\bigl(\pi\bigl(\exp\chi(g_1^{-1}g_2\exp X)\bigr)\bigr)
   =(\exp|_V)^{-1}\bigl(\exp\chi(g_1^{-1}g_2\exp X)\bigr)\\
  &=\chi(g_1^{-1}g_2\exp X)
   =\bigl(\chi\circ L_{g_1^{-1}g_2}\circ(\exp|_V)\bigr)(X).
\end{split} 
\] 
   Accordingly $\psi_{g_1}\circ\psi_{g_2}^{-1}=\chi\circ L_{g_1^{-1}g_2}\circ(\exp|_V)$, and thus $\psi_{g_1}\circ\psi_{g_2}^{-1}:\psi_{g_2}(U_{g_1}\cap U_{g_2})\to\psi_{g_1}(U_{g_1}\cap U_{g_2})$ is real analytic, because all the mappings $\chi:\pi^{-1}(U)\to V$, $L_{g_1^{-1}g_2}:G\to G$ and $\exp:V\to N$ are real analytic due to Lemma \ref{lem-1.1.11}, Proposition \ref{prop-1.1.9}-(3). 
   We have shown \ref{eq-1}. 
   Henceforth, $G/H$ is a real analytic manifold having the atlas $\mathcal{S}=\{(U_g,\psi_g)\}_{g\in G}$.\par
   
   Our second aim is to verify that 
\begin{equation}\label{eq-2}\tag*{\textcircled{2}}
   \mbox{$\pi:G\to G/H$, $g\mapsto gH$, is a surjective, open, real analytic mapping}.
\end{equation} 
   By virtue of Lemma \ref{lem-1.1.5}-(1) it suffices to verify that the projection $\pi:G\to G/H$ is real analytic. 
   Let $B_1$ denote the open neighborhood of $0\in\frak{h}$ given in Lemma \ref{lem-1.1.8}, and let $W:=\exp V\exp B_1$. 
   Here, we know that $W$ is an open neighborhood of $e\in G$ and $\varphi:V\times B_1\to W$, $(X,Y)\mapsto\exp X\exp Y$, is a real analytic diffeomorphism (cf.\ the proof of Lemma \ref{lem-1.1.11}).
   For an arbitrary $g\in G$, it follows that 
\begin{enumerate}
\item[4.]
   $gW$ is an open neighborhood of $g\in G$, 
\item[5.]
   $\pi(gW)\subset U_g$, 
\item[6.]
   $(gW,\varphi^{-1}\circ L_{g^{-1}})$ is a coordinate neighborhood of $G$, 
\end{enumerate}
where we identify $\frak{h}$ with $\mathbb{R}^k$ by fixing a real basis $\{Y_j\}_{j=1}^k\subset\frak{h}$.
   For any $(X,Y)\in(\varphi^{-1}\circ L_{g^{-1}})(gW)\subset V\times B_1$ we obtain 
\[
   \bigl(\psi_g\circ\pi\circ(\varphi^{-1}\circ L_{g^{-1}})^{-1}\bigr)(X,Y)
   =\psi_g\bigl(\pi(g\exp X\exp Y)\bigr)
   =X
\]
from \eqref{eq-b} and \eqref{eq-a}.
   Consequently $\psi_g\circ\pi\circ(\varphi^{-1}\circ L_{g^{-1}})^{-1}:(\varphi^{-1}\circ L_{g^{-1}})(gW)\to\psi_g(U_g)$, $(X,Y)\mapsto X$, is real analytic, and so $\pi:gW\to G/H$ is real analytic.\par
   
   Now, let us define a continuous mapping $\sigma_g:U_g\to G$ by 
\begin{equation}\label{eq-c}\tag{c}
   \mbox{$\sigma_g(x):=L_g\bigl(\sigma(\tau_g^{-1}(x))\bigr)$ for $x\in U_g$ ($g\in G$)}.
\end{equation}
   Our third aim is to prove the following proposition: for each $g\in G$   
\begin{equation}\label{eq-3}\tag*{\textcircled{3}}
   \mbox{$\sigma_g:U_g\to G$ is a real analytic mapping such that $\sigma_g(U_g)\subset gW$ and $\pi\circ\sigma_g=\operatorname{id}$ on $U_g$}.   
\end{equation}
   It is immediate from \eqref{eq-c} and \eqref{eq-b} that $\sigma_g(U_g)=L_g(\sigma(U))\subset L_g(N)\subset gW$.
   For any $x\in U_g=\tau_g(U)$, there exists an $X\in V$ satisfying $x=\tau_g\bigl(\pi(\exp X)\bigr)$, and then it follows from \eqref{eq-c} and \eqref{eq-a} that $\pi\bigl(\sigma_g(x)\bigr)=\pi\bigl(L_g\bigl(\sigma(\pi(\exp X))\bigr)\bigr)=\pi(g\exp X)=x$; hence $\pi\circ\sigma_g=\operatorname{id}$ on $U_g$. 
   Let us demonstrate that $\sigma_g:U_g\to G$ is real analytic. 
   For any $X\in\psi_g(U_g)\subset V$, we deduce 
\[
   \bigl((\varphi^{-1}\circ L_{g^{-1}})\circ\sigma_g\circ\psi_g^{-1}\bigr)(X)
   =\bigl((\varphi^{-1}\circ L_{g^{-1}})\circ\sigma_g\bigr)\bigl(\pi(g\exp X)\bigr)
   =(\varphi^{-1}\circ L_{g^{-1}})(g\exp X)
   =(X,0)
\] 
by \eqref{eq-b}, \eqref{eq-a} and \eqref{eq-c}.
   This implies that $(\varphi^{-1}\circ L_{g^{-1}})\circ\sigma_g\circ\psi_g^{-1}:\psi_g(U_g)\to(\varphi^{-1}\circ L_{g^{-1}})(gW)$, $X\mapsto(X,0)$, is real analytic, so that $\sigma_g:U_g\to G$ is a real analytic mapping.\par 

   Our last aim is to conclude that 
\begin{equation}\label{eq-4}\tag*{\textcircled{4}}
   \mbox{$\mu:G\times G/H\to G/H$, $(g_1,g_2H)\mapsto g_1g_2H$, is a real analytic mapping}.
\end{equation}
   We denote by $f$ the multiplication in $G$, namely $f:G\times G\to G$, $(g_1,g_2)\mapsto g_1g_2$. 
   Let us take any $(g_1,g_2H)\in G\times G/H$.
   Then, $G\times U_{g_2}$ is an open neighborhood of $(g_1,g_2H)\in G\times G/H$. 
   For any $(g,x)\in G\times U_{g_2}$ we assert that   
\[
   \pi\bigl(f(g,\sigma_{g_2}(x))\bigr)
   =\pi\bigl(g\sigma_{g_2}(x)\bigr)
   =\tau_{g}\bigl(\pi(\sigma_{g_2}(x))\bigr)
   =\tau_g(x)
   =\mu(g,x)
\]
because of $\pi\circ\sigma_{g_2}=\operatorname{id}$ on $U_{g_2}$.
   This assures that $\mu:G\times U_{g_2}\to G/H$, $(g,x)\mapsto\mu(g,x)$, is a real analytic mapping, since all the mappings $\pi:G\to G/H$, $f:G\times G\to G$ and $\sigma_{g_2}:U_{g_2}\to G$ are real analytic due to \ref{eq-2}, \ref{eq-3}.\par
   
   Theorem \ref{thm-1.1.2} comes from \ref{eq-1}, \ref{eq-2}, \ref{eq-3} and \ref{eq-4}. 
\end{proof} 

   Lemma \ref{lem-1.1.6}-(3) and the proof of Theorem \ref{thm-1.1.2} lead to  
\begin{corollary}\label{cor-1.1.12}
   The homogeneous space $G/H$ is an $n$-dimensional real analytic manifold which satisfies the second countability axiom, where $n=\dim_\mathbb{R}G-\dim_\mathbb{R}H$.
\end{corollary}

   The following lemma will be needed later (e.g.\ Chapter \ref{ch-9}):
\begin{lemma}\label{lem-1.1.13}
   Equip the homogeneous space $G/H$ with the real analytic structure $\mathcal{S}$ in Theorem {\rm \ref{thm-1.1.2}}, and define a mapping $F:\frak{g}\to T_{\pi(e)}(G/H)$ by 
\[
   \mbox{$F(X):=(d\pi)_eX_e$ for $X\in\frak{g}$}.
\] 
   Then, $F$ is a surjective, linear mapping and $\frak{h}$ coincides with the kernel $\ker(F)$.
\end{lemma}
\begin{proof}
   It is clear that $F:\frak{g}\to T_{\pi(e)}(G/H)$, $X\mapsto(d\pi)_eX_e$, is a linear mapping. 
   In the proof of Theorem \ref{thm-1.1.2} we have shown that $\pi:G\to G/H$ is real analytic. 
   By the arguments we conclude that the linear mapping $(d\pi)_e:T_eG\to T_{\pi(e)}(G/H)$, $v\mapsto(d\pi)_ev$, is surjective. 
   Accordingly $F$ is surjective linear because the mapping $\frak{g}\ni X\mapsto X_e\in T_eG$ is a linear isomorphism.\par
   
   Now, let us prove that $\frak{h}=\ker(F)$. 
   For any $Z\in\frak{h}$ and $f\in\mathcal{C}^\infty(G/H)$ one obtains $\bigl((d\pi)_eZ_e\bigr)f=d/dt\big|_{t=0}f\bigl(\pi(\exp tZ)\bigr)=d/dt\big|_{t=0}f\bigl(\pi(e)\bigr)=0$; and hence $F(Z)=(d\pi)_eZ_e=0$. 
   This gives rise to 
\[
  \frak{h}\subset\ker(F).
\] 
   Furthermore, since $F:\frak{g}\to T_{\pi(e)}(G/H)$ is surjective linear, Corollary \ref{cor-1.1.12} implies that 
\[
   \dim_\mathbb{R}\ker(F)=\dim_\mathbb{R}\frak{g}-\dim_\mathbb{R}T_{\pi(e)}(G/H)=\dim_\mathbb{R}\frak{h},
\]
so that $\frak{h}=\ker(F)$ holds.   
\end{proof}

\subsection{Supplementation}\label{subsec-1.1.4}
   Let us confirm the proposition in Remark \ref{rem-1.1.4}-(i) for the sake of completeness.
   
   Suppose $G/H$ to admit another real analytic structure $\mathcal{S}'$ so that $\mu:G\times G/H\to G/H$, $(g_1,g_2H)\mapsto g_1g_2H$, is real analytic, where the topology for $G/H$ is the quotient one relative to $\pi$. 
   We denote by $M$ the real analytic manifold $G/H$ having the atlas $\mathcal{S}'$. 
   Since the topology for $G/H$ is the same as that for $M$, the identity mapping $\operatorname{id}:G/H\to M$ is a $G$-equivariant homeomorphism. 
   For any $p\in G/H$, Theorem \ref{thm-1.1.2} allows us to have an open neighborhood $U_p$ of $p\in G/H$ and a real analytic mapping $\sigma_p:U_p\to G$ such that $\pi\bigl(\sigma_p(x)\bigr)=x$ for all $x\in U_p$.
   Then, $\operatorname{id}=\pi\circ\sigma_p$ on $U_p$, which implies that $\operatorname{id}:U_p\to M$ is real analytic because $\pi:G\to M$ is real analytic (cf.\ Remark \ref{rem-1.1.3}). 
   Consequently $\operatorname{id}:G/H\to M$ is $G$-equivariant homeomorphic and real analytic.
   Since $\operatorname{id}:G/H\to M$ is real analytic, one can consider the differential of $\operatorname{id}$ at each point, which is a real linear isomorphism. 
   Therefore the inverse mapping theorem assures that the inverse mapping $\operatorname{id}:M\to G/H$ is also real analytic. 
   For this reason  $G/H=(G/H,\mathcal{S})$ is $G$-equivariant real analytic diffeomorphic to $M=(G/H,\mathcal{S}')$ via $\operatorname{id}$.

\section{Complex case}\label{sec-1.2}

   Let $G$, $H$ be the same Lie groups as in Theorem \ref{thm-1.1.2}. 
   Suppose further that (s1) $G$ is a complex Lie group and (s2) $H$ is a complex Lie subgroup of $G$. 
   Then, one can show 

\begin{theorem}\label{thm-1.2.1}
   There exists a holomorphic structure $\mathcal{S}=\{(U_\alpha,\psi_\alpha)\}_{\alpha\in A}$ on the homogeneous space $G/H$ so that 
\begin{enumerate}
\item[{\rm (1)}]
   $\pi:G\to G/H$, $a\mapsto aH$, is a surjective, open, holomorphic mapping,
\item[{\rm (2)}]
   $\mu:G\times G/H\to G/H$, $(a_1,a_2H)\mapsto a_1a_2H$, is a holomorphic mapping.
\end{enumerate} 
   Moreover, for each $\alpha\in A$ there exists a holomorphic mapping $\sigma_\alpha:U_\alpha\to G$ such that $\pi\bigl(\sigma_\alpha(z)\bigr)=z$ for all $z\in U_\alpha$.
\end{theorem}   
\begin{proof}
   We get the conclusion by substituting the words ``complex'' for the words ``real'' in Subsections \ref{subsec-1.1.2} and \ref{subsec-1.1.3}.  
\end{proof}

\begin{remark}[Uniqueness]\label{rem-1.2.2}
   Suppose $G/H$ to admit another holomorphic structure $\mathcal{S}'$ so that $\mu:G\times G/H\to G/H$, $(a_1,a_2H)\mapsto a_1a_2H$, is holomorphic, where the topology for $G/H$ is the quotient one relative to $\pi$. 
   Then, $(G/H_,\mathcal{S})$ is $G$-equivariant biholomorphic to $(G/H,\mathcal{S}')$ via the identity mapping of $G/H$.
\end{remark}   

\begin{remark}\label{rem-1.2.3}
   For a complex Lie group $G$ satisfying the second countability axiom and a closed complex Lie subgroup $H$ of $G$, we always consider the complex homogeneous space $G/H$ to be a homogeneous complex manifold of $G$ with respect to the invariant complex structure $J$ induced by the $\mathcal{S}$ in Theorem \ref{thm-1.2.1}.    
\end{remark}   

\section{Principal fiber bundles and homogeneous spaces}\label{sec-1.3}
   Let $G$ be a Lie group which satisfies the second countability axiom, and $H$ a closed subgroup of $G$. 
   Denote by $\pi$ the projection of $G$ onto $G/H$, and consider the homogeneous space $G/H$ as a real analytic manifold having the atlas $\mathcal{S}=\{(U_\alpha,\psi_\alpha)\}_{\alpha\in A}$ in Theorem \ref{thm-1.1.2}. 
   In addition, let $\sigma_\alpha:U_\alpha\to G$ be the real analytic mapping in Theorem \ref{thm-1.1.2} ($\alpha\in A$).
   In this setting, we will show that this $(G,\pi,G/H)$ is a principal fiber bundle.\par  

   For an $\alpha\in A$, it follows that $\pi^{-1}(U_\alpha)$ is an open subset of $G$. 
   Then we set 
\begin{equation}\label{eq-1.3.1}
   \mbox{$\zeta_\alpha(g):=\bigl(\sigma_\alpha(\pi(g))\bigr)^{-1}g$ for $g\in\pi^{-1}(U_\alpha)$}.
\end{equation} 
   Since $\pi(g)\in U_\alpha$ and $\pi\bigl(\sigma_\alpha(x)\bigr)=x$ for all $x\in U_\alpha$, it is natural that $\sigma_\alpha(\pi(g))H=\pi\bigl(\sigma_\alpha(\pi(g))\bigr)=\pi(g)=gH$, and therefore $\zeta_\alpha(g)=\bigl(\sigma_\alpha(\pi(g))\bigr)^{-1}g$ belongs to $H$. 
   Moreover, the following lemma holds:

\begin{lemma}\label{lem-1.3.2}
   For each $\alpha\in A$, the $\zeta_\alpha:\pi^{-1}(U_\alpha)\to H$, $g\mapsto\zeta_\alpha(g)$, is a real analytic mapping such that 
\begin{enumerate}
\item[{\rm (1)}]
   $\zeta_\alpha(gh)=\zeta_\alpha(g)h$ for all $(g,h)\in \pi^{-1}(U_\alpha)\times H$,
\item[{\rm (2)}]
   $\zeta_\alpha\bigl(\sigma_\alpha(x)\bigr)=e$ for all $x\in U_\alpha$.
\end{enumerate} 
\end{lemma} 

   Now, we see that 
\begin{enumerate}
\item   
   $H$ acts real analytically and freely on $G$ to the right, $G\times H\ni(g,h)\mapsto R_h(g)=gh\in G$, 
\item
   $R_{h_1h_2}(g)=R_{h_2}\bigl(R_{h_1}(g)\bigr)$ for all $h_1,h_2\in H$ and $g\in G$, 
\item
   $\pi:G\to G/H$, $g\mapsto gH$, is a surjective, real analytic mapping, 
\item
   for given $g_1,g_1\in G$, $\pi(g_1)=\pi(g_2)$ if and only if there exists an $h\in H$ such that $g_2=R_h(g_1)$,
\item
   $\{U_\alpha:\alpha\in A\}$ is an open covering of $G/H$. 
\end{enumerate}
   Furthermore, \eqref{eq-1.3.1} and Lemma \ref{lem-1.3.2} enable one to see that for each $\alpha\in A$,
\begin{enumerate}
\item[6.]
   $\theta_\alpha:\pi^{-1}(U_\alpha)\to U_\alpha\times H$, $g\mapsto\bigl(\pi(g),\zeta_\alpha(g)\bigr)$, is a real analytic diffeomorphism, $\theta_\alpha^{-1}(y,h)=\sigma_\alpha(y)h$ for all $(y,h)\in U_\alpha\times H$,
\item[7.]
   $\zeta_\alpha(gh)=R_h\bigl(\zeta_\alpha(g)\bigr)$ for all $(g,h)\in\pi^{-1}(U_\alpha)\times H$.   
\end{enumerate}
   These lead to  
\begin{proposition}\label{prop-1.3.3}
   $(G,\pi,G/H)$ is a real analytic, principal fiber bundle over $G/H$ with group $H$.
\end{proposition}   

\begin{remark}\label{rem-1.3.4}
   The principal fiber bundle $(G,\pi,G/H)$ in Proposition \ref{prop-1.3.3} is able to be holomorphic, provided that the $G$ is a complex Lie group and the $H$ is a complex Lie subgroup of $G$.
\end{remark}

\chapter{Homogeneous vector bundles over homogeneous spaces}\label{ch-2}
   In this chapter we deal with homogeneous vector bundles over homogeneous spaces.
   The setting of Chapter \ref{ch-2} is as follows:
\begin{itemize}
\item 
   $G$ is a Lie group which satisfies the second countability axiom, 
\item 
   $H$ is a closed subgroup of $G$, 
\item
   $\pi$ is the projection of $G$ onto the left quotient space $G/H$,
\item
   $\mathcal{S}=\{(U_\alpha,\psi_\alpha)\}_{\alpha\in A}$ is the real analytic structure on $G/H$ given in Theorem \ref{thm-1.1.2}, 
\item
   $\sigma_\alpha:U_\alpha\to G$ is the real analytic mapping in Theorem \ref{thm-1.1.2} ($\alpha\in A$).
\end{itemize}
   The topology for $G/H$ is the quotient topology relative to $\pi:g\mapsto gH$, and the homogeneous space $G/H$ is an $n$-dimensional real analytic manifold having the atlas $\mathcal{S}$.

\section{Definition of homogeneous vector bundle}\label{sec-2.1}
   First of all, we are going to recall the definition of homogeneous vector bundle. 
   Let ${\sf V}$ be a finite-dimensional real vector space, and let $\rho:H\to GL({\sf V})$, $h\mapsto\rho(h)$, be a continuous (group) homomorphism,\footnote{Remark.\ Since $\rho:H\to GL({\sf V})$, $h\mapsto\rho(h)$, is a continuous homomorphism, it is a real analytic mapping. 
      e.g.\ \begin{CJK}{UTF8}{min}定理 2.3.7 in 杉浦\end{CJK} \cite[p.48]{Su2}.} where we fix a real basis $\{{\sf e}_i\}_{i=1}^m$ of ${\sf V}$ and identify ${\sf V}$ with $\mathbb{R}^m$, and we consider the vector space ${\sf V}$ and the general linear group $GL({\sf V})$ as a real analytic manifold and a Lie group, respectively.
   For two elements $(g_1,{\sf v}_1),(g_2,{\sf v}_2)\in G\times {\sf V}$ we say that $(g_1,{\sf v}_1)$ is {\it equivalent} to $(g_2,{\sf v}_2)$, if there exists an $h\in H$ satisfying 
\begin{equation}\label{eq-2.1.1}
\begin{array}{ll}
   g_2=g_1h, & {\sf v}_2=\rho(h)^{-1}({\sf v}_1). 
\end{array}
\end{equation} 
   This gives rise to an equivalence relation on $G\times{\sf V}$.
   We denote by $[(g,{\sf v})]$ the equivalence class of an element $(g,{\sf v})\in G\times{\sf V}$, put $G\times_\rho{\sf V}:=\big\{[(g,{\sf v})] : (g,{\sf v})\in G\times{\sf V}\big\}$, and define two surjective mappings $\varpi:G\times{\sf V}\to G\times_\rho{\sf V}$ and $\Pr:G\times_\rho{\sf V}\to G/H$ by 
\begin{equation}\label{eq-2.1.2}
\begin{array}{ll}
   \mbox{$\varpi(g,{\sf v}):=[(g,{\sf v})]$ for $(g,{\sf v})\in G\times{\sf V}$}, &
   \mbox{$\Pr\bigl([(g,{\sf v})]\bigr):=\pi(g)$ for $[(g,{\sf v})]\in G\times_\rho{\sf V}$},
\end{array}
\end{equation} 
respectively.
   Provide $G\times_\rho{\sf V}$ with the quotient topology relative to this $\varpi$. 
    
\begin{definition}[{cf.\ Bott \cite[p.207]{Bo}}]\label{def-2.1.3}
   In the setting above, $G\times_\rho{\sf V}=(G\times_\rho{\sf V},\Pr,G/H)$ is called a {\it homogeneous vector bundle over $G/H$ associated with $\rho$}\index{homogeneous vector bundle@homogeneous vector bundle\dotfill} or called an {\it associated fiber bundle of the principal fiber bundle $(G,\pi,G/H)$ with fiber ${\sf V}$}.
\end{definition} 

\begin{center}
\unitlength=1mm
\begin{picture}(43,20)
\put(1,17){$G\times{\sf V}$}
\put(17,19){$\varpi$}
\put(14,18){$\vector(1,0){15}$}
\put(32,17){$G\times_\rho{\sf V}$}
\put(34,1){$G/H$}
\put(32,10){$\Pr$}
\put(37,15){$\vector(0,-1){10}$}
\end{picture}
\end{center}

   We will confirm that 
\begin{enumerate}
\item
   $G\times_\rho{\sf V}$ is a real analytic manifold, cf.\ Section \ref{sec-2.2},
\item
   $(G\times_\rho{\sf V},\Pr,G/H)$ is a fiber bundle with fiber ${\sf V}$ and group $\rho(H)$ ($\subset GL({\sf V})$), cf.\ Section \ref{sec-2.3}.
\end{enumerate}
   In addition, we will study the real vector space $\Gamma^r(G\times_\rho{\sf V})$ of differentiable cross-sections of the bundle $(G\times_\rho{\sf V},\Pr,G/H)$, cf.\ Section \ref{sec-2.4}.

\section{Real analytic structures on homogeneous vector bundles}\label{sec-2.2}
   Our purpose of this section is to define a real analytic structure $\mathscr{S}=\{(\Pr^{-1}(U_\alpha),\varphi_\alpha)\}_{\alpha\in A}$ on $G\times_\rho{\sf V}$. 
   Here, $G\times_\rho{\sf V}=(G\times_\rho{\sf V},\Pr,G/H)$ is a homogeneous vector bundle over $G/H$ associated with $\rho$.    
   In order to accomplish the purpose, we first define a real analytic mapping $\Phi_\alpha$ needed later.
   By use of $\zeta_\alpha:\pi^{-1}(U_\alpha)\to H$ in \eqref{eq-1.3.1}, we define a real analytic mapping $\Phi_\alpha:\pi^{-1}(U_\alpha)\times{\sf V}\to U_\alpha\times{\sf V}$ as follows:
\begin{equation}\label{eq-2.2.1}
   \mbox{$\Phi_\alpha(g,{\sf v}):=\bigl(\pi(g),\rho(\zeta_\alpha(g)){\sf v}\bigr)$ for $(g,{\sf v})\in\pi^{-1}(U_\alpha)\times{\sf V}$}
\end{equation} 
($\alpha\in A$).

\subsection{Topological properties of $G\times_\rho{\sf V}=(G\times_\rho{\sf V},\Pr,G/H)$}\label{subsec-2.2.1}
   We want to deduce that $G\times_\rho{\sf V}$ is a topological manifold (see Proposition \ref{prop-2.2.9}).
   Recalling that the topologies for $G\times_\rho{\sf V}$ and $G/H$ are the quotient topologies relative to $\varpi:G\times{\sf V}\to G\times_\rho{\sf V}$, $(g,{\sf v})\mapsto[(g,{\sf v})]$ and $\pi:G\to G/H$, $g\mapsto gH$, respectively, we first prove
\begin{lemma}\label{lem-2.2.2}
   $\Pr:G\times_\rho{\sf V}\to G/H$, $[(g,{\sf v})]\mapsto\pi(g)$, is a surjective, continuous mapping.
\end{lemma}
\begin{proof}
   We only verify that $\Pr:G\times_\rho{\sf V}\to G/H$ is continuous. 
   Let $U$ be any open subset of $G/H$. 
   By a direct computation we obtain 
\begin{equation}\label{eq-2.2.3}
   \varpi^{-1}\bigl(\Pr{}^{-1}(U)\bigr)=\pi^{-1}(U)\times{\sf V};
\end{equation} 
besides, $\pi^{-1}(U)\times{\sf V}$ is an open subset of $G\times{\sf V}$. 
   Hence $\varpi^{-1}\bigl(\Pr{}^{-1}(U)\bigr)$ is open in $G\times{\sf V}$, and so $\Pr^{-1}(U)$ is open in $G\times_\rho{\sf V}$.
\end{proof}

\begin{corollary}\label{cor-2.2.4}
   $\{\Pr^{-1}(U_\alpha):\alpha\in A\}$ is an open covering of $G\times_\rho{\sf V}$.
\end{corollary}
\begin{proof}
   Since $\{U_\alpha:\alpha\in A\}$ is an open covering of $G/H$, Lemma \ref{lem-2.2.2} enables us to get the conclusion.
\end{proof}

   For an $\alpha\in A$, it follows from $\Pr\bigl([(g,{\sf v})]\bigr)=\pi(g)$ that $\pi(g)\in U_\alpha$ for all $[(g,{\sf v})]\in\Pr^{-1}(U_\alpha)$. 
   Then one can set
\begin{equation}\label{eq-2.2.5}
   \mbox{$\phi_\alpha\bigl([(g,{\sf v})]\bigr):=\bigl(\pi(g),\rho(\zeta_\alpha(g)){\sf v}\bigr)$ for $[(g,{\sf v})]\in\Pr^{-1}(U_\alpha)$}.
\end{equation} 
   Here it is necessary to confirm that this \eqref{eq-2.2.5} is well-defined. 
   Let us confirm that. 
   Suppose that $(g_1,{\sf v}_1)$ is equivalent to $(g_2,{\sf v}_2)$ with $[(g_1,{\sf v}_1)]\in\Pr^{-1}(U_\alpha)$. 
   By the definition \eqref{eq-2.1.1} of equivalence relation, there exists an $h\in H$ such that $g_2=g_1h$, ${\sf v}_2=\rho(h)^{-1}({\sf v}_1)$.
   In this case it turns out that $\pi(g_2)=\pi(g_1h)=\pi(g_1)$. 
   Moreover, Lemma \ref{lem-1.3.2}-(1) implies 
\[
   \rho(\zeta_\alpha(g_2)){\sf v}_2
   =\rho(\zeta_\alpha(g_1h)){\sf v}_2
   =\rho(\zeta_\alpha(g_1)h){\sf v}_2
   =\rho(\zeta_\alpha(g_1))\bigl(\rho(h){\sf v}_2\bigr)
   =\rho(\zeta_\alpha(g_1)){\sf v}_1
\]
since $\rho:H\to GL({\sf V})$ is a homomorphism. 
   Therefore \eqref{eq-2.2.5} is well-defined. 
   Now, let us prove 
\begin{proposition}\label{prop-2.2.6}
   For each $\alpha\in A$, the mapping $\phi_\alpha:\Pr^{-1}(U_\alpha)\to U_\alpha\times{\sf V}$, $[(g,{\sf v})]\mapsto\bigl(\pi(g),\rho(\zeta_\alpha(g)){\sf v}\bigr)$, is a homeomorphism. 
   In addition, $\phi_\alpha^{-1}(x,{\sf v})=[(\sigma_\alpha(x),{\sf v})]$ for all $(x,{\sf v})\in U_\alpha\times{\sf V}$. 
   cf.\ \eqref{eq-1.3.1}.
\end{proposition}   
\begin{proof}
   Let $\phi_\alpha'(x,{\sf v}):=[(\sigma_\alpha(x),{\sf v})]$ for $(x,{\sf v})\in U_\alpha\times{\sf V}$.\par

   (Bijective). 
   For any $[(g,{\sf v})]\in\Pr^{-1}(U_\alpha)$ we have
\[
\begin{split}
   \phi_\alpha'\bigl(\phi_\alpha\bigl([(g,{\sf v})]\bigr)\bigr)
  &=\phi_\alpha'\bigl(\pi(g),\rho(\zeta_\alpha(g)){\sf v}\bigr)
   =\big[\bigl(\sigma_\alpha(\pi(g)),\rho(\zeta_\alpha(g)){\sf v}\bigr)\big]
   \stackrel{\eqref{eq-1.3.1}}{=}\big[\bigl(\sigma_\alpha(\pi(g)),\rho\bigl((\sigma_\alpha(\pi(g)))^{-1}g\bigr){\sf v}\bigr)\big]\\
  &\stackrel{\eqref{eq-2.1.1}}{=}[(g,{\sf v})].
\end{split} 
\]
   For any $(x,{\sf v})\in U_\alpha\times{\sf V}$ we deduce
\[
   \phi_\alpha\bigl(\phi_\alpha'(x,{\sf v})\bigr)
   =\phi_\alpha\bigl([(\sigma_\alpha(x),{\sf v})]\bigr)
   =\bigl(\pi(\sigma_\alpha(x)),\rho(\zeta_\alpha(\sigma_\alpha(x))){\sf v}\bigr)
   =(x,{\sf v})
\]
by $\pi\circ\sigma_\alpha=\operatorname{id}$ on $U_\alpha$ and Lemma \ref{lem-1.3.2}-(2). 
   Therefore $\phi_\alpha$ is bijective and $\phi_\alpha^{-1}=\phi_\alpha'$.\par

   (Continuous 1). 
   Let us show that $\phi_\alpha:\Pr^{-1}(U_\alpha)\to U_\alpha\times{\sf V}$, $[(g,{\sf v})]\mapsto\bigl(\pi(g),\rho(\zeta_\alpha(g)){\sf v}\bigr)$, is continuous.
   It follows from \eqref{eq-2.2.3} that $\varpi\bigl(\pi^{-1}(U_\alpha)\times{\sf V}\bigr)\subset\Pr^{-1}(U_\alpha)$, and it follows from \eqref{eq-2.1.2}, \eqref{eq-2.2.1} and \eqref{eq-2.2.5} that  
\[
   \Phi_\alpha=\phi_\alpha\circ\varpi
\]
on $\pi^{-1}(U_\alpha)\times{\sf V}$.
   Consequently $\phi_\alpha:\Pr^{-1}(U_\alpha)\to U_\alpha\times{\sf V}$ is continuous, because $\Phi_\alpha:\pi^{-1}(U_\alpha)\times{\sf V}\to U_\alpha\times{\sf V}$ is continuous and the topology for $G\times_\rho{\sf V}$ is the quotient topology relative to $\varpi$.\par
   
   (Continuous 2). 
   The inverse mapping $\phi_\alpha^{-1}:U_\alpha\times{\sf V}\to\Pr^{-1}(U_\alpha)$, $(x,{\sf v})\mapsto[(\sigma_\alpha(x),{\sf v})]$, is also continuous because it is the composition of two continuous mappings $U_\alpha\times{\sf V}\ni(x,{\sf v})\mapsto(\sigma_\alpha(x),{\sf v})\in\pi^{-1}(U_\alpha)\times{\sf V}$ and $G\times{\sf V}\ni(g,{\sf v})\mapsto\varpi(g,{\sf v})\in G\times_\rho{\sf V}$.   
\end{proof}

   Proposition \ref{prop-2.2.6} leads to
\begin{corollary}\label{cor-2.2.7}
   For each $\alpha\in A$, $\psi_\alpha(U_\alpha)\times{\sf V}$ is an open subset of $\mathbb{R}^{n+m}$, the mapping $\varphi_\alpha:=(\psi_\alpha\times\operatorname{id}_{\sf V})\circ\phi_\alpha:\Pr^{-1}(U_\alpha)\to\psi_\alpha(U_\alpha)\times{\sf V}$, $[(g,{\sf v})]\mapsto\bigl(\psi_\alpha(\pi(g)),\rho(\zeta_\alpha(g)){\sf v}\bigr)$, is homeomorphic, and $\varphi_\alpha^{-1}(X,{\sf v})=\big[\bigl(\sigma_\alpha(\psi_\alpha^{-1}(X)),{\sf v}\bigr)\big]$ for all $(X,{\sf v})\in\psi_\alpha(U_\alpha)\times{\sf V}$. 
   Here $m=\dim_\mathbb{R}{\sf V}$.
\end{corollary}

\begin{center}
\unitlength=1mm
\begin{picture}(82,31)
\put(1,26){$G\times_\rho{\sf V}\supset\Pr^{-1}(U_\alpha)$}
\put(39,30){$\phi_\alpha$}
\put(33,29){$\vector(1,0){15}$}
\put(48,26){$\vector(-1,0){15}$}
\put(39,22.5){$\phi_\alpha^{-1}$}
\put(51,26){$U_\alpha\times{\sf V}\subset G/H\times{\sf V}$}
\put(51,1){$\psi_\alpha(U_\alpha)\times{\sf V}\subset\mathbb{R}^{n+m}$}
\put(61,14){$\psi_\alpha\times\operatorname{id}_{\sf V}$}
\put(60,24){$\vector(0,-1){19}$}
\put(49,5){$\vector(-1,1){18}$}
\put(41,13){$\varphi_\alpha^{-1}$}
\put(26,23){$\vector(1,-1){19}$}
\put(31,13){$\varphi_\alpha$}
\end{picture}
\end{center}

\begin{corollary}\label{cor-2.2.8}
   $G\times_\rho{\sf V}$ is a Hausdorff space.
\end{corollary}
\begin{proof}
   For $[(g_1,{\sf v}_1)],[(g_2,{\sf v}_2)]\in G\times_\rho{\sf V}$ we suppose that $[(g_1,{\sf v}_1)]\neq[(g_2,{\sf v}_2)]$. 
   Let us investigate two cases $\pi(g_1)\neq\pi(g_2)$ and $\pi(g_1)=\pi(g_2)$, individually.\par

   $\bullet$ In case of $\pi(g_1)\neq\pi(g_2)$, there exist open neighborhoods $U_1$ of $\pi(g_1)$ and $U_2$ of $\pi(g_2)\in G/H$ such that $U_1\cap U_2=\emptyset$ because $G/H$ is a Hausdorff space.
   Then, Lemma \ref{lem-2.2.2} implies that $\Pr^{-1}(U_1)$, $\Pr^{-1}(U_2)$ are open neighborhoods of $[(g_1,{\sf v}_1)]$, $[(g_2,{\sf v}_2)]\in G\times_\rho{\sf V}$ and $\Pr^{-1}(U_1)\cap \Pr^{-1}(U_2)=\emptyset$.\par
   
   $\bullet$ In case of $\pi(g_1)=\pi(g_2)$, Corollary \ref{cor-2.2.4} and $\Pr\bigl([(g_1,{\sf v}_1)]\bigr)=\pi(g_1)=\pi(g_2)=\Pr\bigl([(g_2,{\sf v}_2)]\bigr)$ assure the existence of an element $\alpha\in A$ satisfying $[(g_1,{\sf v}_1)],[(g_2,{\sf v}_2)]\in\Pr^{-1}(U_\alpha)$.
   So, one has $[(g_1,{\sf v}_1)],[(g_2,{\sf v}_2)]\in\Pr^{-1}(U_\alpha)$ and $[(g_1,{\sf v}_1)]\neq[(g_2,{\sf v}_2)]$.
   Then there exist open subsets $W_1, W_2\subset\Pr^{-1}(U_\alpha)$ such that $[(g_1,{\sf v}_1)]\in W_1$, $[(g_2,{\sf v}_2)]\in W_2$ and 
\[
   W_1\cap W_2=\emptyset,
\]
because Corollary \ref{cor-2.2.7} implies that $\Pr^{-1}(U_\alpha)$ is a Hausdorff space.
   Remark here that both $W_1$ and $W_2$ are open subsets of $G\times_\rho{\sf V}$, since $\Pr^{-1}(U_\alpha)$ is open in $G\times_\rho{\sf V}$.
\end{proof}

   Corollaries \ref{cor-2.2.8}, \ref{cor-2.2.4} and \ref{cor-2.2.7} allow us to assert
\begin{proposition}\label{prop-2.2.9}
   The following three items hold$:$
\begin{enumerate}
\item[{\rm (1)}]
   $G\times_\rho{\sf V}$ is an $(n+m)$-dimensional topological manifold, where $m=\dim_\mathbb{R}{\sf V}$, 
\item[{\rm (2)}]
   each pair $(\Pr^{-1}(U_\alpha),\varphi_\alpha)$ is a coordinate neighborhood of $G\times_\rho{\sf V}$ $(\alpha\in A)$,
\item[{\rm (3)}]
   $\mathscr{S}:=\{(\Pr^{-1}(U_\alpha),\varphi_\alpha)\}_{\alpha\in A}$ is an atlas of $G\times_\rho{\sf V}$.
\end{enumerate}
   Here we refer to Corollary {\rm \ref{cor-2.2.7}} for $\varphi_\alpha$.
\end{proposition}

   We end this subsection with proving 
\begin{proposition}\label{prop-2.2.10}
\begin{enumerate}
\item[]
\item[{\rm (i)}]
   $\Pr:G\times_\rho{\sf V}\to G/H$, $[(g,{\sf v})]\mapsto\pi(g)$, is a surjective, open, continuous mapping.
\item[{\rm (ii)}]
   $G\times_\rho{\sf V}$ satisfies the second countability axiom.
\item[{\rm (iii)}] 
   $\varpi:G\times{\sf V}\to G\times_\rho{\sf V}$, $(g,{\sf v})\mapsto[(g,{\sf v})]$, is a surjective, open, continuous mapping.
\item[{\rm (iv)}]
   $\nu:G\times(G\times_\rho{\sf V})\to G\times_\rho{\sf V}$, $(g_1,[(g_2,{\sf v})])\mapsto[(g_1g_2,{\sf v})]$, is a continuous mapping.
\end{enumerate}
\end{proposition}   
\begin{proof}
   (i). 
   By Lemma \ref{lem-2.2.2} it suffices to confirm that $\Pr:G\times_\rho{\sf V}\to G/H$ is an open mapping.
   For any open subset $W\subset G\times_\rho{\sf V}$ we see that $\varpi^{-1}(W)$ is open in $G\times{\sf V}$. 
   Considering an open mapping $\operatorname{proj}:G\times{\sf V}\to G$, $(g,{\sf v})\mapsto g$, we conclude that $\operatorname{proj}\bigl(\varpi^{-1}(W)\bigr)$ is open in $G$, so that
\[
   \mbox{$\operatorname{proj}\bigl(\varpi^{-1}(W)\bigr)H=\bigcup_{h\in H}R_h\bigl(\operatorname{proj}(\varpi^{-1}(W))\bigr)$ is an open subset of $G$}.
\] 
   A direct computation yields $\operatorname{proj}\bigl(\varpi^{-1}(W)\bigr)H=\pi^{-1}\bigl(\Pr(W)\bigr)$. 
   Accordingly $\pi^{-1}\bigl(\Pr(W)\bigr)\subset G$ is open, and hence $\Pr(W)$ is an open subset of $G/H$.\par

   (ii). 
   Since $G\times_\rho{\sf V}$ is a topological manifold, it is enough to show that $G\times_\rho{\sf V}$ is a Lindel\"{o}f space.
   Let $\{W_\lambda:\lambda\in\Lambda\}$ be an arbitrary open covering of $G\times_\rho{\sf V}$. 
   Needless to say, $\{\varpi^{-1}(W_\lambda):\lambda\in\Lambda\}$ is an open covering of $G\times{\sf V}$. 
   Since both $G$ and ${\sf V}$ satisfy the second countability axiom, the product $G\times{\sf V}$ also satisfies the same axiom. 
   Therefore one can find a countable subset $\{\varpi^{-1}(W_k):k\in\mathbb{N}\}$ of $\{\varpi^{-1}(W_\lambda):\lambda\in\Lambda\}$ so that
\[
   G\times{\sf V}=\bigcup_{k\in\mathbb{N}}\varpi^{-1}(W_k).
\]
   Then $G\times_\rho{\sf V}=\varpi(G\times{\sf V})=\varpi\bigl(\bigcup_{k\in\mathbb{N}}\varpi^{-1}(W_k)\bigr)\subset\bigcup_{k\in\mathbb{N}}W_k$, which implies that $G\times_\rho{\sf V}$ is a Lindel\"{o}f space.\par

   (iii). 
   We only verify that $\varpi:G\times{\sf V}\to G\times_\rho{\sf V}$, $(g,{\sf v})\mapsto[(g,{\sf v})]$,  is an open mapping. 
   For any non-empty open subset $Q\subset G\times{\sf V}$, we are going to show that $\varpi^{-1}\bigl(\varpi(Q)\bigr)$ is an open subset of $G\times{\sf V}$. 
   For any $(g,{\sf v})\in\varpi^{-1}\bigl(\varpi(Q)\bigr)$, it follows that $\varpi(g,{\sf v})\in\varpi(Q)$, so that there exists a $(a,{\sf w})\in Q$ satisfying 
\[
   \varpi(g,{\sf v})=\varpi(a,{\sf w}).
\]
   Since $(a,{\sf w})\in Q$ and $Q\subset G\times{\sf V}$ is open, there exist open subsets $O\subset G$ and $B\subset{\sf V}$ such that $(a,{\sf w})\in O\times B\subset Q$. 
   Moreover, a direct computation, together with \eqref{eq-2.1.1}, yields
\[
   \varpi^{-1}\bigl(\varpi(O\times B)\bigr)=\bigcup_{h\in H}\bigl(R_h(O)\times\rho(h)^{-1}(B)\bigr),
\]
which implies that $\varpi^{-1}\bigl(\varpi(O\times B)\bigr)$ is an open subset of $G\times{\sf V}$ because each $R_h(O)\times\rho(h)^{-1}(B)$ is open in $G\times{\sf V}$. 
   Consequently, $\varpi^{-1}\bigl(\varpi(O\times B)\bigr)$ is an open neighborhood of $(g,{\sf v})\in G\times{\sf V}$ and $\varpi^{-1}\bigl(\varpi(O\times B)\bigr)\subset\varpi^{-1}\bigl(\varpi(Q)\bigr)$. 
   This implies that $\varpi^{-1}\bigl(\varpi(Q)\bigr)$ is an open subset of $G\times{\sf V}$.\par
   
   (iv). 
   Take any $(a_1,[(a_2,{\sf w})])\in G\times(G\times_\rho{\sf V})$ and any open neighborhood $W$ of $\nu(a_1,[(a_2,{\sf w})])=[(a_1a_2,{\sf w})]\in G\times_\rho{\sf V}$.
   Since $\varpi^{-1}(W)$ is an open neighborhood of $(a_1a_2,{\sf w})\in G\times{\sf V}$ and since $\hat{\nu}:G\times(G\times{\sf V})\to G\times{\sf V}$, $(g_1,(g_2,{\sf v}))\mapsto(g_1g_2,{\sf v})$, is a continuous mapping, there exist open subsets $O_1,O_2\subset G$ and $B\subset{\sf V}$ such that $a_1\in O_1$, $a_2\in O_2$, ${\sf w}\in B$ and 
\[
   \hat{\nu}\bigl(O_1,(O_2\times B)\bigr)\subset\varpi^{-1}(W).
\] 
   Then $\varpi(O_2\times B)$ is an open subset of $G\times_\rho{\sf V}$ due to (iii), and so $O_1\times\varpi(O_2\times B)$ is an open neighborhood of $(a_1,[(a_2,{\sf w})])\in G\times(G\times_\rho{\sf V})$; besides, 
\[
   \nu\bigl(O_1\times\varpi(O_2\times B)\bigr)
   \subset\varpi\bigl(\hat{\nu}\bigl(O_1,(O_2\times B)\bigr)\bigr)
   \subset W.
\] 
   Accordingly $\nu:G\times(G\times_\rho{\sf V})\to G\times_\rho{\sf V}$, $(g_1,[(g_2,{\sf v})])\mapsto[(g_1g_2,{\sf v})]$, is continuous.  
\end{proof}

\subsection{A real analytic structure on  $G\times_\rho{\sf V}$}\label{subsec-2.2.2}
   We have shown that $\mathscr{S}=\{(\Pr^{-1}(U_\alpha),\varphi_\alpha)\}_{\alpha\in A}$ is an atlas of $G\times_\rho{\sf V}$ (cf.\ Proposition \ref{prop-2.2.9}).
   In this subsection we aim to confirm that the $\mathscr{S}$ defines a real analytic structure in $G\times_\rho{\sf V}$.

\begin{lemma}\label{lem-2.2.11}
   Suppose that $\Pr^{-1}(U_\alpha)\cap\Pr^{-1}(U_\beta)\neq\emptyset$ $(\alpha,\beta\in A)$. 
   Then, 
\begin{enumerate}
\item[{\rm (1)}]
   $\varphi_\beta\bigl(\Pr^{-1}(U_\alpha)\cap\Pr^{-1}(U_\beta)\bigr)=\psi_\beta(U_\alpha\cap U_\beta)\times{\sf V}$.
\item[{\rm (2)}]
   $(\varphi_\alpha\circ\varphi_\beta^{-1})(X,{\sf v})=\left(\psi_\alpha\bigl(\psi_\beta^{-1}(X)\bigr),\rho\left(\bigl(\sigma_\alpha(\psi_\beta^{-1}(X))\bigr)^{-1}\sigma_\beta\bigl(\psi_\beta^{-1}(X)\bigr)\right){\sf v}\right)$ for all $(X,{\sf v})\in\psi_\beta(U_\alpha\cap U_\beta)\times{\sf V}$.
\item[{\rm (3)}] 
   $\varphi_\alpha\circ\varphi_\beta^{-1}$ is a real analytic diffeomorphism of $\varphi_\beta\bigl(\Pr^{-1}(U_\alpha)\cap\Pr^{-1}(U_\beta)\bigr)$ onto $\varphi_\alpha\bigl(\Pr^{-1}(U_\alpha)\cap\Pr^{-1}(U_\beta)\bigr)$.   
\end{enumerate}   
\end{lemma}
\begin{proof}
   (1). 
   First, let us demonstrate that $\varphi_\beta\bigl(\Pr^{-1}(U_\alpha)\cap\Pr^{-1}(U_\beta)\bigr)\subset\psi_\beta(U_\alpha\cap U_\beta)\times{\sf V}$. 
   For any $(X_1,{\sf v}_1)\in\varphi_\beta\bigl(\Pr^{-1}(U_\alpha)\cap\Pr^{-1}(U_\beta)\bigr)=\varphi_\beta\bigl(\Pr^{-1}(U_\alpha\cap U_\beta)\bigr)$, there exists a $[(a_1,{\sf w}_1)]\in\Pr^{-1}(U_\alpha\cap U_\beta)$ satisfying $(X_1,{\sf v}_1)=\varphi_\beta\bigl([(a_1,{\sf w}_1)]\bigr)$. 
   From $\varphi_\beta=(\psi_\beta\times\operatorname{id}_{\sf V})\circ\phi_\beta$ we obtain 
\[
   (X_1,{\sf v}_1)
   =\varphi_\beta\bigl([(a_1,{\sf w}_1)]\bigr)
   =\bigl(\psi_\beta(\pi(a_1)),\rho(\zeta_\beta(a_1)){\sf w}_1\bigr).
\]
   This, combined with $\pi(a_1)=\Pr\bigl([(a_1,{\sf w}_1)]\bigr)\in U_\alpha\cap U_\beta$, implies that $(X_1,{\sf v}_1)\in\psi_\beta(U_\alpha\cap U_\beta)\times{\sf V}$; hence $\varphi_\beta\bigl(\Pr^{-1}(U_\alpha)\cap\Pr^{-1}(U_\beta)\bigr)\subset\psi_\beta(U_\alpha\cap U_\beta)\times{\sf V}$.
   Next, let us show that the converse inclusion also holds.
   For any $(X_2,{\sf v}_2)\in\psi_\beta(U_\alpha\cap U_\beta)\times{\sf V}$, it follows from Corollary \ref{cor-2.2.7} that 
\[
   (X_2,{\sf v}_2)
   =\varphi_\beta\bigl(\varphi_\beta^{-1}(X_2,{\sf v}_2)\bigr)
   =\varphi_\beta\bigl(\big[\bigl(\sigma_\beta(\psi_\beta^{-1}(X_2)),{\sf v}_2\bigr)\big]\bigr).
\] 
   Furthermore $\Pr\bigl(\big[\bigl(\sigma_\beta(\psi_\beta^{-1}(X_2)),{\sf v}_2\bigr)\big]\bigr)=\pi\bigl(\sigma_\beta(\psi_\beta^{-1}(X_2))\bigr)=\psi_\beta^{-1}(X_2)\in U_\alpha\cap U_\beta$, and so $\big[\bigl(\sigma_\beta(\psi_\beta^{-1}(X_2)),{\sf v}_2\bigr)\big]\in\Pr^{-1}(U_\alpha\cap U_\beta)$. 
   Consequently we have $(X_2,{\sf v}_2)\in\varphi_\beta\bigl(\Pr^{-1}(U_\alpha\cap U_\beta)\bigr)$. 
   Hence, $\psi_\beta(U_\alpha\cap U_\beta)\times{\sf V}\subset\varphi_\beta\bigl(\Pr^{-1}(U_\alpha)\cap\Pr^{-1}(U_\beta)\bigr)$.\par

   (2). 
   For any $(X,{\sf v})\in\psi_\beta(U_\alpha\cap U_\beta)\times{\sf V}$, Corollary \ref{cor-2.2.7} enables us to have
\[
\begin{split}
   (\varphi_\alpha\circ\varphi_\beta^{-1})(X,{\sf v})
  &=\varphi_\alpha\bigl(\big[\bigl(\sigma_\beta(\psi_\beta^{-1}(X)),{\sf v}\bigr)\big]\bigr)
   =\bigl(\psi_\alpha(\pi(\sigma_\beta(\psi_\beta^{-1}(X)))),\rho\bigl(\zeta_\alpha(\sigma_\beta(\psi_\beta^{-1}(X)))\bigr){\sf v}\bigr)\\
  &\stackrel{\eqref{eq-1.3.1}}{=}\bigl(\psi_\alpha(\pi(\sigma_\beta(\psi_\beta^{-1}(X)))),\rho\bigl(\bigl(\sigma_\alpha(\pi(\sigma_\beta(\psi_\beta^{-1}(X))))\bigr)^{-1}\sigma_\beta(\psi_\beta^{-1}(X))\bigr){\sf v}\bigr)\\
  &=\bigl(\psi_\alpha(\psi_\beta^{-1}(X)),\rho\bigl( \bigl(\sigma_\alpha(\psi_\beta^{-1}(X))\bigr)^{-1}\sigma_\beta(\psi_\beta^{-1}(X))\bigr){\sf v}\bigr)
\end{split} 
\]
since $\psi_\beta^{-1}(X)\in U_\alpha\cap U_\beta$ and $\pi\circ\sigma_\beta=\operatorname{id}$ on $U_\beta$.\par

   (3) comes from (1) and (2).
\end{proof}

   By Proposition \ref{prop-2.2.9} and Lemma \ref{lem-2.2.11} we conclude 
\begin{theorem}\label{thm-2.2.12}
   The atlas $\mathscr{S}=\{(\Pr^{-1}(U_\alpha),\varphi_\alpha)\}_{\alpha\in A}$ in Proposition {\rm \ref{prop-2.2.9}} defines a real analytic structure in $G\times_\rho{\sf V}$.
\end{theorem} 

   Theorem \ref{thm-2.2.12} and Proposition \ref{prop-2.2.10}-(ii) lead to
\begin{corollary}\label{cor-2.2.13}
   $G\times_\rho{\sf V}$ is an $(n+m)$-dimensional real analytic manifold which satisfies the second countability axiom, where $m=\dim_\mathbb{R}{\sf V}$.
\end{corollary}   

   We end this section with confirming
\begin{proposition}\label{prop-2.2.14}
   $\varpi:G\times{\sf V}\to G\times_\rho{\sf V}$, $(g,{\sf v})\mapsto[(g,{\sf v})]$, is a surjective, open, real analytic mapping.
   Here $G\times_\rho{\sf V}$ is a real analytic manifold having the atlas $\mathscr{S}=\{(\Pr^{-1}(U_\alpha),\varphi_\alpha)\}_{\alpha\in A}$ in Proposition {\rm \ref{prop-2.2.9}}.
\end{proposition} 
\begin{proof}
   We only demonstrate that $\varpi:G\times{\sf V}\to G\times_\rho{\sf V}$ is real analytic (cf.\ Proposition \ref{prop-2.2.10}-(iii)).
   For any $\alpha\in A$ and $(g,{\sf v})\in\pi^{-1}(U_\alpha)\times{\sf V}$, one has $\varpi(\pi^{-1}(U_\alpha)\times{\sf V})\subset\Pr^{-1}(U_\alpha)$, and 
\[
   (\varphi_\alpha\circ\varpi)(g,{\sf v})
   =\varphi_\alpha\bigl([(g,{\sf v})]\bigr)
   =\bigl(\psi_\alpha(\pi(g)),\rho(\zeta_\alpha(g)){\sf v}\bigr)
\] 
due to Corollary \ref{cor-2.2.7}.
   This mapping $G\times{\sf V}\supset\pi^{-1}(U_\alpha)\times{\sf V}\ni(g,{\sf v})\mapsto\bigl(\psi_\alpha(\pi(g)),\rho(\zeta_\alpha(g)){\sf v}\bigr)\in\varphi_\alpha(\Pr^{-1}(U_\alpha))\subset\mathbb{R}^{n+m}$ is real analytic; hence $\varpi:\pi^{-1}(U_\alpha)\times{\sf V}\to G\times_\rho{\sf V}$ is real analytic for each $\alpha\in A$. 
   Therefore $\varpi:G\times{\sf V}\to G\times_\rho{\sf V}$ is a real analytic mapping, since $G=\bigcup_{\alpha\in A}\pi^{-1}(U_\alpha)$. 
\end{proof}

\section{Fiber bundles}\label{sec-2.3}
   For a homogeneous vector bundle $G\times_\rho{\sf V}=(G\times_\rho{\sf V},\Pr,G/H)$ over $G/H$ associated with $\rho:H\to GL({\sf V})$, we will demonstrate that it is a fiber bundle with fiber ${\sf V}$ and group $\rho(H)$ ($\subset GL({\sf V})$), cf.\ Theorem \ref{thm-2.3.5}. 

\subsection{Vector space structures on fibers}\label{subsec-2.3.1}   
   First, let us show that for every $x_0\in G/H$, the fiber $\Pr^{-1}(\{x_0\})$ can be a real vector space which is real linear isomorphic to ${\sf V}$.
   Since $G/H=\bigcup_{\alpha\in A}U_\alpha$, there exists an $\alpha\in A$ such that $x_0\in U_\alpha$.
   Then, $f_\alpha:{\sf V}\to\Pr^{-1}(\{x_0\})$, ${\sf v}\mapsto[(\sigma_\alpha(x_0),{\sf v})]$, is a homeomorphism due to Proposition \ref{prop-2.2.6}.\footnote{$f_\alpha^{-1}\bigl([(g,{\sf v})]\bigr)=\rho(\zeta_\alpha(g)){\sf v}$ for all $[(g,{\sf v})]\in\Pr^{-1}(\{x_0\})$.} 
   Setting 
\begin{equation}\label{eq-2.3.1}
\left\{
\begin{array}{l}
   [(g_1,{\sf v}_1)]+[(g_2,{\sf v}_2)]:=\big[\bigl(\sigma_\alpha(x_0),\rho\bigl((\sigma_\alpha(x_0))^{-1}g_1\bigr){\sf v}_1+\rho\bigl((\sigma_\alpha(x_0))^{-1}g_2\bigr){\sf v}_2\bigr)\big],\\
   \lambda[(g,{\sf v})]:=\big[\bigl(\sigma_\alpha(x_0),\lambda\rho\bigl((\sigma_\alpha(x_0))^{-1}g\bigr){\sf v}\bigr)\big]
\end{array}\right.
\end{equation}
for $[(g_1,{\sf v}_1)], [(g_2,{\sf v}_2)], [(g,{\sf v})]\in\Pr^{-1}(\{x_0\})$ and $\lambda\in\mathbb{R}$, one can assert that the fiber $\Pr^{-1}(\{x_0\})$ is a real vector space, where we note that 
\[
   \mbox{$[(g,{\sf v})]=\big[\bigl(\sigma_\alpha(x_0),\rho\bigl((\sigma_\alpha(x_0))^{-1}g\bigr){\sf v}\bigr)\big]$ for all $[(g,{\sf v})]\in\Pr^{-1}(\{x_0\})$}.
\]
   With respect to this vector space $\Pr^{-1}(\{x_0\})$, the homeomorphism $f_\alpha:{\sf V}\to\Pr^{-1}(\{x_0\})$, ${\sf v}\mapsto[(\sigma_\alpha(x_0),{\sf v})]$, is real linear isomorphic, and hence the vector space structure on $\Pr^{-1}(\{x_0\})$ is independent of the choice of $\alpha\in A$ satisfying $x_0\in U_\alpha$.
   
\begin{remark}\label{rem-2.3.2}
\begin{enumerate}
\item[] 
\item[(1)]
   For each $x_0\in G/H$ we see that  
\[
   \Pr\bigl(\lambda_1[(g_1,{\sf v}_1)]+\lambda_2[(g_2,{\sf v}_2)]\bigr)=x_0
\] 
for all $[(g_1,{\sf v}_1)], [(g_2,{\sf v}_2)]\in\Pr^{-1}(\{x_0\})$ and $\lambda_1, \lambda_2\in\mathbb{R}$ because of \eqref{eq-2.3.1}.  
\item[(2)]
   Hereafter, for each $x_0\in G/H$ we regard the fiber $\Pr^{-1}(\{x_0\})$ as an $m$-dimensional real vector space by means of \eqref{eq-2.3.1}.
   Here $m=\dim_\mathbb{R}{\sf V}$.
\end{enumerate}   
\end{remark} 

\subsection{Transition functions}\label{subsec-2.3.2} 
    Let us set
\begin{equation}\label{eq-2.3.3}
   \mbox{$g_{\alpha\beta}(y):=(\sigma_\alpha(y))^{-1}\sigma_\beta(y)$ for $y\in U_\alpha\cap U_\beta$}
\end{equation}
whenever $U_\alpha\cap U_\beta\neq\emptyset$ ($\alpha,\beta\in A$). 
   Since $\pi\bigl(\sigma_\alpha(y)\bigr)=y=\pi\bigl(\sigma_\beta(y)\bigr)$ we have $g_{\alpha\beta}(y)\in H$, and therefore $g_{\alpha\beta}:U_\alpha\cap U_\beta\to H$ is a real analytic mapping, where we remark that $H$ is a regular submanifold of $G$.
   It is easy to prove 
\begin{proposition}\label{prop-2.3.4}
   For the real analytic mapping $g_{\alpha\beta}:U_\alpha\cap U_\beta\to H$, $y\mapsto(\sigma_\alpha(y))^{-1}\sigma_\beta(y)$, the following two items hold$:$
\begin{enumerate}
\item[{\rm (a)}]   
   $g_{\alpha\alpha}(x)=e$ for all $x\in U_\alpha$.
\item[{\rm (b)}]  
   $g_{\alpha\beta}(z)g_{\beta\gamma}(z)g_{\gamma\alpha}(z)=e$ for all $z\in U_\alpha\cap U_\beta\cap U_\gamma$.
\end{enumerate}
\end{proposition} 

\subsection{Proof of Theorem \ref{thm-2.3.5}}\label{subsec-2.3.3}   
   Now, we are in a position to demonstrate
 
\begin{theorem}\label{thm-2.3.5}
   Provide $G\times_\rho{\sf V}$ with the real analytic structure $\mathscr{S}=\{(\Pr^{-1}(U_\alpha),\varphi_\alpha)\}_{\alpha\in A}$ in Proposition {\rm \ref{prop-2.2.9}}. 
   Then, the following two items hold$:$
\begin{enumerate}
\item[{\rm (1)}]
   $\Pr:G\times_\rho{\sf V}\to G/H$, $[(g,{\sf v})]\mapsto\pi(g)=gH$, is a surjective, open, real analytic mapping.
\item[{\rm (2)}] 
   For each $\alpha\in A$, the mapping $\phi_\alpha^{-1}:U_\alpha\times{\sf V}\to\Pr^{-1}(U_\alpha)$, $(x,{\sf v})\mapsto[(\sigma_\alpha(x),{\sf v})]$, is a real analytic diffeomorphism$;$ besides, $\phi_\alpha\bigl([(g,{\sf v})]\bigr)=\bigl(\pi(g),\rho(\zeta_\alpha(g)){\sf v}\bigr)$ for all $[(g,{\sf v})]\in\Pr^{-1}(U_\alpha)$.
   cf.\ \eqref{eq-1.3.1}.
\end{enumerate} 
   Moreover, for each $x_0\in U_\alpha$ it follows that 
\begin{enumerate}
\item[{\rm (3)}]
   $\Pr\bigl(\phi_\alpha^{-1}(x_0,{\sf v})\bigr)=x_0$ for all ${\sf v}\in{\sf V}$, 
\item[{\rm (4)}] 
   the mapping ${\sf V}\ni{\sf v}\mapsto\phi_\alpha^{-1}(x_0,{\sf v})\in\Pr^{-1}(\{x_0\})$ is a real linear isomorphism.
\end{enumerate}
   In addition, suppose that $U_\alpha\cap U_\beta\neq\emptyset$ $(\alpha,\beta\in A)$.
   Then, 
\begin{enumerate}
\item[{\rm (5)}] 
   $(\phi_\alpha\circ\phi_\beta^{-1})(y,{\sf v})=\bigl(y,\rho(g_{\alpha\beta}(y)){\sf v}\bigr)$ for all $(y,{\sf v})\in(U_\alpha\cap U_\beta)\times{\sf V}$.   
\end{enumerate} 
   Here we refer to {\rm \eqref{eq-2.3.3}} for $g_{\alpha\beta}$. 
\end{theorem}
\begin{proof}
   (1). 
   By Proposition \ref{prop-2.2.10}-(i), it suffices to confirm that $\Pr:G\times_\rho{\sf V}\to G/H$ is real analytic. 
   For any $\alpha\in A$ and $(X,{\sf v})\in\varphi_\alpha\bigl(\Pr^{-1}(U_\alpha)\bigr)=\psi_\alpha(U_\alpha)\times{\sf V}$, Corollary \ref{cor-2.2.7}, \eqref{eq-2.1.2} and $\pi\circ\sigma_\alpha=\operatorname{id}$ on $U_\alpha$ imply that 
\[
   (\psi_\alpha\circ\Pr\circ\varphi_\alpha^{-1})(X,{\sf v})
   =(\psi_\alpha\circ\Pr)(\big[\bigl(\sigma_\alpha(\psi_\alpha^{-1}(X)),{\sf v}\bigr)\big])
   =\psi_\alpha\bigl(\pi(\sigma_\alpha(\psi_\alpha^{-1}(X)))\bigr)
   =X,
\]
so that $\Pr:\Pr^{-1}(U_\alpha)\to G/H$ is real analytic.
   Thus $\Pr:G\times_\rho{\sf V}\to G/H$ is real analytic.\par
   
   (2). 
   For any $(X,{\sf v})\in(\psi_\alpha\times\operatorname{id}_{\sf V})(U_\alpha\times{\sf V})=\psi_\alpha(U_\alpha)\times{\sf V}$, we deduce  
\[
   \bigl(\varphi_\alpha\circ\phi_\alpha^{-1}\circ(\psi_\alpha\times\operatorname{id}_{\sf V})^{-1}\bigr)(X,{\sf v})
   =(\varphi_\alpha\circ\varphi_\alpha^{-1})(X,{\sf v})
   =(X,{\sf v})
\]
by Corollary \ref{cor-2.2.7}. 
   Hence $\phi_\alpha^{-1}:U_\alpha\times{\sf V}\to\Pr^{-1}(U_\alpha)$ is a real analytic diffeomorphism.\par
   
   (3). 
   By a direct computation we have $\Pr\bigl(\phi_\alpha^{-1}(x_0,{\sf v})\bigr)=\Pr\bigl([(\sigma_\alpha(x_0),{\sf v})]\bigr)=\pi\bigl(\sigma_\alpha(x_0)\bigr)=x_0$.\par
   
   (4). 
   Recall that $f_\alpha:{\sf V}\to\Pr^{-1}(\{x_0\})$, ${\sf v}\mapsto[(\sigma_\alpha(x_0),{\sf v})]$, is real linear isomorphic, cf.\ Subsection \ref{subsec-2.3.1}.\par
   
   (5). 
   For any $(y,{\sf v})\in(U_\alpha\cap U_\beta)\times{\sf V}$, Proposition \ref{prop-2.2.6} enables us to obtain 
\[
\begin{split}
   (\phi_\alpha\circ\phi_\beta^{-1})(y,{\sf v})
  &=\phi_\alpha\bigl([(\sigma_\beta(y),{\sf v})]\bigr)
   =\bigl(\pi(\sigma_\beta(y)),\rho(\zeta_\alpha(\sigma_\beta(y))){\sf v}\bigr)\\
  &\stackrel{\eqref{eq-1.3.1}}{=}\bigl(\pi(\sigma_\beta(y)),\rho\bigl(\bigl(\sigma_\alpha(\pi(\sigma_\beta(y)))\bigr)^{-1}\sigma_\beta(y)\bigr){\sf v}\bigr)
  =\bigl(y,\rho\bigl((\sigma_\alpha(y))^{-1}\sigma_\beta(y)\bigr){\sf v}\bigr) \quad\mbox{($\because$ $\pi(\sigma_\beta(y))=y$)}\\
  &\stackrel{\eqref{eq-2.3.3}}{=}\bigl(y,\rho\bigl(g_{\alpha\beta}(y)\bigr){\sf v}\bigr).
\end{split}
\]   
\end{proof}

   We end Section \ref{sec-2.3} with proving 
\begin{proposition}\label{prop-2.3.6}
   $\nu:G\times(G\times_\rho{\sf V})\to G\times_\rho{\sf V}$, $(g_1,[(g_2,{\sf v})])\mapsto[(g_1g_2,{\sf v})]$, is a real analytic mapping.
   Here $G\times_\rho{\sf V}$ is a real analytic manifold having the atlas $\mathscr{S}=\{(\Pr^{-1}(U_\alpha),\varphi_\alpha)\}_{\alpha\in A}$ in Proposition {\rm \ref{prop-2.2.9}}.
\end{proposition} 
\begin{proof}
   Fix any $\alpha\in A$.
   On the one hand; since $\hat{\nu}:G\times(G\times{\sf V})\to G\times{\sf V}$, $(g_1,(g_2,{\sf v}))\mapsto(g_1g_2,{\sf v})$, is a real analytic mapping, it follows from Proposition \ref{prop-2.2.14} and Theorem \ref{thm-2.3.5}-(2) that $\varpi\circ\hat{\nu}\circ\bigl(\operatorname{id}_G\times(\sigma_\alpha\times\operatorname{id}_{\sf V})\bigr)\circ(\operatorname{id}_G\times\phi_\alpha)$ is real analytic on $G\times\Pr^{-1}(U_\alpha)$. 
   On the other hand; for any $(g_1,[(g_2,{\sf v})])\in G\times\Pr^{-1}(U_\alpha)$ one has
\[ 
\begin{split}
  &\bigl(\varpi\circ\hat{\nu}\circ\bigl(\operatorname{id}_G\times(\sigma_\alpha\times\operatorname{id}_{\sf V})\bigr)\circ(\operatorname{id}_G\times\phi_\alpha)\bigr)
    (g_1,[(g_2,{\sf v})])
   =\bigl(\varpi\circ\hat{\nu}\circ\bigl(\operatorname{id}_G\times(\sigma_\alpha\times\operatorname{id}_{\sf V})\bigr)\bigr)
    (g_1,\bigl(\pi(g_2),\rho(\zeta_\alpha(g_2)){\sf v}\bigr))\\
  &=\bigl(\varpi\circ\hat{\nu}\bigr)
    (g_1,\bigl(\sigma_\alpha(\pi(g_2)),\rho(\zeta_\alpha(g_2)){\sf v}\bigr))
   =\varpi
    \bigl(g_1\sigma_\alpha(\pi(g_2)),\rho(\zeta_\alpha(g_2)){\sf v}\bigr) 
   \stackrel{\eqref{eq-1.3.1}}{=}\varpi
    \bigl(g_1\sigma_\alpha(\pi(g_2)),\rho(\bigl(\sigma_\alpha(\pi(g_2))\bigr)^{-1}g_2){\sf v}\bigr)\\
   &\stackrel{\eqref{eq-2.1.1}}{=}\varpi
    \bigl(g_1g_2,{\sf v}\bigr)
    =\nu(g_1,[(g_2,{\sf v})]).
\end{split}
\]   
   Hence $\nu:G\times\Pr^{-1}(U_\alpha)\to G\times_\rho{\sf V}$ is real analytic, and so $\nu:G\times(G\times_\rho{\sf V})\to G\times_\rho{\sf V}$ is real analytic.
\end{proof}

\section{Vector spaces of cross-sections of homogeneous vector bundles}\label{sec-2.4}
   For a homogeneous vector bundle $G\times_\rho{\sf V}=(G\times_\rho{\sf V},\Pr,G/H)$ we provide $G\times_\rho{\sf V}$ with the real analytic structure $\mathscr{S}=\{(\Pr^{-1}(U_\alpha),\varphi_\alpha)\}_{\alpha\in A}$ in Proposition \ref{prop-2.2.9} (cf.\ Theorem \ref{thm-2.2.12}).
   In this section we study the real vector space $\Gamma^r(G\times_\rho{\sf V})$ of differentiable cross-sections of the bundle $G\times_\rho{\sf V}$.\par

   For an $r\in\mathbb{N}\cup\{0,\infty,\omega\}$, let us set
\begin{equation}\label{eq-2.4.1}
   \Gamma^r(G\times_\rho{\sf V})
   :=\left\{\begin{array}{@{}c|c@{}}
   \gamma:G/H\to G\times_\rho{\sf V} 
   & \begin{array}{@{}l@{}} \mbox{(1) $\gamma$ is of class $C^r$},\\ \mbox{(2) $\Pr\bigl(\gamma(x)\bigr)=x$ for all $x\in G/H$}\end{array}
   \end{array}\right\}.
\end{equation}
   Then, for any $\gamma_i\in\Gamma^r(G\times_\rho{\sf V})$ and $x\in G/H$ ($i=1,2$), it follows from \eqref{eq-2.4.1}-(2) that $\gamma_i(x)\in\Pr^{-1}(\{x\})$. 
   Accordingly, since $\Pr^{-1}(\{x\})$ is a real vector space, one can define an element $\gamma_1+\gamma_2\in\Gamma^r(G\times_\rho{\sf V})$ as follows:
\begin{equation}\label{eq-2.4.2}
   \mbox{$(\gamma_1+\gamma_2)(x):=\gamma_1(x)+\gamma_2(x)$ for $x\in G/H$}.
\end{equation}
   Similarly, $\lambda\gamma\in\Gamma^r(G\times_\rho{\sf V})$ as follows: 
\begin{equation}\label{eq-2.4.3}
   \mbox{$(\lambda\gamma)(x):=\lambda\gamma(x)$ for $x\in G/H$},
\end{equation}
where $\gamma\in\Gamma^r(G\times_\rho{\sf V})$ and $\lambda\in\mathbb{R}$ (cf.\ Remark \ref{rem-2.3.2}). 
   Hereafter, we regard $\Gamma^r(G\times_\rho{\sf V})$ as a real vector space by means of \eqref{eq-2.4.2} and \eqref{eq-2.4.3}.\par

   The main purpose of this section is to verify Theorem \ref{thm-2.4.15} which implies that the real vector space $\Gamma^r(G\times_\rho{\sf V})$ is isomorphic to 
\begin{equation}\label{eq-2.4.4}
   \mathcal{V}^r(G\times_\rho{\sf V})
   :=\left\{\begin{array}{@{}c|c@{}}
   \xi:G\to{\sf V} 
   & \begin{array}{@{}l@{}} \mbox{(i) $\xi$ is of class $C^r$},\\ \mbox{(ii) $\xi(gh)=\rho(h)^{-1}\bigl(\xi(g)\bigr)$ for all $(g,h)\in G\times H$}\end{array}
   \end{array}\right\}.
\end{equation}
   Here $\mathcal{V}^r(G\times_\rho{\sf V})$ is a real vector space with respect to the following addition of vectors and scalar multiplication:
\begin{equation}\label{eq-2.4.5}
   \mbox{$(\xi_1+\xi_2)(g):=\xi_1(g)+\xi_2(g)$, $(\lambda\xi)(g):=\lambda\xi(g)$ for $g\in G$},
\end{equation}
where $\xi_1,\xi_1,\xi\in\mathcal{V}^r(G\times_\rho{\sf V})$ and $\lambda\in\mathbb{R}$.

\subsection{A linear mapping $F_1:\mathcal{V}^r(G\times_\rho{\sf V})\to\Gamma^r(G\times_\rho{\sf V})$}\label{subsec-2.4.1}
   Let $\xi$ be an arbitrary element of $\mathcal{V}^r(G\times_\rho{\sf V})$.
   From it we are going to construct an element of $\Gamma^r(G\times_\rho{\sf V})$. 
   Set $\Gamma_\xi$ as 
\begin{equation}\label{eq-2.4.6}
   \mbox{$\Gamma_\xi\bigl(\pi(g)\bigr):=[(g,\xi(g))]$ for $\pi(g)\in G/H$}.
\end{equation} 
   This \eqref{eq-2.4.6} is well-defined by virtue of \eqref{eq-2.4.4}-(ii) and \eqref{eq-2.1.1}. 
   So, $\Gamma_\xi$ is a mapping of $G/H$ into $G\times_\rho{\sf V}$. 
   Moreover,
\begin{lemma}\label{lem-2.4.7}
   In the setting of {\rm \eqref{eq-2.4.1}}, {\rm \eqref{eq-2.4.4}} and {\rm \eqref{eq-2.4.6}}$;$ 
\begin{enumerate}
\item[{\rm (1)}]
   $\Gamma_\xi$ belongs to $\Gamma^r(G\times_\rho{\sf V})$ for each $\xi\in\mathcal{V}^r(G\times_\rho{\sf V});$
\item[{\rm (2)}]
   $F_1:\mathcal{V}^r(G\times_\rho{\sf V})\to\Gamma^r(G\times_\rho{\sf V})$, $\xi\mapsto\Gamma_\xi$, is a real linear mapping.  
\end{enumerate}    
\end{lemma} 
\begin{proof}
   (1).
   Let $\xi$ be an arbitrary element of $\mathcal{V}^r(G\times_\rho{\sf V})$.
   It is easy to see that $\Pr\bigl(\Gamma_\xi(\pi(g))\bigr)\stackrel{\eqref{eq-2.4.6}}{=}\Pr\bigl([(g,\xi(g))]\bigr)\stackrel{\eqref{eq-2.1.2}}{=}\pi(g)$ for all $\pi(g)\in G/H$. 
   Thus, the rest of proof is to confirm that $\Gamma_\xi:G/H\to G\times_\rho{\sf V}$ is a differentiable mapping of class $C^r$.
   For each $\alpha\in A$, it follows from $\Pr\circ\Gamma_\xi=\operatorname{id}_{G/H}$ that $\Gamma_\xi(U_\alpha)\subset\Pr^{-1}(U_\alpha)$. 
   Then for any $X\in\psi_\alpha(U_\alpha)$ we have 
\[
\begin{split}
   (\varphi_\alpha\circ\Gamma_\xi\circ\psi_\alpha^{-1})(X)
  &=(\varphi_\alpha\circ\Gamma_\xi)\bigl(\psi_\alpha^{-1}(X)\bigr)
   =(\varphi_\alpha\circ\Gamma_\xi)\bigl(\pi\bigl(\sigma_\alpha(\psi_\alpha^{-1}(X))\bigr)\bigr)\quad \mbox{($\because$ $\pi\circ\sigma_\alpha=\operatorname{id}$ on $U_\alpha$)}\\
  &\stackrel{\eqref{eq-2.4.6}}{=}\varphi_\alpha\bigl(\big[\bigl(\sigma_\alpha(\psi_\alpha^{-1}(X)),\xi(\sigma_\alpha(\psi_\alpha^{-1}(X)))\bigr)\big]\bigr)\\
  &=\bigl(\psi_\alpha(\pi(\sigma_\alpha(\psi_\alpha^{-1}(X)))),\rho(\zeta_\alpha(\sigma_\alpha(\psi_\alpha^{-1}(X))))\xi(\sigma_\alpha(\psi_\alpha^{-1}(X)))\bigr)
  =\bigl(X,\xi(\sigma_\alpha(\psi_\alpha^{-1}(X)))\bigr)
\end{split} 
\] 
by Corollary \ref{cor-2.2.7} and Lemma \ref{lem-1.3.2}-(2).
   This mapping $\mathbb{R}^n\supset\psi_\alpha(U_\alpha)\ni X\mapsto\bigl(X,\xi(\sigma_\alpha(\psi_\alpha^{-1}(X)))\bigr)\in\varphi_\alpha\bigl(\Pr^{-1}(U_\alpha)\bigr)\subset\mathbb{R}^{n+m}$ is of class $C^r$ because both $\sigma_\alpha$ and $\psi_\alpha^{-1}$ are of class $C^\omega$ and $\xi$ is of class $C^r$. 
   Consequently $\Gamma_\xi:U_\alpha\to G\times_\rho{\sf V}$ is of class $C^r$ for each $\alpha\in A$, and hence $\Gamma_\xi:G/H\to G\times_\rho{\sf V}$ is a differentiable mapping of class $C^r$.\par

   (2). 
   For any $\xi_1,\xi_2\in\mathcal{V}^r(G\times_\rho{\sf V})$ and $\pi(g)\in G/H$, a direct computation yields
\[
\begin{split}
   \bigl(F_1(\xi_1+\xi_2)\bigr)(\pi(g))
  &=\Gamma_{\xi_1+\xi_2}(\pi(g))
   \stackrel{\eqref{eq-2.4.6}}{=}\big[\bigl(g,(\xi_1+\xi_2)(g)\bigr)\big]
   \stackrel{\eqref{eq-2.4.5}}{=}\big[\bigl(g,\xi_1(g)+\xi_2(g)\bigr)\big]
   \stackrel{\eqref{eq-2.3.1}}{=}\big[\bigl(g,\xi_1(g)\bigr)\big]+\big[\bigl(g,\xi_2(g)\bigr)\big]\\
  &\stackrel{\eqref{eq-2.4.6}}{=}\Gamma_{\xi_1}(\pi(g))+\Gamma_{\xi_2}(\pi(g))
   \stackrel{\eqref{eq-2.4.2}}{=}\bigl(\Gamma_{\xi_1}+\Gamma_{\xi_2}\bigr)(\pi(g))
   =\bigl(F_1(\xi_1)+F_1(\xi_2)\bigr)(\pi(g)).
\end{split}
\] 
   Thus $F_1(\xi_1+\xi_2)=F_1(\xi_1)+F_1(\xi_2)$ for all $\xi_1,\xi_2\in\mathcal{V}^r(G\times_\rho{\sf V})$.
   Similarly, $F_1(\lambda\xi)=\lambda F_1(\xi)$ for all $(\xi,\lambda)\in\mathcal{V}^r(G\times_\rho{\sf V})\times\mathbb{R}$.    
\end{proof}

\subsection{A mapping $F_2:\Gamma^r(G\times_\rho{\sf V})\to\mathcal{V}^r(G\times_\rho{\sf V})$}\label{subsec-2.4.2}
   Fix an arbitrary $\gamma\in\Gamma^r(G\times_\rho{\sf V})$. 
   From it we will construct an element of $\mathcal{V}^r(G\times_\rho{\sf V})$.
   For any $\alpha\in A$, Theorem \ref{thm-2.3.5}-(2) implies that $\phi_\alpha:\Pr^{-1}(U_\alpha)\to U_\alpha\times{\sf V}$, $[(g_1,{\sf v})]\mapsto\bigl(\pi(g_1),\rho(\zeta_\alpha(g_1)){\sf v}\bigr)$, is real analytic, so that we can define a real analytic mapping $\chi_\alpha:\Pr^{-1}(U_\alpha)\to{\sf V}$ by
\begin{equation}\label{eq-2.4.8}
   \mbox{$\chi_\alpha\bigl([(g_1,{\sf v})]\bigr):=\rho(\zeta_\alpha(g_1)){\sf v}$ for $[(g_1,{\sf v})]\in\Pr^{-1}(U_\alpha)$}.
\end{equation} 
   Furthermore, by use of this $\chi_\alpha$ we define a differentiable mapping $\Xi_{\gamma,\alpha}:\pi^{-1}(U_\alpha)\to{\sf V}$ of class $C^r$ as follows: 
\begin{equation}\label{eq-2.4.9}
   \mbox{$\Xi_{\gamma,\alpha}(g_1):=\rho(\zeta_\alpha(g_1))^{-1}\bigl(\chi_\alpha(\gamma(\pi(g_1)))\bigr)$ for $g_1\in\pi^{-1}(U_\alpha)$}.
\end{equation} 
   Then, Lemma \ref{lem-1.3.2}-(1) assures that 
\begin{equation}\label{eq-2.4.10}
   \mbox{$\Xi_{\gamma,\alpha}(g_1h)=\rho(h)^{-1}\bigl(\Xi_{\gamma,\alpha}(g_1)\bigr)$ for all $(g_1,h)\in\pi^{-1}(U_\alpha)\times H$}.
\end{equation} 
   Now, in terms of $\gamma\bigl(\pi(g_1)\bigr)\in G\times_\rho{\sf V}$, we obtain a $(a,{\sf w})\in G\times{\sf V}$ satisfying $\gamma\bigl(\pi(g_1)\bigr)=[(a,{\sf w})]$. 
   Then \eqref{eq-2.4.1}-(2) yields $\pi(g_1)=\Pr\bigl(\gamma(\pi(g_1))\bigr)=\Pr\bigl([(a,{\sf w})]\bigr)=\pi(a)$. 
   Therefore $g_1^{-1}a\in H$, ${\sf v}:=\rho(g_1^{-1}a){\sf w}\in{\sf V}$ and
\[
    \gamma\bigl(\pi(g_1)\bigr)=[(a,{\sf w})]\stackrel{\eqref{eq-2.1.1}}{=}[(g_1,{\sf v})].  
\] 
   This gives  $\Xi_{\gamma,\alpha}(g_1)\stackrel{\eqref{eq-2.4.9}}{=}\rho(\zeta_\alpha(g_1))^{-1}\bigl(\chi_\alpha(\gamma(\pi(g_1)))\bigr)=\rho(\zeta_\alpha(g_1))^{-1}\bigl(\chi_\alpha([(g_1,{\sf v})])\bigr)\stackrel{\eqref{eq-2.4.8}}{=}\rho(\zeta_\alpha(g_1))^{-1}\bigl(\rho(\zeta_\alpha(g_1)){\sf v}\bigr)={\sf v}$.
   Hence, it follows that 
\begin{equation}\label{eq-2.4.11}
   \mbox{$\gamma\bigl(\pi(g_1)\bigr)=\big[\bigl(g_1,\Xi_{\gamma,\alpha}(g_1)\bigr)\big]$ for all $g_1\in\pi^{-1}(U_\alpha)$}.
\end{equation} 
   In a similar way, we conclude that for any $g_2\in\pi^{-1}(U_\alpha)\cap\pi^{-1}(U_\beta)$, there exists a ${\sf v}_2\in{\sf V}$ such that $\gamma\bigl(\pi(g_2)\bigr)=[(g_2,{\sf v}_2)]$; moreover $\Xi_{\gamma,\alpha}(g_2)={\sf v}_2=\Xi_{\gamma,\beta}(g_2)$.
   Therefore one can define a differentiable mapping $\Xi_\gamma:G\to{\sf V}$ of class $C^r$ by 
\begin{equation}\label{eq-2.4.12}
   \mbox{$\Xi_\gamma(g):=\Xi_{\gamma,\alpha}(g)$ if $g\in\pi^{-1}(U_\alpha)$},
\end{equation} 
where we remark that $G=\bigcup_{\alpha\in A}\pi^{-1}(U_\alpha)$.
   Needless to say, it follows from \eqref{eq-2.4.10}, \eqref{eq-2.4.11} and \eqref{eq-2.4.12} that 
\begin{equation}\label{eq-2.4.13}
\left\{
\begin{array}{l}
   \mbox{$\Xi_\gamma(gh)=\rho(h)^{-1}\bigl(\Xi_\gamma(g)\bigr)$ for all $(g,h)\in G\times H$},\\
   \mbox{$\gamma\bigl(\pi(g)\bigr)=\big[\bigl(g,\Xi_\gamma(g)\bigr)\big]$ for all $g\in G$}.
\end{array}\right. 
\end{equation} 

   Summarizing the statements above, we conclude    
\begin{lemma}\label{lem-2.4.14}
   For each $\gamma\in\Gamma^r(G\times_\rho{\sf V})$, $\Xi_\gamma$ belongs to $\mathcal{V}^r(G\times_\rho{\sf V})$.
   Therefore, one can get a mapping $F_2:\Gamma^r(G\times_\rho{\sf V})\to\mathcal{V}^r(G\times_\rho{\sf V})$ by setting $F_2(\gamma):=\Xi_\gamma$ for $\gamma\in\Gamma^r(G\times_\rho{\sf V})$.
   Moreover, $\gamma\bigl(\pi(g)\bigr)=\big[\bigl(g,(F_2(\gamma))(g)\bigr)\big]$ for all $(\gamma,g)\in\Gamma^r(G\times_\rho{\sf V})\times G$.
   Here we refer to {\rm \eqref{eq-2.4.12}}, {\rm \eqref{eq-2.4.9}} for $\Xi_\gamma$.
\end{lemma} 

   Now, let us verify 
\begin{theorem}\label{thm-2.4.15}
   In the setting of {\rm \eqref{eq-2.4.1}} and {\rm \eqref{eq-2.4.4}}$;$ there exists a real linear isomorphism $F:\Gamma^r(G\times_\rho{\sf V})\to\mathcal{V}^r(G\times_\rho{\sf V})$, $\gamma\mapsto F(\gamma)$, such that $\gamma\bigl(\pi(g)\bigr)=\big[\bigl(g,(F(\gamma))(g)\bigr)\big]$ for all $(\gamma,g)\in\Gamma^r(G\times_\rho{\sf V})\times G$.
   Here $r\in\mathbb{N}\cup\{0,\infty,\omega\}$.
\end{theorem} 
\begin{proof}
   By Lemmas \ref{lem-2.4.7} and \ref{lem-2.4.14} it is enough to confirm that (1) $F_1\circ F_2=\operatorname{id}$ on $\Gamma^r(G\times_\rho{\sf V})$ and (2) $F_2\circ F_1=\operatorname{id}$ on $\mathcal{V}^r(G\times_\rho{\sf V})$.\par
   
   (1).
   Let us take any $\alpha\in A$, $x\in U_\alpha$ and $\delta\in\Gamma^r(G\times_\rho{\sf V})$. 
   Since $\delta(x)\in\Pr^{-1}(U_\alpha)\subset\varpi(\pi^{-1}(U_\alpha)\times{\sf V})$, there exists a $(g_1,{\sf v})\in\pi^{-1}(U_\alpha)\times{\sf V}$ such that $\delta(x)=[(g_1,{\sf v})]$.
   Then we have $x=\pi(g_1)$, $\delta\bigl(\pi(g_1)\bigr)=[(g_1,{\sf v})]$ and
\[
\begin{split}
   \bigl(F_1(F_2(\delta))\bigr)(x)
  &=\Gamma_{F_2(\delta)}(x)
   \stackrel{\eqref{eq-2.4.6}}{=}\big[\bigl(g_1,(F_2(\delta))(g_1)\big)\big]
   =\big[\bigl(g_1,\Xi_\delta(g_1)\big)\big]
   \stackrel{\eqref{eq-2.4.9}, \eqref{eq-2.4.12}}{=}\big[\bigl(g_1,\rho(\zeta_\alpha(g_1))^{-1}\bigl(\chi_\alpha(\delta(\pi(g_1)))\bigr)\bigr)\big]\\
  &=\big[\bigl(g_1,\rho(\zeta_\alpha(g_1))^{-1}\bigl(\chi_\alpha([(g_1,{\sf v})])\bigr)\bigr)\big]
  \stackrel{\eqref{eq-2.4.8}}{=}[(g_1,{\sf v})]=\delta(x).
\end{split} 
\] 
   Therefore we see that $F_1(F_2(\delta))=\delta$ on $U_\alpha$ ($\alpha\in A$). 
   This, together with $G/H=\bigcup_{\alpha\in A}U_\alpha$, assures that $F_1(F_2(\delta))=\delta$ on $G/H$.
   For this reason $F_1\circ F_2=\operatorname{id}$ on $\Gamma^r(G\times_\rho{\sf V})$.\par
   
   (2). 
   By arguments similar to those stated in (1), one can conclude that (2) $F_2\circ F_1=\operatorname{id}$ on $\mathcal{V}^r(G\times_\rho{\sf V})$.
   However, let us confirm (2) for the sake of completeness.
   For any $\alpha\in A$, $g\in\pi^{-1}(U_\alpha)$ and $\eta\in\mathcal{V}^r(G\times_\rho{\sf V})$, we have
\[
\begin{split}
   \bigl(F_2(F_1(\eta))\bigr)(g)
  &=\Xi_{F_1(\eta)}(g)
   \stackrel{\eqref{eq-2.4.9}, \eqref{eq-2.4.12}}{=}\rho(\zeta_\alpha(g))^{-1}\bigl(\chi_\alpha\bigl((F_1(\eta))(\pi(g))\bigr)\bigr)
   =\rho(\zeta_\alpha(g))^{-1}\bigl(\chi_\alpha(\Gamma_\eta(\pi(g)))\bigr)\\
  &\stackrel{\eqref{eq-2.4.6}}{=}\rho(\zeta_\alpha(g))^{-1}\bigl(\chi_\alpha([(g,\eta(g))])\bigr)
   \stackrel{\eqref{eq-2.4.8}}{=}\eta(g).
\end{split} 
\] 
   Hence $F_2(F_1(\eta))=\eta$ on $\pi^{-1}(U_\alpha)$ ($\alpha\in A$), and $F_2(F_1(\eta))=\eta$ on $G=\bigcup_{\alpha\in A}\pi^{-1}(U_\alpha)$. 
   Consequently (2) holds.
\end{proof}

\section{The restrictions of bundles to open subsets and their cross-sections}\label{sec-2.5}
   Fix a homogeneous vector bundle $G\times_\rho{\sf V}=(G\times_\rho{\sf V},\Pr,G/H)$ and provide $G\times_\rho{\sf V}$ with the real analytic structure $\mathscr{S}=\{(\Pr^{-1}(U_\alpha),\varphi_\alpha)\}_{\alpha\in A}$ in Proposition \ref{prop-2.2.9}.
   For a non-empty open subset $U\subset G/H$ we define an open subset $(G\times_\rho{\sf V})_U$ of $G\times_\rho{\sf V}$ by 
\begin{equation}\label{eq-2.5.1}
   (G\times_\rho{\sf V})_U:=\Pr{}^{-1}(U);
\end{equation}
besides, we induce real analytic structures $\mathcal{S}_U$ on $U$ and $\mathscr{S}_U$ on $(G\times_\rho{\sf V})_U$ from $G/H$ and $G\times_\rho{\sf V}$, respectively. 
   Then, Proposition \ref{prop-2.3.4} and Theorem \ref{thm-2.3.5} tell us that $(G\times_\rho{\sf V})_U=\bigl((G\times_\rho{\sf V})_U,\Pr|_{(G\times_\rho{\sf V})_U},U\bigr)$ is a fiber bundle with fiber ${\sf V}$ and group $\rho(H)$ ($\subset GL({\sf V})$), which is called the {\it restriction of the bundle $(G\times_\rho{\sf V},\Pr,G/H)$ to $U$}.\index{restriction of a bundle@restriction of a bundle\dotfill}\par
   
   Let $\Gamma^r(G\times_\rho{\sf V})_U$ be the real vector space of differentiable cross-sections of the bundle $(G\times_\rho{\sf V})_U=\bigl((G\times_\rho{\sf V})_U,\Pr|_{(G\times_\rho{\sf V})_U},U\bigr)$, namely
\begin{equation}\label{eq-2.5.2}
   \Gamma^r(G\times_\rho{\sf V})_U
   =\left\{\begin{array}{@{}c|c@{}}
   \gamma:U\to(G\times_\rho{\sf V})_U 
   & \begin{array}{@{}l@{}} \mbox{(1) $\gamma$ is of class $C^r$},\\ \mbox{(2) $\Pr\bigl(\gamma(x)\bigr)=x$ for all $x\in U$}\end{array}
   \end{array}\right\}.
\end{equation}
   This $\Gamma^r(G\times_\rho{\sf V})_U$ corresponds to the following real vector space: 
\begin{equation}\label{eq-2.5.3}
   \mathcal{V}^r(G\times_\rho{\sf V})_U
   :=\left\{\begin{array}{@{}c|c@{}}
   \xi:\pi^{-1}(U)\to{\sf V} 
   & \begin{array}{@{}l@{}} \mbox{(i) $\xi$ is of class $C^r$},\\ \mbox{(ii) $\xi(gh)=\rho(h)^{-1}\bigl(\xi(g)\bigr)$ for all $(g,h)\in \pi^{-1}(U)\times H$}\end{array}
   \end{array}\right\}.
\end{equation}

\begin{proposition}\label{prop-2.5.4}
   In the setting of {\rm \eqref{eq-2.5.1}}, {\rm \eqref{eq-2.5.2}} and {\rm \eqref{eq-2.5.3}}$;$ there exists a real linear isomorphism $F:\Gamma^r(G\times_\rho{\sf V})_U\to\mathcal{V}^r(G\times_\rho{\sf V})_U$, $\gamma\mapsto F(\gamma)$, such that $\gamma\bigl(\pi(g)\bigr)=\big[\bigl(g,(F(\gamma))(g)\bigr)\big]$ for all $(\gamma,g)\in\Gamma^r(G\times_\rho{\sf V})_U\times\pi^{-1}(U)$.
   Here $r\in\mathbb{N}\cup\{0,\infty,\omega\}$ and $U$ is a non-empty open subset of $G/H$.
\end{proposition} 
\begin{proof}
   Refer to Section \ref{sec-2.4} for the proof of this proposition.
\end{proof}

\chapter[Homogeneous holomorphic vector bundles]{Homogeneous holomorphic vector bundles over complex homogeneous spaces}\label{ch-3}
   In this chapter we deal with homogeneous holomorphic vector bundles over complex homogeneous spaces.
   The setting of Chapter \ref{ch-3} is as follows:
\begin{itemize}
\item 
   $G$ is a complex Lie group which satisfies the second countability axiom, 
\item 
   $H$ is a closed complex Lie subgroup of $G$, 
\item
   $\pi$ is the projection of $G$ onto the left quotient space $G/H$,
\item
   $\mathcal{S}=\{(U_\alpha,\psi_\alpha)\}_{\alpha\in A}$ is the holomorphic structure on $G/H$ given in Theorem \ref{thm-1.2.1}, 
\item
   $\sigma_\alpha:U_\alpha\to G$ is the holomorphic mapping in Theorem \ref{thm-1.2.1} ($\alpha\in A$).
\end{itemize}
   The topology for $G/H$ is the quotient topology relative to $\pi:g\mapsto gH$, and the homogeneous space $G/H$ is a complex manifold having the atlas $\mathcal{S}$.

\section{Definition of homogeneous holomorphic vector bundle}\label{sec-3.1}
   Let ${\sf V}$ be a finite-dimensional complex vector space, and let $\rho:H\to GL({\sf V})$, $h\mapsto\rho(h)$, be a holomorphic homomorphism, where we fix a complex basis $\{{\sf e}_i\}_{i=1}^m$ of ${\sf V}$ and identify ${\sf V}$ with $\mathbb{C}^m$, and consider ${\sf V}$ and $GL({\sf V})$ as a complex manifold and a complex Lie group, respectively. 
\begin{definition}\label{def-3.1.1}
   In the setting above, the homogeneous vector bundle $G\times_\rho{\sf V}=(G\times_\rho{\sf V},\Pr,G/H)$ over $G/H$ associated with $\rho$ is said to be {\it holomorphic}\index{homogeneous holomorphic vector bundle@homogeneous holomorphic vector bundle\dotfill}. 
   cf.\ Definition \ref{def-2.1.3}.
\end{definition}

    In the next section we state results about homogeneous holomorphic vector bundles.

\section{Results about homogeneous holomorphic vector bundles}\label{sec-3.2}
   Let $G\times_\rho{\sf V}=(G\times_\rho{\sf V},\Pr,G/H)$ be a homogeneous holomorphic vector bundle over $G/H$ associated with $\rho:H\to GL({\sf V})$.
   Referring to Chapter \ref{ch-2} we are going to state results about this bundle.

\begin{theorem}\label{thm-3.2.1}
   There exists a holomorphic structure $\mathscr{S}=\{(\Pr^{-1}(U_\alpha),\varphi_\alpha)\}_{\alpha\in A}$ on the homogeneous holomorphic vector bundle $G\times_\rho{\sf V}$ so that 
\begin{enumerate}
\item[{\rm (1)}]
   $\varpi:G\times{\sf V}\to G\times_\rho{\sf V}$, $(g,{\sf v})\mapsto[(g,{\sf v})]$, is a surjective, open, holomorphic mapping,
\item[{\rm (2)}]
   $\nu:G\times(G\times_\rho{\sf V})\to G\times_\rho{\sf V}$, $(g_1,[(g_2,{\sf v})])\mapsto[(g_1g_2,{\sf v})]$, is a holomorphic mapping, 
\item[{\rm (4)}]
   $\Pr:G\times_\rho{\sf V}\to G/H$, $[(g,{\sf v})]\mapsto\pi(g)=gH$, is a surjective, open, holomorphic mapping,
\item[{\rm (5)}]
   for each $\alpha\in A$, the mapping $\phi_\alpha^{-1}:U_\alpha\times{\sf V}\to\Pr^{-1}(U_\alpha)$, $(x,{\sf v})\mapsto[(\sigma_\alpha(x),{\sf v})]$, is a biholomorphism$;$ besides, $\phi_\alpha\bigl([(g,{\sf v})]\bigr)=\bigl(\pi(g),\rho(\zeta_\alpha(g)){\sf v}\bigr)$ for all $[(g,{\sf v})]\in\Pr^{-1}(U_\alpha)$.
   cf.\ \eqref{eq-1.3.1}.
\end{enumerate} 
   Moreover, for each $x_0\in U_\alpha$ it follows that 
\begin{enumerate}
\item[{\rm (6)}]
   $\Pr\bigl(\phi_\alpha^{-1}(x_0,{\sf v})\bigr)=x_0$ for all ${\sf v}\in{\sf V}$, 
\item[{\rm (7)}] 
   the mapping ${\sf V}\ni{\sf v}\mapsto\phi_\alpha^{-1}(x_0,{\sf v})\in\Pr^{-1}(\{x_0\})$ is a complex linear isomorphism, where the complex vector space structure on $\Pr^{-1}(\{x_0\})$ is defined by a similar way to \eqref{eq-2.3.1}.
\end{enumerate}
   In addition, suppose that $U_\alpha\cap U_\beta\neq\emptyset$ $(\alpha,\beta\in A)$.
   Then, 
\begin{enumerate}
\item[{\rm (8)}]
   $g_{\alpha\beta}:U_\alpha\cap U_\beta\to H$, $y\mapsto(\sigma_\alpha(y))^{-1}\sigma_\beta(y)$, is a holomorphic mapping such that 
   \begin{enumerate}
   \item[{\rm (8.a)}]   
      $g_{\alpha\alpha}(x)=e$ for all $x\in U_\alpha$,
   \item[{\rm (8.b)}]  
   $g_{\alpha\beta}(z)g_{\beta\gamma}(z)g_{\gamma\alpha}(z)=e$ for all $z\in U_\alpha\cap U_\beta\cap U_\gamma$.
   \end{enumerate}
\item[{\rm (9)}] 
   $(\phi_\alpha\circ\phi_\beta^{-1})(y,{\sf v})=\bigl(y,\rho(g_{\alpha\beta}(y)){\sf v}\bigr)$ for all $(y,{\sf v})\in(U_\alpha\cap U_\beta)\times{\sf V}$.   
\end{enumerate}
\end{theorem} 
\begin{proof}
   ref.\ Proposition \ref{prop-2.2.14}, Proposition \ref{prop-2.3.6}, Theorem \ref{thm-2.3.5} and Proposition \ref{prop-2.3.4}.
\end{proof}

   Provide $G\times_\rho{\sf V}$ with the holomorphic structure $\mathscr{S}=\{(\Pr^{-1}(U_\alpha),\varphi_\alpha)\}_{\alpha\in A}$ in Theorem \ref{thm-3.2.1}, and define complex vector spaces $\Gamma(G\times_\rho{\sf V})$ and $\mathcal{V}(G\times_\rho{\sf V})$ by   
\begin{equation}\label{eq-3.2.2}
   \Gamma(G\times_\rho{\sf V})
   :=\left\{\begin{array}{@{}c|c@{}}
   \gamma:G/H\to G\times_\rho{\sf V} 
   & \begin{array}{@{}l@{}} \mbox{(1) $\gamma$ is holomorphic},\\ \mbox{(2) $\Pr\bigl(\gamma(x)\bigr)=x$ for all $x\in G/H$}\end{array}
   \end{array}\right\}
\end{equation}
and  
\begin{equation}\label{eq-3.2.3}
   \mathcal{V}(G\times_\rho{\sf V})
   :=\left\{\begin{array}{@{}c|c@{}}
   \xi:G\to{\sf V} 
   & \begin{array}{@{}l@{}} \mbox{(i) $\xi$ is holomorphic},\\ \mbox{(ii) $\xi(gh)=\rho(h)^{-1}\bigl(\xi(g)\bigr)$ for all $(g,h)\in G\times H$}\end{array}
   \end{array}\right\},
\end{equation}
respectively.
   This $\Gamma(G\times_\rho{\sf V})$ is the complex vector space of holomorphic cross-sections of the homogeneous holomorphic vector bundle $G\times_\rho{\sf V}$, and 

\begin{theorem}\label{thm-3.2.4}
   In the setting of {\rm \eqref{eq-3.2.2}} and {\rm \eqref{eq-3.2.3}}$;$ there exists a complex linear isomorphism $F:\Gamma(G\times_\rho{\sf V})\to\mathcal{V}(G\times_\rho{\sf V})$, $\gamma\mapsto F(\gamma)$, such that $\gamma\bigl(\pi(g)\bigr)=\big[\bigl(g,(F(\gamma))(g)\bigr)\big]$ for all $(\gamma,g)\in\Gamma(G\times_\rho{\sf V})\times G$.
\end{theorem} 
\begin{proof}
   ref.\ Theorem \ref{thm-2.4.15}.
\end{proof}

   For a non-empty open subset $U\subset G/H$, the restriction $(G\times_\rho{\sf V})_U=\bigl((G\times_\rho{\sf V})_U,\Pr|_{(G\times_\rho{\sf V})_U},U\bigr)$ of the bundle $G\times_\rho{\sf V}$ to $U$ is a fiber bundle with ${\sf V}$ and $\rho(H)$ ($\subset GL({\sf V})$).
   Here we induce holomorphic structures $\mathcal{S}_U$ on $U$ and $\mathscr{S}_U$ on $(G\times_\rho{\sf V})_U$ from $G/H$ and $G\times_\rho{\sf V}$, respectively. 
   Let $\Gamma(G\times_\rho{\sf V})_U$ be the complex vector space of holomorphic cross-sections of the bundle $(G\times_\rho{\sf V})_U$, that is,  
\begin{equation}\label{eq-3.2.5}
   \Gamma(G\times_\rho{\sf V})_U
   :=\left\{\begin{array}{@{}c|c@{}}
   \gamma:U\to(G\times_\rho{\sf V})_U 
   & \begin{array}{@{}l@{}} \mbox{(1) $\gamma$ is holomorphic},\\ \mbox{(2) $\Pr\bigl(\gamma(x)\bigr)=x$ for all $x\in U$}\end{array}
   \end{array}\right\}.
\end{equation}
   This $\Gamma(G\times_\rho{\sf V})_U$ corresponds to the complex vector space 
\begin{equation}\label{eq-3.2.6}
   \mathcal{V}(G\times_\rho{\sf V})_U
   :=\left\{\begin{array}{@{}c|c@{}}
   \xi:\pi^{-1}(U)\to{\sf V} 
   & \begin{array}{@{}l@{}} \mbox{(i) $\xi$ is holomorphic},\\ \mbox{(ii) $\xi(gh)=\rho(h)^{-1}\bigl(\xi(g)\bigr)$ for all $(g,h)\in \pi^{-1}(U)\times H$}\end{array}
   \end{array}\right\}
\end{equation}
as follows (ref.\ Proposition \ref{prop-2.5.4}):

\begin{proposition}\label{prop-3.2.7}
   In the setting of {\rm \eqref{eq-3.2.5}} and {\rm \eqref{eq-3.2.6}}$;$ there exists a complex linear isomorphism $F:\Gamma(G\times_\rho{\sf V})_U\to\mathcal{V}(G\times_\rho{\sf V})_U$, $\gamma\mapsto F(\gamma)$, such that $\gamma\bigl(\pi(g)\bigr)=\big[\bigl(g,(F(\gamma))(g)\bigr)\big]$ for all $(\gamma,g)\in\Gamma(G\times_\rho{\sf V})_U\times\pi^{-1}(U)$.
   Here $U$ is a non-empty open subset of $G/H$.
\end{proposition} 

\chapter{Topological vector spaces of mappings}\label{ch-4}
   The main purpose of Chapter \ref{ch-4} is to study topological vector spaces. 
   This chapter consists of four sections. 
   In Section \ref{sec-4.1} we first define an important metric (which is called the Fr\'{e}chet metric) on the vector space of continuous mappings of a certain topological space into a finite-dimensional vector space, next confirm that the metric topology for the vector space coincides with the topology of uniform convergence on compact sets, and finally conclude that the vector space is a Fr\'{e}chet space.
   In Sections \ref{sec-4.2} and \ref{sec-4.3} we apply the arguments in Section \ref{sec-4.1} to the real vector space of continuous cross-sections of a homogeneous vector bundle and the complex vector space of holomorphic cross-sections of a homogeneous holomorphic vector bundle, respectively.
   In the last section we give a proposition about complete metric spaces.

\section{A topological vector space of continuous mappings}\label{sec-4.1}
   Let $\mathbb{K}=\mathbb{R}$ or $\mathbb{C}$.
   Let $X$ be a locally compact Hausdorff space which satisfies the second countability axiom, let ${\sf V}$ be a finite-dimensional vector space over $\mathbb{K}$, and let
\begin{equation}\label{eq-4.1.1}
   \mathcal{C}(X,{\sf V}):=\{\xi:X\to{\sf V} \,|\, \mbox{$\xi$ is continuous}\},
\end{equation} 
where we fix a basis $\{{\sf e}_i\}_{i=1}^m$ of ${\sf V}$, identify ${\sf V}$ with $\mathbb{K}^m$, and consider ${\sf V}$ as a topological space. 
   For $\xi_1,\xi_2,\xi\in\mathcal{C}(X,{\sf V})$, $\alpha\in\mathbb{K}$, one defines the addition $\xi_1+\xi_2$ and the scalar multiplication $\alpha\xi$ by $(\xi_1+\xi_2)(x):=\xi_1(x)+\xi_2(x)$ and $(\alpha\xi)(x):=\alpha\xi(x)$ for $x\in X$, respectively. 
      In this setting we will first endow the vector space $\mathcal{C}(X,{\sf V})$ with a metric topology so that $\mathcal{C}(X,{\sf V})$ is a Hausdorff topological vector space, and afterwards show that the topological vector space $\mathcal{C}(X,{\sf V})$ is a Fr\'{e}chet space. 

\subsection{A metric topology, the Fr\'{e}chet metric}\label{subsec-4.1.1} 
   We want to set a metric $d$ on $\mathcal{C}(X,{\sf V})$.    
   For a non-empty compact subset $E\subset X$, we first define a function $d_E:\mathcal{C}(X,{\sf V})\times\mathcal{C}(X,{\sf V})\to\mathbb{R}$ by 
\begin{equation}\label{eq-4.1.2}
   d_E(\xi_1,\xi_2):=\sup\big\{\|\xi_1(y)-\xi_2(y)\|:y\in E\big\}
\end{equation} 
for $\xi_1,\xi_2\in\mathcal{C}(X,{\sf V})$. 
   Here $\|\cdot\|$ is an arbitrary norm on the vector space ${\sf V}$.\footnote{Remark.\ Two norms on ${\sf V}$ are always equivalent to each other because of $\dim_\mathbb{K}{\sf V}=m<\infty$.
   e.g.\ \begin{CJK}{UTF8}{min}補題 1.38 in 黒田\end{CJK} \cite[p.22]{Ku}.}
   Since $X$ satisfies the second countability axiom and is a locally compact Hausdorff space, there exist non-empty open subsets $O_n\subset X$ such that 
\begin{enumerate}
\item
   $X=\bigcup_{n=1}^\infty O_n$ (countable union),
\item 
   the closure $\overline{O_n}$ in $X$ is compact for each $n\in\mathbb{N}$.   
\end{enumerate} 
   Then, we put $E_n:=\overline{O_n}$ for $n\in\mathbb{N}$. 
   Taking \eqref{eq-4.1.2} into consideration we set  
\begin{equation}\label{eq-4.1.3}  
   d(\xi_1,\xi_2):=\sum_{n=1}^\infty\dfrac{1}{2^n}\dfrac{d_{E_n}(\xi_1,\xi_2)}{1+d_{E_n}(\xi_1,\xi_2)}
\end{equation} 
for $\xi_1,\xi_2\in\mathcal{C}(X,{\sf V})$. 

\begin{lemma}\label{lem-4.1.4}
   The $d$ in {\rm \eqref{eq-4.1.3}} is a metric on $\mathcal{C}(X,{\sf V})$ such that 
\begin{enumerate}
\item[{\rm (1)}]
   $d(\xi_1,\xi_2)\leq 1$ for all $\xi_1,\xi_2\in\mathcal{C}(X,{\sf V})$,
\item[{\rm (2)}]
   $d(\xi_1,\xi_2)=d(\xi_1+\xi_3,\xi_2+\xi_3)$ for all $\xi_1,\xi_2,\xi_3\in\mathcal{C}(X,{\sf V})$,
\item[{\rm (3)}]
   $d(\alpha\xi_1,\alpha\xi_2)\leq d(\xi_1,\xi_2)$ for all $\alpha\in\mathbb{K}$ with $|\alpha|\leq1$ and all $\xi_1,\xi_2\in\mathcal{C}(X,{\sf V})$.
\end{enumerate}
\end{lemma} 
\begin{proof}
   For any $\xi_1,\xi_2\in\mathcal{C}(X,{\sf V})$, we deduce $0\leq d_E(\xi_1,\xi_2), d(\xi_1,\xi_2)$ by \eqref{eq-4.1.2} and \eqref{eq-4.1.3}. 
   Furthermore,
\begin{equation}\label{eq-1}\tag*{\textcircled{1}}
   d(\xi_1,\xi_2)
   =\sum_{n=1}^\infty\dfrac{1}{2^n}\dfrac{d_{E_n}(\xi_1,\xi_2)}{1+d_{E_n}(\xi_1,\xi_2)}
   \leq\sum_{n=1}^\infty\dfrac{1}{2^n}
   =1<\infty.
\end{equation}
   Hence $d$ is a non-negative function on $\mathcal{C}(X,{\sf V})\times\mathcal{C}(X,{\sf V})$.
   It is immediate from \eqref{eq-4.1.2} and \eqref{eq-4.1.3} that 
\begin{enumerate}
\item
   $d(\xi_1,\xi_2)=d(\xi_1+\xi_3,\xi_2+\xi_3)$ for all $\xi_1,\xi_2,\xi_3\in\mathcal{C}(X,{\sf V})$,
\item
   $d(\alpha\xi_1,\alpha\xi_2)\leq d(\xi_1,\xi_2)$ for all $\alpha\in\mathbb{K}$ with $|\alpha|\leq1$ and all $\xi_1,\xi_2\in\mathcal{C}(X,{\sf V})$,   
\item
   $d(\xi_1,\xi_2)=d(\xi_2,\xi_1)$ for all $\xi_1,\xi_2\in\mathcal{C}(X,{\sf V})$,
\item
   $d(\xi,\xi)=0$ for all $\xi\in\mathcal{C}(X,{\sf V})$.
\end{enumerate}  
   Now, for $\xi_1',\xi_2'\in\mathcal{C}(X,{\sf V})$ we suppose that $d(\xi_1',\xi_2')=0$. 
   Then for each $k\in\mathbb{N}$, one has 
\[
   0
   \leq\dfrac{1}{2^k}\dfrac{d_{E_k}(\xi_1',\xi_2')}{1+d_{E_k}(\xi_1',\xi_2')}
   \leq\sum_{n=1}^\infty\dfrac{1}{2^n}\dfrac{d_{E_n}(\xi_1',\xi_2')}{1+d_{E_n}(\xi_1',\xi_2')}
   \stackrel{\eqref{eq-4.1.3}}{=}d(\xi_1',\xi_2')=0.
\]
   This implies that $d_{E_k}(\xi_1',\xi_2')=0$, so that $\xi_1'=\xi_2'$ on $E_k$ for all $k\in\mathbb{N}$. 
   Therefore $\xi_1'=\xi_2'$ on the whole $X$ in terms of $X=\bigcup_{n=1}^\infty O_n$ and $E_n=\overline{O_n}$. 
   Consequently, the rest of proof is to confirm the triangle inequality 
\begin{equation}\label{eq-2}\tag*{\textcircled{2}}
   \mbox{$d(\xi_1,\xi_3)\leq d(\xi_1,\xi_2)+d(\xi_2,\xi_3)$ for all $\xi_1,\xi_2,\xi_3\in\mathcal{C}(X,{\sf V})$}.
\end{equation}
   For any $\xi_1,\xi_2,\xi_3\in\mathcal{C}(X,{\sf V})$, it follows from \eqref{eq-4.1.2} that $d_{E_n}(\xi_1,\xi_3)\leq d_{E_n}(\xi_1,\xi_2)+d_{E_n}(\xi_2,\xi_3)$ for all $n\in\mathbb{N}$. 
   Therefore it follows from $0\leq d_{E_n}(\xi_i,\xi_j)$ that
\[
\begin{split}
   \dfrac{d_{E_n}(\xi_1,\xi_3)}{1+d_{E_n}(\xi_1,\xi_3)}
  &\leq\dfrac{d_{E_n}(\xi_1,\xi_2)+d_{E_n}(\xi_2,\xi_3)}{1+d_{E_n}(\xi_1,\xi_2)+d_{E_n}(\xi_2,\xi_3)}\\
  &=\dfrac{d_{E_n}(\xi_1,\xi_2)}{1+d_{E_n}(\xi_1,\xi_2)+d_{E_n}(\xi_2,\xi_3)}+\dfrac{d_{E_n}(\xi_2,\xi_3)}{1+d_{E_n}(\xi_1,\xi_2)+d_{E_n}(\xi_2,\xi_3)}
   \leq\dfrac{d_{E_n}(\xi_1,\xi_2)}{1+d_{E_n}(\xi_1,\xi_2)}+\dfrac{d_{E_n}(\xi_2,\xi_3)}{1+d_{E_n}(\xi_2,\xi_3)}
\end{split}
\]
for all $n\in\mathbb{N}$.
   This and \ref{eq-1} lead to \ref{eq-2}. 
   Indeed, 
\[
\begin{split}
   d(\xi_1,\xi_3)
   \stackrel{\eqref{eq-4.1.3}}{=}\sum_{n=1}^\infty\dfrac{1}{2^n}\dfrac{d_{E_n}(\xi_1,\xi_3)}{1+d_{E_n}(\xi_1,\xi_3)}
  &\leq\sum_{n=1}^\infty\Big(\dfrac{1}{2^n}\dfrac{d_{E_n}(\xi_1,\xi_2)}{1+d_{E_n}(\xi_1,\xi_2)}+\dfrac{1}{2^n}\dfrac{d_{E_n}(\xi_2,\xi_3)}{1+d_{E_n}(\xi_2,\xi_3)}\Big)\\
  &=\sum_{n=1}^\infty\dfrac{1}{2^n}\dfrac{d_{E_n}(\xi_1,\xi_2)}{1+d_{E_n}(\xi_1,\xi_2)}+\sum_{n=1}^\infty\dfrac{1}{2^n}\dfrac{d_{E_n}(\xi_2,\xi_3)}{1+d_{E_n}(\xi_2,\xi_3)}
   \stackrel{\eqref{eq-4.1.3}}{=}d(\xi_1,\xi_2)+d(\xi_2,\xi_3).
\end{split}
\]
\end{proof}

\begin{definition}\label{def-4.1.5}
   The metric $d$ in \eqref{eq-4.1.3} is called the {\it Fr\'{e}chet metric} on $\mathcal{C}(X,{\sf V})$.\index{Fr\'{e}chet metric@Fr\'{e}chet metric\dotfill}
\end{definition}

\begin{lemma}\label{lem-4.1.6}
   The metric space $(\mathcal{C}(X,{\sf V}),d)$ is complete.
   Here $d$ is the Fr\'{e}chet metric in \eqref{eq-4.1.3}.
\end{lemma}
\begin{proof}
   Let $\{\xi_n\}_{n=1}^\infty$ be an arbitrary Cauchy sequence in $(\mathcal{C}(X,{\sf V}),d)$.\par
 
   Our first aim is to prove that for any $x\in X$, $\{\xi_n(x)\}_{n=1}^\infty$ is a Cauchy sequence in $({\sf V},\|\cdot\|)$. 
   Let us take any $\epsilon>0$. 
   By $x\in X=\bigcup_{n=1}^\infty O_n$ there exists a $k\in\mathbb{N}$ such that $x\in O_k\subset E_k$. 
   Since $\epsilon/(2^k(1+\epsilon))>0$ and $\{\xi_n\}_{n=1}^\infty$ is a Cauchy sequence in $(\mathcal{C}(X,{\sf V}),d)$, there exists an $M\in\mathbb{N}$ such that $n,m\geq M$ implies 
\[
   d(\xi_n,\xi_m)<\dfrac{1}{2^k}\dfrac{\epsilon}{1+\epsilon}.
\] 
   Then, it follows from \eqref{eq-4.1.3} that 
\[
   \dfrac{1}{2^k}\dfrac{d_{E_k}(\xi_n,\xi_m)}{1+d_{E_k}(\xi_n,\xi_m)}\leq  d(\xi_n,\xi_m)<\dfrac{1}{2^k}\dfrac{\epsilon}{1+\epsilon},
\]
so that $d_{E_k}(\xi_n,\xi_m)<\epsilon$; and we deduce $\|\xi_n(x)-\xi_m(x)\|\leq d_{E_k}(\xi_n,\xi_m)<\epsilon$ by virtue of \eqref{eq-4.1.2} and $x\in E_k$. 
   Hence $\{\xi_n(x)\}_{n=1}^\infty$ is a Cauchy sequence in $({\sf V},\|\cdot\|)$ for each $x\in X$.\par
   
   Since the normed vector space $({\sf V},\|\cdot\|)$ is complete, one can get a mapping $\xi:X\to{\sf V}$ by setting 
\[
   \mbox{$\displaystyle{\xi(x):=\lim_{n\to\infty}\xi_n(x)}$ for $x\in X$}.
\]
   Our second aim is to conclude that this $\xi:X\to{\sf V}$ is continuous.
   For any $\epsilon>0$ and $x_0\in X=\bigcup_{n=1}^\infty O_n$, there exists a $k\in\mathbb{N}$ such that $x_0\in O_k\subset E_k$. 
   Moreover, there exists an $N\in\mathbb{N}$ such that $n,m\geq N$ implies
\[
   d(\xi_n,\xi_m)<\dfrac{1}{2^k}\dfrac{(\epsilon/3)}{1+(\epsilon/3)}.
\] 
   Then, it follows that $d_{E_k}(\xi_n,\xi_m)<\epsilon/3$; and $\|\xi_n(y)-\xi_m(y)\|\leq d_{E_k}(\xi_n,\xi_m)<\epsilon/3$ for all $y\in E_k$, $n,m\geq N$.
   For this reason we see that $\|\xi_n(y)-\xi_N(y)\|<\epsilon/3$ for all $y\in E_k$, $n\geq N$. 
   This assures that 
\begin{equation}\label{eq-1}\tag*{\textcircled{1}} 
   \mbox{$\displaystyle{\|\xi(y)-\xi_N(y)\|=\big\|\lim_{n\to\infty}\xi_n(y)-\xi_N(y)\big\|\leq\epsilon/3}$ for all $y\in E_k$}.
\end{equation}
   Besides, since $\xi_N:X\to{\sf V}$ is continuous at $x_0$, there exists an open neighborhood $U_k$ of $x_0\in O_k$ such that $z\in U_k$ implies 
\begin{equation}\label{eq-2}\tag*{\textcircled{2}}
   \|\xi_N(x_0)-\xi_N(z)\|<\epsilon/3. 
\end{equation}
   The $U_k$ is an open neighborhood of $x_0\in X$, and $z\in U_k$ implies 
\[
   \|\xi(x_0)-\xi(z)\|
   \leq\|\xi(x_0)-\xi_N(x_0)\|+\|\xi_N(x_0)-\xi_N(z)\|+\|\xi_N(z)-\xi(z)\|
   <\epsilon
\]
because of \ref{eq-1}, \ref{eq-2} and $x_0\in U_k\subset E_k$.
   Consequently $\xi:X\to{\sf V}$ is continuous at $x_0$. 
   At this stage we can assert $\xi\in\mathcal{C}(X,{\sf V})$.\par

   Our third aim is to demonstrate $\displaystyle{\lim_{m\to\infty}d(\xi,\xi_m)=0}$.
   For any $\epsilon>0$, one can choose an $\ell\in\mathbb{N}$ such that 
\begin{equation}\label{eq-3}\tag*{\textcircled{3}}
   \dfrac{1}{2^\ell}<\dfrac{\epsilon}{2}.
\end{equation}
   For each $1\leq k\leq\ell$, one has $(\epsilon/2)\big/\bigl(2^k(1+(\epsilon/2))\bigr)>0$ and there exists an $N_k\in\mathbb{N}$ such that $n,m\geq N_k$ implies 
\[
   d(\xi_n,\xi_m)<\dfrac{1}{2^k}\dfrac{(\epsilon/2)}{1+(\epsilon/2)}
\]
because $\{\xi_n\}_{n=1}^\infty$ is a Cauchy sequence in $(\mathcal{C}(X,{\sf V}),d)$. 
   Then, it follows from \eqref{eq-4.1.3} that 
\[
   \dfrac{1}{2^k}\dfrac{d_{E_k}(\xi_n,\xi_m)}{1+d_{E_k}(\xi_n,\xi_m)}\leq  d(\xi_n,\xi_m)<\dfrac{1}{2^k}\dfrac{(\epsilon/2)}{1+(\epsilon/2)},  
\]
so that $d_{E_k}(\xi_n,\xi_m)<\epsilon/2$. 
   This and \eqref{eq-4.1.2} enable us to verify that $\|\xi_n(y)-\xi_m(y)\|\leq d_{E_k}(\xi_n,\xi_m)<\epsilon/2$ for all $y\in E_k$, $n,m\geq N_k$. 
   Therefore $\displaystyle{\|\xi(y)-\xi_m(y)\|=\big\|\lim_{n\to\infty}\xi_n(y)-\xi_m(y)\big\|\leq\epsilon/2}$ for all $y\in E_k$, $m\geq N_k$; and thus \eqref{eq-4.1.2} tells us that 
\begin{equation}\label{eq-4}\tag*{\textcircled{4}}
   \mbox{$d_{E_k}(\xi,\xi_m)\leq \epsilon/2$ for all $m\geq N_k$}.
\end{equation}
   Now, let $N:=\max\{N_j:1\leq j\leq\ell\}$. 
   Then, this $N$ belongs to $\mathbb{N}$ and we deduce  
\[
   \mbox{$d_{E_j}(\xi,\xi_m)\leq \epsilon/2$ for all $1\leq j\leq\ell$ and $m\geq N$}   
\]
by \ref{eq-4}. 
   Accordingly, $m\geq N$ implies 
\[
\begin{split}
   d(\xi,\xi_m)
  &\stackrel{\eqref{eq-4.1.3}}{=}\sum_{j=1}^\ell\dfrac{1}{2^j}\dfrac{d_{E_j}(\xi,\xi_m)}{1+d_{E_j}(\xi,\xi_m)}+\sum_{i=\ell+1}^\infty\dfrac{1}{2^i}\dfrac{d_{E_i}(\xi,\xi_m)}{1+d_{E_i}(\xi,\xi_m)}
   \leq\sum_{j=1}^\ell\dfrac{1}{2^j}d_{E_j}(\xi,\xi_m)+\sum_{i=\ell+1}^\infty\dfrac{1}{2^i}\\
  &\leq\sum_{j=1}^\ell\dfrac{1}{2^j}\dfrac{\epsilon}{2}+\sum_{i=\ell+1}^\infty\dfrac{1}{2^i}
  =\Big(1-\dfrac{1}{2^\ell}\Big)\dfrac{\epsilon}{2}+\dfrac{1}{2^\ell}
  <\epsilon \quad\mbox{($\because$ \ref{eq-3})}.
\end{split} 
\] 
   Hence $\displaystyle{\lim_{m\to\infty}d(\xi,\xi_m)=0}$ follows.  
\end{proof}

\begin{lemma}\label{lem-4.1.7}
   With respect to the Fr\'{e}chet metric $d$ in \eqref{eq-4.1.3}, 
\begin{enumerate}
\item[{\rm (1)}] 
   the addition $\mathcal{C}(X,{\sf V})\times\mathcal{C}(X,{\sf V})\ni(\xi_1,\xi_2)\mapsto\xi_1+\xi_2\in\mathcal{C}(X,{\sf V})$ is continuous,
\item[{\rm (2)}]
   the scalar multiplication $\mathbb{K}\times\mathcal{C}(X,{\sf V})\ni(\alpha,\xi)\mapsto\alpha\xi\in\mathcal{C}(X,{\sf V})$ is continuous.
\end{enumerate} 
   Therefore $\mathcal{C}(X,{\sf V})=(\mathcal{C}(X,{\sf V}),d)$ is a Hausdorff topological vector space over $\mathbb{K}$. 
\end{lemma}
\begin{proof}
   (1). 
   Take any $\eta_1,\eta_2\in\mathcal{C}(X,{\sf V})$ and $\epsilon>0$. 
   If $\xi_1,\xi_2\in\mathcal{C}(X,{\sf V})$ and $d(\xi_1,\eta_1), d(\xi_2,\eta_2)<\epsilon/2$, then the triangle inequality and Lemma \ref{lem-4.1.4}-(2) assure  
\[
   d(\xi_1+\xi_2,\eta_1+\eta_2)
   \leq d(\xi_1+\xi_2,\eta_1+\xi_2)+d(\eta_1+\xi_2,\eta_1+\eta_2)
   =d(\xi_1,\eta_1)+d(\xi_2,\eta_2)<\epsilon.
\]
   Hence the addition of vectors is a continuous mapping.\par

   (2). 
   Fix any $\beta\in\mathbb{K}$, $\eta\in\mathcal{C}(X,{\sf V})$ and $\epsilon>0$. 
   Since $\epsilon>0$ there exists an $N\in\mathbb{N}$ such that
\begin{equation}\label{eq-1}\tag*{\textcircled{1}}
   \dfrac{1}{2^N}<\dfrac{\epsilon}{4}.
\end{equation}
   By use of $\beta$, $\epsilon$ and the $N$, we define a positive real number $p$ as follows:
\begin{equation}\label{eq-2}\tag*{\textcircled{2}} 
   p:=\dfrac{\epsilon}{4N(1+|\beta|)}.
\end{equation}
   Now, let us suppose that $\alpha\in\mathbb{K}$ and $\xi\in\mathcal{C}(X,{\sf V})$ satisfy 
\begin{enumerate}
\item[(s1)]
   $|\alpha-\beta|<\dfrac{\epsilon}{4N\bigl(p+\max\{d_{E_j}(\eta,0):1\leq j\leq N\}\bigr)}$ and 
\item[(s2)]
   $d(\xi,\eta)<\dfrac{1}{2^N}\dfrac{p}{1+p}$,   
\end{enumerate} 
respectively.
   We want to get 
\begin{equation}\label{eq-a}\tag{a}
   d(\alpha\xi,\beta\eta)<\epsilon.
\end{equation}
   It follows from \eqref{eq-4.1.3} and (s2) that for any $1\leq j\leq N$,
\[
   \dfrac{1}{2^j}\dfrac{d_{E_j}(\xi,\eta)}{1+d_{E_j}(\xi,\eta)}
   \leq d(\xi,\eta)
   <\dfrac{1}{2^N}\dfrac{p}{1+p}
   \leq\dfrac{1}{2^j}\dfrac{p}{1+p},
\]
so that   
\begin{equation}\label{eq-3}\tag*{\textcircled{3}} 
   \mbox{$d_{E_j}(\xi,\eta)<p$ for all $1\leq j\leq N$}.
\end{equation}
   On the one hand; we obtain
\allowdisplaybreaks{
\begin{align*}
   d(\alpha\xi,\beta\xi)
  &\stackrel{\eqref{eq-4.1.3}}{=}\sum_{j=1}^N\dfrac{1}{2^j}\dfrac{d_{E_j}(\alpha\xi,\beta\xi)}{1+d_{E_j}(\alpha\xi,\beta\xi)}+\sum_{k=N+1}^\infty\dfrac{1}{2^k}\dfrac{d_{E_k}(\alpha\xi,\beta\xi)}{1+d_{E_k}(\alpha\xi,\beta\xi)}
   \leq\sum_{j=1}^Nd_{E_j}(\alpha\xi,\beta\xi)+\sum_{k=N+1}^\infty\dfrac{1}{2^k}\\
  &=\Big(\sum_{j=1}^Nd_{E_j}(\alpha\xi,\beta\xi)\Big)+\dfrac{1}{2^N}
   \stackrel{\eqref{eq-4.1.2}}{=}\Big(\sum_{j=1}^N|\alpha-\beta|d_{E_j}(\xi,0)\Big)+\dfrac{1}{2^N}
   \leq\Big(\sum_{j=1}^N|\alpha-\beta|\bigl(d_{E_j}(\xi,\eta)+d_{E_j}(\eta,0)\bigr)\Big)+\dfrac{1}{2^N}\\
  &\leq|\alpha-\beta|N\bigl(p+\max\{d_{E_j}(\eta,0):1\leq j\leq N\}\bigr)+\dfrac{1}{2^N} \quad\mbox{($\because$ \ref{eq-3})}\\
  &<\dfrac{\epsilon}{2} \quad\mbox{($\because$ (s1), \ref{eq-1})}.
\end{align*}}On the other hand; 
\[
\begin{split}
   d(\beta\xi,\beta\eta)
  &\stackrel{\eqref{eq-4.1.3}}{=}\sum_{j=1}^N\dfrac{1}{2^j}\dfrac{d_{E_j}(\beta\xi,\beta\eta)}{1+d_{E_j}(\beta\xi,\beta\eta)}+\sum_{k=N+1}^\infty\dfrac{1}{2^k}\dfrac{d_{E_k}(\beta\xi,\beta\eta)}{1+d_{E_k}(\beta\xi,\beta\eta)}
   \leq\Big(\sum_{j=1}^Nd_{E_j}(\beta\xi,\beta\eta)\Big)+\dfrac{1}{2^N}\\
  &\stackrel{\eqref{eq-4.1.2}}{=}\Big(\sum_{j=1}^N|\beta|d_{E_j}(\xi,\eta)\Big)+\dfrac{1}{2^N}
  \leq|\beta|Np+\dfrac{1}{2^N}\quad\mbox{($\because$ \ref{eq-3})}\\
  &\stackrel{\ref{eq-2}}{=}\dfrac{|\beta|\epsilon}{4(1+|\beta|)}+\dfrac{1}{2^N}
  <\dfrac{\epsilon}{2} \quad\mbox{($\because$ \ref{eq-1})}.
\end{split}
\] 
   These yield $d(\alpha\xi,\beta\xi)+d(\beta\xi,\beta\eta)<\epsilon$. 
   This, combined with the triangle inequality, enables us to conclude \eqref{eq-a}.
   So, the scalar multiplication $\mathbb{K}\times\mathcal{C}(X,{\sf V})\ni(\alpha,\xi)\mapsto\alpha\xi\in\mathcal{C}(X,{\sf V})$ is continuous.
\end{proof}

   We here give a supplementation about Hausdorff topological vector spaces.
\begin{proposition}\label{prop-4.1.8}
   Let $\mathcal{X}$ be a Hausdorff topological vector space over $\mathbb{K}$,\footnote{This proposition can hold even if $\mathcal{X}$ is a topological vector space satisfying the first separation axiom.} and let $\mathcal{Y}$ be a real or complex vector subspace of $\mathcal{X}$ according as $\mathbb{K}=\mathbb{R}$ or $\mathbb{K}=\mathbb{C}$.
   Suppose that $\dim_\mathbb{K}\mathcal{Y}<\infty$. 
   Then, 
\begin{enumerate}
\item[{\rm (1)}]
   the topological vector space $\mathcal{Y}$ is isomorphic to $\mathbb{K}^k$, where $k=\dim_\mathbb{K}\mathcal{Y}$, 
\item[{\rm (2)}]
   $\mathcal{Y}$ is a closed subset of $\mathcal{X}$.
\end{enumerate}   
\end{proposition} 
\begin{proof}
   (1).
   Fix a basis $\{e_i\}_{i=1}^k$ of $\mathcal{Y}$, and define a linear isomorphism $f:\mathbb{K}^k\to\mathcal{Y}$ by 
\[
   \mbox{$f(\alpha_1,\alpha_2,\dots,\alpha_k):=\sum_{i=1}^k\alpha_ie_i$ for $(\alpha_1,\alpha_2,\dots,\alpha_k)\in\mathbb{K}^k$}.
\] 
   Let us show that 
\begin{equation}\label{eq-1}\tag*{\textcircled{1}}
   \mbox{the linear isomorphism $f:\mathbb{K}^k\to\mathcal{Y}$ is homeomorphic}.
\end{equation}
   It is natural that $f:\mathbb{K}^k\to\mathcal{Y}$, $(\alpha_1,\alpha_2,\dots,\alpha_k)\mapsto\sum_{i=1}^k\alpha_ie_i$, is continuous because $f$ is the composition of the following three continuous mappings $f_1$, $f_2$ and $f_3$:
\begin{center}
\begin{tabular}{l}
   $f_1:\mathbb{K}^k\to(\mathbb{K}\times\mathcal{Y})\times(\mathbb{K}\times\mathcal{Y})\times\cdots\times(\mathbb{K}\times\mathcal{Y})$, $(\alpha_1,\alpha_2,\dots,\alpha_k)\mapsto(\alpha_1,e_1,\alpha_2,e_2,\dots,\alpha_k,e_k)$,\\
   $f_2:(\mathbb{K}\times\mathcal{Y})\times(\mathbb{K}\times\mathcal{Y})\times\cdots\times(\mathbb{K}\times\mathcal{Y})\to\mathcal{Y}\times\mathcal{Y}\times\cdots\times\mathcal{Y}$, $(\alpha_1,y_1,\alpha_2,y_2,\dots,\alpha_k,y_k)\mapsto(\alpha_1y_1,\alpha_2y_2,\dots,\alpha_ky_k)$,\\
   $f_3:\mathcal{Y}\times\mathcal{Y}\times\cdots\times\mathcal{Y}\to\mathcal{Y}$, $(y_1,y_2,\dots,y_k)\mapsto\sum_{i=1}^ky_i$.
\end{tabular}   
\end{center}
   We need to confirm that the inverse $f^{-1}:\mathcal{Y}\to\mathbb{K}^k$ is also continuous. 
   It suffices to confirm that $f^{-1}$ is continuous at the zero $0\in\mathcal{Y}$. 
   Let $\|\bm{x}\|:=\sqrt{|x_1|^2+|x_2|^2+\cdots+|x_k|^2}$ for $\bm{x}=(x_1,x_2,\dots,x_k)\in\mathbb{K}^k$. 
   We will show that for any $\epsilon>0$ there exists an open neighborhood $U$ of $0\in\mathcal{Y}$ satisfying 
\begin{equation}\label{eq-2}\tag*{\textcircled{2}}
   f^{-1}(U)\subset B_\epsilon,
\end{equation}
where $B_\epsilon:=\{\bm{x}\in\mathbb{K}^k : \|\bm{x}\|<\epsilon\}$. 
   It turns out that the sphere $S_\epsilon:=\{\bm{y}\in\mathbb{K}^k : \|\bm{y}\|=\epsilon\}$ is a compact subset of $\mathbb{K}^k$ and $\bm{0}\not\in S_\epsilon$. 
   Therefore, since $f:\mathbb{K}^k\to\mathcal{Y}$ is injective continuous and $\mathcal{Y}$ is a Hausdorff space, we conclude that $f(S_\epsilon)\subset\mathcal{Y}$ is closed and $0=f({\bf 0})\not\in f(S_\epsilon)$.
   Accordingly there exist two open subsets $U_1,U_2\subset\mathcal{Y}$ such that
\begin{equation}\label{eq-a}\tag{a}
\begin{array}{lll}
   0\in U_1, & f(S_\epsilon)\subset U_2, & U_1\cap U_2=\emptyset
\end{array}   
\end{equation}
because a Hausdorff topological vector space is a regular space.\footnote{More generally, a topological group satisfying the first separation axiom is a regular space. 
   e.g.\ \begin{CJK}{UTF8}{min}定理 1.8 in 村上\end{CJK} \cite[p.28]{Mu}.} 
   Since $U_1$ is an open neighborhood of $0\in\mathcal{Y}$ and the scalar multiplication $\mathbb{K}\times\mathcal{Y}\ni(\alpha,y)\mapsto\alpha y\in\mathcal{Y}$ is continuous at $(0,0)$, there exist a positive real number $p$ and an open neighborhood $V$ of $0\in\mathcal{Y}$ such that $\alpha v\in U_1$ for all $|\alpha|<p$ and $v\in V$.
   Setting $U:=\bigcup_{0<|\beta|<p}\beta V$, one deduces the following: 
\begin{equation}\label{eq-b}\tag{b}
   \mbox{(i) $U$ is an open subset of $\mathcal{Y}$, (ii) $0\in U\subset U_1$, (iii) $tu\in U$ for all $|t|\leq 1$ and $u\in U$},
\end{equation}
where we remark that each $\beta V$ is an open neighborhood of $0\in\mathcal{Y}$.
   Now, we are in a position to prove \ref{eq-2}.
   Let us use proof by contradiction.
   Suppose that there exists a $\bm{z}\in f^{-1}(U)$ which does not belong to $B_\epsilon$.
   Then, it follows from $\bm{z}\not\in B_\epsilon$ that $\|\bm{z}\|\geq\epsilon$. 
   Therefore the intermediate-value theorem enables us to obtain a real number $t_0$ such that $0\leq t_0\leq 1$ and $\|t_0\bm{z}\|=\epsilon$ (because the mapping $[0,1]\ni t\mapsto t\|\bm{z}\|=\|t\bm{z}\|\in\mathbb{R}$ is continuous).
   This $\|t_0\bm{z}\|=\epsilon$ implies $t_0\bm{z}\in S_\epsilon$, and so \eqref{eq-a} yields
\[
   f(t_0\bm{z})\in U_2.
\]
   However, from $f(\bm{z})\in U$, $0\leq t_0\leq 1$ and \eqref{eq-b} we deduce $f(t_0\bm{z})=t_0f(\bm{z})\in U\subset U_1$. 
   Hence $f(t_0\bm{z})\in U_1\cap U_2$, which contradicts \eqref{eq-a}.
   For this reason \ref{eq-2} holds, and $f^{-1}:\mathcal{Y}\to\mathbb{K}^k$ is continuous at $0$. 
   This assures \ref{eq-1}.\par

   (2).   
   Taking the above $U$ and $f$ into account, we are going to prove that $\mathcal{Y}\subset\mathcal{X}$ is closed from now on.
   Let $x$ be an arbitrary element of $\overline{\mathcal{Y}}{}^\mathcal{X}$ (the closure of $\mathcal{Y}$ in $\mathcal{X}$).
   By \eqref{eq-b}-(i) there exists an open subset $O\subset\mathcal{X}$ such that 
\[
   U=(O\cap\mathcal{Y}).
\] 
   Since $O$ is an open neighborhood of $0\in\mathcal{X}$ and the mapping $\mathbb{R}\ni t\mapsto tx\in\mathcal{X}$ is continuous at $0$, there exists a $\nu>0$ such that $(1/\nu)x\in O$.
   Then, one has 
\[
\begin{split}
   x\in(\nu O\cap\overline{\mathcal{Y}}{}^\mathcal{X})
   &\subset\overline{\nu O\cap\mathcal{Y}}{}^\mathcal{X} \quad \mbox{($\because$ $\nu O$ is open in $\mathcal{X}$)}\\
   &=\overline{\nu(O\cap\mathcal{Y})}{}^\mathcal{X}
    =\overline{\nu U}{}^\mathcal{X}
    \subset\overline{\nu f(B_\epsilon)}{}^\mathcal{X} \quad \mbox{($\because$ \ref{eq-2})}\\
   &=\overline{f(\nu B_\epsilon)}{}^\mathcal{X} 
    =\overline{f(B_{\nu\epsilon})}{}^\mathcal{X}
    \subset\overline{f(\overline{B_{\nu\epsilon}}{}^{\mathbb{K}^k})}{}^\mathcal{X} 
    =f(\overline{B_{\nu\epsilon}}{}^{\mathbb{K}^k})
    \subset f(\mathbb{K}^k)=\mathcal{Y},
\end{split}   
\] 
where $\overline{f(\overline{B_{\nu\epsilon}}{}^{\mathbb{K}^k})}{}^\mathcal{X}=f(\overline{B_{\nu\epsilon}}{}^{\mathbb{K}^k})$ follows by $f:\mathbb{K}^k\to(\mathcal{Y}\subset)\mathcal{X}$ being continuous, $f(\overline{B_{\nu\epsilon}}{}^{\mathbb{K}^k})\subset\mathcal{X}$ being compact and $\mathcal{X}$ being a Hausdorff space.
   Therefore $x\in\mathcal{Y}$, and $\overline{\mathcal{Y}}{}^\mathcal{X}\subset\mathcal{Y}$. 
   This implies that $\mathcal{Y}$ is a closed subset of $\mathcal{X}$.
\end{proof}

\subsection{A metric topology, a topology of uniform convergence on compact sets, and a locally convex topology}\label{sec-4.1.2} 
   In the previous subsection we have defined the Fr\'{e}chet metric $d$ on $\mathcal{C}(X,{\sf V})$. 
   So, one can consider the metric topology for $(\mathcal{C}(X,{\sf V}),d)$. 
   Recalling that $X=\bigcup_{n=1}^\infty O_n$ and each $E_n=\overline{O_n}$ is compact in $X$ (cf.\ Subsection \ref{subsec-4.1.1}), we prove two Lemmas \ref{lem-4.1.9} and \ref{lem-4.1.10}, and deduce Theorem \ref{thm-4.1.11} from them.
\begin{lemma}\label{lem-4.1.9}
   The metric topology for $(\mathcal{C}(X,{\sf V}),d)$ coincides with the topology of uniform convergence on compact sets.
\end{lemma}
\begin{proof}
   First, let us demonstrate that the metric topology $\mathscr{D}_d$ for $(\mathcal{C}(X,{\sf V}),d)$ is coarser than the topology $\mathscr{D}_{\rm cu}$ of uniform convergence on compact sets, namely 
\[
   \mathscr{D}_d\subset\mathscr{D}_{\rm cu}.
\]   
   For given $\xi_0\in\mathcal{C}(X,{\sf V})$ and $\epsilon>0$, we set $O_d:=\{\xi\in\mathcal{C}(X,{\sf V}) \,|\, d(\xi,\xi_0)<\epsilon\}$ and take an arbitrary element $\xi\in O_d$. 
   We want to show that there exist a non-empty compact subset $E\subset X$ and a $\delta>0$ satisfying 
\[
   \{\eta\in\mathcal{C}(X,{\sf V}) \,|\, d_E(\eta,\xi)<\delta\}\subset O_d
\]
(see \eqref{eq-4.1.2} for $d_E$).
   Let $r:=d(\xi,\xi_0)$.  
   Since $\epsilon-r>0$ there exists an $m\in\mathbb{N}$ such that   
\[
   1/2^m<(\epsilon-r)/2.
\] 
   By use of $m$, $\epsilon$ and $r$ we put 
\[
\begin{array}{ll}
   E:=\bigcup_{j=1}^mE_j, & \delta:=(\epsilon-r)/(2m).
\end{array}
\] 
   Then, it turns out that $E$ is a non-empty compact subset of $X$ and $\delta>0$.
   Moreover, \eqref{eq-4.1.2} yields 
\[
   d_{E_1}(\xi_1,\xi_2)+\cdots+d_{E_m}(\xi_1,\xi_2)\leq md_E(\xi_1,\xi_2)
\]
for all $\xi_1,\xi_2\in\mathcal{C}(X,{\sf V})$. 
   Hence for any $\eta\in\mathcal{C}(X,{\sf V})$ with $d_E(\eta,\xi)<\delta$, we have  
\[
\begin{split}
   d(\eta,\xi)
   &\stackrel{\eqref{eq-4.1.3}}{=}\sum_{n=1}^\infty\dfrac{1}{2^n}\dfrac{d_{E_n}(\eta,\xi)}{1+d_{E_n}(\eta,\xi)}
   =\sum_{j=1}^m\dfrac{1}{2^j}\dfrac{d_{E_j}(\eta,\xi)}{1+d_{E_j}(\eta,\xi)}+\sum_{k=m+1}^\infty\dfrac{1}{2^k}\dfrac{d_{E_k}(\eta,\xi)}{1+d_{E_k}(\eta,\xi)}\\
   &\leq \sum_{j=1}^md_{E_j}(\eta,\xi)+\sum_{k=m+1}^\infty\dfrac{1}{2^k}
   =\Big(\sum_{j=1}^md_{E_j}(\eta,\xi)\Big)+\dfrac{1}{2^m}
   \leq md_E(\eta,\xi)+\dfrac{1}{2^m}\\
   &<m\delta+\dfrac{1}{2^m}<\epsilon-r.
\end{split}   
\]
   This and $d(\eta,\xi_0)\leq d(\eta,\xi)+d(\xi,\xi_0)=d(\eta,\xi)+r$ imply that $\{\eta\in\mathcal{C}(X,{\sf V}) \,|\, d_E(\eta,\xi)<\delta\}\subset O_d$, and thus $\mathscr{D}_d\subset\mathscr{D}_{\rm cu}$.\par
 
   Next, let us confirm that the converse inclusion $\mathscr{D}_{\rm cu}\subset\mathscr{D}_d$ also holds. 
   For given $\xi_0'\in\mathcal{C}(X,{\sf V})$, $\epsilon'>0$ and non-empty compact subset $E'\subset X$, we set $O_{\rm cu}:=\{\xi'\in\mathcal{C}(X,{\sf V}) \,|\, d_{E'}(\xi',\xi_0')<\epsilon'\}$ and fix any element $\xi'\in O_{\rm cu}$. 
   We are going to show that there exists a $\delta'>0$ satisfying 
\[
   \{\eta'\in\mathcal{C}(X,{\sf V}) \,|\, d(\eta',\xi')<\delta'\}\subset O_{\rm cu}.
\]
   Put $r':=d_{E'}(\xi',\xi_0')$. 
   Since $X=\bigcup_{n=1}^\infty O_n$, $E_n=\overline{O_n}$ and $E'$ is compact, there exist finite elements $n(1),\dots,n(k)\in\mathbb{N}$ such that $n(1)<\cdots<n(k)$ and $E'\subset\bigcup_{i=1}^k E_{n(i)}$.
   Then it follows from \eqref{eq-4.1.2} that
\[
   d_{E'}(\xi_1,\xi_2)\leq d_{E_{n(1)}}(\xi_1,\xi_2)+\cdots+d_{E_{n(k)}}(\xi_1,\xi_2)
\] 
for all $\xi_1,\xi_2\in\mathcal{C}(X,{\sf V})$. 
   Setting 
\[
   \delta':=\dfrac{1}{2^{n(k)}}\dfrac{((\epsilon'-r')/k)}{1+((\epsilon'-r')/k)},
\]  
we deduce $\delta'>0$. 
   In addition; if $\eta'\in\mathcal{C}(X,{\sf V})$ satisfies $d(\eta',\xi')<\delta'$, then it follows from \eqref{eq-4.1.3} that 
\[
   \dfrac{1}{2^{n(i)}}\dfrac{d_{E_{n(i)}}(\eta',\xi')}{1+d_{E_{n(i)}}(\eta',\xi')}
   \leq d(\eta',\xi')
   <\delta'
   =\dfrac{1}{2^{n(k)}}\dfrac{((\epsilon'-r')/k)}{1+((\epsilon'-r')/k)}
   \leq\dfrac{1}{2^{n(i)}}\dfrac{((\epsilon'-r')/k)}{1+((\epsilon'-r')/k)},
\]
so that $d_{E_{n(i)}}(\eta',\xi')<(\epsilon'-r')/k$ for all $1\leq i\leq k$. 
   Consequently, if $\eta'\in\mathcal{C}(X,{\sf V})$ satisfies $d(\eta',\xi')<\delta'$, then $d_{E'}(\eta',\xi')\leq\sum_{i=1}^kd_{E_{n(i)}}(\eta',\xi')<\epsilon'-r'$.
   This and $d_{E'}(\eta',\xi_0')\leq d_{E'}(\eta',\xi')+d_{E'}(\xi',\xi_0')=d_{E'}(\eta',\xi')+r'$ imply that $\{\eta'\in\mathcal{C}(X,{\sf V}) \,|\, d(\eta',\xi')<\delta'\}\subset O_{\rm cu}$, and so $\mathscr{D}_{\rm cu}\subset\mathscr{D}_d$.   
\end{proof}

\begin{lemma}\label{lem-4.1.10}
   The metric topology for $(\mathcal{C}(X,{\sf V}),d)$ coincides with the locally convex topology determined by a countable number of seminorms $\{p_n\}_{n\in\mathbb{N}}$, where $p_n(\xi):=d_{E_n}(\xi,0)$ for $n\in\mathbb{N}$, $\xi\in\mathcal{C}(X,{\sf V})$. 
   Here we refer to \eqref{eq-4.1.2} for $d_{E_n}$.
\end{lemma}
\begin{proof}
   We denote by $\mathscr{D}_{\rm loc}$ the locally convex topology determined by $\{p_n\}_{n\in\mathbb{N}}$, and utilize the same notation $\mathscr{D}_d$, $O_d=\{\xi\in\mathcal{C}(X,{\sf V}) \,|\, d(\xi,\xi_0)<\epsilon\}$ as in the proof of Lemma \ref{lem-4.1.9}. 
   We need to verify that $\mathscr{D}_{\rm loc}=\mathscr{D}_d$, but $\mathscr{D}_{\rm loc}\subset\mathscr{D}_d$ is a consequence of Lemma \ref{lem-4.1.9}. 
   Thus we are going to confirm $\mathscr{D}_d\subset\mathscr{D}_{\rm loc}$ only.\par
   
   For a given $\xi\in O_d$, we put $r:=d(\xi,\xi_0)$. 
   Since $\epsilon-r>0$ there exists an $N\in\mathbb{N}$ such that 
\[
   \dfrac{1}{2^N}<\dfrac{\epsilon-r}{2}.
\]
   Then, any $\eta\in\bigcap_{i=1}^N\{\eta\in\mathcal{C}(X,{\sf V}) \,|\, p_i(\eta-\xi)<(\epsilon-r)/(2N)\}$ satisfies 
\[
\begin{split}
   d(\eta,\xi)
  &\stackrel{\eqref{eq-4.1.3}}{=}\sum_{i=1}^N\dfrac{1}{2^i}\dfrac{d_{E_i}(\eta,\xi)}{1+d_{E_i}(\eta,\xi)}+\sum_{k=N+1}^\infty\dfrac{1}{2^k}\dfrac{d_{E_k}(\eta,\xi)}{1+d_{E_k}(\eta,\xi)}
   \leq\sum_{i=1}^Nd_{E_i}(\eta,\xi)+\sum_{k=N+1}^\infty\dfrac{1}{2^k}
   =\Big(\sum_{i=1}^Nd_{E_i}(\eta,\xi)\Big)+\dfrac{1}{2^N}\\
  &\stackrel{\eqref{eq-4.1.2}}{=}\Big(\sum_{i=1}^Np_i(\eta-\xi)\Big)+\dfrac{1}{2^N}
   <\epsilon-r;
\end{split} 
\]  
furthermore, $d(\eta,\xi_0)\leq d(\eta,\xi)+d(\xi,\xi_0)<\epsilon-r+r=\epsilon$. 
   Hence we see that $\bigcap_{i=1}^N\{\eta\in\mathcal{C}(X,{\sf V}) \,|\, p_i(\eta-\xi)<(\epsilon-r)/(2N)\}\subset O_d$, and $\mathscr{D}_d\subset\mathscr{D}_{\rm loc}$.
\end{proof}

   Summarizing statements above we conclude the following (see \eqref{eq-4.1.1}, \eqref{eq-4.1.2} for $\mathcal{C}(X,{\sf V})$, $d_{E_n}$):
\begin{theorem}\label{thm-4.1.11}
   With respect the Fr\'{e}chet metric $d$ in \eqref{eq-4.1.3}, 
\begin{enumerate}
\item[{\rm (1)}]
   the metric space $(\mathcal{C}(X,{\sf V}),d)$ is complete,
\item[{\rm (2)}] 
   the addition $\mathcal{C}(X,{\sf V})\times\mathcal{C}(X,{\sf V})\ni(\xi_1,\xi_2)\mapsto\xi_1+\xi_2\in\mathcal{C}(X,{\sf V})$ and the scalar multiplication $\mathbb{K}\times\mathcal{C}(X,{\sf V})\ni(\alpha,\xi)\mapsto\alpha\xi\in\mathcal{C}(X,{\sf V})$ are continuous,   
\item[{\rm (3)}]
   the metric topology for $(\mathcal{C}(X,{\sf V}),d)$ coincides with the topology of uniform convergence on compact sets$;$ besides, it also coincides with the locally convex topology determined by a countable number of seminorms $\{p_n\}_{n\in\mathbb{N}}$, where $p_n(\xi)=d_{E_n}(\xi,0)$ for $n\in\mathbb{N}$, $\xi\in\mathcal{C}(X,{\sf V})$.
\end{enumerate}
   Therefore $\mathcal{C}(X,{\sf V})$ is a Fr\'{e}chet space over $\mathbb{K}=\mathbb{R}$ or $\mathbb{C}$.
\end{theorem} 
\begin{proof}
   cf.\ Lemmas \ref{lem-4.1.6}, \ref{lem-4.1.7}, \ref{lem-4.1.9} and \ref{lem-4.1.10}. 
\end{proof}

\section[Real vector spaces of continuous cross-sections]{Real vector spaces of continuous cross-sections of homogeneous vector bundles}\label{sec-4.2}
   The setting of Section \ref{sec-4.2} is as follows:
\begin{itemize}
\item 
   $G$ is a Lie group which satisfies the second countability axiom, 
\item 
   $H$ is a closed subgroup of $G$, 
\item
   $\pi$ is the projection of $G$ onto the left quotient space $G/H$,
\item
   $\mathcal{S}=\{(U_\alpha,\psi_\alpha)\}_{\alpha\in A}$ is the real analytic structure on $G/H$ given in Theorem \ref{thm-1.1.2}, 
\item 
   $G\times_\rho{\sf V}=(G\times_\rho{\sf V},\Pr,G/H)$ is a homogeneous vector bundle over $G/H$ associated with $\rho:H\to GL({\sf V})$, 
\item 
   $\mathscr{S}=\{(\Pr^{-1}(U_\alpha),\varphi_\alpha)\}_{\alpha\in A}$ is the real analytic structure on $G\times_\rho{\sf V}$ in Proposition \ref{prop-2.2.9}.      
\end{itemize}
   The topologies for $G/H$ and $G\times_\rho{\sf V}$ are the quotient topologies relative to $\pi:G\to G/H$, $g\mapsto gH$, and $\varpi:G\times{\sf V}\to G\times_\rho{\sf V}$, $(g,{\sf v})\mapsto[(g,{\sf v})]$, respectively, and the homogeneous space $G/H$ and the homogeneous vector bundle $G\times_\rho{\sf V}$ are real analytic manifolds having the atlases $\mathcal{S}$ and $\mathscr{S}$, respectively.\par
   
   Now, let $U$ be a non-empty open subset of $G/H$. 
   Since $\pi^{-1}(U)$ is open in $G$ and the Lie group $G$ satisfies the second countability axiom, we see that $\pi^{-1}(U)$ is a locally compact Hausdorff space and satisfies the same axiom. 
   For this reason we can apply the arguments and notation ``$d$, $d_{E_n}$, $\|\cdot\|$'' in Section \ref{sec-4.1} to $\mathcal{C}(\pi^{-1}(U),{\sf V})$. 
   Noting \eqref{eq-2.5.3} and 
\[
   \mathcal{V}^0(G\times_\rho{\sf V})_U\subset\mathcal{C}(\pi^{-1}(U),{\sf V}),
\]
we demonstrate  
\begin{proposition}\label{prop-4.2.1}
   With respect the Fr\'{e}chet metric $d$ in \eqref{eq-4.1.3}, 
\begin{enumerate}
\item[{\rm (1)}]
   the metric space $(\mathcal{V}^0(G\times_\rho{\sf V})_U,d)$ is complete,
\item[{\rm (2)}] 
   the addition $\mathcal{V}^0(G\times_\rho{\sf V})_U\times\mathcal{V}^0(G\times_\rho{\sf V})_U\ni(\xi_1,\xi_2)\mapsto\xi_1+\xi_2\in\mathcal{V}^0(G\times_\rho{\sf V})_U$ and the scalar multiplication $\mathbb{R}\times\mathcal{V}^0(G\times_\rho{\sf V})_U\ni(\lambda,\xi)\mapsto\lambda\xi\in\mathcal{V}^0(G\times_\rho{\sf V})_U$ are continuous,   
\item[{\rm (3)}]
   the metric topology for $(\mathcal{V}^0(G\times_\rho{\sf V})_U,d)$ coincides with the topology of uniform convergence on compact sets$;$ besides, it also coincides with the locally convex topology determined by a countable number of seminorms $\{p_n\}_{n\in\mathbb{N}}$, where $p_n(\xi):=d_{E_n}(\xi,0)$ for $n\in\mathbb{N}$, $\xi\in\mathcal{V}^0(G\times_\rho{\sf V})_U$.
\end{enumerate}
   Therefore $\mathcal{V}^0(G\times_\rho{\sf V})_U$ is a Fr\'{e}chet space over $\mathbb{R}$.
\end{proposition} 
\begin{proof}
   Theorem \ref{thm-4.1.11}, together with $\mathcal{V}^0(G\times_\rho{\sf V})_U\subset\mathcal{C}(\pi^{-1}(U),{\sf V})$, enables us to show that $d$ is a metric on $\mathcal{V}^0(G\times_\rho{\sf V})_U$, and to conclude (2), (3). 
   Hence, we only prove that $(\mathcal{V}^0(G\times_\rho{\sf V})_U,d)$ is complete.\par

   Let $\{\eta_n\}_{n=1}^\infty$ be a given Cauchy sequence in $(\mathcal{V}^0(G\times_\rho{\sf V})_U,d)$. 
   By $\mathcal{V}^0(G\times_\rho{\sf V})_U\subset\mathcal{C}(\pi^{-1}(U),{\sf V})$ and Lemma \ref{lem-4.1.6}, there exists a unique $\eta\in\mathcal{C}(\pi^{-1}(U),{\sf V})$ such that $\displaystyle{\lim_{n\to\infty}d(\eta,\eta_n)=0}$. 
   In order to show $\eta\in\mathcal{V}^0(G\times_\rho{\sf V})_U$, it suffices to confirm that 
\[
   \mbox{$\eta(gh)=\rho(h)^{-1}\bigl(\eta(g)\bigr)$ for all $(g,h)\in \pi^{-1}(U)\times H$}
\]
because \eqref{eq-2.5.3}.
   For any $(g,h)\in \pi^{-1}(U)\times H$, it follows from $\displaystyle{\lim_{n\to\infty}d(\eta,\eta_n)=0}$ and $gh,g\in\pi^{-1}(U)$ that 
\[
\begin{array}{ll}
   \displaystyle{\lim_{n\to\infty}\|\eta(gh)-\eta_n(gh)\|=0}, & \displaystyle{\lim_{m\to\infty}\|\eta_m(g)-\eta(g)\|=0}
\end{array}
\]
(ref.\ the beginning of the proof of Lemma \ref{lem-4.1.6}); and therefore 
\[
\begin{split}
   \|\eta(gh)-\rho(h)^{-1}\bigl(\eta(g)\bigr)\|
  &\leq\|\eta(gh)-\eta_n(gh)\|+\|\eta_n(gh)-\rho(h)^{-1}\bigl(\eta(g)\bigr)\|\\
  &=\|\eta(gh)-\eta_n(gh)\|+\|\rho(h)^{-1}\bigl(\eta_n(g)-\eta(g)\bigr)\|
  \longrightarrow 0 \quad(n\to\infty)
\end{split}
\]
because of $\eta_n\in\mathcal{V}^0(G\times_\rho{\sf V})_U$ and because the mapping $\rho(h)^{-1}:{\sf V}\to{\sf V}$, ${\sf v}\mapsto\rho(h)^{-1}({\sf v})$, is continuous.
   Consequently, one has $\eta\in\mathcal{V}^0(G\times_\rho{\sf V})_U$, and the metric space $(\mathcal{V}^0(G\times_\rho{\sf V})_U,d)$ is complete.\footnote{This implies that $\mathcal{V}^0(G\times_\rho{\sf V})_U$ is a closed, real vector subspace of the Fr\'{e}chet space $\mathcal{C}(\pi^{-1}(U),{\sf V})$.}     
\end{proof}

\section[Complex vector spaces of holomorphic cross-sections]{Complex vector spaces of holomorphic cross-sections of homogeneous holomorphic vector bundles}\label{sec-4.3}
   The setting of Section \ref{sec-4.3} is as follows:
\begin{itemize}
\item 
   $G$ is a complex Lie group which satisfies the second countability axiom, 
\item 
   $H$ is a closed complex Lie subgroup of $G$, 
\item
   $\pi$ is the projection of $G$ onto the left quotient space $G/H$,
\item
   $\mathcal{S}=\{(U_\alpha,\psi_\alpha)\}_{\alpha\in A}$ is the holomorphic structure on $G/H$ given in Theorem \ref{thm-1.2.1}, 
\item 
   $G\times_\rho{\sf V}=(G\times_\rho{\sf V},\Pr,G/H)$ is a homogeneous holomorphic vector bundle over $G/H$ associated with $\rho:H\to GL({\sf V})$, 
\item 
   $\mathscr{S}=\{(\Pr^{-1}(U_\alpha),\varphi_\alpha)\}_{\alpha\in A}$ is the holomorphic structure on $G\times_\rho{\sf V}$ in Theorem \ref{thm-3.2.1}.      
\end{itemize}
   The topologies for $G/H$ and $G\times_\rho{\sf V}$ are the quotient topologies relative to $\pi:G\to G/H$, $g\mapsto gH$, and $\varpi:G\times{\sf V}\to G\times_\rho{\sf V}$, $(g,{\sf v})\mapsto[(g,{\sf v})]$, respectively, and the homogeneous space $G/H$ and the homogeneous holomorphic vector bundle $G\times_\rho{\sf V}$ are complex manifolds having the atlases $\mathcal{S}$ and $\mathscr{S}$, respectively. 
   Here we fix a complex basis $\{{\sf e}_i\}_{i=1}^m$ of ${\sf V}$, identify ${\sf V}$ with $\mathbb{C}^m$ and consider ${\sf V}$ as a complex manifold.\par
   
   The following arguments are similar to those in the previous section.
   For a non-empty open subset $U\subset G/H$, it follows from \eqref{eq-3.2.6} and \eqref{eq-4.1.1} that 
\[
   \mathcal{V}(G\times_\rho{\sf V})_U\subset\mathcal{C}(\pi^{-1}(U),{\sf V}),
\] 
and moreover
\begin{proposition}\label{prop-4.3.1}
   With respect the Fr\'{e}chet metric $d$ in \eqref{eq-4.1.3}, 
\begin{enumerate}
\item[{\rm (1)}]
   the metric space $(\mathcal{V}(G\times_\rho{\sf V})_U,d)$ is complete,
\item[{\rm (2)}] 
   the addition $\mathcal{V}(G\times_\rho{\sf V})_U\times\mathcal{V}(G\times_\rho{\sf V})_U\ni(\xi_1,\xi_2)\mapsto\xi_1+\xi_2\in\mathcal{V}(G\times_\rho{\sf V})_U$ and the scalar multiplication $\mathbb{C}\times\mathcal{V}(G\times_\rho{\sf V})_U\ni(\alpha,\xi)\mapsto\alpha\xi\in\mathcal{V}(G\times_\rho{\sf V})_U$ are continuous,   
\item[{\rm (3)}]
   the metric topology for $(\mathcal{V}(G\times_\rho{\sf V})_U,d)$ coincides with the topology of uniform convergence on compact sets$;$ besides, it also coincides with the locally convex topology determined by a countable number of seminorms.
\end{enumerate}
   Therefore $\mathcal{V}(G\times_\rho{\sf V})_U$ is a Fr\'{e}chet space over $\mathbb{C}$.
\end{proposition} 
\begin{proof}
   By Theorem \ref{thm-4.1.11} and $\mathcal{V}(G\times_\rho{\sf V})_U\subset\mathcal{C}(\pi^{-1}(U),{\sf V})$ we conclude that $d$ is a metric on $\mathcal{V}(G\times_\rho{\sf V})_U$, and that both (2) and (3) hold.\par

   Let us prove that $(\mathcal{V}(G\times_\rho{\sf V})_U,d)$ is complete. 
   Let $\{\xi_n\}_{n=1}^\infty$ be a given Cauchy sequence in $(\mathcal{V}(G\times_\rho{\sf V})_U,d)$. 
   By \eqref{eq-3.2.6} and \eqref{eq-2.5.3}, one has
\[
   \mathcal{V}(G\times_\rho{\sf V})_U\subset\mathcal{V}^0(G\times_\rho{\sf V})_U,
\] 
where we regard ${\sf V}$ as a real vector space here.
   Therefore $\{\xi_n\}_{n=1}^\infty\subset\mathcal{V}^0(G\times_\rho{\sf V})_U$ and Proposition \ref{prop-4.2.1}-(1) assure the existence of a unique $\xi\in\mathcal{V}^0(G\times_\rho{\sf V})_U$ satisfying $\displaystyle{\lim_{n\to\infty}d(\xi,\xi_n)=0}$. 
   So, we can get the conclusion if one confirms that 
\begin{equation}\label{eq-1}\tag*{\textcircled{1}}
   \mbox{the continuous mapping $\xi:\pi^{-1}(U)\to{\sf V}=\mathbb{C}^m$ is holomorphic}.
\end{equation}
   For an arbitrary $g\in\pi^{-1}(U)$, we take a holomorphic coordinate neighborhood $(P,\psi)$ of $g$ such that (i) $z^j\bigl(\psi(g)\bigr)=0$ for all $1\leq j\leq N:=\dim_\mathbb{C}\pi^{-1}(U)$ and (ii) $\psi$ is a homeomorphism of $P$ onto an open subset of $\mathbb{C}^N$ defined by $|z^1|<r,|z^2|<r,\dots,|z^N|<r$ for some $r>0$. 
   Let us express $\xi\circ\psi^{-1}:\psi(P)\to\mathbb{C}^m$ as
\[
   (\xi\circ\psi^{-1})(z^1,z^2,\dots,z^N)=\bigl(\xi^1(z^1,z^2,\dots,z^N),\dots,\xi^m(z^1,z^2,\dots,z^N)\bigr),
\]
and set $D:=\{z\in\mathbb{C}:|z|<r\}$.
   If one shows that for each $1\leq i\leq m$ and $1\leq j\leq N$
\begin{equation}\label{eq-1'}\tag*{\textcircled{1}$'$}
   \mbox{the continuous function $D\ni z^j\mapsto \xi^i(\cdots,z^j,\cdots)\in\mathbb{C}$ of one variable is holomorphic},
\end{equation}
then we can conclude \ref{eq-1} by $\psi(P)\ni(z^1,z^2,\dots,z^N)\mapsto\xi^i(z^1,z^2,\dots,z^N)\in\mathbb{C}$ being continuous and $\psi(P)=\underbrace{D\times D\times\cdots\times D}_{N}$. 
   In order to show \ref{eq-1'} we first express $\xi_n\circ\psi^{-1}:\psi(P)\to\mathbb{C}^m$ as
\[
   (\xi_n\circ\psi^{-1})(z^1,z^2,\dots,z^N)=\bigl(\xi_n^1(z^1,z^2,\dots,z^N),\dots,\xi_n^m(z^1,z^2,\dots,z^N)\bigr),
\]
$n\in\mathbb{N}$. 
   Notice that each $\xi_n^i(z^1,z^2,\dots,z^N):\psi(P)\to\mathbb{C}$ is a holomorphic function ($1\leq i\leq m$, $n\in\mathbb{N}$) by virtue of $\xi_n\in\mathcal{V}(G\times_\rho{\sf V})_U$.
   Remark that $\displaystyle{\lim_{n\to\infty}d(\xi,\xi_n)=0}$ and the topology for $(\mathcal{V}(G\times_\rho{\sf V})_U,d)$ coincides with the topology of uniform convergence on compact sets. 
   Substituting a sufficiently small $r'>0$ for $r$ (if necessary), one can assume that $\{\xi_n\circ\psi^{-1}\}_{n=1}^\infty$ is uniformly convergent to $\xi\circ\psi^{-1}$ on the set $\psi(P)$---that is, for any $\epsilon>0$ there exists a $K\in\mathbb{N}$ such that $k\geq K$ implies 
\[
   \mbox{$\big\|(\xi\circ\psi^{-1})(\bm{z})-(\xi_k\circ\psi^{-1})(\bm{z})\big\|<\epsilon$ for all $\bm{z}\in\psi(P)$},
\]
where $\|\bm{w}\|:=\sqrt{|w_1|^2+\cdots+|w_m|^2}$ for $\bm{w}=(w_1,\dots,w_m)\in\mathbb{C}^m$. 
   Consequently it follows that for each $1\leq i\leq m$
\[
   \mbox{$\{\xi^i_n(z^1,z^2,\dots,z^N)\}_{n=1}^\infty$ is uniformly convergent to $\xi^i(z^1,z^2,\dots,z^N)$ on $\psi(P)$},
\]
and in particular, for each $1\leq i\leq m$ and $1\leq j\leq N$
\begin{equation}\label{eq-a}\tag{a}
   \mbox{$\{\xi^i_n(\cdots,z^j,\cdots)\}_{n=1}^\infty$ is uniformly convergent to $\xi^i(\cdots,z^j,\cdots)$ on $D$}.  
\end{equation}
   Now, we are in a position to demonstrate \ref{eq-1'}. 
   Fix any $1\leq i\leq m$ and $1\leq j\leq N$.
   Let $C$ be any piecewise differentiable closed curve of class $C^1$ which is contained in $D$. 
   Then we have
\[
   \int_{C}\xi^i(\cdots,z^j,\cdots)dz^j
   \stackrel{\eqref{eq-a}}{=}\lim_{n\to\infty}\int_{C}\xi_n^i(\cdots,z^j,\cdots)dz^j
   =0
\]
because the function $D\ni z^j\mapsto\xi_n^i(\cdots,z^j,\cdots)\in\mathbb{C}$ is holomorphic for each $n\in\mathbb{N}$, its domain $D$ is a star region and Cauchy's integral theorem.\footnote{e.g.\ \begin{CJK}{UTF8}{min}定理 2.2 in 杉浦\end{CJK} \cite[p.249]{Su1}.}
   Accordingly we obtain \ref{eq-1'} from Morera's theorem.\footnote{e.g.\ \begin{CJK}{UTF8}{min}定理 3.4 in 杉浦\end{CJK} \cite[p.258]{Su1}.}
\end{proof}

\section{An appendix (complete metric spaces, the Baire category theorem)}\label{sec-4.4}
   In Sections \ref{sec-4.1}, \ref{sec-4.2} and \ref{sec-4.3} we have dealt with metric spaces $\mathcal{C}(X,{\sf V})$, $\mathcal{V}^0(G\times_\rho{\sf V})_U$ and $\mathcal{V}(G\times_\rho{\sf V})_U$, respectively. 
   To these spaces we can apply the following proposition:
\begin{proposition}\label{prop-4.4.1}
   Let $X=(X,d)$ be a complete, metric space.
   If $\{F_n\}_{n=1}^\infty$ is a sequence of closed subsets of $X$ and $X=\bigcup_{n=1}^\infty F_n$, then there exists an $N\in\mathbb{N}$ such that $F_N$ includes a non-empty open subset of $X$.  
\end{proposition}
\begin{proof}
   We use proof by contradiction. 
   Suppose that each $F_n$ cannot include any non-empty open subset of $X$ ($n\in\mathbb{N}$). 
   Then, $X\neq F_1$ follows, and $X-F_1$ is a non-empty open subset of $X$. 
   Thus there exist an $a_1\in X-F_1$ and an $r_1>0$ satisfying 
\[
\begin{array}{ll}
   r_1<1/2, & B(a_1,r_1):=\{x\in X \,|\, d(x,a_1)<r_1\}\subset X-F_1.
\end{array}
\]  
   Since $B(a_1,r_1/2)$ is a non-empty open subset of $X$, the supposition assures that $(X-F_2)\cap B(a_1,r_1/2)$ is a non-empty open subset of $X$. 
   Thus there exist an $a_2\in(X-F_2)\cap B(a_1,r_1/2)$ and an $r_2>0$ satisfying 
\[
\begin{array}{ll}
   r_2<r_1/2, & B(a_2,r_2)\subset (X-F_2)\cap B(a_1,r_1/2).
\end{array}   
\]  
   By repeating the arguments above, one has a sequence $\{a_n\}_{n=1}^\infty\subset X$ and a sequence $\{r_n\}_{n=1}^\infty$ of positive real numbers such that 
\begin{equation}\label{eq-1}\tag*{\textcircled{1}}
\begin{array}{ll}
   r_{n+1}<r_n/2, & B(a_{n+1},r_{n+1})\subset (X-F_{n+1})\cap B(a_n,r_n/2),\, n=1,2,\dots.
\end{array}
\end{equation}
   Here we remark that 
\begin{equation}\label{eq-2}\tag*{\textcircled{2}}
\begin{split}
  \cdots\subset B(a_{n+1},r_{n+1})\subset  \bigl((X-F_{n+1})\cap & B(a_n,r_n/2)\bigr) \subset B(a_n,r_n)\subset \bigl((X-F_n)\cap B(a_{n-1},r_{n-1}/2)\bigr)\\
  &\subset\cdots\subset B(a_2,r_2)\subset\bigl((X-F_2)\cap B(a_1,r_1/2)\bigr)\subset B(a_1,r_1) \subset X-F_1. 
\end{split}
\end{equation}
   The \ref{eq-1} assures that $n\geq m$ implies 
\[
\begin{split}
   d(a_m,a_n)
  &\leq d(a_m,a_{m+1})+d(a_{m+1},a_{m+2})+\cdots+d(a_{n-1},a_n)\\
  &<\dfrac{r_m}{2}+\dfrac{r_{m+1}}{2}+\cdots+\dfrac{r_{n-1}}{2}
   <\dfrac{r_1}{2^m}+\dfrac{r_1}{2^{m+1}}+\cdots+\dfrac{r_1}{2^{n-1}}
   <\dfrac{r_1}{2^{m-1}}
   <\dfrac{1}{2^m}.
\end{split}
\]
   Consequently $\{a_n\}_{n=1}^\infty$ is a Cauchy sequence in $(X,d)$, so there exists a unique $a\in X$ satisfying  
\begin{equation}\label{eq-3}\tag*{\textcircled{3}}
   \lim_{n\to\infty}d(a_n,a)=0.
\end{equation}
   For any $k\in\mathbb{N}$, in terms of \ref{eq-3} there exists a natural number $N_k>k$ such that $n\geq N_k$ implies $d(a_n,a)<r_{k+1}/2$, and then 
\[
   d(a_{k+1},a)\leq d(a_{k+1},a_n)+d(a_n,a)<r_{k+1}
\]
because we can deduce $d(a_{k+1},a_n)<r_{k+1}/2$ from $n\geq k+1$ and \ref{eq-1}. 
   Therefore it follows from $a\in B(a_{k+1},r_{k+1})$ and \ref{eq-2} that $a\not\in\bigcup_{j=1}^{k+1}F_j$ for all $k\in\mathbb{N}$. 
   This and $X=\bigcup_{n=1}^\infty F_n$ yield $a\not\in X$, which is a contradiction.
\end{proof}

\chapter{Left-invariant Haar measures}\label{ch-5}
   In this chapter we deal with left-invariant Haar measures on topological groups. 
   The setting of this chapter is as follows: 
\begin{itemize}  
\item    
   $G$ is a locally compact Hausdorff topological group.
\end{itemize}
   Besides, we utilize the following notation:
\begin{itemize}  
\item 
   $\mathcal{T}$ : the set of open subsets of $G$, 
\item 
   $\mathscr{B}$ : the $\sigma$-algebra on $G$ generated by $\mathcal{T}$, i.e., the Borel field on $G$,
\item 
   $\mathcal{C}$ : the set of compact subsets of $G$,
\item 
   $\mathcal{U}$ : the set of open neighborhoods of the unit element $e\in G$,
\item
   $W^\circ$ : the interior of a subset $W\subset G$,
\item
   $l_g$ (resp.\ $r_g$) : $2^G\to 2^G$, $A\mapsto gA$ (resp.\ $Ag$), for $g\in G$, 
\item
   $c_B$ : the characteristic function of a subset $B\subset G$, 
\item 
   $\mathscr{C}_{\geq 0}(G,\mathbb{R})
   :=\big\{\begin{array}{@{}c|c@{}}
   f:G\to\mathbb{R} 
   & \mbox{(1) $f$ is continuous, (2) $\operatorname{supp}(f)\subset G$ is compact, (3) $f(g)\geq 0$ for all $g\in G$}
   \end{array}\big\}$.
\end{itemize} 
   Here $2^G$ stands for the power set of $G$.

\section{Definition of left-invariant Haar measure}\label{sec-5.1}
   We first give a lemma, next state Theorem \ref{thm-5.1.2} and then recall the definition of left-invariant Haar measure.

\begin{lemma}\label{lem-5.1.1}
\begin{enumerate}
\item[]
\item[{\rm (1)}] 
   $\mathcal{C}\subset\mathscr{B}$.
\item[{\rm (2)}]
   $l_g(\mathscr{B})\subset\mathscr{B}$, $r_g(\mathscr{B})\subset\mathscr{B}$ for each $g\in G$.
\end{enumerate}
\end{lemma}
\begin{proof}
   (1).
   Since $G$ is a Hausdorff space, every $C\in\mathcal{C}$ is a closed subset of $G$, and we have (1).\par
  
   (2). 
   Since the left translation $L_{g^{-1}}:G\to G$ is a homeomorphism, $l_{g^{-1}}(\mathscr{B})$ is a $\sigma$-algebra on $G$ and includes $\mathcal{T}=l_{g^{-1}}(\mathcal{T})$. 
   Hence we see that $\mathscr{B}\subset l_{g^{-1}}(\mathscr{B})$ because $\mathscr{B}$ is the least $\sigma$-algebra on $G$ including $\mathcal{T}$.
   It follows from $\mathscr{B}\subset l_{g^{-1}}(\mathscr{B})$ that $l_g(\mathscr{B})\subset\mathscr{B}$. 
   Similarly, $r_g(\mathscr{B})\subset\mathscr{B}$.   
\end{proof} 

   Lemma \ref{lem-5.1.1} ensures that the conditions (p6), (p7) in the following theorem are well-defined:
\begin{theorem}[{cf.\ Haar \cite{Ha}, von Neumann \cite{vN}}]\label{thm-5.1.2}
   There exists a set function $\mu:\mathscr{B}\to\mathbb{R}\amalg\{\infty\}$ such that 
\begin{enumerate}
\item[{\rm (p1)}]
   $0\leq\mu(A)\leq\infty$ for all $A\in\mathscr{B}$, 
\item[{\rm (p2)}]
   $\mu(\emptyset)=0$,
\item[{\rm (p3)}] 
   $A_n\in\mathscr{B}$ $(n=1,2,\dots)$, $A_j\cap A_k=\emptyset$ $(j\neq k)$ imply $\mu(\coprod_{n=1}^\infty A_n)=\sum_{n=1}^\infty\mu(A_n)$,
\item[{\rm (p4)}] 
   $\mu(A)=\inf\{\mu(O):\mbox{$O\in\mathcal{T}$, $A\subset O$}\}$ for every $A\in\mathscr{B}$,
\item[{\rm (p5)}] 
   $\mu(O)=\sup\{\mu(C):\mbox{$C\in\mathcal{C}$, $C\subset O$}\}$ for every $O\in\mathcal{T}$,  
\item[{\rm (p6)}] 
   $\mu(C)<\infty$ for each $C\in\mathcal{C}$,
\item[{\rm (p7)}]  
   $\mu(gA)=\mu(A)$ for all $(g,A)\in G\times\mathscr{B}$, \quad {\rm (left-invariant)}
\item[{\rm (p8)}] 
   $\mu(O)>0$ for each $O\in\mathcal{T}-\{\emptyset\}$.
\end{enumerate}
   In addition, the existence of $\mu$ above is unique up to a positive multiplicative constant whenever $G$ satisfies the second countability axiom.
\end{theorem} 

\begin{remark}\label{rem-5.1.3}
   Here are comments on Theorem \ref{thm-5.1.2}.
\begin{enumerate}
\item[(i)] 
   The conditions (p1), (p2) and (p3) are just the conditions for $\mu$ to be a measure on $\mathscr{B}$.    
\item[(ii)]
   The conditions (p3) and (p6) imply that for a $g\in G$,
\[
   \left\{\begin{array}{l}
   \mu(\{g\})=\mu(\{g\}\amalg\emptyset)=\mu(\{g\})+\mu(\emptyset),\\
   \mu(\{g\})<\infty.
   \end{array}\right.
\]
   Accordingly, these conditions imply (p2) $\mu(\emptyset)=0$. 
\item[(iii)] 
   It seems that one can omit the supposition ``$G$ satisfies the second countability axiom'' from this theorem. 
   e.g.\ Theorem 9.2.6 in Cohn \cite[p.290]{Co}. 
\end{enumerate}
\end{remark}

   We will prove this theorem in the next section.

\begin{definition}\label{def-5.1.4}
   A measure $\mu$ on $\mathscr{B}$ is called a non-zero {\it left-invariant Haar measure} on $G$,\index{left-invariant Haar measure@left-invariant Haar measure\dotfill} if it satisfies the five conditions (p4) through (p8) in Theorem \ref{thm-5.1.2}.
\end{definition}

\section{Proof of Theorem 5.1.2}\label{sec-5.2}
   We take four steps to prove Theorem \ref{thm-5.1.2}. 
   In Subsection \ref{subsec-5.2.1} we first define a non-negative integer $\sharp(C:W)$ and a set function $h_U:\mathcal{C}\to\mathbb{Q}$. 
   In Subsection \ref{subsec-5.2.2} we get a set function $h_\bullet:\mathcal{C}\to\mathbb{R}$ by taking $\sharp(C:W)$ and $h_U$ into consideration.
   In Subsection \ref{subsec-5.2.3} we construct a Carath\'{e}odory outer measure $\mu^*$ on $G$ from the function $h_\bullet$. 
   Finally in Subsection \ref{subsec-5.2.4} we complete the proof of Theorem \ref{thm-5.1.2}. 
   The arguments below will be similar to those in Cohn \cite[Section 9.2]{Co}.

\subsection{Step 1/4, $\sharp(C:W)\in\mathbb{Z}_{\geq 0}$ \& $h_U:\mathcal{C}\to\mathbb{Q}$}\label{subsec-5.2.1}    
   For any $C\in\mathcal{C}$ and any subset $W\subset G$ with $W^\circ\neq\emptyset$, one puts
\begin{equation}\label{eq-5.2.1}
   \sharp(C:W):=\min\{n\in\mathbb{Z}_{\geq 0} \,|\, \mbox{there exist $n$ elements $g_1,g_2,\dots,g_n\in G$ so that $C\subset\bigcup_{i=1}^ng_iW$}\}.
\end{equation}
   This \eqref{eq-5.2.1} is well-defined because $\emptyset\neq\{n\in\mathbb{Z}_{\geq 0} \,|\, \mbox{there exist $n$ elements $g_1,g_2,\dots,g_n\in G$ so that $C\subset\bigcup_{i=1}^ng_iW$}\}$ follows from $C\in\mathcal{C}$ and $W^\circ\neq\emptyset$.
   In view of \eqref{eq-5.2.1} we see that 
\begin{equation}\label{eq-5.2.2} 
\begin{array}{ll}
   \sharp(C:W)\in\mathbb{Z}_{\geq 0}; & \mbox{$C=\emptyset$ if and only if $\sharp(C:W)=0$}.
\end{array}
\end{equation}
   Since $G$ is locally compact, there exists a $C_0\in\mathcal{C}$ whose interior is non-empty. 
   By use of this $C_0$ and a given $U\in\mathcal{U}$, let us define a set function $h_U:\mathcal{C}\to\mathbb{Q}$ by 
\begin{equation}\label{eq-5.2.3} 
   \mbox{$h_U(C):=\dfrac{\sharp(C:U)}{\sharp(C_0:U)}$ for $C\in\mathcal{C}$},  
\end{equation}
where we remark that \eqref{eq-5.2.3} is well-defined due to \eqref{eq-5.2.2}, $C_0\neq\emptyset$ and $U\in\mathcal{U}$.
   The above $h_U$ has the following properties:

\begin{proposition}\label{prop-5.2.4}
   For any $U\in\mathcal{U}$ and $C,C_1,C_2\in\mathcal{C}$,
\begin{enumerate}
\item[{\rm (i)}]
   $0\leq h_U(C)\leq\sharp(C:C_0)\in\mathbb{Z}$, 
\item[{\rm (ii)}]
   $h_U(\emptyset)=0$, 
\item[{\rm (iii)}]
   $h_U(C_0)=1$, 
\item[{\rm (iv)}]
   $h_U(gC)=h_U(C)$ for all $g\in G$, 
\item[{\rm (v)}]
   $C_1\subset C_2$ implies $h_U(C_1)\leq h_U(C_2)$,
\item[{\rm (vi)}]
   $h_U(C_1\cup C_2)\leq h_U(C_1)+h_U(C_2)$,
\item[{\rm (vii)}]
   $C_1U^{-1}\cap C_2U^{-1}=\emptyset$ implies $h_U(C_1\cup C_2)=h_U(C_1)+h_U(C_2)$. 
   Here $U^{-1}:=\{u^{-1} \,|\, u\in U\}$.  
\end{enumerate} 
\end{proposition}
\begin{proof}
   (i). 
   It is enough to show $h_U(C)\leq\sharp(C:C_0)$ because of \eqref{eq-5.2.2} and \eqref{eq-5.2.3}. 
   Let $n:=\sharp(C:C_0)$, $m:=\sharp(C_0:U)$. 
   Then, by \eqref{eq-5.2.1} there exist $n$ elements $g_1,g_2,\dots,g_n\in G$ and $m$ elements $h_1,h_2,\dots,h_m\in G$ such that $C\subset\bigcup_{i=1}^ng_iC_0$ and $C_0\subset\bigcup_{j=1}^mh_jU$, respectively. 
   Accordingly $C\subset\bigcup_{i=1}^n\bigcup_{j=1}^mg_ih_jU$, and so \eqref{eq-5.2.1} implies
\[
   \sharp(C:U)\leq nm=\sharp(C:C_0)\sharp(C_0:U).
\]
   This, \eqref{eq-5.2.3} and $\sharp(C_0:U)>0$ yield $h_U(C)\leq\sharp(C:C_0)$.\par
   
   (ii), (iii) are immediate from \eqref{eq-5.2.2} and \eqref{eq-5.2.3}.\par
   
   (iv). 
   By \eqref{eq-5.2.3} it suffices to show $\sharp(gC:U)=\sharp(C:U)$.
   Let $k:=\sharp(gC:U)$, $\ell:=\sharp(C:U)$. 
   Then, by \eqref{eq-5.2.1} there exist $g_1,g_2,\dots,g_k,h_1,h_2,\dots,h_\ell\in G$ such that $gC\subset\bigcup_{a=1}^kg_aU$, $C\subset\bigcup_{b=1}^\ell h_bU$. 
   On the one hand; from $gC\subset\bigcup_{a=1}^kg_aU$ we obtain $C\subset\bigcup_{a=1}^k(g^{-1}g_a)U$, and hence $\ell=\sharp(C:U)\leq k$ by \eqref{eq-5.2.1}. 
   On the other hand; from $C\subset\bigcup_{b=1}^\ell h_bU$ one obtains $gC\subset\bigcup_{b=1}^\ell(gh_b)U$, and $k=\sharp(gC:U)\leq\ell$. 
   Therefore $k=\ell$ holds, namely $\sharp(gC:U)=\sharp(C:U)$.\par
   
   (v). 
   By \eqref{eq-5.2.3} and $\sharp(C_0:U)>0$ it suffices to show $\sharp(C_2:U)\geq\sharp(C_1:U)$.
   The supposition ``$C_1\subset C_2$'' implies that 
\begin{multline*}
   \{n\in\mathbb{Z}_{\geq 0} \,|\, \mbox{there exist $n$ elements $g_1,g_2,\dots,g_n\in G$ so that $C_2\subset\bigcup_{i=1}^ng_iU$}\}\\
   \subset 
   \{m\in\mathbb{Z}_{\geq 0} \,|\, \mbox{there exist $m$ elements $h_1,h_2,\dots,h_m\in G$ so that $C_1\subset\bigcup_{j=1}^mh_jU$}\},   
\end{multline*}
and hence 
\begin{multline*}
   \min\{n\in\mathbb{Z}_{\geq 0} \,|\, \mbox{there exist $n$ elements $g_1,g_2,\dots,g_n\in G$ so that $C_2\subset\bigcup_{i=1}^ng_iU$}\}\\
   \geq 
   \min\{m\in\mathbb{Z}_{\geq 0} \,|\, \mbox{there exist $m$ elements $h_1,h_2,\dots,h_m\in G$ so that $C_1\subset\bigcup_{j=1}^mh_jU$}\}.   
\end{multline*}
   Consequently we deduce $\sharp(C_2:U)\geq\sharp(C_1:U)$ by \eqref{eq-5.2.1}.\par

   (vi). 
   By \eqref{eq-5.2.3} and $\sharp(C_0:U)>0$ it suffices to show $\sharp(C_1\cup C_2:U)\leq\sharp(C_1:U)+\sharp(C_2:U)$. 
   Let $m:=\sharp(C_1:U)$, $n:=\sharp(C_2:U)$. 
   Then, by \eqref{eq-5.2.1} there exist $h_1,h_2,\dots,h_m,g_1,g_2,\dots,g_n\in G$ such that $C_1\subset\bigcup_{j=1}^mh_jU$, $C_2\subset\bigcup_{i=1}^ng_iU$; and it follows that $C_1\cup C_2\subset \bigcup_{j=1}^mh_jU\cup\bigcup_{i=1}^ng_iU$. 
   So, \eqref{eq-5.2.1} yields $\sharp(C_1\cup C_2:U)\leq m+n=\sharp(C_1:U)+\sharp(C_2:U)$.\par
   
   (vii). 
   By (vi), \eqref{eq-5.2.3} and $\sharp(C_0:U)>0$ it suffices to show $\sharp(C_1:U)+\sharp(C_2:U)\leq\sharp(C_1\cup C_2:U)$.
   Let $\ell:=\sharp(C_1\cup C_2:U)$. 
   Then, there exist $g_1,g_2,\dots,g_\ell\in G$ such that 
\[
   (C_1\cup C_2)\subset\bigcup_{a=1}^\ell g_aU
\]
by \eqref{eq-5.2.1}.
   Here, the supposition ``$C_1U^{-1}\cap C_2U^{-1}=\emptyset$'' enables us to assert that each set $g_aU$ meets at most one of $C_1$ and $C_2$. 
   Therefore one can separate $\{g_a\}_{a=1}^\ell$ into two pieces $\{h_b\}_{b=1}^n$ and $\{k_c\}_{c=1}^m$ so that $C_1\subset\bigcup_{b=1}^nh_bU$ and $C_2\subset\bigcup_{c=1}^mk_cU$.
   This and \eqref{eq-5.2.1} imply $\sharp(C_1:U)+\sharp(C_2:U)\leq n+m=\ell=\sharp(C_1\cup C_2:U)$.  
\end{proof}

\subsection{Step 2/4, $h_\bullet:\mathcal{C}\to\mathbb{R}$}\label{subsec-5.2.2}
   Our goal in this subsection is to demonstrate Proposition \ref{prop-5.2.11}.\par

   For each $C\in\mathcal{C}$ we define a closed (finite) interval $I_C\subset\mathbb{R}$ as 
\[
   I_C:=\big[0,\sharp(C:C_0)\big]
\]
(cf.\ \eqref{eq-5.2.1}), and denote by $X$ the product space of the family $\{I_C\}_{C\in\mathcal{C}}$ of topological spaces. 
   Tikhonov's product theorem implies that 
\begin{equation}\label{eq-5.2.5} 
   \mbox{the topological space $X=\Pi_{C\in\mathcal{C}}I_C$ is compact}.
\end{equation}

\begin{remark}\label{rem-5.2.6}
   In general, one can identify ``a set function $h:\mathcal{C}\to\mathbb{R}$ such that $h(C)\in I_C$ for all $C\in\mathcal{C}$'' with ``an element of $X=\Pi_{C\in\mathcal{C}}I_C$'' via $h\mapsto\bigl(h(C)\bigr)_{C\in\mathcal{C}}\in X$.
   Under this identification we construct arguments hereafter.
\end{remark}

   Proposition \ref{prop-5.2.4}-(i) and Remark \ref{rem-5.2.6} allow us to assume that $h_U\in X$ for all $U\in\mathcal{U}$. 
   For this reason, we can define a closed subset $S(V)\subset X$ by 
\begin{equation}\label{eq-5.2.7}
   S(V):=\overline{\{h_U \,|\, \mbox{$U\in\mathcal{U}$, $U\subset V$}\}} \quad\mbox{(the closure in $X$)}
\end{equation}
for $V\in\mathcal{U}$.

\begin{lemma}\label{lem-5.2.8}
   There exists an $h_\bullet\in\bigcap_{V\in\mathcal{U}}S(V)$.
\end{lemma}
\begin{proof}
   The family $\{S(V)\}_{V\in\mathcal{U}}$ consists of closed subsets of $X$. 
   It has the finite intersection property. 
   Indeed; for any finite elements $V_1,\dots,V_m\in\mathcal{U}$ one sees that $V:=\bigcap_{i=1}^mV_i$ belongs to $\mathcal{U}$, and moreover $h_V\in\bigcap_{i=1}^mS(V_i)$; hence $\{S(V)\}_{V\in\mathcal{U}}$ has the desired property.
   Consequently we deduce $\bigcap_{V\in\mathcal{U}}S(V)\neq\emptyset$ by \eqref{eq-5.2.5}.
\end{proof}

   We prepare two lemmas for proving Proposition \ref{prop-5.2.11}. 
\begin{lemma}\label{lem-5.2.9}
   For each $C\in\mathcal{C}$, the following two items hold$:$ 
\begin{enumerate}
\item[{\rm (1)}] 
   $\Pr_C:X\to I_C$, $h\mapsto h(C)$ is continuous$;$ in particular, it is a continuous mapping of $X$ into $\mathbb{R}$.
   cf.\ Remark {\rm \ref{rem-5.2.6}}.
\item[{\rm (2)}]
   $0\leq h(C)\leq \sharp(C:C_0)$ for all $h\in X$. 
\end{enumerate}
\end{lemma}
\begin{proof}
   (1). 
   $X=\Pi_{C\in\mathcal{C}}I_C$ is the product topological space, and so the projection $X\ni h\mapsto h(C)\in I_C$ is continuous.\par
   
   (2) follows by $h(C)=\Pr_C(h)\in I_C=\big[0,\sharp(C:C_0)\big]$.
\end{proof}

\begin{lemma}\label{lem-5.2.10}
   For any $C_1,C_2\in\mathcal{C}$ with $C_1\cap C_2=\emptyset$, there exist $(O_1,V_1),(O_2,V_2)\in\mathcal{T}\times\mathcal{U}$ which satisfy $O_1\cap O_2=\emptyset$, $C_1V_1\subset O_1$ and $C_2V_2\subset O_2$.
\end{lemma}
\begin{proof}
   In case of $C_1=\emptyset$ we can get the conclusion by setting $O_1:=\emptyset$ and $O_2=V_1=V_2:=G$. 
   Similarly one can do so in case of $C_2=\emptyset$.\par

   Now, let us suppose that $C_1\neq\emptyset$ and $C_2\neq\emptyset$.
   On the one hand; $C_2$ is a closed subset of $G$ since $G$ is a Hausdorff space and $C_2\in\mathcal{C}$. 
   On the other hand; $G$ is a regular space since $G$ is a Hausdorff topological group.
   Consequently, for an arbitrary $g\in C_1$, there exist $P_g,Q_g\in\mathcal{T}$ such that 
\[
   \begin{array}{lll}
   g\in P_g, & C_2\subset Q_g, & P_g\cap Q_g=\emptyset,
   \end{array}
\]
where we remark that $C_1\cap C_2=\emptyset$, $g\in C_1$ lead to $g\not\in C_2$. 
   In terms of $C_1\subset\bigcup_{g\in C_1}P_g$ and $C_1\in\mathcal{C}$, there exist finite elements $g_1,\dots,g_n\in C_1$ such that $C_1\subset\bigcup_{i=1}^nP_{g_i}$. 
   Setting $O_1:=\bigcup_{i=1}^nP_{g_i}$ and $O_2:=\bigcap_{i=1}^nQ_{g_i}$ we deduce 
\[
   \begin{array}{llll}
   O_1,O_2\in\mathcal{T}, & O_1\cap O_2=\emptyset, & C_1\subset O_1, & C_2\subset O_2.
   \end{array}
\] 
   The rest of proof is to confirm that for each $a=1,2$, there exists a $V_a\in\mathcal{U}$ satisfying $C_aV_a\subset O_a$. 
   Fix any element $h\in C_a$. 
   From $h\in C_a\subset O_a\in\mathcal{T}$ we obtain a $W_h\in\mathcal{U}$ such that 
\[
   hW_h\subset O_a.
\] 
   Moreover, since the mapping $G\times G\ni(g_1,g_2)\mapsto g_1g_2\in G$ is continuous at $(e,e)$ and $W_h$ is an open neighborhood of $e\in G$, there exists a $U_h\in\mathcal{U}$ satisfying 
\[
   U_hU_h\subset W_h.
\]
   In terms of $C_a\subset\bigcup_{h\in C_a}hU_h$ and $C_a\in\mathcal{C}$, there exist finite elements $h_1,\dots,h_\ell\in C_a$ such that $C_a\subset\bigcup_{j=1}^\ell h_jU_{h_j}$. 
   Now, let $V_a:=\bigcap_{j=1}^\ell U_{h_j}$. 
   Then it follows that $V_a\in\mathcal{U}$; besides, for any $k\in C_a$ ($\subset\bigcup_{j=1}^\ell h_jU_{h_j}$) there exists a $1\leq i\leq\ell$ such that $k\in h_iU_{h_i}$, and hence 
\[
   kV_a\subset h_iU_{h_i}V_a\subset h_iU_{h_i}U_{h_i}\subset h_iW_{h_i}\subset O_a.
\]
   This implies $C_aV_a\subset O_a$.   
\end{proof}

   Now, let us prove 
\begin{proposition}\label{prop-5.2.11}
   For any $C,C_1,C_2\in\mathcal{C}$,
\begin{enumerate}
\item[{\rm (i)}]
   $0\leq h(C)\leq\sharp(C:C_0)\in\mathbb{Z}$ for all $h\in X=\Pi_{C\in\mathcal{C}}I_C$,
\item[{\rm (ii)}]
   $h(\emptyset)=0$ for all $h\in X$,
\item[{\rm (iii)}]
   $h(C_0)=1$ for all $h\in S(G)$,
\item[{\rm (iv)}]
   $h(gC)=h(C)$ for all $g\in G$ and $h\in S(G)$,
\item[{\rm (v)}]
   $C_1\subset C_2$ implies $h(C_1)\leq h(C_2)$ for all $h\in S(G)$,
\item[{\rm (vi)}]
   $h(C_1\cup C_2)\leq h(C_1)+h(C_2)$ for all $h\in S(G)$,
\item[{\rm (vii)}]
   $C_1\cap C_2=\emptyset$ implies $h_\bullet(C_1\cup C_2)=h_\bullet(C_1)+h_\bullet(C_2)$.
\end{enumerate} 
   Remark here that $h_\bullet\in S(G)\subset X$, and {\rm (i)} through {\rm (vii)} hold for $h_\bullet$.
   cf.\ \eqref{eq-5.2.7}.
\end{proposition}   
\begin{proof}
   (i) follows from \eqref{eq-5.2.2} and Lemma \ref{lem-5.2.9}-(2).\par
   
   (ii). 
   By Lemma \ref{lem-5.2.9}-(1), $\Pr_{\emptyset}(X)\subset I_{\emptyset}=\big[0,\sharp(\emptyset:C_0)\big]\stackrel{\eqref{eq-5.2.2}}{=}\{0\}$. 
   Hence, $h(\emptyset)=\Pr_{\emptyset}(h)=0$ for all $h\in X$.\par

   (iii). 
   Lemma \ref{lem-5.2.9}-(1) implies that $\Pr_{C_0}:S(G)\to I_{C_0}$, $h\mapsto h(C_0)$, is continuous. 
   Proposition \ref{prop-5.2.4}-(iii), combined with \eqref{eq-5.2.7}, implies that $\Pr_{C_0}=1$ on a dense subset $\{h_U \,|\, U\in\mathcal{U}\}$ of $S(G)$. 
   Consequently $h(C_0)=\Pr_{C_0}(h)=1$ for all $h\in S(G)$.\par
   
   (iv). 
   By Lemma \ref{lem-5.2.9}-(1) we see that $\Pr_{gC}-\Pr_C:S(G)\to\mathbb{R}$ is continuous. 
   Proposition \ref{prop-5.2.4}-(iv) and \eqref{eq-5.2.7} imply that $\Pr_{gC}-\Pr_C=0$ on the dense subset $\{h_U \,|\, U\in\mathcal{U}\}$ of $S(G)$. 
   Thus $h(gC)-h(C)=(\Pr_{gC}-\Pr_C)(h)=0$ for all $h\in S(G)$.\par
   
   (v). 
   By virtue of Lemma \ref{lem-5.2.9}-(1), Proposition \ref{prop-5.2.4}-(v) and \eqref{eq-5.2.7} we deduce that $\Pr_{C_2}-\Pr_{C_1}:S(G)\to\mathbb{R}$ is continuous, and that $\Pr_{C_2}-\Pr_{C_1}\geq0$ on the dense subset $\{h_U \,|\, U\in\mathcal{U}\}\subset S(G)$. 
   Hence $h(C_2)-h(C_1)=(\Pr_{C_1}-\Pr_{C_2})(h)\geq0$ for all $h\in S(G)$.\par
   
   (vi). 
   One can conclude (vi) by arguments similar to those in the above (v) and Proposition \ref{prop-5.2.4}-(vi).\par
   
   (vii). 
   Since $C_1,C_2\in\mathcal{C}$ with $C_1\cap C_2=\emptyset$, Lemma \ref{lem-5.2.10} assures that there exist $(O_1,V_1),(O_2,V_2)\in\mathcal{T}\times\mathcal{U}$ satisfying 
\[
\begin{array}{lll}
   O_1\cap O_2=\emptyset, & C_1V_1\subset O_1, & C_2V_2\subset O_2.
\end{array}
\]
   By use of $V_1,V_2$, we put $V_3:=V_1\cap V_2$. 
   Then, it follows that $V_3,V_3^{-1}\in\mathcal{U}$; and moreover, $U\subset V_3^{-1}$ and $U\in\mathcal{U}$ imply
\[
   h_U(C_1)+h_U(C_2)-h_U(C_1\cup C_2)=0
\]
because of Proposition \ref{prop-5.2.4}-(vii) and $(C_1U^{-1}\cap C_2U^{-1})\subset(C_1V_3\cap C_2V_3)\subset(C_1V_1\cap C_2V_2)\subset(O_1\cap O_2)=\emptyset$. 
   Consequently we deduce that $(\Pr_{C_1}+\Pr_{C_2}-\Pr_{C_1\cup C_2})(h_U)=0$ for all $h_U\in\{h_U \,|\, \mbox{$U\in\mathcal{U}$, $U\subset V_3^{-1}$}\}$. 
   Furthermore, one verifies that  
\[
   \mbox{$(\Pr_{C_1}+\Pr_{C_2}-\Pr_{C_1\cup C_2})(h)=0$ for all $h\in S(V_3^{-1})$}
\]
because $\Pr_{C_1}+\Pr_{C_2}-\Pr_{C_1\cup C_2}:S(V_3^{-1})\to\mathbb{R}$ is continuous and $\{h_U \,|\, \mbox{$U\in\mathcal{U}$, $U\subset V_3^{-1}$}\}$ is dense in $S(V_3^{-1})$. 
   Therefore we obtain $h_\bullet(C_1)+h_\bullet(C_2)-h_\bullet(C_1\cup C_2)=(\Pr_{C_1}+\Pr_{C_2}-\Pr_{C_1\cup C_2})(h_\bullet)=0$ from $h_\bullet\in\bigl(\bigcap_{V\in\mathcal{U}}S(V)\bigr)\subset S(V_3^{-1})$. 
\end{proof}

\subsection{Step 3/4, $\mu^*:2^G\to\mathbb{R}\amalg\{\infty\}$}\label{subsec-5.2.3}

\begin{lemma}\label{lem-5.2.12}
   Set 
\begin{equation}\label{eq-5.2.13}
   \mbox{$\mu_1^*(O):=\sup\{h_\bullet(C) : \mbox{$C\in\mathcal{C}$, $C\subset O$}\}$ for $O\in\mathcal{T}$};
\end{equation} 
\begin{equation}\label{eq-5.2.14}
   \mbox{$\mu^*(A):=\inf\{\mu_1^*(O) : \mbox{$O\in\mathcal{T}$, $A\subset O$}\}$ for $A\in 2^G$}.
\end{equation} 
   Then $\mu_1^*(O)=\mu^*(O)$ holds for each $O\in\mathcal{T}$.
\end{lemma}
\begin{proof}
   On the one hand; it follows from $O\in\mathcal{T}$, $O\subset O$ that 
\[
   \mu_1^*(O)
   \geq\inf\{\mu_1^*(P) : \mbox{$P\in\mathcal{T}$, $O\subset P$}\}
   \stackrel{\eqref{eq-5.2.14}}{=}\mu^*(O).
\]  
   On the other hand; for an arbitrary $Q\in\mathcal{T}$ with $O\subset Q$, one has $\{h_\bullet(C) : \mbox{$C\in\mathcal{C}$, $C\subset O$}\}\subset\{h_\bullet(K) : \mbox{$K\in\mathcal{C}$, $K\subset Q$}\}$, and so \eqref{eq-5.2.13} yields $\mu_1^*(O)\leq\mu_1^*(Q)$. 
   This enables us to show 
\[
   \mu_1^*(O)
   \leq\inf\{\mu_1^*(Q) : \mbox{$Q\in\mathcal{T}$, $O\subset Q$}\}
   \stackrel{\eqref{eq-5.2.14}}{=}\mu^*(O).   
\] 
   Hence $\mu_1^*(O)=\mu^*(O)$ holds.
\end{proof}

   Our first aim in this subsection is to prove Proposition \ref{prop-5.2.18}, which tells us that the $\mu^*$ in \eqref{eq-5.2.14} is a Carath\'{e}odory outer measure on $G$.
   We are going to confirm three lemmas and conclude Proposition \ref{prop-5.2.18} from them.

\begin{lemma}\label{lem-5.2.15}
   Let $C\in\mathcal{C}$, and let $O_1,O_2,\dots,O_k\in\mathcal{T}$ such that $C\subset\bigcup_{a=1}^kO_a$.
   Then, there exist $C_1,C_2,\dots,C_k\in\mathcal{C}$ such that $C_a\subset O_a$ $(1\leq a\leq k)$ and $C=\bigcup_{a=1}^kC_a$. 
\end{lemma}
\begin{proof}
   We prove this lemma in case of $k=2$, which enables one to get the conclusion by mathematical induction on $k$.\par
   
   Suppose that $C\subset O_1\cup O_2$, where $O_1,O_2\in\mathcal{T}$. 
   Setting $K_a:=C-O_a$ ($a=1,2$), we obtain $K_1,K_2\in\mathcal{C}$ and $K_1\cap K_2=\emptyset$ from the supposition.
   Hence Lemma \ref{lem-5.2.10} assures the existence of $P_1,P_2\in\mathcal{T}$ such that 
\begin{equation}\label{eq-1}\tag*{\textcircled{1}}
\begin{array}{lll}
   P_1\cap P_2=\emptyset, & K_1\subset P_1, & K_2\subset P_2.
\end{array} 
\end{equation}
   Now, let $C_a:=C-P_a$ for $a=1,2$. 
   Then, it follows from \ref{eq-1} that $C_1,C_2\in\mathcal{C}$, $C_a\subset O_a$ ($a=1,2$) and $C=C_1\cup C_2$.   
\end{proof}

\begin{lemma}\label{lem-5.2.16}
   For any $A,A_1,A_2\in 2^G$,
\begin{enumerate}
\item[{\rm (1)}]
   $0\leq\mu^*(A)\leq\infty$, 
\item[{\rm (2)}]
   $\mu^*(\emptyset)=0$, 
\item[{\rm (3)}]
   $A_1\subset A_2$ implies $\mu^*(A_1)\leq\mu^*(A_2)$.  
\end{enumerate}   
\end{lemma}
\begin{proof}
   (1) (resp.\ (2)) follows by \eqref{eq-5.2.13}, \eqref{eq-5.2.14} and Proposition \ref{prop-5.2.11}-(i) (resp.\ -(ii)).\par
   
   (3). 
   From $A_1\subset A_2$ we deduce that $\{\mu_1^*(O) : \mbox{$O\in\mathcal{T}$, $A_2\subset O$}\}\subset\{\mu_1^*(P) : \mbox{$P\in\mathcal{T}$, $A_1\subset P$}\}$, so that $\mu^*(A_2)\geq\mu^*(A_1)$ due to \eqref{eq-5.2.14}.
\end{proof}

\begin{lemma}\label{lem-5.2.17}
   $O_n\in\mathcal{T}$ $(n=1,2,\dots)$ imply $\mu^*(\bigcup_{n=1}^\infty O_n)\leq\sum_{n=1}^\infty\mu^*(O_n)$.
\end{lemma}
\begin{proof}
   For an arbitrary $C\in\mathcal{C}$ with $C\subset\bigcup_{n=1}^\infty O_n$, one can choose a finite subset $\{O_a\}_{a=1}^k\subset\{O_n\}_{n=1}^\infty$ so that $C\subset\bigcup_{a=1}^kO_a$. 
   Then, there exist $C_1,C_2,\dots,C_k\in\mathcal{C}$ such that $C_a\subset O_a$ $(1\leq a\leq k)$ and $C=\bigcup_{a=1}^kC_a$ by Lemma \ref{lem-5.2.15}. 
   Therefore
\[
\begin{split}
   h_\bullet(C)
  &=h_\bullet\Big(\bigcup_{a=1}^kC_a\Big)
   \leq\sum_{a=1}^k h_\bullet(C_a) \quad\mbox{($\because$ Proposition \ref{prop-5.2.11}-(vi), $C_a\in\mathcal{C}$)}\\
  &\leq\sum_{a=1}^k\mu_1^*(O_a) \quad\mbox{($\because$ \eqref{eq-5.2.13}, $C_a\in\mathcal{C}$, $C_a\subset O_a$)}\\
  &\leq\sum_{n=1}^\infty\mu^*(O_n) \quad\mbox{($\because$ Lemma \ref{lem-5.2.12}, $O_a\in\mathcal{T}$, Lemma \ref{lem-5.2.16}-(1))},
\end{split} 
\] 
namely $h_\bullet(C)\leq\sum_{n=1}^\infty\mu^*(O_n)$ for any $C\in\mathcal{C}$ with $C\subset\bigcup_{n=1}^\infty O_n$. 
   This and \eqref{eq-5.2.13} yield $\mu_1^*(\bigcup_{n=1}^\infty O_n)\leq\sum_{n=1}^\infty\mu^*(O_n)$.
   Hence $\mu^*(\bigcup_{n=1}^\infty O_n)\leq\sum_{n=1}^\infty\mu^*(O_n)$ by Lemma \ref{lem-5.2.12} and $\bigcup_{n=1}^\infty O_n\in\mathcal{T}$.
\end{proof}

   We are in a position to prove
\begin{proposition}\label{prop-5.2.18}
   The $\mu^*$ in \eqref{eq-5.2.14} has the following four properties$:$
\begin{enumerate}
\item[{\rm (i)}]
   $0\leq\mu^*(A)\leq\infty$ for all $A\in 2^G$, 
\item[{\rm (ii)}]
   $\mu^*(\emptyset)=0$,
\item[{\rm (iii)}] 
   $A_1\subset A_2$, $A_1,A_2\in 2^G$ imply $\mu^*(A_1)\leq\mu^*(A_2)$,
\item[{\rm (iv)}] 
   $A_n\in 2^G$ $(n=1,2,\dots)$ imply $\mu^*(\bigcup_{n=1}^\infty A_n)\leq\sum_{n=1}^\infty\mu^*(A_n)$.
\end{enumerate}   
\end{proposition}
\begin{proof}
   By virtue of Lemma \ref{lem-5.2.16} it is enough to prove (iv).\par

   (iv). 
   It is clear in the case where there exists a $j\in\mathbb{N}$ such that $\mu^*(A_j)=\infty$. 
   Henceforth, we investigate the case where $\mu^*(A_n)<\infty$ for all $n\in\mathbb{N}$. 
   Let $\epsilon>0$.
   For each $m\in\mathbb{N}$ there exists an $O_m\in\mathcal{T}$ such that 
\[
\begin{array}{ll}
   A_m\subset O_m, & \mu^*(O_m)=\mu_1^*(O_m)<\mu^*(A_m)+\dfrac{\epsilon}{2^m}
\end{array}
\]   
because of $\mu^*(A_m)<\infty$, \eqref{eq-5.2.14} and Lemma \ref{lem-5.2.12}.
   Then, it follows that
\[
\begin{split}
   \mu^*\Big(\bigcup_{m=1}^\infty A_m\Big)
  &\leq\mu^*\Big(\bigcup_{m=1}^\infty O_m\Big) \quad\mbox{($\because$ $\bigcup_{m=1}^\infty A_m\subset\bigcup_{m=1}^\infty O_m$, Lemma \ref{lem-5.2.16}-(3))}\\
  &\leq\sum_{m=1}^\infty\mu^*(O_m) \quad\mbox{($\because$ Lemma \ref{lem-5.2.17}, $O_m\in\mathcal{T}$ ($m=1,2,\dots$))}\\
  &\leq\sum_{m=1}^\infty\Big(\mu^*(A_m)+\dfrac{\epsilon}{2^m}\Big)
  =\Big(\sum_{m=1}^\infty\mu^*(A_m)\Big)+\epsilon,
\end{split}  
\] 
so that $\mu^*(\bigcup_{m=1}^\infty A_m)\leq\sum_{m=1}^\infty\mu^*(A_m)$.
\end{proof}

   Proposition \ref{prop-5.2.18} tells us that the $\mu^*$ in \eqref{eq-5.2.14} is a Carath\'{e}odory outer measure on $G$, so that we can get a $\sigma$-algebra $\mathscr{M}$ on $G$ by setting
\begin{equation}\label{eq-5.2.19}
   \mathscr{M}:=\{A\in 2^G \,|\, \mbox{$\mu^*(B)=\mu^*(B\cap A)+\mu^*(B-A)$ for every $B\in 2^G$}\}.
\end{equation}
   Our second aim is to prove Proposition \ref{prop-5.2.22} below.
   We first show two lemmas and afterwards accomplish the aim.
   
\begin{lemma}\label{lem-5.2.20}
   Let $A\in 2^G$. 
   Then, 
\begin{enumerate}
\item[{\rm (1)}]
   $\mu^*(B)\leq\mu^*(B\cap A)+\mu^*(B-A)$ for every $B\in 2^G;$ 
\item[{\rm (2)}] 
   $\mu^*(B)=\infty$, $B\in 2^G$ imply $\mu^*(B)\geq\mu^*(B\cap A)+\mu^*(B-A)$.  
\end{enumerate}   
\end{lemma} 
\begin{proof}
   (1). 
   By Proposition \ref{prop-5.2.18}-(iv) we have $\mu^*(B)=\mu^*\bigl((B\cap A)\cup(B-A)\bigr)\leq\mu^*(B\cap A)+\mu^*(B-A)$.\par
   
   (2). 
   Trivial.
\end{proof}

\begin{lemma}\label{lem-5.2.21}
   Let $O\in\mathcal{T}$. 
   Then it follows that $\mu^*(P)<\infty$, $P\in\mathcal{T}$ imply $\mu^*(P)\geq\mu^*(P\cap O)+\mu^*(P-O)$.
\end{lemma} 
\begin{proof}
   Take any $\epsilon>0$. 
   By Lemma \ref{lem-5.2.16}-(3) and $(P\cap O)\subset P$, we see that $\mu^*(P\cap O)\leq\mu^*(P)<\infty$. 
   Therefore there exists a $C_1\in\mathcal{C}$ such that 
\begin{equation}\label{eq-1}\tag*{\textcircled{1}}
\begin{array}{ll}
   C_1\subset(P\cap O), & h_\bullet(C_1)>\mu_1^*(P\cap O)-\epsilon=\mu^*(P\cap O)-\epsilon
\end{array}
\end{equation}
because of \eqref{eq-5.2.13}, $P\cap O\in\mathcal{T}$ and Lemma \ref{lem-5.2.12}.
   Since $(P-C_1)\subset P$ one can conclude that there exists a $C_2\in\mathcal{C}$ satisfying    
\begin{equation}\label{eq-2}\tag*{\textcircled{2}}
\begin{array}{ll}
   C_2\subset(P-C_1), & h_\bullet(C_2)>\mu^*(P-C_1)-\epsilon 
\end{array}
\end{equation}
in a similar way.
   Moreover, $C_1\cup C_2\in\mathcal{C}$, $(C_1\cup C_2)\subset P$, \eqref{eq-5.2.13} and Lemma \ref{lem-5.2.12} yield $\mu^*(P)\geq h_\bullet(C_1\cup C_2)$. 
   Hence 
\[
\begin{split}
   \mu^*(P)
  &\geq h_\bullet(C_1\cup C_2)
   =h_\bullet(C_1)+h_\bullet(C_2) \quad\mbox{($\because$ $C_1\cap C_2=\emptyset$, Proposition \ref{prop-5.2.11}-(vii))}\\
  &>\mu^*(P\cap O)+\mu^*(P-C_1)-2\epsilon \quad\mbox{($\because$ \ref{eq-1}, \ref{eq-2})}\\
  &\geq\mu^*(P\cap O)+\mu^*(P-O)-2\epsilon,
\end{split} 
\]
where we remark that $\mu^*(P-C_1)\geq\mu^*(P-O)$ follows from $(P-C_1)\supset(P-O)$ and Lemma \ref{lem-5.2.16}-(3). 
   This $\mu^*(P)>\mu^*(P\cap O)+\mu^*(P-O)-2\epsilon$ assures that $\mu^*(P)\geq\mu^*(P\cap O)+\mu^*(P-O)$ holds.
\end{proof}

   Lemmas \ref{lem-5.2.20} and \ref{lem-5.2.21} allow us to assert 
\begin{proposition}\label{prop-5.2.22}
   The $\sigma$-algebra $\mathscr{M}$ on $G$ includes $\mathscr{B}$.
   cf.\ \eqref{eq-5.2.19}.
\end{proposition}
\begin{proof}
   It is enough to conclude $\mathcal{T}\subset\mathscr{M}$, since $\mathscr{M}$ is a $\sigma$-algebra on $G$ and $\mathscr{B}$ is the least $\sigma$-algebra on $G$ including $\mathcal{T}$.
   From \eqref{eq-5.2.19} and Lemma \ref{lem-5.2.20} one can obtain $\mathcal{T}\subset\mathscr{M}$, provided that the following inequality holds for each $O\in\mathcal{T}$:
\begin{equation}\label{eq-1}\tag*{\textcircled{1}} 
   \mbox{$\mu^*(B\cap O)+\mu^*(B-O)\leq\mu^*(B)$ for any $B\in 2^G$ with $\mu^*(B)<\infty$}.
\end{equation}
   Let us show \ref{eq-1} from now on. 
   Fix any $\epsilon>0$, $O\in\mathcal{T}$, and $B\in 2^G$ with $\mu^*(B)<\infty$. 
   By virtue of $\mu^*(B)<\infty$ and \eqref{eq-5.2.14} we have a $P\in\mathcal{T}$ such that
\begin{equation}\label{eq-a}\tag{a}
\begin{array}{ll}
   B\subset P, & \mu_1^*(P)<\mu^*(B)+\epsilon<\infty.
\end{array}   
\end{equation}
   Since $O,P\in\mathcal{T}$ and $\mu_1^*(P)<\infty$, Lemmas \ref{lem-5.2.21} and \ref{lem-5.2.12} assure 
\begin{equation}\label{eq-b}\tag{b}
   \mu^*(P\cap O)+\mu^*(P-O)\leq\mu^*(P)=\mu_1^*(P).
\end{equation}
   By $B\subset P$ we deduce $(B\cap O)\subset(P\cap O)$ and $(B-O)\subset(P-O)$. 
   Hence Lemma \ref{lem-5.2.16}-(3) implies that 
\begin{equation}\label{eq-c}\tag{c}
   \mu^*(B\cap O)+\mu^*(B-O)\leq\mu^*(P\cap O)+\mu^*(P-O).
\end{equation}
   Consequently \eqref{eq-a}, \eqref{eq-b} and \eqref{eq-c} yield $\mu^*(B\cap O)+\mu^*(B-O)<\mu^*(B)+\epsilon$, which gives rise to \ref{eq-1}.  
\end{proof}

\subsection{Step 4/4, the proof of Theorem \ref{thm-5.1.2}}\label{subsec-5.2.4}

   In this subsection we demonstrate Theorem \ref{thm-5.1.2}. 
   We prove the existence of a left-invariant Haar measure $\mu$ in the first half, and prove the uniqueness of $\mu$ in the latter half (see Propositions \ref{prop-5.2.23} and \ref{prop-5.2.29}).
\paragraph{The existence of Haar measure}
   First of all, let us show 
\begin{proposition}[Existence]\label{prop-5.2.23}
   There exists a set function $\mu:\mathscr{B}\to\mathbb{R}\amalg\{\infty\}$ satisfying the eight conditions {\rm (p1)} through {\rm (p8)} in Theorem {\rm \ref{thm-5.1.2}}.
\end{proposition}
\begin{proof}
   We have already known that the $\mu^*$ in \eqref{eq-5.2.14} is a Carath\'{e}odory outer measure on $G$, and that the inclusion $\mathscr{B}\subset\mathscr{M}$ holds for the class $\mathscr{M}$ of all sets measurable with respect to $\mu^*$ (recall Proposition \ref{prop-5.2.18}, \eqref{eq-5.2.19}, Proposition \ref{prop-5.2.22}).
   Hence we can define a measure $\mu$ on $\mathscr{B}$ as follows: 
\begin{equation}\label{eq-5.2.24}
   \mu:=\mu^*|_\mathscr{B}.
\end{equation}
   Then \eqref{eq-5.2.14} assures that the four conditions (p1) through (p4) in Theorem \ref{thm-5.1.2} hold for the $\mu=\mu^*|_\mathscr{B}$.
   From now on, let us confirm that the rest of conditions also hold.\par
   
   (p5). 
   Fix an $O\in\mathcal{T}$.
   For a given $C\in\mathcal{C}$, it follows from \eqref{eq-5.2.13} that $h_\bullet(C)\leq\mu(Q)$ for any $Q\in\mathcal{T}$ with $C\subset Q$, so that 
\[
   h_\bullet(C)
   \leq\inf\{\mu(Q) : \mbox{$Q\in\mathcal{T}$, $C\subset Q$}\}
   \stackrel{\eqref{eq-5.2.14}}{=}\mu(C).
\]
   Hence we see that 
\begin{equation}\label{eq-1}\tag*{\textcircled{1}} 
   \mbox{$h_\bullet(C)\leq\mu(C)$ for all $C\in\mathcal{C}$}.
\end{equation}
   This \ref{eq-1} gives us  
\[
   \mu(O)
   \stackrel{\eqref{eq-5.2.13}}{=}\sup\{h_\bullet(C) : \mbox{$C\in\mathcal{C}$, $C\subset O$}\}
   \leq\sup\{\mu(C) : \mbox{$C\in\mathcal{C}$, $C\subset O$}\}
   \leq\mu(O);
\]
and therefore (p5) $\mu(O)=\sup\{\mu(C):\mbox{$C\in\mathcal{C}$, $C\subset O$}\}$ holds. 
   Here, we remark that $\mu(C)\leq\mu(O)$ follows from $C\subset O$ and Lemma \ref{lem-5.2.16}-(3).\par 
   
   (p6). 
   Take any $C\in\mathcal{C}$. 
   Since $C\subset G$ is compact and $G$ is a locally compact Hausdorff space, we can construct an $O\in\mathcal{T}$ so that $C\subset O$ and $\overline{O}\in\mathcal{C}$.
   Then, for any $K\in\mathcal{C}$ with $K\subset O$, one has
\[
   h_\bullet(K)\leq h_\bullet(\overline{O})
\]
due to $K,\overline{O}\in\mathcal{C}$, $K\subset\overline{O}$ and Proposition \ref{prop-5.2.11}-(v). 
   Therefore it follows from \eqref{eq-5.2.13} that 
\[
   \mu(O)\leq h_\bullet(\overline{O}).
\]
   In addition, $C\subset O$ and Lemma \ref{lem-5.2.16}-(3) yield 
\[
   \mu(C)\leq\mu(O).
\]
   Consequently we deduce $\mu(C)\leq\mu(O)\leq h_\bullet(\overline{O})\leq\sharp(\overline{O}:C_0)<\infty$ by Proposition \ref{prop-5.2.11}-(i).
   So, (p6) $\mu(C)<\infty$ holds.\par

   (p7). 
   Fix an arbitrary $(g,A)\in G\times\mathscr{B}$. 
   For a given $O\in\mathcal{T}$ we first obtain 
\[
\begin{split}
   \mu(O)
  &\stackrel{\eqref{eq-5.2.13}}{=}\sup\{h_\bullet(C) : \mbox{$C\in\mathcal{C}$, $C\subset O$}\}
   =\sup\{h_\bullet(gC) : \mbox{$C\in\mathcal{C}$, $gC\subset O$}\} \quad\mbox{($\because$ $l_{g^{-1}}(\mathcal{C})=\mathcal{C}$)}\\
  &\quad=\sup\{h_\bullet(gC) : \mbox{$C\in\mathcal{C}$, $C\subset g^{-1}O$}\}
   =\sup\{h_\bullet(C) : \mbox{$C\in\mathcal{C}$, $C\subset g^{-1}O$}\} \quad\mbox{($\because$ Proposition \ref{prop-5.2.11}-(iv))}\\
  &\stackrel{\eqref{eq-5.2.13}}{=}\mu(g^{-1}O). 
\end{split}
\] 
   This $\mu(O)=\mu(g^{-1}O)$ enables us to conclude that 
\[
\begin{split}
   \mu(gA)
  &\stackrel{\eqref{eq-5.2.14}}{=}\inf\{\mu(O) : \mbox{$O\in\mathcal{T}$, $gA\subset O$}\}
   =\inf\{\mu(g^{-1}O) : \mbox{$O\in\mathcal{T}$, $gA\subset O$}\}\\
  &\quad=\inf\{\mu(g^{-1}O) : \mbox{$O\in\mathcal{T}$, $A\subset g^{-1}O$}\}
   =\inf\{\mu(O) : \mbox{$O\in\mathcal{T}$, $A\subset O$}\} \quad\mbox{($\because$ $l_g(\mathcal{T})=\mathcal{T}$)}\\
  &\stackrel{\eqref{eq-5.2.14}}{=}\mu(A).
\end{split} 
\]
   Hence (p7) $\mu(gA)=\mu(A)$ holds.\par

   (p8). 
   Let us use proof by contradiction.
   Suppose that there exists a $P\in\mathcal{T}-\{\emptyset\}$ satisfying $\mu(P)\leq0$.
   On the one hand; from (p1) and $\mu(P)\leq0$ we obtain $\mu(P)=0$. 
   On the other hand; since $P\neq\emptyset$ there exists a $p\in P$. 
   Setting $P':=p^{-1}P$ we conclude 
\[
\begin{array}{ll}
   P'\in\mathcal{U}, & \mu(P')=0
\end{array}
\]
by (p7). 
   For any $C\in\mathcal{C}$, it follows from $P'\in\mathcal{U}$ that $C\subset\bigcup_{c\in C}cP'$, and so there exist finite elements $c_1,\dots,c_k\in C$ such that $C\subset\bigcup_{i=1}^kc_iP'$. 
   Then (p1), Lemma \ref{lem-5.2.16}-(3), Proposition \ref{prop-5.2.18}-(iv) imply that 
\[
   0\leq\mu(C)
   \leq\mu\Big(\bigcup_{i=1}^kc_iP'\Big)
   \leq\sum_{i=1}^k\mu(c_iP')
   \stackrel{{\rm (p7)}}{=}\sum_{i=1}^k\mu(P')
   =0;
\]
in particular, $\mu(C_0)=0$. 
   However, Proposition \ref{prop-5.2.11}-(iii) and \ref{eq-1} yield $1=h_\bullet(C_0)\leq\mu(C_0)=0$, which is a contradiction. 
   For this reason one sees that $\mu(Q)>0$ for all $Q\in\mathcal{T}-\{\emptyset\}$, and (p8) holds. 
   Consequently we have shown the $\mu:\mathscr{B}\to\mathbb{R}\cup\{\infty\}$ in \eqref{eq-5.2.24} satisfies the eight conditions (p1) through (p8) in Theorem \ref{thm-5.1.2}.
\end{proof}

\paragraph{The uniqueness of Haar measure}
   Our aim is to demonstrate that the existence of left-invariant Haar measure is unique up to a positive multiplicative constant whenever $G$ satisfies the second countability axiom (cf.\ Proposition \ref{prop-5.2.29}).
   For the aim let us give four lemmas first. 

\begin{lemma}\label{lem-5.2.25}
   For any $K\in\mathcal{C}$ and $O\in\mathcal{T}$ with $K\subset O$, there exists an $f\in\mathscr{C}_{\geq 0}(G,\mathbb{R})$ such that $c_K\leq f\leq c_O$ on $G$. 
\end{lemma}   
\begin{proof}
   Since $K\in\mathcal{C}$, $O\in\mathcal{T}$, $K\subset O$ and $G$ is a locally compact Hausdorff space, there exists a $Q\in\mathcal{T}$ such that 
\[
   K\subset Q\subset\overline{Q}\subset O
\]
and $\overline{Q}\in\mathcal{C}$.  
   Here $\overline{Q}$ is a compact Hausdorff space, so it is a normal space.
   Hence Uryson's lemma assures that there exists a continuous function $h:\overline{Q}\to\mathbb{R}$ such that
\begin{center}
\begin{tabular}{lll}
   (i) $0\leq h(q)\leq 1$ for all $q\in\overline{Q}$, & (ii) $h(k)=1$ for all $k\in K$, & (iii) $h(p)=0$ for all $p\in\overline{Q}-Q$.   
\end{tabular}
\end{center}
   Then we define a function $f:G\to\mathbb{R}$ by 
\[
   f(g):=\begin{cases} h(g) & \mbox{if $g\in\overline{Q}$},\\ 0 & \mbox{if $g\in G-Q$}.\end{cases}
\]
   Remark here that the definition of $f$ is well-defined because $h(p)=0$ for all $p\in\overline{Q}\cap(G-Q)$. 
   About this $f$ we assert the following statements, which complete the proof of Lemma \ref{lem-5.2.25}: 
\begin{enumerate}
\item 
   $f$ is continuous since $h:\overline{Q}\to\mathbb{R}$ is continuous, both $\overline{Q}$ and $G-Q$ are closed in $G$ and $G=\overline{Q}\cup(G-Q)$;
\item 
   $\operatorname{supp}(f)$ is compact due to $\operatorname{supp}(f)\subset\overline{Q}$ and $\overline{Q}\in\mathcal{C}$;
\item
   it follows from (i) that $0\leq f(g)\leq 1$ for all $g\in G$;
\item
   $0\leq f\leq 1$, (ii) and $K\subset\overline{Q}$ imply $c_K\leq f$;
\item 
   it follows from $(G-O)\subset(G-Q)$ that $f(x)=0$ for all $x\not\in O$, so that $0\leq f\leq 1$ leads to $f\leq c_O$.    
\end{enumerate}
\end{proof}

   From Lemma \ref{lem-5.2.25} we deduce
\begin{lemma}\label{lem-5.2.26}
   Let $\nu$ be a measure on $\mathscr{B}$ such that 
\begin{enumerate}
\item[{\rm (p5)}] 
   $\nu(O)=\sup\{\nu(C):\mbox{$C\in\mathcal{C}$, $C\subset O$}\}$ for every $O\in\mathcal{T}$.
\end{enumerate} 
   Then, for each $P\in\mathcal{T}$ it follows that 
\[
   \nu(P)
   =\sup\left\{\begin{array}{@{}l|l@{}}
     \displaystyle{\int_Gf(g)d\nu(g)} & \mbox{$f\in\mathscr{C}_{\geq 0}(G,\mathbb{R})$, $f\leq c_P$}\end{array}\right\}.
\] 
\end{lemma}   
\begin{proof}
   First, let us confirm that 
\begin{equation}\label{eq-1}\tag*{\textcircled{1}}
   \sup\left\{\begin{array}{@{}l|l@{}}
     \displaystyle{\int_Gf(g)d\nu(g)} & \mbox{$f\in\mathscr{C}_{\geq 0}(G,\mathbb{R})$, $f\leq c_P$}\end{array}\right\}
   \leq\nu(P).   
\end{equation} 
   For any $f\in\mathscr{C}_{\geq 0}(G,\mathbb{R})$ with $f\leq c_P$, both $f$ and $c_P$ are $\mathscr{B}$-measurable functions on $G$ and $0\leq f\leq c_P$. 
   Therefore we have 
\[
   \int_Gf(g)d\nu(g)\leq\int_Gc_P(g)d\nu(g)=\nu(P).
\] 
   Hence the inequality \ref{eq-1} holds.
   Now, let us show that the converse inequality also holds. 
   From (p5) it suffices to show that for any $K\in\mathcal{C}$ with $K\subset P$, there exists an $h\in\mathscr{C}_{\geq 0}(G,\mathbb{R})$ satisfying   
\[
\begin{array}{ll}
   h\leq c_P, & \nu(K)\leq\displaystyle{\int_Gh(g)d\nu(g)}.
\end{array}
\]
   That comes from Lemma \ref{lem-5.2.25} and $\nu(K)=\int_Gc_K(g)d\nu(g)$. 
\end{proof}

\begin{lemma}\label{lem-5.2.27}
   Let $\nu$ be a measure on $\mathscr{B}$ such that 
\begin{enumerate}
\item[{\rm (p6)}] 
   $\nu(C)<\infty$ for each $C\in\mathcal{C}$.
\end{enumerate} 
   Then, the following three items hold$:$
\begin{enumerate}
\item[{\rm (1)}]
   Any $f\in\mathscr{C}_{\geq 0}(G,\mathbb{R})$ is $\nu$-integrable on $G$.
\item[{\rm (2)}]
   For any $f\in\mathscr{C}_{\geq 0}(G,\mathbb{R})$ and $g\in G$, the non-negative function $G\ni x\mapsto\int_Gf(gx)d\nu(g)\in\mathbb{R}$ is continuous.
\item[{\rm (3)}]
   The measure space $(G,\mathscr{B},\nu)$ is $\sigma$-finite in the case where $G$ satisfies the second countability axiom.
\end{enumerate}
\end{lemma}   
\begin{proof}
   (1). 
   Since $\operatorname{supp}(f)\subset G$ is compact and $f$ is continuous, there exists a positive real number $\lambda$ such that $f\leq\lambda c_{\operatorname{supp}(f)}$ on $G$. 
   In addition, $\nu\bigl(\operatorname{supp}(f)\bigr)<\infty$ due to (p6).
   Then, it follows from $0\leq f\leq\lambda c_{\operatorname{supp}(f)}$ and $\nu\bigl(\operatorname{supp}(f)\bigr)<\infty$ that 
\[
   \int_Gf(g)d\nu(g)
   \leq\int_G\lambda c_{\operatorname{supp}(f)}(g)d\nu(g)
   =\lambda\nu\bigl(\operatorname{supp}(f)\bigr)<\infty.
\]
   Hence $f$ is $\nu$-integrable on $G$.\par
   
   (2).
   The above (1) assures that the function $G\ni x\mapsto\int_Gf(gx)d\nu(g)\in\mathbb{R}$ and the computations below are well-defined.\par
   
   Fix any $\epsilon>0$ and $x_0\in G$. 
   There exists a $V\in\mathcal{U}$ satisfying $\overline{V}\in\mathcal{C}$ because $G$ is a locally compact Hausdorff space. 
   In view of $\operatorname{supp}(f)\in\mathcal{C}$, we see that $\operatorname{supp}(f)\overline{V}x_0^{-1}$ is a compact subset of $G$, and that $f$ is uniformly continuous on $G$.
   Then it follows from (p6) that  
\[
   \mbox{$0<\delta<\infty$ holds for $\delta:=1+\nu\bigl(\operatorname{supp}(f)\overline{V}x_0^{-1}\bigr)$};
\]
and moreover, there exists a $U\in\mathcal{U}$ such that (i) $U=U^{-1}$, (ii) $U\subset V$, and (iii) $a^{-1}b\in U$ implies $|f(a)-f(b)|<\epsilon/\delta$. 
   If $x\in G$ and $g\in G$ satisfy $x_0^{-1}x\in U$ and $g\not\in\operatorname{supp}(f)\overline{V}x_0^{-1}$, respectively, then (i) and (ii) yield $gx\not\in\operatorname{supp}(f)$, and $f(gx)=0$. 
   Consequently, $x\in x_0U$ implies 
\allowdisplaybreaks{
\begin{align*}
   \Big|\int_Gf(gx_0)d\nu(g)&-\int_Gf(gx)d\nu(g)\Big|
   \leq\int_G\bigl|f(gx_0)-f(gx)\bigr|d\nu(g)\\
  &=\int_{\operatorname{supp}(f)\overline{V}x_0^{-1}}\bigl|f(gx_0)-f(gx)\bigr|d\nu(g)+\int_{G-\operatorname{supp}(f)\overline{V}x_0^{-1}}\bigl|f(gx_0)-f(gx)\bigr|d\nu(g)\\
  &=\int_{\operatorname{supp}(f)\overline{V}x_0^{-1}}\bigl|f(gx_0)-f(gx)\bigr|d\nu(g)\\
  &\leq\int_{\operatorname{supp}(f)\overline{V}x_0^{-1}}\dfrac{\epsilon}{\delta}d\nu(g) \quad\mbox{($\because$ $(gx_0)^{-1}gx=x_0^{-1}x\in U$, (iii))}\\
  &=\dfrac{\epsilon}{\delta}\nu\bigl(\operatorname{supp}(f)\overline{V}x_0^{-1}\bigr)
   =\epsilon\dfrac{\nu\bigl(\operatorname{supp}(f)\overline{V}x_0^{-1}\bigr)}{1+\nu\bigl(\operatorname{supp}(f)\overline{V}x_0^{-1}\bigr)}
   \leq\epsilon.
\end{align*}}So, the function $G\ni x\mapsto\int_Gf(gx)d\nu(g)\in\mathbb{R}$ is continuous.\par

   (3). 
   Since $G$ satisfies the second countability axiom and is a locally compact Hausdorff space, there exists a sequence $\{E_n\}_{n=1}^\infty\subset G$ satisfying $E_n\in\mathcal{C}$ ($n\in\mathbb{N}$) and $\bigcup_{n=1}^\infty E_n=G$. 
   Thus we conclude (3) from (p6). 
\end{proof}

\begin{lemma}\label{lem-5.2.28}
   There exists an $h_0\in\mathscr{C}_{\geq 0}(G,\mathbb{R})$ satisfying $\int_Gh_0(gx)d\nu(g)>0$ for all $x\in G$ and all measures $\nu$ on $\mathscr{B}$ such that {\rm (p8)} $\nu(O)>0$ for each $O\in\mathcal{T}-\{\emptyset\}$.
\end{lemma}   
\begin{proof}
   Since $G$ is locally compact, there exists a $K\in\mathcal{C}$ satisfying $\emptyset\neq K^\circ$. 
   Lemma \ref{lem-5.2.25} and $K\in\mathcal{C}$ allow us to find an $h_0\in\mathscr{C}_{\geq 0}(G,\mathbb{R})$ such that $h_0\geq c_K$. 
   In this setting, for each $x\in G$ and each measure $\nu$ with (p8), we obtain
\[
   \int_Gh_0(gx)d\nu(g)
   \geq\int_Gc_K(gx)d\nu(g)
   =\int_Gc_{Kx^{-1}}(g)d\nu(g)
   =\nu(Kx^{-1})
   \geq\nu(K^\circ x^{-1})>0
\] 
from (p8).
\end{proof}

   Lemmas \ref{lem-5.2.26}, \ref{lem-5.2.27} and \ref{lem-5.2.28} enable one to obtain
\begin{proposition}[Uniqueness]\label{prop-5.2.29}
   Suppose that $G$ satisfies the second countability axiom. 
   Then, for non-zero left-invariant Haar measures $\mu$ and $\nu$ on $G$, there exists a positive real number $\lambda$ such that $\mu=\lambda\nu$. 
\end{proposition} 
\begin{proof}
   Throughout this proof, (p$k$) means the condition (p$k$) in Theorem \ref{thm-5.1.2} ($1\leq k\leq 8$).\par
   
   By (p4) and Lemma \ref{lem-5.2.26} it suffices to confirm the following: there exists a $\lambda>0$ such that 
\begin{equation}\label{eq-1}\tag*{\textcircled{1}}
   \int_Gf(x)d\mu(x)=\lambda\int_Gf(y)d\nu(y)
\end{equation}
for all $f\in\mathscr{C}_{\geq 0}(G,\mathbb{R})$.
   Let us fix any $f\in\mathscr{C}_{\geq 0}(G,\mathbb{R})$, and deal with the product measure space $(G\times G,\mathscr{R}, \mu\times\nu)$ obtained from the measure spaces $(G,\mathscr{B},\mu)$ and $(G,\mathscr{B},\nu)$. 
   Recalling that there exists an $h_0\in\mathscr{C}_{\geq 0}(G,\mathbb{R})$ such that $\int_Gh_0(gx)d\mu(g)>0$, $\int_Gh_0(gx)d\nu(g)>0$ for all $x\in G$, we define a function $F:G\times G\to\mathbb{R}$ by 
\begin{equation}\label{eq-a}\tag{a}
   \mbox{$F(x,y):=\dfrac{f(x)h_0(yx)}{\int_Gh_0(gx)d\nu(g)}$ for $(x,y)\in G\times G$}
\end{equation}
(cf.\ Lemma \ref{lem-5.2.28}).
   On the one hand; this function $F$ is non-negative, continuous on $G\times G$ by Lemma \ref{lem-5.2.27}-(2). 
   Hence $F$ is $\mathscr{R}$-measurable on $G\times G$.
   On the other hand; since $\operatorname{supp}(F)\subset\operatorname{supp}(f)\times\operatorname{supp}(h_0)\operatorname{supp}(f)^{-1}$, we see that $\operatorname{supp}(F)$ is a compact subset of $G\times G$, and that $(\mu\times\nu)\bigl(\operatorname{supp}(F)\bigr)\leq\mu\bigl(\operatorname{supp}(f)\bigr)\nu\bigl(\operatorname{supp}(h_0)\operatorname{supp}(f)^{-1}\bigr)<\infty$ by (p6). 
   Accordingly we conclude that  
\begin{equation}\label{eq-b}\tag{b}
   \mbox{the $F(x,y)$ is $\mathscr{R}$-measurable and $(\mu\times\nu)$-integrable on $G\times G$}
\end{equation}
by arguments similar to those in the proof of Lemma \ref{lem-5.2.27}-(1). 
   Now, \eqref{eq-b}, Lemma \ref{lem-5.2.27}-(3) and Fubini's theorem imply
\allowdisplaybreaks{
\begin{align*}
   \int_Gf(x)d\mu(x)
  &=\int_G\Big(\int_GF(x,y)d\nu(y)\Big)d\mu(x)
   =\int_{G\times G}F(x,y)d(\mu\times\nu)(x,y)\\
  &=\int_G\Big(\int_GF(x,y)d\mu(x)\Big)d\nu(y)
   =\int_G\Big(\int_GF(y^{-1}x,y)d\mu(y^{-1}x)\Big)d\nu(y) \quad\mbox{(by $x\mapsto y^{-1}x$)}\\
  &\stackrel{{\rm (p7)}}{=}\int_G\Big(\int_GF(y^{-1}x,y)d\mu(x)\Big)d\nu(y) 
   =\int_{G\times G}F(y^{-1}x,y)d(\mu\times\nu)(x,y)\\
  &=\int_G\Big(\int_GF(y^{-1}x,y)d\nu(y)\Big)d\mu(x)
   =\int_G\Big(\int_GF(y^{-1},xy)d\nu(xy)\Big)d\mu(x) \quad\mbox{(by $y\mapsto xy$)}\\
  &\stackrel{{\rm (p7)}}{=}\int_G\Big(\int_GF(y^{-1},xy)d\nu(y)\Big)d\mu(x)
   \stackrel{\eqref{eq-a}}{=}\int_Gh_0(x)d\mu(x)\int_G\dfrac{f(y^{-1})}{\int_Gh_0(gy^{-1})d\nu(g)}d\nu(y),
\end{align*}}where we remark that $f(x)=\int_GF(x,y)d\nu(y)$. 
   Hence it turns out that 
\begin{equation}\label{eq-c}\tag{c}
   \dfrac{\int_Gf(x)d\mu(x)}{\int_Gh_0(z)d\mu(z)}
   =\int_G\dfrac{f(y^{-1})}{\int_Gh_0(gy^{-1})d\nu(g)}d\nu(y).
\end{equation}
   The above arguments assure that for any non-zero left-invariant Haar measure $\mu'$ on $G$, the equality 
\[
   \dfrac{\int_Gf(x)d\mu'(x)}{\int_Gh_0(z)d\mu'(z)}
   =\int_G\dfrac{f(y^{-1})}{\int_Gh_0(gy^{-1})d\nu(g)}d\nu(y)
\]
always holds, and thus \eqref{eq-c} yields $\dfrac{\int_Gf(x)d\mu(x)}{\int_Gh_0(z)d\mu(z)}=\dfrac{\int_Gf(x)d\mu'(x)}{\int_Gh_0(z)d\mu'(z)}$; in particular, 
\[
   \dfrac{\int_Gf(x)d\mu(x)}{\int_Gh_0(z)d\mu(z)}=\dfrac{\int_Gf(x)d\nu(x)}{\int_Gh_0(z)d\nu(z)}.
\]   
   Setting $\lambda:=\dfrac{\int_Gh_0(z)d\mu(z)}{\int_Gh_0(z)d\nu(z)}$, we have $\lambda>0$ and \ref{eq-1}.     
\end{proof}

   Propositions \ref{prop-5.2.23} and \ref{prop-5.2.29} lead to Theorem \ref{thm-5.1.2}.

\section{An example of unimodular group}\label{sec-5.3}
   Suppose $G$ to satisfy the second countability axiom. 
   Let $\mu$ be a non-zero left-invariant Haar measure on $G$.
   For an $x\in G$, Lemma \ref{lem-5.1.1}-(2) enables us to define a set function $\delta(x)\mu:\mathscr{B}\to\mathbb{R}\amalg\{\infty\}$ by
\begin{equation}\label{eq-5.3.1}
   \mbox{$(\delta(x)\mu)(A):=\mu(Ax)$ for $A\in\mathscr{B}$}.
\end{equation}
   Then $\delta(x)\mu$ is also a non-zero left-invariant Haar measure on $G$, since the right translation $R_x:G\to G$ is a homeomorphism.
   Accordingly there exists a unique positive real number $\triangle(x)$ satisfying 
\begin{equation}\label{eq-5.3.2}
   \delta(x)\mu=\triangle(x)\mu
\end{equation}
by Theorem \ref{thm-5.1.2}. 
   The function $\triangle:G\to\mathbb{R}^+$, $x\mapsto\triangle(x)$, is called the {\it modular function}\index{modular function@modular function\dotfill} of $G$.\footnote{Remark.\ Theorem \ref{thm-5.1.2} assures that this modular function $\triangle$ is independent of the choice of $\mu$.} 
   Besides; the group $G$ is said to be {\it unimodular},\index{unimodular group@unimodular group\dotfill} if $\triangle(x)=1$ for all $x\in G$.
   In this section, we clarify some properties of $\triangle$ and show Proposition \ref{prop-5.3.4} which provides us with an example of unimodular group.
\begin{proposition}\label{prop-5.3.3}
   Suppose that $G$ satisfies the second countability axiom. 
   Let $\triangle$ denote the modular function of $G$. 
   Then,
\begin{enumerate}
\item[{\rm (i)}]
   $\triangle:G\to\mathbb{R}^+$, $x\mapsto\triangle(x)$, is a continuous function.
\item[{\rm (ii)}]
   $\triangle(xy)=\triangle(x)\triangle(y)$ for all $x,y\in G$. 
\item[{\rm (iii)}]
   $\triangle(x)=1$ for all $x\in G$ $($i.e., $G$ is unimodular$)$ if and only if a non-zero left-invariant Haar measure $\mu$ on $G$ is also right-invariant $($i.e., $\mu(Ag)=\mu(A)$ for all $(g,A)\in G\times\mathscr{B})$.  
\end{enumerate} 
\end{proposition} 
\begin{proof}
   Let $\mu$ be a non-zero left-invariant Haar measure on $G$.\par
   
   (i).
   By Lemma \ref{lem-5.2.28} there exists an $h_0\in\mathscr{C}_{\geq 0}(G,\mathbb{R})$ such that $\int_Gh_0(gx)d\mu(g)>0$ for all $x\in G$.
   From \eqref{eq-5.3.1} and \eqref{eq-5.3.2} we obtain $\int_Gh_0(gx)d\mu(g)=\triangle(x)^{-1}\int_Gh_0(g)d\mu(g)$. 
   Then, Lemma \ref{lem-5.2.27}-(2) implies that 
\[
   \mbox{$\displaystyle{G\ni x\mapsto\dfrac{1}{\triangle(x)}\int_Gh_0(g)d\mu(g)\in\mathbb{R}^+}$ is continuous}.
\]
   Hence $\triangle:G\to\mathbb{R}^+$, $x\mapsto\triangle(x)$, is continuous because $\int_Gh_0(g)d\mu(g)$ is a positive constant.\par
   
   (ii).
   By a direct computation, together with \eqref{eq-5.3.2} and \eqref{eq-5.3.1}, we  have $\triangle(xy)\mu(A)=\mu(Axy)=(\delta(y)\mu)(Ax)=\triangle(y)\mu(Ax)=\triangle(x)\triangle(y)\mu(A)$ for all $A\in\mathscr{B}$, and so $\triangle(xy)=\triangle(x)\triangle(y)$ by virtue of $\mu\neq 0$.\par
   
   (iii). 
   For each $x\in G$, it follows from \eqref{eq-5.3.2}, $\mu\neq 0$ and \eqref{eq-5.3.1} that $\triangle(x)=1$ if and only if $\delta(x)\mu=\mu$ if and only if $\mu(Ax)=\mu(A)$ for all $A\in\mathscr{B}$.
   Hence we can get the conclusion.
\end{proof} 

   Now, let us show 
\begin{proposition}\label{prop-5.3.4} 
   $G$ is unimodular if $G$ is a compact Hausdorff topological group, or $G$ is a connected semisimple Lie group.\footnote{Remark.\ A connected Lie group always satisfies the second countability axiom.}
   Here, we say that a Lie group is {\rm semisimple}, if so is its Lie algebra.  
\end{proposition} 
\begin{proof}
   First, let us confirm that a compact Hausdorff topological group $K$ is unimodular. 
   Proposition \ref{prop-5.3.3}-(i), (ii) implies that $\triangle:K\to\mathbb{R}^+$, $k\mapsto\triangle(k)$, is a continuous (group) homomorphism, where we note that $\mathbb{R}^+$ is the identity component of $GL(1,\mathbb{R})$. 
   Thus its image $\triangle(K)$ is a compact subgroup of $\mathbb{R}^+$, and it must be $\{1\}$.\footnote{($\because$) Suppose that $\triangle(K)$ contains an element $\lambda\neq 1$. 
   Then, since $\triangle(K)$ is a subgroup of $\mathbb{R}^+$, we have $\lambda,\lambda^{-1}\in\triangle(K)$ and $(0,1)\cup(1,\infty)\subset\triangle(K)$; so $\triangle(K)=(0,\infty)=\mathbb{R}^+$. 
   This is a contradiction.
   For this reason $\triangle(K)=\{1\}$.} 
   So $K$ is unimodular.\par
 
   Next, let us prove that $G$ is unimodular, where $G$ is a connected semisimple Lie group. 
   Since $\triangle:G\to\mathbb{R}^+$, $g\mapsto\triangle(g)$, is a continuous homomorphism, it is a Lie group homomorphism.
   Therefore one can set its differential $\triangle_*:\frak{g}\to\frak{gl}(1,\mathbb{R})$, and obtain 
\[
   \triangle_*(\frak{g})=\{0\}
\]
from $\frak{g}=[\frak{g},\frak{g}]$. 
   Since $G$ is a connected Lie group, for each $g\in G$ there exist finite elements $X_1,X_2,\dots,X_k\in\frak{g}$ such that $g=\exp X_1\exp X_2\cdots\exp X_k$, and then 
\[
   \triangle(g)
   =\triangle(\exp X_1\cdots\exp X_k)
   =\triangle(\exp X_1)\cdots\triangle(\exp X_k)
   =e^{\triangle_*(X_1)}\cdots e^{\triangle_*(X_k)}
   =1.
\]
   For this reason $\triangle(G)\subset\{1\}$, and $G$ is unimodular.   
\end{proof}

   We end this chapter with commenting on unimodular groups.
\begin{remark}\label{rem-5.3.5}
   Suppose that (s1) $G$ satisfies the second countability axiom and (s2) $G$ is unimodular. 
   Then, for a given non-zero left-invariant Haar measure $\mu$ on $G$ and any $\mu$-integrable function $f$ on $G$, it follows that
\begin{enumerate}
\item[(1)]
   $\displaystyle{\int_Gf(g)d\mu(g)=\int_Gf(xg)d\mu(g)=\int_Gf(g)d\mu(xg)}$ for all $x\in G$;
\item[(2)]
   $\displaystyle{\int_Gf(g)d\mu(g)=\dfrac{1}{\triangle(x)}\int_Gf(g)d\mu(g)=\int_Gf(gx)d\mu(g)}$ for all $x\in G$;
\item[(3)]
   $\displaystyle{\int_Gf(g)d\mu(g)=\int_G\dfrac{1}{\triangle(g)}f(g^{-1})d\mu(g)=\int_Gf(g^{-1})d\mu(g)}$,
\end{enumerate}
where $\triangle$ is the modular function of $G$.
   Remark here, (1), (2$'$) $\triangle(x)^{-1}\int_Gf(g)d\mu(g)=\int_Gf(gx)d\mu(g)$ and (3$'$) $\int_Gf(g)d\mu(g)=\int_G\triangle(g)^{-1}f(g^{-1})d\mu(g)$ come from the measure $\mu$ being left-invariant only.
\end{remark} 

\chapter{Regulated integrals}\label{ch-6}
   In this chapter we study integrals of vector-valued functions.
   cf.\ Lang \cite[Section 4, Chapter I]{La}.

\section{An introduction to regulated integral}\label{sec-6.1} 
   The setting of Section \ref{sec-6.1} is as follows:
\begin{itemize}
\item
   $(X,\mathscr{B},\mu)$ is a measure space which consists of an abstract space $X$, a $\sigma$-algebra $\mathscr{B}$ on $X$, and a measure $\mu$ on $\mathscr{B}$,
\item
   $\mathcal{V}$ is a Fr\'{e}chet space over $\mathbb{K}=\mathbb{R}$ or $\mathbb{C}$, whose topology is determined by a countable number of seminorms $\{p_\ell\}_{\ell\in\mathbb{N}}$,
\item
   $d$ is a metric on $\mathcal{V}$ such that 
   \begin{enumerate}
   \item[(1)] 
      $d$ is a suitable metric for $\mathcal{V}$ which induces a topology identical to the original one, 
   \item[(2)]
      the metric space $(\mathcal{V},d)$ is complete, 
   \item[(3)]
      $d(\xi_1,\xi_2)=d(\xi_1+\xi_3,\xi_2+\xi_3)$ for all $\xi_1,\xi_2,\xi_3\in\mathcal{V}$.
   \end{enumerate}
\end{itemize}
   We remark that for $\xi\in\mathcal{V}$ and $\{\xi_n\}_{n=1}^\infty\subset\mathcal{V}$, $\displaystyle{\lim_{n\to\infty}d(\xi,\xi_n)=0}$ if and only if for any $\epsilon>0$ and each $\ell\in\mathbb{N}$ there exists an $N_\ell\in\mathbb{N}$ such that $n\geq N_\ell$ implies $p_\ell(\xi-\xi_n)<\epsilon$.

\subsection{The regulated integral of a step function}\label{subsec-6.1.1}
   For an $A\in\mathscr{B}$ with $\mu(A)<\infty$, a {\it step function}\index{step function@step function\dotfill} 
\[
   S=S(x):A\to\mathcal{V}
\]
is a mapping such that there exists a finite sequence $\{A_i\}_{i=1}^k\subset 2^A$ satisfying the following three conditions: 
\begin{enumerate}
\item
   $A_i\in\mathscr{B}$ for all $1\leq i\leq k$,
\item
   $A=\coprod_{i=1}^kA_i$ (disjoint union),
\item
   the mapping $S$ is constant on each $A_i$ ($1\leq i\leq k$).
\end{enumerate}
   For this step function $S:A\to\mathcal{V}$, we define its integral $\int_AS(x)d\mu(x)$ on $A$ by
\begin{equation}\label{eq-6.1.1}
   \int_AS(x)d\mu(x)
   :=\sum_{i=1}^k\mu(A_i)\xi_i,
\end{equation}
where $S(A_i)=\{\xi_i\}$, $1\leq i\leq k$.
   Remark that $0\leq\mu(A_i)\leq\mu(A)<\infty$ ($1\leq i\leq k$), that the integral \eqref{eq-6.1.1} is independent of the choice of $A_i$ on which $S$ is constant, and that $\int_AS(x)d\mu(x)\in\mathcal{V}$.
   In addition, $\int_AS(x)d\mu(x)=0$ if $\mu(A)\leq 0$ (i.e., $\mu(A)=0$).

\subsection{Definition of regulated integral}\label{subsec-6.1.2}
   In the previous subsection we have set the integral of a step function, \eqref{eq-6.1.1}.    
   We want to consider the integral $\int_AF(x)d\mu(x)$ on $A$ of a more general function $F:A\to\mathcal{V}$. 
   For this reason, let us first prove Lemma \ref{lem-6.1.2} and afterwards show Proposition \ref{prop-6.1.3}. 
   This proposition grants our want.   

\begin{lemma}\label{lem-6.1.2}
   Let $A\in\mathscr{B}$ with $\mu(A)<\infty$, let $S,T:A\to\mathcal{V}$ be step functions, and let $\alpha,\beta\in\mathbb{K}$. 
\begin{enumerate}
\item[{\rm (1)}]
   $\alpha S+\beta T:A\to\mathcal{V}$ is a step function, and $\displaystyle{\int_A(\alpha S+\beta T)(x)d\mu(x)=\alpha\int_AS(x)d\mu(x)+\beta\int_AT(x)d\mu(x)}$.
\item[{\rm (2)}] 
   If $A=B\amalg C$ and $B\in\mathscr{B}$, then both $S:B\to\mathcal{V}$ and $S:C\to\mathcal{V}$ are step functions, and 
\[
   \int_AS(x)d\mu(x)=\int_BS(y)d\mu(y)+\int_CS(z)d\mu(z).
\]
\item[{\rm (3)}]
   Suppose that $\mathcal{W}$ is a Fr\'{e}chet space over $\mathbb{K}$ and $K:\mathcal{V}\to\mathcal{W}$ is a $\mathbb{K}$-linear mapping.
   Then, $K\circ S:A\to\mathcal{W}$ is a step function, and $\displaystyle{K\Big(\int_AS(x)d\mu(x)\Big)=\int_A(K\circ S)(x)d\mu(x)}$.
\end{enumerate}
\end{lemma}  
\begin{proof}
   Let 
\[
\begin{array}{llll}
   A=\coprod_{i=1}^kA_i=\coprod_{j=1}^hB_j, 
   & A_i,B_j\in\mathscr{B}, 
   & S(x)=\sum_{i=1}^kc_{A_i}(x)\xi_i, 
   & T(x)=\sum_{j=1}^hc_{B_j}(x)\eta_j,
\end{array}
\]
where $\xi_i,\eta_j\in\mathcal{V}$ and $c_Q$ is the characteristic function of a subset $Q\subset A$.\par

   (1).
   It turns out that $A_i\cap B_j\in\mathscr{B}$, $A=\coprod_{1\leq i\leq k, 1\leq j\leq h}(A_i\cap B_j)$, and $\alpha S+\beta T=\sum_{1\leq i\leq k, 1\leq j\leq h}c_{A_i\cap B_j}(\alpha\xi_i+\beta\eta_j)$. 
   Consequently $\alpha S+\beta T:A\to\mathcal{V}$ is a step function; besides,  
\[
\begin{split}
   \int_A(\alpha S+\beta T)(x)d\mu(x)
  &\stackrel{\eqref{eq-6.1.1}}{=}\sum_{1\leq i\leq k, 1\leq j\leq h}\mu(A_i\cap B_j)(\alpha\xi_i+\beta\eta_j)
   =\alpha\sum_{i=1}^k\Big(\sum_{j=1}^h\mu(A_i\cap B_j)\Big)\xi_i+\beta\sum_{j=1}^h\Big(\sum_{i=1}^k\mu(A_i\cap B_j)\Big)\eta_j\\
  &=\alpha\sum_{i=1}^k\mu(A_i)\xi_i+\beta\sum_{j=1}^h\mu(B_j)\eta_j
   \stackrel{\eqref{eq-6.1.1}}{=}\alpha\int_AS(x)d\mu(x)+\beta\int_AT(x)d\mu(x).
\end{split} 
\] 
   Hence (1) holds.\par
   
   (2).
   Both $S(y)=\sum_{i=1}^kc_{A_i\cap B}(y)\xi_i:B\to\mathcal{V}$ and $S(z)=\sum_{i=1}^kc_{A_i\cap C}(z)\xi_i:C\to\mathcal{V}$ are step functions, and
\[
   \int_{B\amalg C}S(x)d\mu(x)
   \stackrel{\eqref{eq-6.1.1}}{=}\sum_{i=1}^k\mu\bigl(A_i\cap(B\amalg C)\bigr)\xi_i
   =\sum_{i=1}^k\mu(A_i\cap B)\xi_i+\sum_{i=1}^k\mu(A_i\cap C)\xi_i
   \stackrel{\eqref{eq-6.1.1}}{=}\int_BS(y)d\mu(y)+\int_CS(z)d\mu(z),
\]
where we remark that $B$, $C=(A-B)\in\mathscr{B}$, $\mu(B)\leq\mu(A)<\infty$, $\mu(C)<\infty$, $B=\coprod_{i=1}^k(A_i\cap B)$, $C=\coprod_{i=1}^k(A_i\cap C)$, and $(A_i\cap B)$, $(A_i\cap C)\in\mathscr{B}$.\par

   (3). 
   Since $K:\mathcal{V}\to\mathcal{W}$ is linear we see that $(K\circ S)(x)=\sum_{i=1}^kc_{A_i}(x)K(\xi_i)$, and so $K\circ S:A\to\mathcal{W}$ is a step function.
   Moreover, 
\[
   K\Big(\int_AS(x)d\mu(x)\Big)
   \stackrel{\eqref{eq-6.1.1}}{=}K\Big(\sum_{i=1}^k\mu(A_i)\xi_i\Big)
   =\sum_{i=1}^k\mu(A_i)K(\xi_i)
   \stackrel{\eqref{eq-6.1.1}}{=}\int_A(K\circ S)(x)d\mu(x).
\]
\end{proof}

   Taking the proof of Lemma \ref{lem-6.1.2} into account, we prove    
\begin{proposition}\label{prop-6.1.3}
   Let $A\in\mathscr{B}$ with $\mu(A)<\infty$, and let $F:A\to\mathcal{V}$ be a mapping.
   Suppose that a sequence $\{S_n:A\to\mathcal{V} \,|\, \mbox{$S_n$ is a step function}\}_{n=1}^\infty$ is uniformly convergent to $F$ on $A$. 
   Then, the following two items hold$:$
\begin{enumerate}
\item[{\rm (i)}]
   There exists a unique $\xi_F\in\mathcal{V}$ such that $\displaystyle{\lim_{n\to\infty}d\Bigl(\xi_F,\int_AS_n(x)d\mu(x)\Bigr)=0}$.
\item[{\rm (ii)}] 
   If another sequence $\{T_m:A\to\mathcal{V}\,|\, \mbox{$T_m$ is a step function}\}_{m=1}^\infty$ is uniformly convergent to $F$ on $A$ also, then the sequence $\big\{\int_AT_m(x)d\mu(x)\big\}_{m=1}^\infty$ in $(\mathcal{V},d)$ converges to the same limit point $\xi_F$ as $\big\{\int_AS_n(x)d\mu(x)\big\}_{n=1}^\infty$. 
\end{enumerate}   
\end{proposition}  
\begin{proof}
   Let $A=\coprod_{i=1}^{k_n}A_{n,i}$, $A_{n,i}\in\mathscr{B}$ and $S_n(x)=\sum_{i=1}^{k_n}c_{A_{n,i}}(x)\xi_{n,i}$, where $\xi_{n,i}\in\mathcal{V}$, $n\in\mathbb{N}$.\par

   (i). 
   Since $(\mathcal{V},d)$ is a Fr\'{e}chet space, it is enough to confirm that $\big\{\int_AS_n(x)d\mu(x)\big\}_{n=1}^\infty$ is a Cauchy sequence in $(\mathcal{V},d)$.  
   In case of $\mu(A)\leq0$, we show that $0\leq\mu(A_{n,i})\leq\mu(A)\leq0$ and 
\[
   \int_AS_n(x)d\mu(x)
   \stackrel{\eqref{eq-6.1.1}}{=}\sum_{i=1}^{k_n}\mu(A_{n,i})\xi_{n,i}
   =0   
\]
for all $n\in\mathbb{N}$; therefore $\big\{\int_AS_n(x)d\mu(x)\big\}_{n=1}^\infty=\{0\}$ is a Cauchy sequence in $(\mathcal{V},d)$. 
   From now on, let us consider the case where $\mu(A)>0$. 
   Fix any $\epsilon>0$ and an arbitrary seminorm $p\in\{p_\ell\}_{\ell\in\mathbb{N}}$. 
   For each $n\in\mathbb{N}$ we define a subset $X_n\subset A$ by 
\[
   X_n:=\big\{a\in A \,\big|\, p\bigl(S_n(a)-F(a)\bigr)<\epsilon/(2\mu(A))\big\}.
\]   
   On the one hand; the supposition allows us to choose an $N_p\in\mathbb{N}$ such that $n\geq N_p$ implies $A=X_n$; and so 
\begin{equation}\label{eq-1}\tag*{\textcircled{1}}
   \mbox{$A=\big\{a\in A \,\big|\, p\bigl(S_n(a)-F(a)\bigr)<\epsilon/(2\mu(A))\big\}$ for all $n\geq N_p$}.
\end{equation}
   On the other hand; a direct computation yields   
\[   
\begin{split}
    p\Big(\int_AS_n(x)d\mu(x)-\int_AS_m(x)d\mu(x)\Big)
   &\stackrel{\eqref{eq-6.1.1}}{=}p\Big(\sum_{1\leq i\leq k_n, 1\leq j\leq k_m}\mu(A_{n,i}\cap A_{m,j})(\xi_{n,i}-\xi_{m,j})\Big)\\
   &\leq\sum_{1\leq i\leq k_n, 1\leq j\leq k_m}\mu(A_{n,i}\cap A_{m,j})\cdot p(\xi_{n,i}-\xi_{m,j}).
\end{split}
\]   
   Here, one estimates the last term at 
\begin{equation}\label{eq-2}\tag*{\textcircled{2}}
   \mu(A_{n,i}\cap A_{m,j})\cdot p(\xi_{n,i}-\xi_{m,j})\leq\mu(A_{n,i}\cap A_{m,j})\cdot\epsilon/\mu(A),
\end{equation}
provided that $n,m\geq N_p$.
   Indeed; in case of $A_{n,i}\cap A_{m,j}=\emptyset$ we have 
\[
   \mu(A_{n,i}\cap A_{m,j})\cdot p(\xi_{n,i}-\xi_{m,j})
   =0
   \leq\mu(A_{n,i}\cap A_{m,j})\cdot\epsilon/\mu(A).
\]
   In case of $A_{n,i}\cap A_{m,j}\neq\emptyset$, one can take an element $a\in A_{n,i}\cap A_{m,j}$, and obtain $S_n(a)=\xi_{n,i}$, $S_m(a)=\xi_{m,j}$ and 
\begin{multline*}
   \mu(A_{n,i}\cap A_{m,j})\cdot p(\xi_{n,i}-\xi_{m,j})
   =\mu(A_{n,i}\cap A_{m,j})\cdot p(S_n(a)-S_m(a))\\
   \leq\mu(A_{n,i}\cap A_{m,j})\cdot \bigl(p(S_n(a)-F(a))+p(F(a)-S_m(a))\bigr)
   <\mu(A_{n,i}\cap A_{m,j})\cdot\epsilon/\mu(A)
\end{multline*} 
from \ref{eq-1} and $n,m\geq N_p$. 
   In any case \ref{eq-2} does hold. 
   Consequently, it follows from \ref{eq-2} that $n,m\geq N_p$ implies
\[
\begin{split}
   p\Big(\int_AS_n(x)d\mu(x)-\int_AS_m(x)d\mu(x)\Big)
  &\leq\sum_{1\leq i\leq k_n, 1\leq j\leq k_m}\mu(A_{n,i}\cap A_{m,j})\cdot p(\xi_{n,i}-\xi_{m,j})\\
  &\leq\dfrac{\epsilon}{\mu(A)}\sum_{1\leq i\leq k_n, 1\leq j\leq k_m}\mu(A_{n,i}\cap A_{m,j})
   \leq\epsilon.
\end{split}
\]
   Hence $d\bigl(\int_AS_n(x)d\mu(x),\int_AS_m(x)d\mu(x)\bigr)\to 0$ ($n,m\to\infty$), and (i) holds.\par
   
   (ii). 
   Suppose that a sequence $\{T_m:A\to\mathcal{V}\,|\, \mbox{$T_m$ is a step function}\}_{m=1}^\infty$ is uniformly convergent to $F$ on $A$. 
   Then, by virtue of (i) there exists a unique $\xi_F'\in\mathcal{V}$ satisfying $\displaystyle{\lim_{m\to\infty}d\Bigl(\xi_F',\int_AT_m(x)d\mu(x)\Bigr)=0}$; and  
\[
   d(\xi_F,\xi_F')
   \leq d\Big(\xi_F,\int_AS_n(x)d\mu(x)\Big)+d\Big(\int_AS_n(x)d\mu(x),\int_AT_n(x)d\mu(x)\Big)+d\Big(\int_AT_n(x)d\mu(x),\xi_F'\Big).
\] 
   Accordingly, it suffices to confirm that 
\begin{equation}\label{eq-a}\tag{a} 
   \lim_{n\to\infty}d\Big(\int_AS_n(x)d\mu(x),\int_AT_n(x)d\mu(x)\Big)=0.
\end{equation}   
   The arguments below will be similar to those in (i) above.\par
   
   Now, let $A=\coprod_{j=1}^{h_m}B_{m,j}$, $B_{m,j}\in\mathscr{B}$ and $T_m(x)=\sum_{j=1}^{h_m}c_{B_{m,j}}(x)\eta_{m,j}$, where $\eta_{m,j}\in\mathcal{V}$, $m\in\mathbb{N}$. 
   In case of $\mu(A)\leq0$ we know that $\mu(B_{n,j})=0$ and 
\[
   \int_AS_n(x)d\mu(x)=0=\int_AT_n(x)d\mu(x)
\] 
for all $n\in\mathbb{N}$. 
   Accordingly one has \eqref{eq-a} in case of $\mu(A)=0$.
   So, we investigate the case where $\mu(A)>0$ henceforth. 
   Let us fix any $\epsilon>0$ and seminorm $p\in\{p_\ell\}_{\ell\in\mathbb{N}}$.   
   Since $\{S_n\}_{n=1}^\infty$ and $\{T_m\}_{m=1}^\infty$ are uniformly convergent to $F$ on $A$, there exist $N_p\in\mathbb{N}$ and $M_p\in\mathbb{N}$ such that $n\geq N_p$ and $m\geq M_p$ imply 
\begin{equation}\label{eq-1'}\tag*{\textcircled{1}$'$}
   \mbox{$A=\big\{a\in A \,\big|\, p\bigl(S_n(a)-F(a)\bigr)<\epsilon/(2\mu(A))\big\}$ and $A=\big\{b\in A \,\big|\, p\bigl(T_m(b)-F(b)\bigr)<\epsilon/(2\mu(A))\big\}$},
\end{equation}
respectively.
   For any $n\geq\max\{N_p,M_p\}$, one knows  
\begin{equation}\label{eq-2'}\tag*{\textcircled{2}$'$} 
   \mu(A_{n,i}\cap B_{n,j})\cdot p(\xi_{n,i}-\eta_{n,j})\leq\mu(A_{n,i}\cap B_{n,j})\cdot\epsilon/\mu(A)
\end{equation}
in a similar way.
   Consequently $n\geq\max\{N_p,M_p\}$ implies 
\[
\begin{split}
   p\Big(\int_AS_n(x)d\mu(x)-\int_AT_n(x)d\mu(x)\Big)
  &\leq\sum_{1\leq i\leq k_n, 1\leq j\leq h_n}\mu(A_{n,i}\cap B_{n,j})\cdot p(\xi_{n,i}-\eta_{n,j})\\
  &\leq\dfrac{\epsilon}{\mu(A)}\sum_{1\leq i\leq k_n, 1\leq j\leq h_n}\mu(A_{n,i}\cap B_{n,j})
   \leq\epsilon.
\end{split}
\]
   Therefore $d\bigl(\int_AS_n(x)d\mu(x),\int_AT_n(x)d\mu(x)\bigr)\to 0$ ($n\to\infty$), and \eqref{eq-a} holds.
\end{proof}

   Proposition \ref{prop-6.1.3} assures that the following Definition \ref{def-6.1.4}-(2) is well-defined:
\begin{definition}\label{def-6.1.4}
   Let $A\in\mathscr{B}$ with $\mu(A)<\infty$.
\begin{enumerate}
\item[(1)]
   A mapping $F:A\to\mathcal{V}$ is said to be {\it regulated},\index{regulated mapping@regulated mapping\dotfill} if there exists a sequence $\{S_n:A\to\mathcal{V} \,|\, \mbox{$S_n$ is a step function}\}_{n=1}^\infty$ which is uniformly convergent to $F$ on $A$. 
\item[(2)]
   Let $F:A\to\mathcal{V}$ be a regulated mapping. 
   Suppose that a sequence $\{S_n\}_{n=1}^\infty$ of step functions is uniformly convergent to $F$ on $A$.
   Then, there exists a unique $\xi_F\in\mathcal{V}$ such that 
\[
   \lim_{n\to\infty}d\Bigl(\xi_F,\int_AS_n(x)d\mu(x)\Bigr)=0.
\]
   This $\xi_F$ is called the {\it regulated integral}\index{regulated integral@regulated integral\dotfill} or the {\it integral} on $A$ of $F$ and we write $\int_AF(x)d\mu(x)$.     
\end{enumerate}
\end{definition} 

   Needless to say, the above integral $\int_AF(x)d\mu(x)$ accords with the integral in \eqref{eq-6.1.1} whenever $F$ is a step function.

\subsection{Properties of regulated integrals}\label{subsec-6.1.3}
   Let us clarify some properties of regulated integrals.
\begin{lemma}\label{lem-6.1.5}
   Let $A\in\mathscr{B}$ with $\mu(A)<\infty$, let $F,G:A\to\mathcal{V}$ be regulated, and let $\alpha,\beta\in\mathbb{K}$.
\begin{enumerate}
\item[{\rm (1)}]
   $\alpha F+\beta G:A\to\mathcal{V}$ is a regulated mapping and
\[
   \int_A(\alpha F+\beta G)(x)d\mu(x)=\alpha\int_AF(x)d\mu(x)+\beta\int_AG(x)d\mu(x).
\] 
\item[{\rm (2)}]
   If $A=B\amalg C$ and $B\in\mathscr{B}$, then both $F:B\to\mathcal{V}$ and $F:C\to\mathcal{V}$ are regulated mappings, and 
\[
   \int_AF(x)d\mu(x)=\int_BF(y)d\mu(y)+\int_CF(z)d\mu(z).  
\]
\item[{\rm (3)}]
   Suppose that $\mathcal{W}$ is a Fr\'{e}chet space over $\mathbb{K}$ and $L:\mathcal{V}\to\mathcal{W}$ is a continuous, $\mathbb{K}$-linear mapping.
   Then, $L\circ F:A\to\mathcal{W}$ is a regulated mapping and
\[
   L\Big(\int_AF(x)d\mu(x)\Big)=\int_A(L\circ F)(x)d\mu(x).
\]
\item[{\rm (4)}]
   For any continuous seminorm $\hat{p}$ on $\mathcal{V}$, it follows that $\hat{p}\circ F:A\to\mathbb{R}$ is regulated, and the inequality
\[
   \hat{p}\Big(\int_AF(x)d\mu(x)\Big)\leq\int_A(\hat{p}\circ F)(x)d\mu(x)
\]
holds.
\end{enumerate}
   cf.\ Definition {\rm \ref{def-6.1.4}}.
\end{lemma}   
\begin{proof}
   Let $\{S_n\}_{n=1}^\infty$ and $\{T_m\}_{m=1}^\infty$ be sequences of step functions which are uniformly convergent to $F$, $G$ on $A$, respectively.\par
   
   (1). 
   In case of $|\alpha|+|\beta|\leq0$, one has $\alpha=\beta=0$; thus $\alpha F+\beta G=0$ is a regulated mapping and 
\[
   \int_A(\alpha F+\beta G)(x)d\mu(x)=0=\alpha\int_AF(x)d\mu(x)+\beta\int_AG(x)d\mu(x).
\]
   So, let us consider the case where $|\alpha|+|\beta|>0$ henceforth.
   Fix any $\epsilon>0$ and any seminorm $p\in\{p_\ell\}_{\ell\in\mathbb{N}}$.
   Since $\{S_n\}_{n=1}^\infty$, $\{T_m\}_{m=1}^\infty$ are uniformly convergent to $F$, $G$ on $A$, there exist $N_p$, $M_p\in\mathbb{N}$ such that $n\geq N_p$, $m\geq M_p$ imply 
\[
\begin{array}{ll}
   p\bigl(F(x)-S_n(x)\bigr)<\dfrac{\epsilon}{|\alpha|+|\beta|}, &
   p\bigl(G(x)-T_m(x)\bigr)<\dfrac{\epsilon}{|\alpha|+|\beta|}
\end{array}   
\] 
for all $x\in A$, respectively. 
   Hence $k\geq\max\{N_p,M_p\}$ implies 
\[
   p\bigl(\alpha F(x)+\beta G(x)-\alpha S_k(x)-\beta T_k(x)\bigr)
   \leq|\alpha|p\bigl(F(x)-S_k(x)\bigr)+|\beta|p\bigl(G(x)-T_k(x)\bigr)
   <\epsilon
\]
for all $x\in A$.
   Consequently the sequence $\{\alpha S_n+\beta T_n\}_{n=1}^\infty$ of step functions is uniformly convergent to $\alpha F+\beta G$ on $A$ (cf.\ Lemma \ref{lem-6.1.2}-(1)), and $\alpha F+\beta G:A\to\mathcal{V}$ is a regulated mapping.
   Furthermore, we obtain 
\allowdisplaybreaks{
\begin{align*}
  p\Big(\int_A&(\alpha F+\beta G)(x)d\mu(x)-\alpha\int_AF(x)d\mu(x)-\beta\int_AG(x)d\mu(x)\Big)\\
  &\leq p\Big(\int_A(\alpha F+\beta G)(x)d\mu(x)-\int_A(\alpha S_n+\beta T_n)(x)d\mu(x)\Big)\\
&\qquad+p\Big(\int_A(\alpha S_n+\beta T_n)(x)d\mu(x)-\alpha\int_AF(x)d\mu(x)-\beta\int_AG(x)d\mu(x)\Big)\\
  &=p\Big(\int_A(\alpha F+\beta G)(x)d\mu(x)-\int_A(\alpha S_n+\beta T_n)(x)d\mu(x)\Big)\\
  &\qquad+p\Big(\alpha\int_AS_n(x)d\mu(x)+\beta\int_AT_n(x)d\mu(x)-\alpha\int_AF(x)d\mu(x)-\beta\int_AG(x)d\mu(x)\Big)
  \quad\mbox{($\because$ Lemma \ref{lem-6.1.2}-(1))}\\
  &\leq p\Big(\int_A(\alpha F+\beta G)(x)d\mu(x)-\int_A(\alpha S_n+\beta T_n)(x)d\mu(x)\Big)\\
  &\qquad+|\alpha|p\Big(\int_AS_n(x)d\mu(x)-\int_AF(x)d\mu(x)\Big)+|\beta|p\Big(\int_AT_n(x)d\mu(x)-\int_AG(x)d\mu(x)\Big)\\
  &\longrightarrow 0 \quad(n\to\infty)
\end{align*}}because of Definition \ref{def-6.1.4}-(2) and $\{\alpha S_n+\beta T_n\}_{n=1}^\infty$, $\{S_n\}_{n=1}^\infty$, $\{T_m\}_{m=1}^\infty$ being uniformly convergent to $\alpha F+\beta G$, $F$, $G$ on $A$, respectively. 
   The above computation leads to $\int_A(\alpha F+\beta G)(x)d\mu(x)=\alpha\int_AF(x)d\mu(x)+\beta\int_AG(x)d\mu(x)$.\par

   (2). 
   cf.\ Lemma \ref{lem-6.1.2}-(2).\par

   (3).
   Suppose that the topology for $\mathcal{W}$ is determined by a countable number of seminorms $\{q_m\}_{m\in\mathbb{N}}$. 
   Since $\{S_n\}_{n=1}^\infty$ is uniformly convergent to $F$ on $A$, it follows from Definition \ref{def-6.1.4}-(2) that
\begin{equation}\label{eq-1}\tag*{\textcircled{1}} 
   \lim_{n\to\infty}d\Big(\int_AF(x)d\mu(x),\int_AS_n(x)d\mu(x)\Big)=0.
\end{equation} 
   From now on, let us confirm that $L\circ F:A\to\mathcal{W}$ is regulated. 
   By Lemma \ref{lem-6.1.2}-(3),  $L\circ S_n:A\to\mathcal{W}$ is a step function ($n\in\mathbb{N}$).
   We want to show that the sequence $\{L\circ S_n\}_{n=1}^\infty$ of step functions is uniformly convergent to $L\circ F$ on $A$. 
   Take any $\epsilon>0$ and any seminorm $q\in\{q_m\}_{m\in\mathbb{N}}$. 
   Since $L:\mathcal{V}\to\mathcal{W}$ is continuous, there exist finite $p_{\ell_1},\dots,p_{\ell_r}\in\{p_\ell\}_{\ell\in\mathbb{N}}$ and $\lambda_1,\dots,\lambda_r>0$ such that 
\begin{equation}\label{eq-2}\tag*{\textcircled{2}}
   \mbox{$q\bigl(L(\xi)\bigr)\leq\lambda_1p_{\ell_1}(\xi)+\cdots+\lambda_rp_{\ell_r}(\xi)$ for all $\xi\in\mathcal{V}$}.
\end{equation} 
   Then, for each $1\leq j\leq r$ there exists an $N_{q,j}\in\mathbb{N}$ such that $n\geq N_{q,j}$ implies 
\begin{equation}\label{eq-3}\tag*{\textcircled{3}}
   p_{\ell_j}\bigl(S_n(x)-F(x)\bigr)<\dfrac{\epsilon}{\lambda_1+\cdots+\lambda_r}
\end{equation} 
for all $x\in A$, because $\{S_n\}_{n=1}^\infty$ is uniformly convergent to $F$ on $A$.
   Therefore it follows from \ref{eq-2} and \ref{eq-3} that $n\geq\max\{N_{q,j}:1\leq j\leq r\}$ implies 
\[
   q\bigl(L(S_n(x))-L(F(x))\bigr)
   =q\bigl(L(S_n(x)-F(x))\bigr)
   \leq\sum_{j=1}^r\lambda_jp_{\ell_j}\bigl(S_n(x)-F(x)\bigr)
   <\dfrac{\epsilon}{\lambda_1+\cdots+\lambda_r}\sum_{j=1}^r\lambda_j
   =\epsilon
\]
for all $x\in A$.
   Hence $\{L\circ S_n\}_{n=1}^\infty$ is uniformly convergent to $L\circ F$ on $A$. 
   Consequently we assert that $L\circ F:A\to\mathcal{W}$ is a regulated mapping. 
   Note here, at this stage we see that 
\begin{equation}\label{eq-4}\tag*{\textcircled{4}} 
   q\Big(\int_A(L\circ S_n)(x)d\mu(x)-\int_A(L\circ F)d\mu(x)\Big)\longrightarrow 0 \quad(n\to\infty)
\end{equation} 
by Definition \ref{def-6.1.4}-(2).
   The rest of proof is to verify that $L\bigl(\int_AF(x)d\mu(x)\bigr)=\int_A(L\circ F)(x)d\mu(x)$, which comes from
\allowdisplaybreaks{
\begin{align*}
   q\Big(L&\bigl(\int_AF(x)d\mu(x)\bigr)-\int_A(L\circ F)(x)d\mu(x)\Big)\\
  &\leq q\Big(L\bigl(\int_AF(x)d\mu(x)\bigr)-\int_A(L\circ S_n)(x)d\mu(x)\Big)+q\Big(\int_A(L\circ S_n)(x)d\mu(x)-\int_A(L\circ F)(x)d\mu(x)\Big)\\
  &=q\Big(L\bigl(\int_AF(x)d\mu(x)\bigr)-L\bigl(\int_AS_n(x)d\mu(x)\bigr)\Big)+q\Big(\int_A(L\circ S_n)(x)d\mu(x)-\int_A(L\circ F)(x)d\mu(x)\Big)
  \quad\mbox{($\because$ Lemma \ref{lem-6.1.2}-(3))}\\
  &=q\Big(L\bigl(\int_AF(x)d\mu(x)-\int_AS_n(x)d\mu(x)\bigr)\Big)+q\Big(\int_A(L\circ S_n)(x)d\mu(x)-\int_A(L\circ F)(x)d\mu(x)\Big)\\
  &\longrightarrow 0 \quad(n\to\infty)
\end{align*}}because $q\circ L$ is continuous, \ref{eq-1} and \ref{eq-4}.\par

   (4). 
   Let $A=\coprod_{i=1}^{k_n}A_{n,i}$, $A_{n,i}\in\mathscr{B}$ and $S_n(x)=\sum_{i=1}^{k_n}c_{A_{n,i}}(x)\xi_{n,i}$, where $\xi_{n,i}\in\mathcal{V}$, $n\in\mathbb{N}$.
   By a direct computation we obtain $\hat{p}\bigl(S_n(y)\bigr)=\hat{p}(\xi_{n,i})$ if $y\in A_{n,i}$. 
   This assures that 
\[
   (\hat{p}\circ S_n)(x)=\sum_{i=1}^{k_n}c_{A_{n,i}}(x)\hat{p}(\xi_{n,i}),
\]
and that $\{\hat{p}\circ S_n\}_{n=1}^\infty$ is a sequence of step functions. 
   Now, let us show that $\{\hat{p}\circ S_n\}_{n=1}^\infty$ is uniformly convergent to $\hat{p}\circ F$ on $A$. 
   Take an arbitrary $\epsilon>0$. 
   On the one hand; since $\hat{p}:\mathcal{V}\to\mathbb{R}$ is continuous at $0$, there exists a $\delta>0$ such that $\eta\in\mathcal{V}$, $d(\eta,0)<\delta$ implies 
\[
   \hat{p}(\eta)=\big|\hat{p}(\eta)-\hat{p}(0)\big|<\epsilon.
\] 
   On the other hand; since $\{S_n\}_{n=1}^\infty$ is uniformly convergent to $F$ on $A$, there exists an $N\in\mathbb{N}$ such that $n\geq N$ implies 
\[
   d\bigl(F(x)-S_n(x),0\bigr)
   =d\bigl(F(x),S_n(x)\bigr)
   <\delta
\]
for all $x\in A$. 
   Consequently, $n\geq N$ implies $\hat{p}\bigl(F(x)-S_n(x)\bigr)<\epsilon$, and then
\[
   \big|\hat{p}(F(x))-\hat{p}(S_n(x))\big|
   \leq \hat{p}\bigl(F(x)-S_n(x)\bigr)
   <\epsilon
\]   
for all $x\in A$. 
   Hence $\{\hat{p}\circ S_n\}_{n=1}^\infty$ is uniformly convergent to $\hat{p}\circ F$ on $A$, and we conclude that $\hat{p}\circ F:A\to\mathbb{R}$ is regulated. 
   In addition, one has 
\[
   \int_A(\hat{p}\circ S_n)(x)d\mu(x)
   \stackrel{\eqref{eq-6.1.1}}{=}\sum_{i=1}^{k_n}\mu(A_{n,i})\hat{p}(\xi_{n,i})
   \geq\hat{p}\Big(\sum_{i=1}^{k_n}\mu(A_{n,i})\xi_{n,i}\Big)
   \stackrel{\eqref{eq-6.1.1}}{=}\hat{p}\Big(\int_AS_n(x)d\mu(x)\Big)
\] 
for all $n\in\mathbb{N}$, and therefore  
\[
\begin{split}
   \int_A(\hat{p}\circ F)(x)d\mu(x)
   =\lim_{n\to\infty}\int_A(\hat{p}\circ S_n)(x)d\mu(x)
   \geq\lim_{n\to\infty}\hat{p}\Big(\int_AS_n(x)d\mu(x)\Big)
  &=\hat{p}\Big(\lim_{n\to\infty}\int_AS_n(x)d\mu(x)\Big) \quad\mbox{($\because$ $\hat{p}$ is continuous)}\\
  &=\hat{p}\Big(\int_AF(x)d\mu(x)\Big). 
\end{split} 
\]  
\end{proof}

\begin{remark}\label{rem-6.1.6}
   In terms of Lemmas \ref{lem-6.1.2}-(1) and \ref{lem-6.1.5}-(1), the sets of step functions and regulated mappings are vector spaces over $\mathbb{K}$, respectively.   
\end{remark}

\subsection{A remark on regulated integrals of real-valued functions}\label{subsec-6.1.4}
   In case of $\mathcal{V}=\mathbb{R}$, we can consider two kinds of integrals, the regulated integral in the sense of Definition \ref{def-6.1.4}-(2) and the $\mu$-integral in the sense of the measure theory, which we here dare to write $\int_AF(x)d\mu(x)$ the former and $\mu$-$\int_AF(x)d\mu(x)$ the latter.
   In this subsection we are going to confirm that
\[
   \int_AF(x)d\mu(x)=\mbox{$\mu$-}\!\!\int_AF(x)d\mu(x)
\]
for all regulated mappings $F:A\to\mathbb{R}$, where $A\in\mathscr{B}$ with $\mu(A)<\infty$.\par

   Suppose that $\mathcal{V}=\mathbb{R}$, $A\in\mathscr{B}$ and $\mu(A)<\infty$. 
   Here, $(\mathcal{V},\{p_\ell\}_{\ell\in\mathbb{N}})=(\mathbb{R},|\cdot|)$ follows.\par

   Any step function $S=\sum_{i=1}^kc_{A_i}(x)\lambda_i:A\to\mathbb{R}$ is $\mu$-integrable on $A$, and it is immediate from \eqref{eq-6.1.1} that 
\begin{equation}\label{eq-1}\tag*{\textcircled{1}}
   \int_AS(x)d\mu(x)
   =\sum_{i=1}^k\mu(A_i)\lambda_i
   =\mbox{$\mu$-}\!\!\int_AS(x)d\mu(x).
\end{equation}   
   Now, let $F:A\to\mathbb{R}$ be a regulated mapping.
   Then, by Definition \ref{def-6.1.4}-(1) there exists a sequence $\{S_n\}_{n=1}^\infty$ of step functions which is uniformly convergent to $F$ on $A$.
   On the one hand; Definition \ref{def-6.1.4}-(2) implies that  
\begin{equation}\label{eq-2}\tag*{\textcircled{2}}
   \Big|\int_AF(x)d\mu(x)-\int_AS_n(x)d\mu(x)\Big|\longrightarrow 0 \quad(n\to\infty).
\end{equation}   
   On the other hand; since $\{S_n\}_{n=1}^\infty$ is uniformly convergent to $F$ on $A$, there exists an $M\in\mathbb{N}$ such that $m\geq M$ implies $|S_m(x)-S_M(x)|<1$ for all $x\in A$; and it follows that  
\[
   \mbox{$|S_m(x)|\leq |S_M(x)|+1$ for all $m\geq M$ and $x\in A$}.
\]
   This, together with Lebesgue's convergence theorem, tells us that $F$ is $\mu$-integrable on $A$ and   
\begin{equation}\label{eq-3}\tag*{\textcircled{3}}
   \Big|\mbox{$\mu$-}\!\!\int_AS_n(x)d\mu(x)-\mbox{$\mu$-}\!\!\int_AF(x)d\mu(x)\Big|\longrightarrow 0 \quad(n\to\infty),
\end{equation}   
since $\mu(A)<\infty$, each $S_n$ is $\mu$-integrable on $A$ and $\{S_n\}_{n=1}^\infty$ is uniformly convergent to $F$ on $A$.   
   In view of \ref{eq-1}, \ref{eq-2} and \ref{eq-3} we see that 
\begin{multline*}
   \Big|\int_AF(x)d\mu(x)-\mbox{$\mu$-}\!\!\int_AF(x)d\mu(x)\Big|
   \leq\Big|\int_AF(x)d\mu(x)-\int_AS_n(x)d\mu(x)\Big|+\Big|\int_AS_n(x)d\mu(x)-\mbox{$\mu$-}\!\!\int_AF(x)d\mu(x)\Big|\\
   =\Big|\int_AF(x)d\mu(x)-\int_AS_n(x)d\mu(x)\Big|+\Big|\mbox{$\mu$-}\!\!\int_AS_n(x)d\mu(x)-\mbox{$\mu$-}\!\!\int_AF(x)d\mu(x)\Big|
   \longrightarrow 0 \quad(n\to\infty).
\end{multline*}
   So, one has $\int_AF(x)d\mu(x)=\mu$-$\int_AF(x)d\mu(x)$.

\begin{remark}\label{rem-6.1.7}
   By the arguments above and the measure theory we deduce that for any $A\in\mathscr{B}$ and $\mu(A)<\infty$, all regulated mappings $f,g:A\to\mathbb{R}$ are $\mu$-integrable on $A$, and that the inequalities 
\[
   0\leq\int_Af(x)d\mu(x)\leq\int_Ag(x)d\mu(x)
\]
hold in the case where $0\leq f(x)\leq g(x)$ for all $x\in A$.
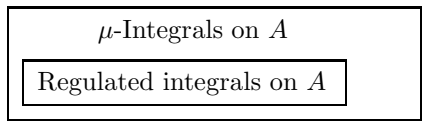
\begin{center}
\unitlength=1mm
\begin{picture}(96,16)
\put(40,1){\line(1,0){55}}
\put(42,3){\line(1,0){43}}
\put(42,3){\line(0,1){6}}
\put(44,5){Regulated integrals on $A$}
\put(42,9){\line(1,0){43}}
\put(85,3){\line(0,1){6}} 
\put(52,12){$\mu$-Integrals on $A$}
\put(40,1){\line(0,1){15}}
\put(40,16){\line(1,0){55}}
\put(95,1){\line(0,1){15}} 
\put(0,10){\underline{Case $\mathcal{V}=\mathbb{R}$, $\mu(A)<\infty$}:}
\end{picture}
\end{center}
\end{remark}

\subsection{An example of regulated mapping}\label{subsec-6.1.5}
   The following proposition provides us with examples of regulated mappings.

\begin{proposition}\label{prop-6.1.8}
   Suppose that 
\begin{enumerate}
\item[{\rm (s1)}]
   $X$ is a Hausdorff $($topological$)$ space, 
\item[{\rm (s2)}]
   $\mathscr{B}$ includes the set of open subsets of $X$, 
\item[{\rm (s3)}] 
   $K$ is a compact subset of $X$ with $\mu(K)<\infty$.  
\end{enumerate}
   Then, every continuous mapping $F:K\to\mathcal{V}$ is regulated.
   Accordingly $F$ has a regulated integral on $K$. 
\end{proposition}
\begin{proof}
   For $\xi_0\in\mathcal{V}$ and $r>0$ we set an open subset $B(\xi_0,r)\subset\mathcal{V}$ as $B(\xi_0,r):=\{\xi\in\mathcal{V} \,|\, d(\xi,\xi_0)<r\}$.\par

   By Definition \ref{def-6.1.4} it suffices to construct a sequence of step functions which is uniformly convergent to $F$ on $K$. 
   For any $n\in\mathbb{N}$, one has $F(K)\subset\bigcup_{\eta\in F(K)}B(\eta,1/n)$. 
   Since $F(K)$ is compact in $\mathcal{V}$, there exist finite elements $\eta_1,\eta_2,\dots,\eta_{\ell_n}\in F(K)$ satisfying $F(K)\subset\bigcup_{i=1}^{\ell_n}B(\eta_i,1/n)$. 
   Then, we put 
\[
   \mbox{$A_i:=F^{-1}\bigl(B(\eta_i,1/n)\bigr)-\bigcup_{j=1}^{i-1}F^{-1}\bigl(B(\eta_j,1/n)\bigr)$ for $1\leq i\leq\ell_n$}.
\]
   Since $F:K\to\mathcal{V}$ is continuous, it follows from (s2) that $F^{-1}\bigl(B(\eta_i,1/n)\bigr)\in\mathscr{B}$, so that $A_i\in\mathscr{B}$ ($1\leq i\leq\ell_n$).
   Moreover, we deduce $K=\coprod_{i=1}^{\ell_n}A_i$ by a direct computation. 
   Now, let $S_n(x):=\sum_{i=1}^{\ell_n}c_{A_i}(x)\eta_i$ for $x\in K$.
   This $S_n:K\to\mathcal{V}$ is a step function and satisfies $d\bigl(S_n(x),F(x)\bigr)<1/n$ for all $x\in K$.
   Accordingly $\{S_n\}_{n=1}^\infty$ is the desired sequence of step functions.
\end{proof}

\begin{remark}\label{rem-6.1.9}
   By considering Lemma \ref{lem-6.1.5}-(3) in case $\mathcal{W}=\mathbb{K}$ we see that the integral $\int_KF(x)d\mu(x)$ in Proposition \ref{prop-6.1.8} accords with the integral in Bourbaki \cite[Section 3, Chapter III]{Br}, the integral in Rudin \cite[p.77, Definition 3.26]{Ru}, and so on.
\end{remark}

\section{An application ($K$-finite vectors)}\label{sec-6.2}
   The setting of Section \ref{sec-6.2} is as follows:
\begin{itemize}
\item
   $K$ is a compact Lie group, 
\item
   $\mathcal{V}$ is a Fr\'{e}chet space over $\mathbb{C}$, 
\item
   $d$ is a suitable complete metric for $\mathcal{V}$ which induces a topology identical to  the original one,    
\item 
   $\varrho:K\to GL(\mathcal{V})$, $k\mapsto\varrho(k)$, is a (group) homomorphism, where it does not matter whether $\varrho$ is continuous here.   
\end{itemize}
   In this section we apply the theory on regulated integrals to conclude 
\begin{proposition}\label{prop-6.2.1}
   Suppose that {\rm (S)} the mapping $\pi_\varrho:K\times\mathcal{V}\to\mathcal{V}$, $(k,\xi)\mapsto\varrho(k)\xi$, is continuous.\footnote{This supposition (S) means that $\varrho$ is a continuous representation of $K$ on $\mathcal{V}$ (see Definition \ref{def-11.0.1}).}  
   Then, 
\[
   \mathcal{V}_K:=\{\eta\in\mathcal{V} \,|\, \dim_\mathbb{C}\operatorname{span}_\mathbb{C}\{\varrho(k)\eta : k\in K\}<\infty\}
\]
is a $\varrho(K)$-invariant, complex vector subspace of $\mathcal{V}$, and moreover, it is dense in $\mathcal{V}$.
\end{proposition}
   
   The main purpose of this section is to prove Proposition \ref{prop-6.2.1}.

\subsection{A preparation for proving Proposition \ref{prop-6.2.1}}\label{subsec-6.2.1} 
   In order to prove Proposition \ref{prop-6.2.1} we study $\mathcal{C}(K,\mathbb{C})$ first, where $\mathcal{C}(K,\mathbb{C})=\{\phi:K\to\mathbb{C} \,|\, \mbox{$\phi$ is continuous}\}$. 
   Let us define a norm $\|\cdot\|$ on the complex vector space $\mathcal{C}(K,\mathbb{C})$ by 
\[
   \mbox{$\|\phi\|:=\sup\big\{|\phi(x)| : x\in K\big\}$ for $\phi\in\mathcal{C}(K,\mathbb{C})$}, 
\]
consider $\mathcal{C}(K,\mathbb{C})$ as a complex Banach space with this norm, and define a homomorphism $\rho:K\to GL\bigl(\mathcal{C}(K,\mathbb{C})\bigr)$, $k\mapsto\rho(k)$, as follows:
\[
   \mbox{$\bigl(\rho(k)\phi\bigr)(x):=\phi(k^{-1}x)$ for $\phi\in\mathcal{C}(K,\mathbb{C})$ and $x\in K$}.
\] 
   Then, it turns out that $\|\rho(k)\phi\|=\|\phi\|$ for all $(k,\phi)\in K\times\mathcal{C}(K,\mathbb{C})$, and that the mapping $\pi_\rho:K\times\mathcal{C}(K,\mathbb{C})\to\mathcal{C}(K,\mathbb{C})$, $(k,\phi)\mapsto\rho(k)\phi$, is continuous. 
   Furthermore, by the Peter-Weyl theory one knows

\begin{proposition}[{e.g.\ Sugiura \cite[p.27, Theorem 3.5]{Su0}}]\label{prop-6.2.2}
\[
   \mathcal{C}(K,\mathbb{C})_K:=\{\varphi\in\mathcal{C}(K,\mathbb{C}) \,|\, \dim_\mathbb{C}\operatorname{span}_\mathbb{C}\{\rho(k)\varphi : k\in K\}<\infty\}
\]
is a $\rho(K)$-invariant, complex vector subspace of $\mathcal{C}(K,\mathbb{C})$ and is dense in $\mathcal{C}(K,\mathbb{C})$.\footnote{Remark.\ $\mathcal{C}(K,\mathbb{C})_K$ includes the representative ring of the compact Lie group $K$.}
\end{proposition}

   From now on, let us suppose that 
\begin{enumerate}
\item[(S)]
   $\pi_\varrho:K\times\mathcal{V}\to\mathcal{V}$, $(k,\xi)\mapsto\varrho(k)\xi$, is continuous,
\end{enumerate}
and denote by $\mu$ a non-zero left-invariant Haar measure on $\mathscr{B}$, where $\mathscr{B}$ is the Borel field on the compact Lie group $K$.   
   Then, for a $\xi\in\mathcal{V}$, Proposition \ref{prop-6.1.8} and (S) allow us to define a mapping $F_\xi:\mathcal{C}(K,\mathbb{C})\to\mathcal{V}$ by 
\begin{equation}\label{eq-6.2.3}
   \mbox{$\displaystyle{F_\xi(\phi):=\int_K\phi(x)(\varrho(x)\xi)d\mu(x)}$ for $\phi\in\mathcal{C}(K,\mathbb{C})$},
\end{equation} 
since $f_\phi:K\to\mathcal{V}$, $x\mapsto\phi(x)(\varrho(x)\xi)$, is continuous and $\mu(K)<\infty$ (cf.\ (p6) in Theorem \ref{thm-5.1.2}).

\begin{lemma}\label{lem-6.2.4}
   On the supposition {\rm (S)}$;$ for each $\xi\in\mathcal{V}$, the mapping $F_\xi:\mathcal{C}(K,\mathbb{C})\to\mathcal{V}$, $\phi\mapsto\int_K\phi(x)(\varrho(x)\xi)d\mu(x)$, is continuous.
\end{lemma}
\begin{proof}
   Fix any $\xi\in\mathcal{V}$, $\epsilon>0$, $\psi\in\mathcal{C}(K,\mathbb{C})$ and any continuous seminorm $\hat{p}$ on $\mathcal{V}$.
   Let us verify that $F_\xi:\mathcal{C}(K,\mathbb{C})\to\mathcal{V}$ is continuous at $\psi$.
   For $\phi\in\mathcal{C}(K,\mathbb{C})$ we suppose that 
\[
   \|\psi-\phi\|<\dfrac{\epsilon}{1+\int_K\hat{p}\bigl(\varrho(k)\xi\bigr)d\mu(k)},
\] 
where it should be remarked that $0\leq\int_K\hat{p}\bigl(\varrho(k)\xi\bigr)d\mu(k)<\infty$, cf.\ Remark \ref{rem-6.1.7}.  
   The mapping $f_\psi$ (resp.\ $f_\phi$)$:K\to\mathcal{V}$, $x\mapsto\psi(x)(\varrho(x)\xi)$ (resp.\ $\mapsto\phi(x)(\varrho(x)\xi)$), is continuous, and so it is regulated due to Proposition \ref{prop-6.1.8}. 
   Hence one shows 
\[
\begin{split}
   \hat{p}\bigl(F_\xi(\psi)-F_\xi(\phi)\bigr)
  &\stackrel{\eqref{eq-6.2.3}}{=}\hat{p}\Big(\int_Kf_\psi(x)d\mu(x)-\int_Kf_\phi(x)d\mu(x)\Big)
   =\hat{p}\Big(\int_K(f_\psi-f_\phi)(x)d\mu(x)\Big) \quad\mbox{($\because$ Lemma \ref{lem-6.1.5}-(1))}\\
  &\leq\int_K\hat{p}\bigl((f_\psi-f_\phi)(x)\bigr)d\mu(x) 
   =\int_K|\psi(x)-\phi(x)|\hat{p}\bigl(\varrho(x)\xi\bigr)d\mu(x)
   \leq\int_K\|\psi-\phi\|\hat{p}\bigl(\varrho(x)\xi\bigr)d\mu(x)\\
  &=\|\psi-\phi\|\int_K\hat{p}\bigl(\varrho(x)\xi\bigr)d\mu(x)
  \leq\dfrac{\epsilon}{1+\int_K\hat{p}\bigl(\varrho(k)\xi\bigr)d\mu(k)}\int_K\hat{p}\bigl(\varrho(x)\xi\bigr)d\mu(x) 
  \leq\epsilon
\end{split}
\]
(see Lemma \ref{lem-6.1.5}-(4), Remark \ref{rem-6.1.7} also).
   Therefore $F_\xi:\mathcal{C}(K,\mathbb{C})\to\mathcal{V}$ is continuous at $\psi$.
\end{proof}

\begin{lemma}\label{lem-6.2.5}
   On the supposition {\rm (S)}$;$ $F_\xi\bigl(\mathcal{C}(K,\mathbb{C})_K\bigr)\subset\mathcal{V}_K$ for all $\xi\in\mathcal{V}$.   
\end{lemma}
\begin{proof}
   Fix a $\varphi\in\mathcal{C}(K,\mathbb{C})_K$. 
   There exist finite vectors $\varphi_1,\dots,\varphi_n\in\mathcal{C}(K,\mathbb{C})$ and functions $\alpha_1,\dots,\alpha_n:K\to\mathbb{C}$ such that
\begin{equation}\label{eq-1}\tag*{\textcircled{1}}
   \mbox{$\rho(k)\varphi=\sum_{i=1}^n\alpha_i(k)\varphi_i$ for all $k\in K$} 
\end{equation} 
by virtue of $\dim_\mathbb{C}\operatorname{span}_\mathbb{C}\{\rho(k)\varphi : k\in K\}<\infty$. 
   Therefore, for any $k\in K$ 
\allowdisplaybreaks{
\begin{align*}
   \varrho(k)\bigl(F_\xi(\varphi)\bigr)
  &\stackrel{\eqref{eq-6.2.3}}{=}\varrho(k)\Big(\int_K\varphi(x)(\varrho(x)\xi)d\mu(x)\Big)
   =\int_K\varrho(k)\bigl(\varphi(x)(\varrho(x)\xi)\bigr)d\mu(x) \quad\mbox{($\because$ (S), Lemma \ref{lem-6.1.5}-(3))}\\
  &=\int_K\varphi(x)(\varrho(kx)\xi)d\mu(x)
   =\int_K\varphi(k^{-1}y)(\varrho(y)\xi)d\mu(k^{-1}y) \quad\mbox{(by $x\to k^{-1}y$)}\\
  &\stackrel{\ref{eq-1}}{=}\int_K\sum_{i=1}^n\alpha_i(k)\varphi_i(y)(\varrho(y)\xi)d\mu(k^{-1}y) 
   =\int_K\sum_{i=1}^n\alpha_i(k)\varphi_i(y)(\varrho(y)\xi)d\mu(y) \quad\mbox{($\because$ $\mu$ is left-invariant)}\\
  &=\sum_{i=1}^n\alpha_i(k)\int_K\varphi_i(y)(\varrho(y)\xi)d\mu(y)
   \stackrel{\eqref{eq-6.2.3}}{=}\sum_{i=1}^n\alpha_i(k)F_\xi(\varphi_i). 
\end{align*}}Hence $\dim_\mathbb{C}\operatorname{span}_\mathbb{C}\{\varrho(k)\bigl(F_\xi(\varphi)\bigr) : k\in K\}\leq n<\infty$, and so $F_\xi(\varphi)\in\mathcal{V}_K$.
\end{proof}

\begin{lemma}\label{lem-6.2.6}
   On the supposition {\rm (S)}$;$ for any $\xi\in\mathcal{V}$, there exists a sequence $\{\phi_n\}_{n=1}^\infty\subset\mathcal{C}(K,\mathbb{C})$ satisfying 
\[
   \lim_{n\to\infty}d\bigl(\xi,F_\xi(\phi_n)\bigr)=0.
\]   
\end{lemma}
\begin{proof}
   Fix an $\epsilon>0$ and a $\xi\in\mathcal{V}$.
   Since $K$ is a smooth manifold, one can construct a strictly decreasing sequence $\{U_n\}_{n=1}^\infty$ of open neighborhoods of $e\in K$ so that 
\begin{equation}\label{eq-1}\tag*{\textcircled{1}}
   U_1\supset\overline{U_2}\supset U_2\supset\overline{U_3}\supset\cdots\supset U_n\supset\overline{U_{n+1}}\supset U_{n+1}\supset\overline{U_{n+2}}\supset\cdots\supset\{e\}
\end{equation}
and a sequence $\{\psi_n\}_{n=1}^\infty$ of smooth functions (which are sometimes called {\it bump functions}) so that
\begin{enumerate}
\item
   $0\leq\psi_n(x)\leq 1$ for all $x\in K$, 
\item 
   $\psi_n(y)=1$ for all $y\in U_{n+1}$,
\item
   $\psi_n(z)=0$ for all $z\in K-U_n$.   
\end{enumerate}
   For these functions one can assert that $0<\mu(U_{n+1})\leq\int_K\psi_n(x)d\mu(x)\leq\mu(U_n)\leq\mu(K)<\infty$ for all $n\in\mathbb{N}$ because $\mu$ is a Haar measure on the compact Lie group $K$, cf.\ (p8), (p6) in Theorem \ref{thm-5.1.2}. 
   Setting $\phi_n:=\dfrac{1}{\int_K\psi_n(x)d\mu(x)}\psi_n$ for $n\in\mathbb{N}$, we conclude that $\phi_n\in\mathcal{C}(K,\mathbb{C})$, $\int_K\phi_n(x)d\mu(x)=1$, and $0\leq\phi_n(x)$ for all $n\in\mathbb{N}$ and $x\in K$; in addition, 
\begin{equation}\label{eq-2}\tag*{\textcircled{2}}
\begin{split}
   \xi-F_\xi(\phi_n)
  &\stackrel{\eqref{eq-6.2.3}}{=}\xi-\int_K\phi_n(x)(\varrho(x)\xi)d\mu(x)
   =\int_K\phi_n(x)\xi d\mu(x)-\int_K\phi_n(x)(\varrho(x)\xi)d\mu(x)\\
  &=\int_K\phi_n(x)(\xi-\varrho(x)\xi)d\mu(x) \quad\mbox{($\because$ Lemma \ref{lem-6.1.5}-(1), Proposition \ref{prop-6.1.8})},
\end{split}  
\end{equation}
where $\xi=\int_K\phi_n(x)\xi d\mu(x)$ follows from $\int_K\phi_n(x)d\mu(x)=1$.\footnote{Indeed; suppose a sequence $\{f_m=\sum_ic_{A_{m,i}}\lambda_{m,i}:K\to\mathbb{R}\}_{m=1}^\infty$ of step functions to be uniformly convergent to $\phi_n$ on $K$. 
   Then it follows from Definition \ref{def-6.1.4}-(2) and $\int_K\phi_n(x)d\mu(x)=1$ that $\lim_m\int_Kf_m(x)d\mu(x)=\int_K\phi_n(x)d\mu(x)=1$. 
   Besides, the sequence $\{f_m\xi=\sum_ic_{A_{m,i}}\lambda_{m,i}\xi:K\to\mathcal{V}\}_{m=1}^\infty$ consists of step functions and is uniformly convergent to $\phi_n\xi$ on $K$. 
   Accordingly $\int_K\phi_n(x)\xi d\mu(x)=\lim_m\int_K f_m(x)\xi d\mu(x)\stackrel{\eqref{eq-6.1.1}}{=}\lim_m\sum_i\mu(A_{m,i})\lambda_{m,i}\xi=\bigl(\lim_m(\sum_i\mu(A_{m,i})\lambda_{m,i})\bigr)\xi\stackrel{\eqref{eq-6.1.1}}{=}\bigl(\lim_m\int_Kf_m(x)d\mu(x)\bigr)\xi=\xi$.
   Hence $\int_K\phi_n(x)\xi d\mu(x)=\xi$.}
   Now, let $\hat{p}$ be an arbitrary continuous seminorm on $\mathcal{V}$. 
   Since $K\ni k\mapsto\varrho(k)\xi\in\mathcal{V}$ is continuous at $e\in K$ and $\varrho(e)\xi=\xi$, there exists an open neighborhood $O$ of $e\in K$ satisfying  
\begin{equation}\label{eq-3}\tag*{\textcircled{3}}
   \mbox{$\hat{p}\bigl(\xi-\varrho(y)\xi\bigr)<\epsilon$ for all $y\in O$}.
\end{equation}
   By \ref{eq-1} there exists an $N_{\hat{p}}\in\mathbb{N}$ such that $n\geq N_{\hat{p}}$ implies $U_n\subset O$, and then
\allowdisplaybreaks{
\begin{align*}
   \hat{p}\bigl(\xi-F_\xi(\phi_n)\bigr)
  &\stackrel{\ref{eq-2}}{=}\hat{p}\Big(\int_K\phi_n(x)\bigl(\xi-\varrho(x)\xi\bigr)d\mu(x)\Big)
   \leq\int_K\phi_n(x)\hat{p}\bigl(\xi-\varrho(x)\xi\bigr)d\mu(x) \quad\mbox{($\because$ Lemma \ref{lem-6.1.5}-(4), $0\leq\phi_n(x)$)}\\
  &=\int_{U_n}\phi_n(y)\hat{p}\bigl(\xi-\varrho(y)\xi\bigr)d\mu(y)+\int_{K-U_n}\phi_n(z)\hat{p}\bigl(\xi-\varrho(z)\xi\bigr)d\mu(z) \quad\mbox{($\because$ Lemma \ref{lem-6.1.5}-(2))}\\
  &=\int_{U_n}\phi_n(y)\hat{p}\bigl(\xi-\varrho(y)\xi\bigr)d\mu(y) \quad\mbox{($\because$ $\phi_n=0$ on $K-U_n$)}\\
  &\leq\epsilon\int_{U_n}\phi_n(y)d\mu(y) \quad\mbox{($\because$ $U_n\subset O$, \ref{eq-3})}\\
  &\leq\epsilon\int_K\phi_n(x)d\mu(x) \quad\mbox{($\because$ $0\leq\phi_n(x)$)}\\
  &=\epsilon.
\end{align*}}For this reason the sequence $\{\phi_n\}_{n=1}^\infty$ satisfies $\displaystyle{\lim_{n\to\infty}d\bigl(\xi,F_\xi(\phi_n)\bigr)=0}$.
\end{proof}

\subsection{Proof of Proposition \ref{prop-6.2.1}}\label{subsec-6.2.2} 
   Now, let us prove Proposition \ref{prop-6.2.1}.
\begin{proof}[Proof of Proposition {\rm \ref{prop-6.2.1}}]
   We only prove that $\mathcal{V}_K$ is dense in $\mathcal{V}=(\mathcal{V},d)$.
   Take any $\xi_0\in\mathcal{V}$ and $\epsilon>0$. 
   By Lemma \ref{lem-6.2.6} there exists a $\psi\in\mathcal{C}(K,\mathbb{C})$ satisfying  
\[
   d\bigl(\xi_0,F_{\xi_0}(\psi)\bigr)<\epsilon/2.
\] 
   Since $F_{\xi_0}:\mathcal{C}(K,\mathbb{C})\to\mathcal{V}$ is continuous at $\psi$ (cf.\ Lemma \ref{lem-6.2.4}), there exists an open neighborhood $W$ of $\psi\in\mathcal{C}(K,\mathbb{C})$ such that  
\[
   d\bigl(F_{\xi_0}(\psi),F_{\xi_0}(\phi)\bigr)<\epsilon/2
\]
for all $\phi\in W$.
   Proposition \ref{prop-6.2.2} enables us to take an element $\varphi\in W\cap \mathcal{C}(K,\mathbb{C})_K$. 
   Then, it follows from Lemma \ref{lem-6.2.5} that $F_{\xi_0}(\varphi)\in\mathcal{V}_K$, and we have 
\[
   d\bigl(\xi_0,F_{\xi_0}(\varphi)\bigr)
   \leq d\bigl(\xi_0,F_{\xi_0}(\psi)\bigr)+d\bigl(F_{\xi_0}(\psi),F_{\xi_0}(\varphi)\bigr)
   <\epsilon.
\]   
   This completes the proof of Proposition \ref{prop-6.2.1}.  
\end{proof}

\chapter{Elliptic elements and elliptic adjoint orbits}\label{ch-7}
   In this chapter we recall the definitions of elliptic element and elliptic (adjoint) orbit, and show some fundamental properties of elliptic elements and elliptic orbits. 
   The setting of Chapter \ref{ch-7} is as follows: 
\begin{itemize}
\item
   $G$ is a connected, real semisimple Lie group, 
\item 
   $\frak{g}_\mathbb{C}$ is the complexification of the Lie algebra $\frak{g}$.
\end{itemize}
   Remark that the Lie group $G$ satisfies the second countability axiom since it is connected.

\section{Definitions of elliptic element and elliptic orbit}\label{sec-7.1}
   Here are the definitions of elliptic element and elliptic orbit.
\begin{definition}[{cf.\ Kobayashi \cite{Ko}}]\label{def-7.1}
\begin{enumerate}
\item[]
\item[(i)]
   An element $Z\in\frak{g}$ is said to be {\it semisimple},\index{semisimple element@semisimple element\dotfill} if the linear transformation $\operatorname{ad}Z:\frak{g}\to\frak{g}$, $X\mapsto[Z,X]$, is semisimple.\footnote{This condition is equivalent to the condition that $\operatorname{ad}Z:\frak{g}_\mathbb{C}\to\frak{g}_\mathbb{C}$ is represented by a diagonal matrix relative to some complex basis of $\frak{g}_\mathbb{C}$.}
\item[(ii)]
   The adjoint orbit $\operatorname{Ad}G(Z)$ of $G$ through a semisimple element $Z\in\frak{g}$ is called a {\it semisimple} ({\it adjoint}) {\it orbit}.\index{semisimple (adjoint) orbit@semisimple (adjoint) orbit\dotfill}   
\item[(iii)]
   An element $T\in\frak{g}$ is said to be {\it elliptic},\index{elliptic element@elliptic element\dotfill} if it is semisimple and all the eigenvalues of $\operatorname{ad}T$ are purely imaginary.
\item[(iv)]
   The adjoint orbit $\operatorname{Ad}G(T)$ of $G$ through an elliptic element $T\in\frak{g}$ is called an {\it elliptic adjoint orbit} or an {\it elliptic orbit}.\index{elliptic (adjoint) orbit@elliptic (adjoint) orbit\dotfill} 
\end{enumerate} 
\end{definition}

   Needless to say, $0$ is an elliptic element of $\frak{g}$.

\section{Properties of elliptic elements}\label{sec-7.2}
   Let us clarify some properties of elliptic elements.

\begin{lemma}\label{lem-7.2.1}
   Let $T$ be a non-zero, elliptic element of $\frak{g}$, and let $S^1:=\{\exp tT : t\in\mathbb{R}\}$.
\begin{enumerate}
\item[{\rm (1)}]
   $S^1$ is a $1$-dimensional connected, closed Abelian subgroup of $G$. 
\item[{\rm (2)}]
   Suppose that the center $Z(G)$ of $G$ is finite. 
   Then, $S^1$ is compact. 
\end{enumerate}   
\end{lemma}

\begin{remark}\label{rem-7.2.2}
   We cannot omit the supposition ``the center $Z(G)$ of $G$ is finite'' from Lemma \ref{lem-7.2.1}-(2). 
   cf.\ Example \ref{ex-7.2.3}.
\end{remark}

   From now on, we are going to prove Lemma \ref{lem-7.2.1}.
\begin{proof}[Proof of Lemma {\rm \ref{lem-7.2.1}}]
   (1). 
   $S^1$ coincides with the connected Lie subgroup of $G$ corresponding to the subalgebra $\operatorname{span}_\mathbb{R}\{T\}$ of $\frak{g}$. 
   Hence, let us only verify that $S^1$ is a closed subset of $G$.
   The element $T$ is non-zero elliptic, so there exist $\lambda_1,\dots,\lambda_k>0$ and an ordered real basis $\{X_1,Y_1,X_2,Y_2,\dots,X_k,Y_k,Z_{2k+1},\dots,Z_N\}$ of $\frak{g}$ such that 
\[
\begin{array}{ll}
   \mbox{$\operatorname{ad}T(X_i)=\lambda_iY_i$, $\operatorname{ad}T(Y_i)=-\lambda_iX_i$ ($1\leq i\leq k$)},
   & \mbox{$\operatorname{ad}T(Z_j)=0$ ($2k+1\leq j\leq N=\dim_\mathbb{R}\frak{g}$)}.
\end{array}   
\]
   Relative to this ordered basis, $\operatorname{Ad}(\exp tT)=\exp t\operatorname{ad}T:\frak{g}\to\frak{g}$ is represented by 
\[
\begin{array}{ll}
   \operatorname{Ad}(\exp tT)
   =\left(\begin{array}{@{}c@{\,\,}c@{\,\,}c@{\,\,}c@{}}
                  R_1 &        &     & \mbox{\huge{O}}\\
                      & \ddots &     &                \\
                      &        & R_k &                \\
      \mbox{\huge{O}} &        &     & I_{N-2k}
     \end{array}\right)\!,
& R_i=\begin{pmatrix} \cos(t\lambda_i) & -\sin(t\lambda_i)\\ \sin(t\lambda_i) & \cos(t\lambda_i)\end{pmatrix}\!,\, 1\leq i\leq k, 
\end{array}
\]
where $I_{N-2k}$ stands for the identity matrix of order $(N-2k)$. 
   Accordingly 
\begin{equation}\label{eq-1}\tag*{\textcircled{1}}
   \mbox{$\operatorname{Ad}S^1$ is a compact subgroup of $GL(\frak{g})=GL(N,\mathbb{R})$}.
\end{equation}
   Since the adjoint representation $\operatorname{Ad}:G\to GL(\frak{g})$ is a continuous homomorphism and the Lie group $G$ is connected, we conclude that 
\begin{equation}\label{eq-2}\tag*{\textcircled{2}}
   \mbox{$\operatorname{Ad}^{-1}\bigl(\operatorname{Ad}S^1\bigr)=S^1Z(G)$ is a closed subgroup of $G$}
\end{equation}
by \ref{eq-1}.
   Here it follows from $\dim_\mathbb{R}Z(G)=0$ that $S^1$ is the identity component of the Lie group $S^1Z(G)$. 
   Therefore $S^1$ is closed in $S^1Z(G)$. 
   This and \ref{eq-2} assure that $S^1$ is a closed subset of $G$.\par

   (2). 
   Suppose that the center $Z(G)$ is finite. 
   On the one hand; $S^1\cap Z(G)$ is compact by the supposition. 
   On the other hand; $S^1/(S^1\cap Z(G))$ is homeomorphic to $\operatorname{Ad}S^1$ because $S^1$ satisfies the second countability axiom ($\because$ (1)) and $\operatorname{Ad}:S^1\to\operatorname{Ad}S^1$ is a surjective continuous homomorphism with kernel $S^1\cap Z(G)$. 
   Consequently, $S^1$ is compact due to \ref{eq-1}.  
\end{proof}

\begin{example}\label{ex-7.2.3}
   Let $\frak{g}:=\frak{sl}(2,\mathbb{R})=\left\{\begin{array}{@{}c|c@{}} \begin{pmatrix} x & y\\ z & -x\end{pmatrix} & x,y,z\in\mathbb{R}\end{array}\right\}$ and let 
\[
   T:=\begin{pmatrix} 0 & 1\\ -1 & 0\end{pmatrix}\!.
\]   
   Then $T$ belongs to $\frak{g}$, $\left\{\begin{array}{@{}c@{\,}c@{\,}c@{}} E_1:=\begin{pmatrix} -i & 1\\ 1 & i\end{pmatrix}\!, & E_2:=\begin{pmatrix} i & 1\\ 1 & -i\end{pmatrix}\!, & E_3:=\begin{pmatrix} 0 & 1\\ -1 & 0\end{pmatrix}\end{array}\right\}$ is a complex basis of $\frak{g}_\mathbb{C}=\frak{sl}(2,\mathbb{C})$, and $\operatorname{ad}T:\frak{g}_\mathbb{C}\to\frak{g}_\mathbb{C}$ is represented by 
\[
   \operatorname{ad}T=\begin{pmatrix} 2i & 0 & 0\\ 0 & -2i & 0\\ 0 & 0 & 0\end{pmatrix}
\] 
relative to the basis. 
   This implies that $T$ is a non-zero elliptic element of $\frak{g}$. 
   Incidentally, $\operatorname{ad}T:\frak{g}\to\frak{g}$ is represented by  
\[
   \operatorname{ad}T=\begin{pmatrix} 0 & -2 & 0\\ 2 & 0 & 0\\ 0 & 0 & 0\end{pmatrix}
\]
relative to the real basis $\left\{\begin{array}{@{}c@{\,}c@{\,}c@{}} X_1:=\begin{pmatrix} 0 & 1\\ 1 & 0\end{pmatrix}\!, & Y_1:=\begin{pmatrix} 1 & 0\\ 0 & -1\end{pmatrix}\!, & Z_3:=\begin{pmatrix} 0 & 1\\ -1 & 0\end{pmatrix}\end{array}\right\}$ of $\frak{g}=\frak{sl}(2,\mathbb{R})$. 
   Now, $\operatorname{span}_\mathbb{R}\{T\}=\frak{so}(2)$ holds, and one can get a Cartan decomposition $\frak{g}=\frak{k}\oplus\frak{p}$ of $\frak{g}$ by setting
\[
\begin{array}{ll}
   \frak{k}:=\operatorname{span}_\mathbb{R}\{T\}, & \frak{p}:=\left\{\begin{array}{@{}c|c@{}} \begin{pmatrix} x & y\\ y & -x\end{pmatrix} & x,y\in\mathbb{R}\end{array}\right\}\!.
\end{array}   
\] 
   Let $G$ be a simply connected Lie group with Lie algebra $\frak{g}$, and let $G=KP$ denote the Cartan decomposition of $G$ corresponding to $\frak{g}=\frak{k}\oplus\frak{p}$. 
   In this setting, we have $\{\exp tT : t\in\mathbb{R}\}=K$, but $K$ is not compact because $K$ includes the center $Z(G)$ and $Z(G)$ is infinite (more precisely, $Z(G)=\mathbb{Z}$). 
   This implies that we cannot omit the supposition ``the center $Z(G)$ of $G$ is finite'' from Lemma \ref{lem-7.2.1}-(2). 
\end{example}

   The following lemma provides us with a criterion for judging whether an $X\in\frak{g}$ is elliptic or not. 
\begin{lemma}\label{lem-7.2.4}
   An element $X\in\frak{g}$ is elliptic if and only if there exists a Cartan involution $\theta_*$ of $\frak{g}$ so that $\theta_*(X)=X$.
\end{lemma}
\begin{proof}
   Assume that $X\neq0$ (otherwise our assertions are trivial).
   If $X$ is elliptic, then Lemma \ref{lem-7.2.1} assures the existence of a Cartan involution $\theta_*$ of $\frak{g}$ satisfying $\theta_*(X)=X$ (because, for a given compact subalgebra $\frak{s}^1\subset\frak{g}$ there always exists a maximal compact subalgebra $\frak{k}$ of $\frak{g}$ such that $\frak{s}^1\subset\frak{k}$).\par
   
   Conversely, suppose that a Cartan involution $\theta_*$ of $\frak{g}$ satisfies $\theta_*(X)=X$, and define an inner product $\langle\,\cdot\,,\,\cdot\,\rangle$ on $\frak{g}$ by 
\begin{equation}\label{eq-7.2.5}
   \mbox{$\langle Y,Z\rangle:=-B_\frak{g}\bigl(Y,\theta_*(Z)\bigr)$ for $Y,Z\in\frak{g}$},
\end{equation} 
where $B_\frak{g}$ stands for the Killing form of $\frak{g}$. 
   For an arbitrary $\operatorname{ad}X$-invariant vector subspace $\frak{m}\subset\frak{g}$, we take its orthogonal complement $\frak{m}^\perp$ in $\frak{g}$ with respect to $\langle\,\cdot\,,\,\cdot\,\rangle$. 
   Then, $\frak{g}=\frak{m}\oplus\frak{m}^\perp$ holds, and moreover, the vector space $\frak{m}^\perp$ is also $\operatorname{ad}X$-invariant because it follows from $\theta_*(X)=X$ that 
\begin{equation}\label{eq-1}\tag*{\textcircled{1}}
   \mbox{$\langle\operatorname{ad}X(Y),Z\rangle=-\langle Y,\operatorname{ad}X(Z)\rangle$ for all $Y,Z\in\frak{g}$}.
\end{equation}
   Consequently the linear transformation $\operatorname{ad}X:\frak{g}\to\frak{g}$ is semisimple; besides, all the eigenvalues of $\operatorname{ad}X$ are purely imaginary by \ref{eq-1}. 
   Hence, $X$ is elliptic.
\end{proof}

   By Lemma \ref{lem-7.2.4} one has
\begin{corollary}\label{cor-7.2.6}
   All elements of a compact semisimple Lie algebra are elliptic.
\end{corollary}

   Corollary \ref{cor-7.2.6} provides us with examples of elliptic orbits. 
\begin{example}[A complex Grassmann manifold]\label{ex-7.2.7}
   Let $G:=SU(n)=\{g\in SL(n,\mathbb{C}) \,|\, {}^t\overline{g}=g^{-1}\}$, and let 
\[
   T:=\sqrt{-1}\begin{pmatrix}
   (n-k)I_k & \mbox{O}\\
   \mbox{O} & -kI_{n-k}
   \end{pmatrix},
\] 
where $n\geq 2$ and $1\leq k\leq n-1$. 
   Then $T$ is an element of $\frak{g}$, and it is elliptic because $\frak{g}=\frak{su}(n)$ is a compact semisimple Lie algebra. 
   A direct computation yields 
\[
   C_G(T)
   =\left\{\begin{array}{@{\,}c|c@{\,}}
   \begin{pmatrix}
    A & \mbox{O}\\
    \mbox{O} & D
   \end{pmatrix}\in SL(n,\mathbb{C})
   & A\in U(k), D\in U(n-k)
   \end{array}\right\} 
   =S(U(k)\times U(n-k)),
\]
and the elliptic orbit $\operatorname{Ad}G(T)=G/C_G(T)=SU(n)/S(U(k)\times U(n-k))$ is a complex Grassmann manifold; in particular, it is a complex projective space in case of $k=1$ or $k=n-1$.
   Incidentally, the eigenvalue of $\operatorname{ad}T$ is $\pm n\sqrt{-1}$ or zero.
\end{example}

   The following lemma will be needed later (e.g.\ Chapter \ref{ch-8}):
\begin{lemma}\label{lem-7.2.8}
   Let $T$ be an elliptic element of $\frak{g}$, and set complex vector subspaces $\frak{g}^\lambda,\frak{u}^\pm\subset\frak{g}_\mathbb{C}$ as 
\begin{equation}\label{eq-7.2.9}
\begin{array}{lll}
   \mbox{$\frak{g}^\lambda:=\{W\in\frak{g}_\mathbb{C} \,|\, \operatorname{ad}T(W)=i\lambda W\}$ for $\lambda\in\mathbb{R}$}, 
   & \frak{u}^+:=\bigoplus_{\lambda>0}\frak{g}^\lambda,
   & \frak{u}^-:=\bigoplus_{\lambda>0}\frak{g}^{-\lambda},
\end{array}   
\end{equation} 
where $i:=\sqrt{-1}$ and $\frak{g}^\lambda=\{0\}$ in the case where $i\lambda$ is different from the eigenvalues of $\operatorname{ad}T$.
   In addition, denote by $\frak{l}_\mathbb{C}$ $($resp.\ $\frak{l})$ the centralizer of $T$ in $\frak{g}_\mathbb{C}$ $($resp.\ $\frak{g})$, by $\frak{u}$ the image of the linear mapping $\operatorname{ad}T:\frak{g}\to\frak{g}$, and by $\overline{\sigma}$ the conjugation of $\frak{g}_\mathbb{C}$ with respect to $\frak{g}$.
   Then, it follows that 
\begin{enumerate}
\item[{\rm (1)}]
   $\frak{g}_\mathbb{C}=\bigoplus_{\lambda\in\mathbb{R}}\frak{g}^\lambda=\frak{u}^+\oplus\frak{l}_\mathbb{C}\oplus\frak{u}^-$, $\frak{l}_\mathbb{C}=\frak{g}^0$,
\item[{\rm (2)}] 
   $\operatorname{Ad}z(\frak{g}^\lambda)\subset\frak{g}^\lambda$ for all $(z,\lambda)\in C_G(T)\times\mathbb{R}$, where $C_G(T):=\{z\in G \,|\, \operatorname{Ad}z(T)=T\}$,
\item[{\rm (2$'$)}]  
   $\operatorname{Ad}z(\frak{l}_\mathbb{C})\subset\frak{l}_\mathbb{C}$, $\operatorname{Ad}z(\frak{u}^+)\subset\frak{u}^+$, $\operatorname{Ad}z(\frak{u}^-)\subset\frak{u}^-$ for all $z\in C_G(T)$,
\item[{\rm (3)}]  
   $[\frak{g}^\lambda,\frak{g}^\mu]\subset\frak{g}^{\lambda+\mu}$ for all $\lambda,\mu\in\mathbb{R}$,
\item[{\rm (3$'$)}]  
   $[\frak{l}_\mathbb{C},\frak{l}_\mathbb{C}]\subset\frak{l}_\mathbb{C}$, $[\frak{l}_\mathbb{C},\frak{u}^+]\subset\frak{u}^+$, $[\frak{l}_\mathbb{C},\frak{u}^-]\subset\frak{u}^-$, $[\frak{u}^+,\frak{u}^+]\subset\frak{u}^+$, $[\frak{u}^-,\frak{u}^-]\subset\frak{u}^-$,  
\item[{\rm (4)}]  
   $B_{\frak{g}_\mathbb{C}}(\frak{g}^\lambda,\frak{g}^\mu)=\{0\}$ if $\lambda+\mu\neq0$, where $B_{\frak{g}_\mathbb{C}}$ is the Killing form of $\frak{g}_\mathbb{C}$, 
\item[{\rm (4$'$)}]  
   $B_{\frak{g}_\mathbb{C}}(\frak{l}_\mathbb{C},\frak{u}^+)=\{0\}$, $B_{\frak{g}_\mathbb{C}}(\frak{l}_\mathbb{C},\frak{u}^-)=\{0\}$, $B_{\frak{g}_\mathbb{C}}(\frak{u}^+,\frak{u}^+)=\{0\}$, $B_{\frak{g}_\mathbb{C}}(\frak{u}^-,\frak{u}^-)=\{0\}$,   
\item[{\rm (5)}]  
   $\overline{\sigma}(\frak{g}^\lambda)=\frak{g}^{-\lambda}$ for all $\lambda\in\mathbb{R}$,
\item[{\rm (5$'$)}] 
   $\overline{\sigma}(\frak{l}_\mathbb{C})=\frak{l}_\mathbb{C}$, $\overline{\sigma}(\frak{u}^+)=\frak{u}^-$, $\overline{\sigma}(\frak{u}^-)=\frak{u}^+$,
\item[{\rm (6)}] 
   $\overline{\sigma}\bigl(\operatorname{Ad}g(W)\bigr)=\operatorname{Ad}g\bigl(\overline{\sigma}(W)\bigr)$ for all $(g,W)\in G\times\frak{g}_\mathbb{C}$, 
\item[{\rm (i)}] 
   $\frak{g}=\frak{l}\oplus\frak{u}$, $T\in\frak{l}$, 
\item[{\rm (ii)}]  
   $\operatorname{Ad}z(\frak{l})\subset\frak{l}$, $\operatorname{Ad}z(\frak{u})\subset\frak{u}$ for all $z\in C_G(T)$, 
\item[{\rm (iii)}]  
   $\frak{l}=\{Y\in\frak{l}_\mathbb{C} \,|\, \overline{\sigma}(Y)=Y\}$, $\frak{u}=\{V+\overline{\sigma}(V) \,|\, V\in\frak{u}^+\}$.  
\end{enumerate}   
\end{lemma} 
\begin{proof}
   (1). 
   Since $T\in\frak{g}$ is elliptic and \eqref{eq-7.2.9} we obtain $\frak{g}_\mathbb{C}=\bigoplus_{\lambda\in\mathbb{R}}\frak{g}^\lambda$. 
   Moreover, it follows from \eqref{eq-7.2.9} that $\frak{l}_\mathbb{C}=\frak{g}^0$ and $\bigoplus_{\lambda\in\mathbb{R}}\frak{g}^\lambda=\bigoplus_{\lambda>0}\frak{g}^\lambda\oplus\frak{g}^0\oplus\bigoplus_{\mu<0}\frak{g}^\mu=\frak{u}^+\oplus\frak{l}_\mathbb{C}\oplus\frak{u}^-$.\par
   
   (2). 
   For any $(z,\lambda)\in C_G(T)\times\mathbb{R}$ and $X^\lambda\in\frak{g}^\lambda$, one has $[T,\operatorname{Ad}z(X^\lambda)]=\operatorname{Ad}z\bigl([T,X^\lambda]\bigr)=i\lambda\operatorname{Ad}z(X^\lambda)$, and $\operatorname{Ad}z(X^\lambda)\in\frak{g}^\lambda$. 
   Hence $\operatorname{Ad}z(\frak{g}^\lambda)\subset\frak{g}^\lambda$ holds.\par
   
   (3). 
   From the Jacobi identity and \eqref{eq-7.2.9} we deduce (3).\par
 
   (4). 
   For any $X^\lambda\in\frak{g}^\lambda$, $X^\mu\in\frak{g}^\mu$ we see that $i\lambda B_{\frak{g}_\mathbb{C}}(X^\lambda,X^\mu)=B_{\frak{g}_\mathbb{C}}([T,X^\lambda],X^\mu)=-B_{\frak{g}_\mathbb{C}}(X^\lambda,[T,X^\mu])=-i\mu B_{\frak{g}_\mathbb{C}}(X^\lambda,X^\mu)$, and therefore $B_{\frak{g}_\mathbb{C}}(X^\lambda,X^\mu)=0$ if $\lambda\neq-\mu$.
   This yields (4).\par

   (5). 
   For any $X^\lambda\in\frak{g}^\lambda$ we obtain $[T,\overline{\sigma}(X^\lambda)]=\overline{\sigma}\bigl([T,X^\lambda]\bigr)=\overline{\sigma}(i\lambda X^\lambda)=-i\lambda\overline{\sigma}(X^\lambda)$, and $\overline{\sigma}(X^\lambda)\in\frak{g}^{-\lambda}$. 
   Thus $\overline{\sigma}(\frak{g}^\lambda)\subset\frak{g}^{-\lambda}$. 
   This enables us to have $\overline{\sigma}(\frak{g}^{-\lambda})\subset\frak{g}^{-(-\lambda)}=\frak{g}^\lambda$, and moreover $\frak{g}^{-\lambda}=\overline{\sigma}^2(\frak{g}^{-\lambda})\subset\overline{\sigma}(\frak{g}^\lambda)\subset\frak{g}^{-\lambda}$. 
   Consequently we can show $\overline{\sigma}(\frak{g}^\lambda)=\frak{g}^{-\lambda}$.\par
   
   (${\rm b}'$) is a consequence of $\frak{l}_\mathbb{C}=\frak{g}^0$, \eqref{eq-7.2.9} and (${\rm b}$), where ${\rm b}=2,3,4,5$.\par
   
   (6). 
   Take an arbitrary $g\in G$ and $V,W\in\frak{g}_\mathbb{C}$. 
   On the one hand; for any $X\in\frak{g}$ one has 
\[
\begin{split}
   \overline{\sigma}\bigl(\operatorname{Ad}\exp X(V)\bigr)
  &=\overline{\sigma}\bigl(\exp\operatorname{ad}X(V)\bigr)
   =\overline{\sigma}\Big(\sum_{n\geq 0}\dfrac{1}{n!}(\operatorname{ad}X)^nV\Big)
   =\sum_{n\geq 0}\dfrac{1}{n!}\bigl(\operatorname{ad}\overline{\sigma}(X)\bigr)^{\!n}\overline{\sigma}(V)\\
  &=\sum_{n\geq 0}\dfrac{1}{n!}(\operatorname{ad}X)^n\overline{\sigma}(V)
  =\operatorname{Ad}\exp X\bigl(\overline{\sigma}(V)\bigr)
\end{split} 
\] 
because $\overline{\sigma}(X)=X$.
   On the other hand; since the Lie group $G$ is connected, there exist finite elements $X_1,X_2,\dots,X_k\in\frak{g}$ such that $g=\exp X_1\exp X_2\cdots\exp X_k$. 
   Accordingly 
\[
\begin{split}
   \overline{\sigma}\bigl(\operatorname{Ad}g(W)\bigr)
  &=\overline{\sigma}\bigl(\operatorname{Ad}\exp X_1(\operatorname{Ad}\exp X_2(\cdots(\operatorname{Ad}\exp X_k(W))))\bigr)
   =\operatorname{Ad}\exp X_1(\overline{\sigma}\bigl(\operatorname{Ad}\exp X_2(\cdots(\operatorname{Ad}\exp X_k(W)))\bigr))\\
  &=\cdots
   =\operatorname{Ad}\exp X_1(\operatorname{Ad}\exp X_2(\cdots(\operatorname{Ad}\exp X_k(\overline{\sigma}\bigl(W\bigr)))))
   =\operatorname{Ad}g\bigl(\overline{\sigma}(W)\bigr).
\end{split}  
\] 

   (i). 
   Since the linear transformation $\operatorname{ad}T:\frak{g}\to\frak{g}$ is semisimple, we conclude that $\frak{g}=\ker(\operatorname{ad}T)\oplus\operatorname{ad}T(\frak{g})=\frak{l}\oplus\frak{u}$.
   It is natural that $[T,T]=0$ and $T\in\frak{g}$, so that $T\in\frak{l}$.\par
   
   (ii). 
   For any $z\in C_G(T)$, we show that $[T,\operatorname{Ad}z(Y)]=\operatorname{Ad}z\bigl([T,Y]\bigr)=0$ for all $Y\in\frak{l}$; hence $\operatorname{Ad}z(\frak{l})\subset\frak{l}$. 
   It follows from $\frak{u}=[T,\frak{g}]$ and $\operatorname{Ad}z(\frak{g})\subset\frak{g}$ that $\operatorname{Ad}z(\frak{u})=\operatorname{Ad}z\bigl([T,\frak{g}]\bigr)\subset[T,\operatorname{Ad}z(\frak{g})]\subset\frak{u}$.\par
   
   (iii). 
   For a given $W\in\frak{g}_\mathbb{C}$, $W\in\frak{l}$ if and only if ``$[T,W]=0$ and $\overline{\sigma}(W)=W$'' if and only if ``$W\in\frak{l}_\mathbb{C}$ and $\overline{\sigma}(W)=W$.'' 
   This implies that 
\begin{equation}\label{eq-1}\tag*{\textcircled{1}}
   \frak{l}=\{Y\in\frak{l}_\mathbb{C} \,|\, \overline{\sigma}(Y)=Y\}.
\end{equation}
   Now, let us prove that $\frak{u}=\{V+\overline{\sigma}(V) \,|\, V\in\frak{u}^+\}$.
   For any $U\in\frak{u}$, there exists a $X\in\frak{g}$ satisfying $U=[T,X]$ in view of $\frak{u}=[T,\frak{g}]$. 
   Since $X\in\frak{g}\subset\frak{g}_\mathbb{C}$ and (1) there exists a unique $(V^+,Z,V^-)\in\frak{u}^+\times\frak{g}^0\times\frak{u}^-$ such that $X=V^++Z+V^-$. 
   Here $\overline{\sigma}(X)=X$ yields $\overline{\sigma}(V^+)=V^-$, and $X=V^++Z+\overline{\sigma}(V^+)$. 
   Therefore 
\[
   U
   =[T,X]
   =[T,V^++Z+\overline{\sigma}(V^+)]
   =[T,V^+]+[T,\overline{\sigma}(V^+)]
   =[T,V^+]+\overline{\sigma}\bigl([T,V^+]\bigr).
\]  
   This, together with $[T,V^+]\in[\frak{l}_\mathbb{C},\frak{u}^+]\subset\frak{u}^+$, enables us to assert that 
\begin{equation}\label{eq-2}\tag*{\textcircled{2}}
   \frak{u}\subset\{V+\overline{\sigma}(V) \,|\, V\in\frak{u}^+\}.
\end{equation}
   From $\{W\in\frak{u}^+\oplus\frak{u}^- \,|\, \overline{\sigma}(W)=W\}=\{V+\overline{\sigma}(V) \,|\, V\in\frak{u}^+\}$ we obtain 
\[
   2\dim_\mathbb{R}\{V+\overline{\sigma}(V) \,|\, V\in\frak{u}^+\}
   =\dim_\mathbb{R}\frak{u}^++\dim_\mathbb{R}\frak{u}^-
   \stackrel{{\rm (1)}}{=}\dim_\mathbb{R}\frak{g}_\mathbb{C}-\dim_\mathbb{R}\frak{l}_\mathbb{C}
   \stackrel{\ref{eq-1}}{=}2(\dim_\mathbb{R}\frak{g}-\dim_\mathbb{R}\frak{l})
   \stackrel{{\rm (i)}}{=}2\dim_\mathbb{R}\frak{u}.
\]  
   Hence one concludes $\frak{u}=\{V+\overline{\sigma}(V) \,|\, V\in\frak{u}^+\}$ by \ref{eq-2}.
\end{proof}

\section{The centralizer of an elliptic element}\label{sec-7.3}
   In this section we first clarify relation between the centralizer $C_G(T^r)$ of a torus subgroup $T^r$ in $G$ and the centralizer $C_G(T)$ of an elliptic element $T\in\frak{g}$ (cf.\ Proposition \ref{prop-7.3.2}), and then confirm that $C_G(T)$ is a connected, closed subgroup of $G$ (cf.\ Lemma \ref{lem-7.3.3}).    
   Here we utilize the following notation:
\begin{itemize}
\item
   $C_G(A):=\{g\in G \,|\, \mbox{$gag^{-1}=a$ for all $a\in A$}\}$ for a subset $A\subset G$,
\item
   $C_G(X):=\{g\in G \,|\, \operatorname{Ad}g(X)=X\}$ for an element $X\in\frak{g}$.
\end{itemize}

   In order to prove Proposition \ref{prop-7.3.2} we need the following lemma:
\begin{lemma}\label{lem-7.3.1}
   For an $X\in\frak{g}$, we put $A:=\{\exp tX : t\in\mathbb{R}\}$ and denote by $\overline{A}$ the closure of $A$ in $G$.
   Then, 
\[
   C_G(X)=C_G(A)=C_G(\overline{A}).
\]   
\end{lemma}
\begin{proof}
   We will show $C_G(\overline{A})\subset C_G(A)\subset C_G(X)\subset C_G(\overline{A})$ and conclude this lemma.\par
   
   ($C_G(\overline{A})\subset C_G(A)$). 
   The inclusion $C_G(\overline{A})\subset C_G(A)$ is immediate from $A\subset\overline{A}$.\par
   
   ($C_G(A)\subset C_G(X)$). 
   Take an arbitrary $g\in C_G(A)$. 
   Then, for all $t\in\mathbb{R}$ one has $\exp tX=g(\exp tX)g^{-1}=\exp t\operatorname{Ad}g(X)$.
   Differentiating this equation at $t=0$, we obtain $X=\operatorname{Ad}g(X)$. 
   Hence $g\in C_G(X)$, and so $C_G(A)\subset C_G(X)$ follows.\par
   
   ($C_G(X)\subset C_G(\overline{A})$).
   Take an $h\in C_G(X)$. 
   For any $a\in A$, there exists a $t\in\mathbb{R}$ satisfying $a=\exp tX$, and it follows from $\operatorname{Ad}h(X)=X$ that $hah^{-1}=h(\exp tX)h^{-1}=\exp t\operatorname{Ad}h(X)=\exp tX=a$.
   Therefore 
\begin{equation}\label{eq-1}\tag*{\textcircled{1}}
   \mbox{$hah^{-1}=a$ for all $a\in A$}.
\end{equation}
   The mapping $\overline{A}\ni x\mapsto hxh^{-1}\in G$ is continuous, and $A$ is dense in $\overline{A}$. 
   Hence \ref{eq-1} implies that $hxh^{-1}=x$ for all $x\in\overline{A}$, which allows us to show $h\in C_G(\overline{A})$, and $C_G(X)\subset C_G(\overline{A})$.   
\end{proof}

   From Lemma \ref{lem-7.3.1} we deduce 
\begin{proposition}\label{prop-7.3.2}
   For any torus subgroup $T^r\subset G$, there exists an elliptic element $X\in\frak{g}$ such that $C_G(X)=C_G(T^r)$.
\end{proposition} 
\begin{proof}
   Since $T^r$ is a torus, Kronecker's approximation theorem enables us to obtain an element $X\in\operatorname{Lie}(T^r)$ such that the closure $\overline{A}{}^{T^r}$ in $T^r$ coincides with the whole $T^r$, where $A:=\{\exp tX : t\in\mathbb{R}\}$. 
   This $X$ is an elliptic element of $\frak{g}$ by Lemma \ref{lem-7.2.4} and $\operatorname{Lie}(T^r)$ being a compact subalgebra of $\frak{g}$. 
   Furthermore, $A\subset\overline{A}{}^{T^r}\subset\overline{A}{}^G$ and $T^r=\overline{A}{}^{T^r}$ yield 
\[
   C_G(\overline{A}{}^G)\subset C_G(T^r)\subset C_G(A).
\]
   Thus $C_G(X)=C_G(T^r)$ follows by Lemma \ref{lem-7.3.1}.   
\end{proof} 

   Recalling that the Lie group $G$ is connected, we demonstrate
\begin{lemma}\label{lem-7.3.3}
   For any elliptic element $T\in\frak{g}$, the centralizer $C_G(T)$ is a connected, closed subgroup of $G$.
\end{lemma}
\begin{proof}
   Needless to say, $C_G(T)$ is a closed subgroup of $G$. 
   So, we only prove that $C_G(T)$ is connected. 
   By Lemma \ref{lem-7.2.4} there exists a Cartan involution $\theta_*$ of $\frak{g}$ so that $\theta_*(T)=T$. 
   By use of this $\theta_*$ we put $\frak{k}:=\{X\in\frak{g} \,|\, \theta_*(X)=X\}$, $\frak{p}:=\{Y\in\frak{g} \,|\, \theta_*(Y)=-Y\}$.  
   Denote by $G=KP$ the Cartan decomposition of $G$ corresponding to $\frak{g}=\frak{k}\oplus\frak{p}$. 
   Then, it turns out that 
\begin{enumerate}
\item[(i)]
   $[\frak{k},\frak{k}]\subset\frak{k}$, $[\frak{k},\frak{p}]\subset\frak{p}$, $[\frak{p},\frak{p}]\subset\frak{k}$,
\item[(ii)]
   $K$ is a closed subgroup of $G$, 
\item[(iii)]
   $T\in\operatorname{Lie}(K)=\frak{k}$,   
\item[(iv)] 
   $P$ is a regular submanifold of $G$,     
\item[(v)]
   $\exp:\frak{p}\to P$, $Y\mapsto\exp Y$, is a real analytic diffeomorphism, 
\item[(vi)]
   $\varphi:K\times P\to G$, $(k,p)\mapsto kp$, is a real analytic diffeomorphism. 
\end{enumerate} 
   Let us prove that $C_G(T)$ is connected by taking three steps (S1), (S2) and (S3): 
\begin{enumerate} 
\item[(S1)] 
   $C_K(T)\times(P\cap C_G(T))$ is homeomorphic to $C_G(T)$ via $\varphi$, where we equip $C_K(T)\times(P\cap C_G(T))$ with the induced topology from $K\times P$;   
\item[(S2)] 
   $P\cap C_G(T)$ is connected; 
\item[(S3)]  
   $C_K(T)$ is connected.   
\end{enumerate}
   (S1): 
   Since $\varphi\bigl(C_K(T)\times(P\cap C_G(T))\bigr)\subset C_G(T)$ is clear, it suffices to confirm that for a given $x\in C_G(T)$, there exist $k\in C_K(T)$ and $p\in P\cap C_G(T)$ satisfying $kp=\varphi(k,p)=x$ by virtue of (vi). 
   Now, let us take any $x\in C_G(T)$. 
   In view of $T=\operatorname{Ad}(x)T$ one sees that  
\begin{equation}\label{eq-1}\tag*{\textcircled{1}}
  \exp tT=\exp t\operatorname{Ad}(x)T=x(\exp tT)x^{-1} \quad(t\in\mathbb{R}).
\end{equation}
   For the $x\in G$ there exists a unique $(k,p)\in K\times P$ such that $kp=x$ by (vi). 
   We want to show that both $k$ and $p$ belong to $C_G(T)$. 
   It follows from \ref{eq-1} that $x=(\exp tT)x\exp(-tT)$, so that   
\[
   k\cdot p
   =x
   =(\exp tT)x\exp(-tT)
  =\bigl((\exp tT)k\exp(-tT)\bigr)\cdot\bigl((\exp tT)p\exp(-tT)\bigr)
\]
for all $t\in\mathbb{R}$.
   Here (i), (iii) and (v) yield $(\exp tT)k\exp(-tT)\in K$ and $(\exp tT)p\exp(-tT)\in P$; hence we conclude 
\begin{equation}\label{eq-2}\tag*{\textcircled{2}}
\begin{array}{ll}
   k=(\exp tT)k\exp(-tT), & p=(\exp tT)p\exp(-tT)
\end{array}
\end{equation}
by $\varphi:K\times P\to G$ being injective. 
   This gives rise to $\exp tT=k(\exp tT)k^{-1}=\exp t\operatorname{Ad}k(T)$ and $\exp tT=\exp t\operatorname{Ad}p(T)$. 
   Differentiating $\exp tT=\exp t\operatorname{Ad}k(T)=\exp t\operatorname{Ad}p(T)$ at $t=0$, we have $T=\operatorname{Ad}k(T)=\operatorname{Ad}p(T)$. 
   This assures that $k,p\in C_G(T)$, and accordingly we deduce $k\in(K\cap C_G(T))=C_K(T)$ and $p\in P\cap C_G(T)$.\par 
 
   (S2): 
   Let us demonstrate that $P\cap C_G(T)$ is (arcwise) connected. 
   Take any $y\in P\cap C_G(T)$ and express it as $y=\exp Y$ ($Y\in\frak{p}$). 
   For any $t\in\mathbb{R}$, one deduces $y=(\exp tT)y\exp(-tT)$ by $y\in C_G(T)$ and arguments similar to those in (S1). 
   Then, we have   
\[
   \exp Y
   =y
   =(\exp tT)y\exp(-tT)
   =\exp\operatorname{Ad}(\exp tT)Y.
\]  
   Therefore, it follows from (v) and $\operatorname{Ad}(\exp tT)Y\in\frak{p}$ that $Y=\operatorname{Ad}(\exp tT)Y=\exp t\operatorname{ad}T(Y)$ for all $t\in\mathbb{R}$; and hence $[T,Y]=0$ holds.  
   By $[Y,T]=0$ we conclude that for every $t\in\mathbb{R}$, 
\[
   \operatorname{Ad}(\exp tY)T
   =\exp t\operatorname{ad}Y(T)
   =\sum_{n\geq 0}\frac{t^n}{n!}(\operatorname{ad}Y)^nT
   =T. 
\]  
   This implies that the whole $1$-parameter subgroup $\{\exp tY\,|\,t\in\mathbb{R}\}$ lies in $P\cap C_G(T)$, where $\exp tY\in P$ follows from $tY\in\frak{p}$ and (v). 
   So, one can joint $y=\exp tY|_{t=1}$ to the unite element $e=\exp tY|_{t=0}\in P\cap C_G(T)$ by an arc in $P\cap C_G(T)$.\par

   (S3):
   Note that $K$ is connected because (vi) and $G$ is connected. 
   Since $\frak{k}$ is compact one can decompose it as 
\[
\begin{array}{ll} 
   \frak{k}=\frak{k}_{\rm ss}\oplus\frak{z}(\frak{k}) & \mbox{(direct sum of Lie algebras)}, 
\end{array}
\]  
where $\frak{k}_{\rm ss}$ (resp.\ $\frak{z}(\frak{k})$) stands for the semisimple part (resp.\ the center) of $\frak{k}$. 
   This and (iii) enable us to uniquely express the $T$ as follows: 
\[
   T=T_{\rm ss}+T_{\rm z}
\] 
($T_{\rm ss}\in\frak{k}_{\rm ss}$, $T_{\rm z}\in\frak{z}(\frak{k})$). 
   Denote by $K_{\rm ss}$ and $Z(K)_0$ the connected Lie subgroups of $K$ corresponding to $\frak{k}_{\rm ss}$ and $\frak{z}(\frak{k})$, respectively.  
   From now on, let us verify that $C_K(T)$ is connected. 
   Since the Lie group $K$ is connected, one sees that $K=K_{\rm ss}\cdot Z(K)_0$; so that  
\begin{equation}\label{eq-3}\tag*{\textcircled{3}}
   C_K(T)=C_{K_{\rm ss}}(T_{\rm ss})\cdot Z(K)_0
\end{equation} 
because $\operatorname{Ad}(k)T_{\rm z}=T_{\rm z}$ for all $k\in K$, and $\operatorname{Ad}(c)X=X$ for all $(c,X)\in Z(K)_0\times\frak{k}$. 
  Since $K_{\rm ss}$ is connected and $\frak{k}_{\rm ss}$ is compact semisimple, $K_{\rm ss}$ is compact. 
  This implies that $C_{K_{\rm ss}}(T_{\rm ss})$ is connected, and it follows from \ref{eq-3} that $C_K(T)$ is connected.
\end{proof}

   By Lemma \ref{lem-7.3.3} one can conclude
\begin{proposition}\label{prop-7.3.4}
   For any elliptic element $T\in\frak{g}$, the homogeneous space $G/C_G(T)$ is simply connected.
\end{proposition} 
\begin{proof}
   Let $(\widetilde{G},p)$ be a universal covering group of the connected Lie group $G$. 
   Then, the differential homomorphism $p_*:\widetilde{\frak{g}}\to\frak{g}$ is a Lie algebra isomorphism, and so we assume $\widetilde{\frak{g}}=\frak{g}$ via $p_*$. 
   On this assumption, $T$ is an elliptic element of $\widetilde{\frak{g}}$ and $C_{\widetilde{G}}(T)$ is connected by Lemma \ref{lem-7.3.3}. 
   Therefore $\widetilde{G}/C_{\widetilde{G}}(T)$ is simply connected, and hence $G/C_G(T)$ is also simply connected because $\widetilde{G}/C_{\widetilde{G}}(T)$ is homeomorphic to $G/C_G(T)$.
\end{proof}

\section{An appendix (semisimple orbits)}\label{sec-7.4}
   We investigate relation between semisimple (adjoint) orbits and reductive homogeneous spaces, where we refer to Nomizu \cite[p.41]{No} for the definition of reductive homogeneous space.
   Let $Z$ be a semisimple element of $\frak{g}$.
   Since the linear transformation $\operatorname{ad}Z:\frak{g}\to\frak{g}$ is semisimple, it follows that    
\begin{enumerate}
\item[(1)]
   its image $\operatorname{ad}Z(\frak{g})$ is a vector subspace of $\frak{g}$, 
\item[(2)]
   $\frak{g}=\ker(\operatorname{ad}Z)\oplus\operatorname{ad}Z(\frak{g})$,
\item[(3)]
   $\operatorname{Ad}x(Y)\subset\operatorname{ad}Z(\frak{g})$ for all $(x,Y)\in C_G(Z)\times\operatorname{ad}Z(\frak{g})$.
\end{enumerate}
   Here we remark $\ker(\operatorname{ad}Z)=\operatorname{Lie}(C_G(Z))$.
   Hence, the semisimple adjoint orbit $G/C_G(Z)$ is a {\it reductive} homogeneous space.
   Moreover, one can assert that
\begin{lemma}[Uniqueness]\label{lem-7.4.1}
   Let $X$ be any element of $\frak{g}$.
   If $\frak{m}$ is a vector subspace of $\frak{g}$ such that {\rm (2$'$)} $\frak{g}=\ker(\operatorname{ad}X)\oplus\frak{m}$ and {\rm (3$'$)} $\operatorname{Ad}x(Y)\subset\frak{m}$ for all $(x,Y)\in C_G(X)\times\frak{m}$, then $\frak{m}$ coincides with $\operatorname{ad}X(\frak{g})$.
\end{lemma}
\begin{proof}
   $\operatorname{ad}X(\frak{m})\subset\frak{m}$ is a consequence of (3$'$).
   Furthermore, (2$'$) assures that the linear transformation $\operatorname{ad}X:\frak{m}\to\frak{m}$ is injective, and hence is isomorphic.
   From $\frak{m}=\operatorname{ad}X(\frak{m})$ we obtain $\frak{m}\subset\operatorname{ad}X(\frak{g})$. 
   Therefore $\frak{m}=\operatorname{ad}X(\frak{g})$ holds because of $\operatorname{ad}X(\frak{g})=\operatorname{ad}X\bigl(\ker(\operatorname{ad}X)\oplus\frak{m}\bigr)\subset\operatorname{ad}X(\frak{m})\subset\frak{m}$.
\end{proof}

\begin{proposition}\label{prop-7.4.2}
   For an $X\in\frak{g}$ the following {\rm (a)} and {\rm (b)} are equivalent$:$
\begin{enumerate}
\item[{\rm (a)}]
   $X$ is a semisimple element of $\frak{g}$.
\item[{\rm (b)}]
   $G/C_G(X)$ is a reductive homogeneous space.   
\end{enumerate}  
\end{proposition}
\begin{proof}
   (a)$\Rightarrow$(b). 
   cf.\ the beginning of this section.\par
   
   (b)$\Rightarrow$(a). 
   Suppose that $G/C_G(X)$ is a reductive homogeneous space. 
   By the definition of reductive homogeneous space, there exists a vector subspace $\frak{m}\subset\frak{g}$ such that 
\begin{equation}\label{eq-1}\tag*{\textcircled{1}}
\begin{array}{ll}
   \frak{g}=\frak{c}_\frak{g}(X)\oplus\frak{m}, & \operatorname{Ad}\bigl(C_G(X)\bigr)\frak{m}\subset\frak{m}.
\end{array}
\end{equation}
   Then, Lemma \ref{lem-7.4.1} implies  
\begin{equation}\label{eq-2}\tag*{\textcircled{2}}
   \frak{m}=\operatorname{ad}X(\frak{g}).
\end{equation}
   Now, let $\frak{h}$ be a Cartan subalgebra of $\frak{c}_\frak{g}(X)$---this is, $\frak{h}$ is a subalgebra of $\frak{c}_\frak{g}(X)$ such that  
\begin{enumerate}
\item[(i)]
   $\frak{h}$ is nilpotent, 
\item[(ii)]
   the normalizer of $\frak{h}$ in $\frak{c}_\frak{g}(X)$ coincides with $\frak{h}$.
\end{enumerate} 
   We will verify that this $\frak{h}$ is also a Cartan subalgebra of $\frak{g}$. 
   Since (ii), $X\in\frak{c}_\frak{g}(X)$ and $[\frak{h},X]=\{0\}\subset\frak{h}$, one obtains   
\begin{equation}\label{eq-3}\tag*{\textcircled{3}}
   X\in\frak{h}.
\end{equation}
   We want to show that the normalizer $\frak{n}_\frak{g}(\frak{h})$ of $\frak{h}$ in $\frak{g}$ also coincides with $\frak{h}$. 
   Let $Y$ be any element of $\frak{g}$ with $[\frak{h},Y]\subset\frak{h}$. 
   On the one hand; \ref{eq-3} yields $\operatorname{ad}X(Y)\in\frak{h}\subset\frak{c}_\frak{g}(X)$.
   On the other hand; \ref{eq-2} yields $\operatorname{ad}X(Y)\in\frak{m}$. 
   Thus $\operatorname{ad}X(Y)\in(\frak{c}_\frak{g}(X)\cap\frak{m})=\{0\}$ by \ref{eq-1}, and hence $Y\in\frak{c}_\frak{g}(X)$.
   Accordingly, (ii) implies that $Y\in\frak{h}$, so that $\frak{n}_\frak{g}(\frak{h})\subset\frak{h}$, and $\frak{n}_\frak{g}(\frak{h})=\frak{h}$.
   This, together with (i), assures that $\frak{h}$ is a Cartan subalgebra of $\frak{g}$. 
   So, since $\frak{g}$ is a semisimple Lie algebra, $\operatorname{ad}H:\frak{g}\to\frak{g}$ is semisimple for each $H\in\frak{h}$. 
   In particular, $\operatorname{ad}X:\frak{g}\to\frak{g}$ is semisimple. 
   For this reason $X$ is a semisimple element of $\frak{g}$. 
\end{proof}

\begin{remark}\label{rem-7.4.3}
   It is known that $\operatorname{Ad}G(X)$ is a closed subset of $\frak{g}$ if and only if $X$ is a semisimple element of $\frak{g}$.
   cf.\ Proposition \ref{prop-7.4.2}.
\end{remark}

\chapter{Complex flag manifolds}\label{ch-8}
   By a {\it complex flag manifold}\index{complex flag manifold@complex flag manifold\dotfill}, we mean the complex homogeneous space $G_\mathbb{C}/Q$ of a connected complex semisimple Lie group $G_\mathbb{C}$ over a connected, closed complex parabolic (Lie) subgroup $Q\subset G_\mathbb{C}$.
   Here, a complex flag manifold is also called a K\"{a}hler C-space or a generalized flag manifold.
   In this chapter we study complex flag manifolds.    
   The setting of this chapter is as follows:
\begin{itemize}
\item
   $G_\mathbb{C}$ is a connected complex semisimple Lie group,
\item 
   $T$ is a non-zero, elliptic element of $\frak{g}_\mathbb{C}$, 
\item 
   $L_\mathbb{C}:=C_{G_\mathbb{C}}(T)=\{x\in G_\mathbb{C} \,|\, \operatorname{Ad}x(T)=T\}$,
\item 
   $\frak{g}^\lambda:=\{X\in\frak{g}_\mathbb{C} \,|\, \operatorname{ad}T(X)=i\lambda X\}$ for $\lambda\in\mathbb{R}$,
\item
   $\frak{u}^+:=\bigoplus_{\lambda>0}\frak{g}^\lambda$, $\frak{u}^-:=\bigoplus_{\lambda>0}\frak{g}^{-\lambda}$,
\item 
   $U^+:=\exp\frak{u}^+$, $U^-:=\exp\frak{u}^-$, 
\item 
   $Q^+:=N_{G_\mathbb{C}}(\bigoplus_{\nu\geq 0}\frak{g}^\nu)=\{q\in G_\mathbb{C} \,|\, \operatorname{Ad}q(\bigoplus_{\nu\geq 0}\frak{g}^\nu)\subset\bigoplus_{\nu\geq 0}\frak{g}^\nu\}$, $Q^-:=N_{G_\mathbb{C}}(\bigoplus_{\nu\geq 0}\frak{g}^{-\nu})$,
\end{itemize} 
where $\frak{g}^\lambda=\{0\}$ in the case where $i\lambda$ is different from the eigenvalues of $\operatorname{ad}T$ and $\exp:\frak{g}_\mathbb{C}\to G_\mathbb{C}$ is the exponential mapping. 
   In addition, let $\overline{\theta}_*$ be a Cartan involution of $\frak{g}_\mathbb{C}$ satisfying $\overline{\theta}_*(T)=T$ (cf.\ Lemma \ref{lem-7.2.4}).
   Since $G_\mathbb{C}$ is semisimple, $\overline{\theta}_*$ can be lifted to $G_\mathbb{C}$. 
   Denote its lift by $\overline{\theta}$, and set closed subgroups $G_u\subset G_\mathbb{C}$ and $L_u\subset G_u$ as
\begin{itemize}
\item 
   $G_u:=\{k\in G_\mathbb{C} \,|\, \overline{\theta}(k)=k\}$,
\item
   $L_u:=C_{G_u}(T)$,   
\end{itemize} 
respectively.
   We remark here that $T\in\frak{g}_u$, that $\overline{\theta}$ is an anti-holomorphic involutive automorphism of $G_\mathbb{C}$, and that $G_u$ is connected and is  a maximal compact subgroup of $G_\mathbb{C}$.\footnote{Here we assert that the Lie group $G_u$ is connected since so is $G_\mathbb{C}$; and that $G_u$ is compact since it is a connected Lie group whose Lie algebra is compact semisimple.}
   One can show 
\begin{lemma}\label{lem-8.0.1}
\begin{enumerate}
\item[]
\item[{\rm (a)}]
   $L_\mathbb{C}$ is a connected closed complex $($Lie$)$ subgroup of $G_\mathbb{C}$ with $\frak{l}_\mathbb{C}=\frak{c}_{\frak{g}_\mathbb{C}}(T)=\frak{g}^0$,
\item[{\rm (b)}]
   $Q^s$ is a closed complex subgroup of $G_\mathbb{C}$ with $\frak{q}^s=\{X\in\frak{g}_\mathbb{C} : [X,\bigoplus_{\nu\geq 0}\frak{g}^{s\nu}]\subset\bigoplus_{\nu\geq 0}\frak{g}^{s\nu}\}$ $(s=\pm)$,   
\item[{\rm (c)}] 
   $L_u$ is a connected compact subgroup of $G_u$ and $L_u=(G_u\cap L_\mathbb{C})$,
\item[{\rm (1)}]
   $\frak{g}_\mathbb{C}=\bigoplus_{\lambda\in\mathbb{R}}\frak{g}^\lambda=\frak{u}^+\oplus\frak{l}_\mathbb{C}\oplus\frak{u}^-$, $\bigoplus_{\nu\geq 0}\frak{g}^\nu=\frak{l}_\mathbb{C}\oplus\frak{u}^+$, $\bigoplus_{\nu\geq 0}\frak{g}^{-\nu}=\frak{l}_\mathbb{C}\oplus\frak{u}^-$,
\item[{\rm (2)}] 
   $\operatorname{Ad}x(\frak{g}^\lambda)\subset\frak{g}^\lambda$ for all $(x,\lambda)\in L_\mathbb{C}\times\mathbb{R}$, 
\item[{\rm (2$'$)}]  
   $\operatorname{Ad}x(\frak{l}_\mathbb{C})\subset\frak{l}_\mathbb{C}$, $\operatorname{Ad}x(\frak{u}^+)\subset\frak{u}^+$, $\operatorname{Ad}x(\frak{u}^-)\subset\frak{u}^-$ for all $x\in L_\mathbb{C}$,
\item[{\rm (2$''$)}]
   $L_\mathbb{C}\subset Q^+$, $L_\mathbb{C}\subset Q^-$,
\item[{\rm (3)}]  
   $[\frak{g}^\lambda,\frak{g}^\mu]\subset\frak{g}^{\lambda+\mu}$ for all $\lambda,\mu\in\mathbb{R}$,
\item[{\rm (3$'$)}]  
   $[\frak{l}_\mathbb{C},\frak{l}_\mathbb{C}]\subset\frak{l}_\mathbb{C}$, $[\frak{l}_\mathbb{C},\frak{u}^+]\subset\frak{u}^+$, $[\frak{l}_\mathbb{C},\frak{u}^-]\subset\frak{u}^-$, $[\frak{u}^+,\frak{u}^+]\subset\frak{u}^+$, $[\frak{u}^-,\frak{u}^-]\subset\frak{u}^-$,  
\item[{\rm (3$''$)}]
   both $\frak{u}^+$ and $\frak{u}^-$ are complex nilpotent subalgebras of $\frak{g}_\mathbb{C}$, both $\bigoplus_{\nu\geq 0}\frak{g}^\nu$ and $\bigoplus_{\nu\geq 0}\frak{g}^{-\nu}$ are complex subalgebras of $\frak{g}_\mathbb{C}$,
\item[{\rm (4)}]  
   $B_{\frak{g}_\mathbb{C}}(\frak{g}^\lambda,\frak{g}^\mu)=\{0\}$ if $\lambda+\mu\neq0$, where $B_{\frak{g}_\mathbb{C}}$ is the Killing form of $\frak{g}_\mathbb{C}$, 
\item[{\rm (4$'$)}]  
   $B_{\frak{g}_\mathbb{C}}(\frak{l}_\mathbb{C},\frak{u}^+)=\{0\}$, $B_{\frak{g}_\mathbb{C}}(\frak{l}_\mathbb{C},\frak{u}^-)=\{0\}$, $B_{\frak{g}_\mathbb{C}}(\frak{u}^+,\frak{u}^+)=\{0\}$, $B_{\frak{g}_\mathbb{C}}(\frak{u}^-,\frak{u}^-)=\{0\}$,   
\item[{\rm (5)}]  
   $\overline{\theta}_*(\frak{g}^\lambda)=\frak{g}^{-\lambda}$ for all $\lambda\in\mathbb{R}$,
\item[{\rm (5$'$)}] 
   $\overline{\theta}_*(\frak{l}_\mathbb{C})=\frak{l}_\mathbb{C}$, $\overline{\theta}_*(\frak{u}^+)=\frak{u}^-$, $\overline{\theta}_*(\frak{u}^-)=\frak{u}^+$,
\item[{\rm (5$''$)}] 
   $\overline{\theta}(L_\mathbb{C})=L_\mathbb{C}$, $\overline{\theta}(U^+)=U^-$, $\overline{\theta}(U^-)=U^+$, $\overline{\theta}(Q^+)=Q^-$, $\overline{\theta}(Q^-)=Q^+$,
\item[{\rm (6)}] 
   $\overline{\theta}_*\bigl(\operatorname{Ad}k(X)\bigr)=\operatorname{Ad}k\bigl(\overline{\theta}_*(X)\bigr)$ for all $(k,X)\in G_u\times\frak{g}_\mathbb{C}$, 
\item[{\rm (i)}] 
   $\frak{g}_u=\frak{l}_u\oplus\operatorname{ad}T(\frak{g}_u)$, $T\in\frak{l}_u$, 
\item[{\rm (ii)}]  
   $\operatorname{Ad}z(\frak{l}_u)\subset\frak{l}_u$, $\operatorname{Ad}z\bigl(\operatorname{ad}T(\frak{g}_u)\bigr)\subset\operatorname{ad}T(\frak{g}_u)$ for all $z\in L_u$, 
\item[{\rm (iii)}]  
   $\frak{l}_u=(\frak{g}_u\cap\frak{l}_\mathbb{C})=\{Y\in\frak{l}_\mathbb{C} \,|\, \overline{\theta}_*(Y)=Y\}$, $\operatorname{ad}T(\frak{g}_u)=\{A\in\operatorname{ad}T(\frak{g}_\mathbb{C}) \,|\, \overline{\theta}_*(A)=A\}=\{V+\overline{\theta}_*(V) \,|\, V\in\frak{u}^+\}$.  
\end{enumerate}   
\end{lemma} 
\begin{proof}
   cf.\ Lemmas \ref{lem-7.3.3} and \ref{lem-7.2.8}.
\end{proof}

\section{A complex flag manifold and a fundamental root system}\label{sec-8.1}
   From the next section we will prove propositions related to complex flag manifolds by taking a root system into consideration. 
   In this section we set up a root system and give a lemma.

\subsection{A root space decomposition}\label{subsec-8.1.1} 
   Take a maximal torus $i\frak{h}_\mathbb{R}$ of the compact semisimple Lie algebra $\frak{g}_u$ containing the element $T$, and denote by $\triangle=\triangle(\frak{g}_\mathbb{C},\frak{h}_\mathbb{C})$ the (non-zero) root system of $\frak{g}_\mathbb{C}$ relative to $\frak{h}_\mathbb{C}$, where $\frak{h}_\mathbb{C}$ is the complex vector subspace of $\frak{g}_\mathbb{C}$ generated by $i\frak{h}_\mathbb{R}$.
   Let $\frak{g}_\alpha$ be the root subspace of $\frak{g}_\mathbb{C}$ for $\alpha\in\triangle$. 
   In this setting $T\in i\frak{h}_\mathbb{R}$, $\overline{\theta}_*(\frak{h}_\mathbb{C})=\frak{h}_\mathbb{C}$, $\frak{h}_\mathbb{C}=i\frak{h}_\mathbb{R}\oplus\frak{h}_\mathbb{R}$, $\overline{\theta}_*=\operatorname{id}$ on $i\frak{h}_\mathbb{R}$, $\overline{\theta}_*=-\operatorname{id}$ on $\frak{h}_\mathbb{R}$, and $\frak{g}_\mathbb{C}$ is decomposed into a direct sum of vector subspaces: $\frak{g}_\mathbb{C}=\frak{h}_\mathbb{C}\oplus\bigoplus_{\alpha\in\triangle}\frak{g}_\alpha$. 
   Here $\frak{h}_\mathbb{R}:=i(i\frak{h}_\mathbb{R})$.

\subsection{Chevalley's canonical basis}\label{subsec-8.1.2}    
   For each root $\alpha\in\triangle$, there exists a unique $H_\alpha\in\frak{h}_\mathbb{C}$ such that $\alpha(X)=B_{\frak{g}_\mathbb{C}}(H_\alpha,X)$ for all $X\in\frak{h}_\mathbb{C}$.
   Then $\frak{h}_\mathbb{R}=\operatorname{span}_\mathbb{R}\{H_\alpha\,|\,\alpha\in\triangle\}$, and for every $\alpha\in\triangle$ there exist vectors $E_{\pm\alpha}\in\frak{g}_{\pm\alpha}$ satisfying 
\begin{equation}\label{eq-8.1.1}
   \mbox{$(E_{\alpha}-E_{-\alpha}), i(E_{\alpha}+E_{-\alpha})\in\frak{g}_u$ and $[E_{\alpha},E_{-\alpha}]=\bigl(2/\alpha(H_\alpha)\bigr)H_\alpha$}.
\end{equation}
   Here it follows that $\frak{g}_u=i\frak{h}_\mathbb{R}\oplus\bigoplus_{\alpha\in\triangle}\operatorname{span}_\mathbb{R}\{E_{\alpha}-E_{-\alpha}\}\oplus\operatorname{span}_\mathbb{R}\{i(E_{\alpha}+E_{-\alpha})\}$, and
\begin{equation}\label{eq-8.1.2}
   \mbox{$\overline{\theta}_*(E_{\alpha})=-E_{-\alpha}$ for all $\alpha\in\triangle=\triangle(\frak{g}_\mathbb{C},\frak{h}_\mathbb{C})$}.
\end{equation}
   Remark that $\frak{g}_\alpha=\operatorname{span}_\mathbb{C}\{E_\alpha\}$ for all $\alpha\in\triangle$. 
   Setting $H_\alpha^*:=\bigl(2/\alpha(H_\alpha)\bigr)H_\alpha$ for $\alpha\in\triangle$, one has $[H_\alpha^*,E_\alpha]=2E_\alpha$, $[H_\alpha^*,E_{-\alpha}]=-2E_{-\alpha}$, $[E_\alpha,E_{-\alpha}]=H_\alpha^*$, and thus $\frak{s}_\alpha:=\operatorname{span}_\mathbb{C}\{H_\alpha^*,E_\alpha,E_{-\alpha}\}$ is a complex subalgebra of $\frak{g}_\mathbb{C}$ which is isomorphic to $\frak{sl}(2,\mathbb{C})$ for each $\alpha\in\triangle$.

\subsection{A Weyl group $\mathcal{W}$}\label{subsec-8.1.3} 
   Define a Weyl group $\mathcal{W}$ of $G_\mathbb{C}$ and an action $\zeta$ of $\mathcal{W}$ on the dual space $(\frak{h}_\mathbb{C})^*$ by 
\begin{equation}\label{eq-8.1.3}
   \left\{\begin{array}{l}
   \mathcal{W}:=N_{G_u}(i\frak{h}_\mathbb{R})/C_{G_u}(i\frak{h}_\mathbb{R}),\\
   \mbox{$\zeta([w])\eta:={}^t\!\operatorname{Ad}w^{-1}(\eta)$ for $[w]\in\mathcal{W}$ and $\eta\in(\frak{h}_\mathbb{C})^*$},
   \end{array}\right.
\end{equation}
where $[w]$ stands for the left coset $wC_{G_u}(i\frak{h}_\mathbb{R})$. 
   By use of $E_\alpha$ in \eqref{eq-8.1.1} we set 
\begin{equation}\label{eq-8.1.4}
   \mbox{$w_\alpha:=\exp(\pi/2)(E_{\alpha}-E_{-\alpha})$ for $\alpha\in\triangle$}.
\end{equation}
   Remark that $\zeta:\mathcal{W}\to GL((\frak{h}_\mathbb{C})^*)$, $[w]\mapsto\zeta([w])$, is a group homomorphism and $\operatorname{Ad}w(\frak{g}_\beta)=\frak{g}_{\zeta([w])\beta}$ for all $([w],\beta)\in\mathcal{W}\times\triangle$.
   For every root $\alpha\in\triangle$, it follows from \eqref{eq-8.1.1} and \eqref{eq-8.1.4} that $\operatorname{Ad}w_\alpha(X)=X-\alpha(X)H_\alpha^*$ for all $X\in\frak{h}_\mathbb{C}$, so that $w_\alpha$ belongs to the normalizer $N_{G_u}(i\frak{h}_\mathbb{R})$ and so $[w_\alpha]\in\mathcal{W}$; besides, $\zeta([w_\alpha])$ is the reflection along $\alpha$ which leaves $\triangle$ invariant.

\subsection{A fundamental root system, Borel subalgebras, and Iwasawa decompositions}\label{subsec-8.1.4} 
   Let $\Pi_\triangle$ be a fundamental root system of $\triangle=\triangle(\frak{g}_\mathbb{C},\frak{h}_\mathbb{C})$ satisfying\footnote{There is such a system with \eqref{eq-8.1.5}---for example, consider the lexicographic linear ordering on the dual space $(\frak{h}_\mathbb{R})^*$ associated with an ordered real basis $-iT=:A_1,A_2,\dots,A_\ell$ of $\frak{h}_\mathbb{R}$.} 
\begin{equation}\label{eq-8.1.5}
   \mbox{$\alpha(-iT)\geq 0$ for all $\alpha\in\Pi_\triangle$}.
\end{equation}
   Relative to this $\Pi_\triangle$ we fix the set $\triangle^+$ of positive roots, and put $\triangle^-:=-\triangle^+$. 
   Needless to say, $\beta(-iT)\geq 0$ for all $\beta\in\triangle^+$.
   Setting 
\begin{equation}\label{eq-8.1.6}
\begin{array}{llll}
   \frak{n}^+:=\bigoplus_{\beta\in\triangle^+}\frak{g}_\beta, 
   & \frak{n}^-:=\bigoplus_{\beta\in\triangle^+}\frak{g}_{-\beta}, 
   & \frak{b}^+:=\frak{h}_\mathbb{C}\oplus\frak{n}^+,
   & \frak{b}^-:=\frak{h}_\mathbb{C}\oplus\frak{n}^-,   
\end{array}
\end{equation}
one has Iwasawa decompositions $\frak{g}_\mathbb{C}=\frak{g}_u\oplus\frak{h}_\mathbb{R}\oplus\frak{n}^\pm$ and complex Borel subalgebras $\frak{b}^\pm$ of $\frak{g}_\mathbb{C}$.
   Moreover, $\frak{g}_\mathbb{C}=\frak{n}^+\oplus\frak{h}_\mathbb{C}\oplus\frak{n}^-$, $\overline{\theta}_*(\frak{n}^\pm)=\frak{n}^\mp$, $\overline{\theta}_*(\frak{b}^\pm)=\frak{b}^\mp$ and 
\begin{equation}\label{eq-8.1.7}
\left\{
\begin{array}{@{\,}l}
   \frak{l}_\mathbb{C}=\frak{c}_{\frak{g}_\mathbb{C}}(T)=\frak{g}^0=\frak{h}_\mathbb{C}\oplus\bigoplus_{\gamma\in\blacktriangle}\frak{g}_\gamma,\\
   \frak{u}^+=\bigoplus_{\lambda>0}\frak{g}^\lambda=\bigoplus_{\alpha\in\triangle^+-\blacktriangle}\frak{g}_\alpha\subset\bigoplus_{\beta\in\triangle^+}\frak{g}_\beta=\frak{n}^+\subset\frak{b}^+\subset\bigoplus_{\nu\geq 0}\frak{g}^\nu=\frak{l}_\mathbb{C}\oplus\frak{u}^+,\\ 
   \frak{u}^-=\bigoplus_{\alpha\in\triangle^+-\blacktriangle}\frak{g}_{-\alpha}\subset\frak{n}^-\subset\frak{b}^-\subset\bigoplus_{\nu\geq 0}\frak{g}^{-\nu}=\frak{l}_\mathbb{C}\oplus\frak{u}^-,
\end{array}\right.   
\end{equation}
where $\blacktriangle:=\{\gamma\in\triangle(\frak{g}_\mathbb{C},\frak{h}_\mathbb{C}) \,|\, \gamma(T)=0\}$.
   Denote by $G_\mathbb{C}=G_uH_\mathbb{R}N^\pm$ the Iwasawa decompositions of $G_\mathbb{C}$ corresponding to the $\frak{g}_\mathbb{C}=\frak{g}_u\oplus\frak{h}_\mathbb{R}\oplus\frak{n}^\pm$, respectively.

\begin{remark}\label{rem-8.1.8}
\begin{enumerate}[(i)]
\item[]
\item 
   $\frak{l}_\mathbb{C}$ is a complex reductive Lie algebra by \eqref{eq-8.1.7}.
\item
   A complex subalgebra of $\frak{g}_\mathbb{C}$ is said to be {\it parabolic},\index{parabolic subalgebra@parabolic subalgebra\dotfill} if it includes a complex Borel subalgebra of $\frak{g}_\mathbb{C}$.
\item  
   \eqref{eq-8.1.7} implies that $\bigoplus_{\nu\geq 0}\frak{g}^{s\nu}=\frak{l}_\mathbb{C}\oplus\frak{u}^s$ is a complex parabolic subalgebra of $\frak{g}_\mathbb{C}$ ($s=\pm$). 
\item
   We have constructed complex parabolic subalgebras $\frak{l}_\mathbb{C}\oplus\frak{u}^\pm\subset\frak{g}_\mathbb{C}$ from the elliptic element $T\in\frak{g}_\mathbb{C}$. 
   Similarly one can do so from every elliptic element of $\frak{g}_\mathbb{C}$. 
   This construction provides us with all complex parabolic subalgebras of $\frak{g}_\mathbb{C}$. 
\item
   Let $\{Z_a\}_{a=1}^\ell\subset\frak{h}_\mathbb{C}$ be the dual basis of $\Pi_\triangle=\{\alpha_a\}_{a=1}^\ell$. 
   Then, by \eqref{eq-8.1.5} one can express $-iT$ as $-iT=\sum_{a=1}^\ell\lambda_aZ_a$ with $\lambda_1,\lambda_2,\dots,\lambda_\ell\geq 0$.
   In fact; for any elliptic element $T'\in\frak{g}_\mathbb{C}$, there exist an inner automorphism $\psi$ of $\frak{g}_\mathbb{C}$ and $\lambda_1',\lambda_2',\dots,\lambda_\ell'\geq 0$ such that $\psi(-iT')=\sum_{a=1}^\ell\lambda_a'Z_a$.
\item
   $\blacktriangle$, $\triangle^+-\blacktriangle$ and $\triangle^--\blacktriangle$ are closed subsets of $\triangle$, and furthermore, $\blacktriangle$ is symmetric (i.e., $\blacktriangle=-\blacktriangle$).
   Here a subset $\Gamma\subset\triangle$ is said to be {\it closed},\index{closed subset of a root system@closed subset of a root system\dotfill} if $\alpha,\beta\in\Gamma$ and $\alpha+\beta\in\triangle$ imply $\alpha+\beta\in\Gamma$.
\item
   $\triangle^+-\blacktriangle=\{\alpha\in\triangle \,|\, \alpha(-iT)>0\}$, $\triangle^--\blacktriangle=\{\alpha\in\triangle \,|\, \alpha(-iT)<0\}$.
\item 
   If $\alpha(-iT)>0$ for all $\alpha\in\Pi_\triangle$ (cf.\ \eqref{eq-8.1.5}), then $\frak{l}_\mathbb{C}=\frak{h}_\mathbb{C}$, $\frak{u}^\pm=\frak{n}^\pm$, $\bigoplus_{\nu\geq 0}\frak{g}^{\pm\nu}=\frak{l}_\mathbb{C}\oplus\frak{u}^\pm=\frak{b}^\pm$ and $\blacktriangle=\emptyset$. 
\end{enumerate}
\end{remark}

   In addition to Remark \ref{rem-8.1.8} we pay attention to 
\begin{remark}\label{rem-8.1.9}
   Let 
\begin{equation}\label{eq-8.1.10}
\begin{array}{lllll}
   H_\mathbb{C}:=C_{G_\mathbb{C}}(\frak{h}_\mathbb{C}), & B^\pm:=N_{G_\mathbb{C}}(\frak{b}^\pm), & \blacktriangle^\pm:=\blacktriangle\cap\triangle^\pm, & \frak{n}^\pm_1:=\bigoplus_{\alpha\in\blacktriangle^\pm}\frak{g}_\alpha,
   & N^\pm_1:=\exp\frak{n}^\pm_1.
\end{array}
\end{equation} 
   Then it turns out that 
\begin{enumerate}[(i)]
\item 
   $L_\mathbb{C}=L_uH_\mathbb{R}N^\pm_1$ are Iwasawa decompositions of the reductive Lie group $L_\mathbb{C}$, 
\item 
   If $\alpha(-iT)>0$ for all $\alpha\in\Pi_\triangle$, then $L_\mathbb{C}=H_\mathbb{C}$, $U^\pm=N^\pm$ and $Q^\pm=B^\pm$.    
\end{enumerate}   
\end{remark}

   In view of \eqref{eq-8.1.7} we see
\begin{lemma}\label{lem-8.1.11}
   Let $s=+$ or $-$. 
\begin{enumerate}
\item[{\rm (1)}]
   $G_\mathbb{C}=G_uQ^s$. 
\item[{\rm (2)}]
   $N^s\subset L_\mathbb{C}U^s$.   
\end{enumerate}
\end{lemma}
\begin{proof}
   (1). 
   We prove $G_\mathbb{C}\subset G_uQ^s$ only. 
   For any $g\in G_\mathbb{C}$, there exists a unique $(k,a,n)\in G_uH_\mathbb{R}N^s$ satisfying $g=kan$, since $G_\mathbb{C}=G_uH_\mathbb{R}N^s$.
   From \eqref{eq-8.1.7} and Lemma \ref{lem-8.0.1}-(3$''$) we obtain $[\frak{h}_\mathbb{R},\bigoplus_{\nu\geq 0}\frak{g}^{s\nu}]\subset\bigoplus_{\nu\geq 0}\frak{g}^{s\nu}$, $[\frak{n}^s,\bigoplus_{\nu\geq 0}\frak{g}^{s\nu}]\subset\bigoplus_{\nu\geq 0}\frak{g}^{s\nu}$. 
   So, both $H_\mathbb{R}=\exp\frak{h}_\mathbb{R}$ and $N^s=\exp\frak{n}^s$ are subsets of $N_{G_\mathbb{C}}(\bigoplus_{\nu\geq 0}\frak{g}^{s\nu})=Q^s$, and thus $an\in H_\mathbb{R}N^s\in Q^sQ^s\subset Q^s$. 
   This yields $g=k(an)\in G_uQ^s$, and $G_\mathbb{C}\subset G_uQ^s$.\par

   (2).
   By a direct computation with \eqref{eq-8.1.7} we deduce $\frak{n}^s=(\frak{l}_\mathbb{C}\cap\frak{n}^s)\oplus\frak{u}^s$, and both $\frak{l}_\mathbb{C}\cap\frak{n}^s$ and $\frak{u}^s$ are subalgebras of $\frak{n}^s$. 
   Therefore we conclude $N^s=\exp(\frak{l}_\mathbb{C}\cap\frak{n}^s)\exp\frak{u}^s\subset L_\mathbb{C}U^s$ since the nilpotent Lie group $N^s=\exp\frak{n}^s$ is simply connected.
\end{proof}

\section{Propositions related to complex flag manifolds}\label{sec-8.2}

\subsection{Some properties of $U^s$, $Q^s$ and $G_\mathbb{C}/Q^s$}\label{subsec-8.2.1}
   Let us clarify some properties of $U^\pm$, $Q^\pm$ and $G_\mathbb{C}/Q^\pm$.

\begin{proposition}\label{prop-8.2.1}
   The following seven items hold for each $s=\pm:$
\begin{enumerate}
\item[{\rm (i)}]
   $U^s$ is a simply connected, closed complex nilpotent subgroup of $G_\mathbb{C}$ whose Lie algebra is $\frak{u}^s$, and $\exp:\frak{u}^s\to U^s$ is biholomorphic.  
\item[{\rm (ii)}]
   $L_u$ coincides with $G_u\cap Q^s$.
\item[{\rm (iii)}]
   $Q^s=N_{G_\mathbb{C}}(\bigoplus_{\nu\geq 0}\frak{g}^{s\nu})$ is a connected, closed complex parabolic subgroup of $G_\mathbb{C}$ such that $Q^s=L_\mathbb{C}\ltimes U^s$ $($semidirect$)$ and $\frak{q}^s=(\frak{l}_\mathbb{C}\oplus\frak{u}^s)=\bigoplus_{\nu\geq 0}\frak{g}^{s\nu}$.
\item[{\rm (iv)}]
   The product mapping $U^{-s}\times Q^s\ni(u,q)\mapsto uq\in G_\mathbb{C}$ is a biholomorphism of $U^{-s}\times Q^s$ onto a domain in $G_\mathbb{C}$.\footnote{This statement will be improved later (see Corollary \ref{cor-8.3.16}-(i)).}
\item[{\rm (v)}] 
   $\iota:G_u/L_u\to G_\mathbb{C}/Q^s$, $kL_u\mapsto kQ^s$, is a $G_u$-equivariant real analytic diffeomorphism.  
\item[{\rm (vi)}] 
   $Q^s$ includes the center $Z(G_\mathbb{C})$.
\item[{\rm (vii)}]
   $Q^{-s}\cap Q^s=L_\mathbb{C}$. 
\end{enumerate}   
\end{proposition}   
\begin{proof}
   By Lemma \ref{lem-8.0.1}-(5$''$), (5$'$), (1) and $\overline{\theta}\bigl(Z(G_\mathbb{C})\bigr)=Z(G_\mathbb{C})$, it suffices to investigate the case of $s=+$ only.
   Let us obey the setting of Section \ref{sec-8.1}.\par

   (i). 
   $N^+$ is a closed complex nilpotent subgroup of $G_\mathbb{C}$ whose Lie algebra is $\frak{n}^+$, and $\exp:\frak{n}^+\to N^+$ is biholomorphic. 
   Hence we conclude (i) from $\frak{u}^+\subset\frak{n}^+$ and $U^+=\exp\frak{u}^+$. 
   cf.\ \eqref{eq-8.1.7}.\par
   
   (ii). 
   It is immediate from $L_u=(G_u\cap L_\mathbb{C})$ and Lemma \ref{lem-8.0.1}-(2$''$) that 
\[
   L_u
   \subset(G_u\cap L_\mathbb{C})
   \subset(G_u\cap Q^+). 
\]
   Let us show that the converse inclusion also holds.
   Take an arbitrary $k\in G_u\cap Q^+$. 
   We are going to conclude $k\in C_{G_u}(T)$.
   Since $Q^+=N_{G_\mathbb{C}}(\bigoplus_{\nu\geq 0}\frak{g}^\nu)$ we have $\operatorname{Ad}k\bigl(\bigoplus_{\nu\geq 0}\frak{g}^\nu\bigr)\subset\bigoplus_{\nu\geq 0}\frak{g}^\nu$. 
   Furthermore, $\overline{\theta}(k)=k$ and Lemma \ref{lem-8.0.1}-(5) give rise to $\operatorname{Ad}k\bigl(\bigoplus_{\nu\geq 0}\frak{g}^{-\nu}\bigr)\subset\bigoplus_{\nu\geq 0}\frak{g}^{-\nu}$. 
   Therefore it follows from $\frak{c}_{\frak{g}_\mathbb{C}}(T)=\frak{g}^0$ and $\frak{u}^+=\bigoplus_{\lambda>0}\frak{g}^\lambda$ that 
\[
\left\{
\begin{array}{@{\,}l}
   \operatorname{Ad}k\bigl(\frak{c}_{\frak{g}_\mathbb{C}}(T)\bigr)
   =\operatorname{Ad}k\bigl(\bigoplus_{\nu\geq 0}\frak{g}^\nu\cap\bigoplus_{\mu\geq 0}\frak{g}^{-\mu}\bigr)
   \subset\bigl(\bigoplus_{\nu\geq 0}\frak{g}^\nu\cap\bigoplus_{\mu\geq 0}\frak{g}^{-\mu}\bigr)
   =\frak{c}_{\frak{g}_\mathbb{C}}(T),\\
   \operatorname{Ad}k(\frak{u}^+)
   =\operatorname{Ad}k\bigl([T,\frak{u}^+]\bigr)
   \subset[\operatorname{Ad}k(T),\bigoplus_{\nu\geq 0}\frak{g}^\nu]
   =[\operatorname{Ad}k(T),\frak{c}_{\frak{g}_\mathbb{C}}(T)\oplus\frak{u}^+]
   \subset\frak{u}^+,
\end{array}\right.
\]
where we note that $\operatorname{ad}T:\frak{u}^+\to\frak{u}^+$ is linear isomorphic and $\operatorname{Ad}k(T)$ belongs to the center of $\frak{c}_{\frak{g}_\mathbb{C}}(T)$. 
   From $\operatorname{Ad}k(\frak{g}_u)\subset\frak{g}_u$ and $\operatorname{Ad}k\bigl(\frak{c}_{\frak{g}_\mathbb{C}}(T)\bigr)\subset\frak{c}_{\frak{g}_\mathbb{C}}(T)$ one obtains 
\begin{equation}\label{eq-1}\tag*{\textcircled{1}}
   \operatorname{Ad}k\bigl(\frak{c}_{\frak{g}_u}(T)\bigr)
   =\operatorname{Ad}k\bigl(\frak{g}_u\cap\frak{c}_{\frak{g}_\mathbb{C}}(T)\bigr)
   \subset\frak{c}_{\frak{g}_u}(T).
\end{equation}
   Here $i\frak{h}_\mathbb{R}\subset\frak{c}_{\frak{g}_u}(T)$ and $i\frak{h}_\mathbb{R}$ is a maximal torus of the compact Lie algebra $\frak{c}_{\frak{g}_u}(T)$. 
   So, by \ref{eq-1} there exists an $x\in C_{G_u}(T)$ satisfying 
\[
\begin{array}{ll}
   \operatorname{Ad}(xk)(i\frak{h}_\mathbb{R})=i\frak{h}_\mathbb{R}, 
   & {}^t\!\operatorname{Ad}(xk)^{-1}(\blacktriangle^+)\subset\blacktriangle^+,
\end{array}   
\]
where $\blacktriangle=\{\gamma\in\triangle(\frak{g}_\mathbb{C},\frak{h}_\mathbb{C}) \,|\, \gamma(T)=0\}$ and $\blacktriangle^+=\blacktriangle\cap\triangle^+$.
   In view of $x\in C_{G_u}(T)\subset L_\mathbb{C}$, $\operatorname{Ad}k(\frak{u}^+)\subset\frak{u}^+$ and Lemma \ref{lem-8.0.1}-(2$'$) we see that $\operatorname{Ad}(xk)(\frak{u}^+)\subset\frak{u}^+$. 
   This, combined with $\operatorname{Ad}(xk)(i\frak{h}_\mathbb{R})=i\frak{h}_\mathbb{R}$ and $\frak{u}^+=\bigoplus_{\alpha\in\triangle^+-\blacktriangle}\frak{g}_\alpha$, assures that 
\[
   {}^t\!\operatorname{Ad}(xk)^{-1}(\triangle^+-\blacktriangle)\subset\triangle^+-\blacktriangle.
\]   
   Consequently it follows that ${}^t\!\operatorname{Ad}(xk)^{-1}(\triangle^+)\subset\triangle^+$, and so $\operatorname{Ad}(xk)=\operatorname{id}$ on $\frak{h}_\mathbb{C}$. 
   Hence we conclude $k\in C_{G_u}(T)$ by $T\in\frak{h}_\mathbb{C}$ and $\operatorname{Ad}x(T)=T$. 
   For this reason we show that $(G_u\cap Q^+)\subset C_{G_u}(T)=L_u$, so that $L_u=(G_u\cap Q^+)$.\par
    
   (iii). 
   First, let us verify 
\begin{equation}\label{eq-2}\tag*{\textcircled{2}}
   L_\mathbb{C}U^+\subset Q^+.
\end{equation}
   It follows from \eqref{eq-8.1.7} and Lemma \ref{lem-8.0.1}-(3$''$) that $[\frak{u}^+,\bigoplus_{\nu\geq 0}\frak{g}^\nu]\subset\bigoplus_{\nu\geq 0}\frak{g}^\nu$. 
   Accordingly Lemma \ref{lem-8.0.1}-(2$''$), together with (i), assures that both $L_\mathbb{C}$ and $U^+$ are subsets of $N_{G_\mathbb{C}}(\bigoplus_{\nu\geq 0}\frak{g}^\nu)=Q^+$, and thus one can assert \ref{eq-2} $L_\mathbb{C}U^+\subset Q^+Q^+\subset Q^+$.
   Next, let us demonstrate  
\begin{equation}\label{eq-3}\tag*{\textcircled{3}}
   Q^+\subset L_\mathbb{C}U^+.
\end{equation}
   Take an arbitrary $q\in Q^+$. 
   By $q\in G_\mathbb{C}=G_uH_\mathbb{R}N^+$ there exists a unique $(k,a,n)\in G_u\times H_\mathbb{R}\times N^+$ satisfying $q=kan$. 
   Then $H_\mathbb{R}\subset L_\mathbb{C}$, Lemma \ref{lem-8.1.11}-(2) and \ref{eq-2} tell us that 
\[
   an\in L_\mathbb{C}L_\mathbb{C}U^+\subset L_\mathbb{C}U^+\subset Q^+.
\]
   This and (ii) yield $k=q(an)^{-1}\in(G_u\cap Q^+)=L_u\subset L_\mathbb{C}$. 
   Accordingly, $q=k(an)\in L_\mathbb{C}L_\mathbb{C}U^+\subset L_\mathbb{C}U^+$, and one has \ref{eq-3}.
   At this stage we know that $Q^+=L_\mathbb{C}U^+$, and that $Q^+$ is a connected, closed complex subgroup of $G_\mathbb{C}$ due to (i) and Lemma \ref{lem-8.0.1}-(a), (b). 
   In addition, $U^+$ is a normal subgroup of $Q^+=L_\mathbb{C}U^+$ by virtue of Lemma \ref{lem-8.0.1}-(2$'$) and $U^+=\exp\frak{u}^+$.
   Therefore, the rest of proof is to confirm  
\begin{equation}\label{eq-4}\tag*{\textcircled{4}}
   L_\mathbb{C}\cap U^+=\{e\}
\end{equation}
because \ref{eq-4}, $Q^+=L_\mathbb{C}U^+$, Lemma \ref{lem-8.0.1}-(1) and Remark \ref{rem-8.1.8}-(iii) assure that $\operatorname{Lie}(Q^+)=(\frak{l}_\mathbb{C}\oplus\frak{u}^+)=\bigoplus_{\nu\geq 0}\frak{g}^\nu$ is parabolic. 
   Take an arbitrary $y\in L_\mathbb{C}\cap U^+$. 
   By $y\in U^+$ and (i) there exists a unique $Y\in\frak{u}^+$ satisfying $y=\exp Y$. 
   It follows from $y\in L_\mathbb{C}=C_{G_\mathbb{C}}(T)$ that $\operatorname{Ad}y(T)=T$. 
   So, for any $t\in\mathbb{R}$ we have $y(\exp tT)y^{-1}=\exp tT$, and then $y=(\exp tT)y\exp(-tT)$. 
   Therefore $\exp Y=\exp\operatorname{Ad}(\exp tT)(Y)$. 
   This, together with $Y,\operatorname{Ad}(\exp tT)(Y)\in\frak{u}^+$ and (i), assures that $Y=\operatorname{Ad}(\exp tT)(Y)$. 
   Differentiating this equation at $t=0$, we obtain $0=[T,Y]$. 
   Thus $Y\in(\frak{c}_{\frak{g}_\mathbb{C}}(T)\cap\frak{u}^+)=\{0\}$, and $y=\exp Y=e$. 
   For this reason \ref{eq-4} holds.\par
   
   (iv).
   Since (i), (iii) and $\frak{g}_\mathbb{C}=\frak{u}^-\oplus(\frak{l}_\mathbb{C}\oplus\frak{u}^+)=\frak{u}^-\oplus\frak{q}^+$, we only show that 
\[
   U^-\cap Q^+=\{e\}.
\]   
   Take any $z\in U^-\cap Q^+$. 
   On the one hand; by (i) there exists a unique $Z\in\frak{u}^-$ such that $z=\exp Z$. 
   On the other hand; from $z\in Q^+=N_{G_\mathbb{C}}(\bigoplus_{\nu\geq 0}\frak{g}^\nu)$ and $T\in\frak{g}^0$ we have $\operatorname{Ad}z(T)-T\in\bigoplus_{\nu\geq 0}\frak{g}^\nu$. 
   Hence 
\[
   \bigoplus_{\nu\geq 0}\frak{g}^\nu
   \ni\operatorname{Ad}z(T)-T
   =\sum_{n\geq 1}\dfrac{1}{n!}(\operatorname{ad}Z)^nT
   \in\frak{u}^-
   =\bigoplus_{\lambda>0}\frak{g}^{-\lambda}.
\]  
   This implies that $\operatorname{Ad}z(T)-T=0$, so that $z\in(L_\mathbb{C}\cap U^-)=\{e\}$ by (iii). 
   Thus $U^-\cap Q^+=\{e\}$ follows.\par
   
   (v). 
   By (ii) and Lemma \ref{lem-8.1.11}-(1), $\iota:G_u/L_u\to G_\mathbb{C}/Q^+$, $kL_u\mapsto kQ^+$, is bijective and $G_u$-equivariant real analytic. 
   Hence it suffices to show that the differential $(d\iota)_{o_u}$ of $\iota$ at $o_u$ is a real linear isomorphism of the tangent vector space $T_{o_u}(G_u/L_u)$ onto $T_{\iota(o_u)}(G_\mathbb{C}/Q^+)$. 
   Here $o_u$ denotes the origin of $G_u/L_u$.
   From (ii) we obtain $\frak{l}_u=(\frak{g}_u\cap\frak{q}^+)$, and so the differential $(d\iota)_{o_u}$ is a linear injection. 
   Moreover, since $\frak{g}_u$ (resp.\ $\frak{l}_u$) is a real form of $\frak{g}_\mathbb{C}$ (resp.\ $\frak{l}_\mathbb{C}$), one shows
\[
\begin{split}
   \dim_\mathbb{R}G_\mathbb{C}/Q^+
  &=\dim_\mathbb{R}\frak{g}_\mathbb{C}-\dim_\mathbb{R}\frak{q}^+
   \stackrel{{\rm (iii)}}{=}\dim_\mathbb{R}\frak{g}_\mathbb{C}-(\dim_\mathbb{R}\frak{l}_\mathbb{C}+\dim_\mathbb{R}\frak{u}^+)\\
  &=2\dim_\mathbb{R}\frak{g}_u-2\dim_\mathbb{R}\frak{l}_u-\dim_\mathbb{R}\frak{u}^+
   =2\dim_\mathbb{R}\operatorname{ad}T(\frak{g}_u)-\dim_\mathbb{R}\frak{u}^+\\
  &=\dim_\mathbb{R}\operatorname{ad}T(\frak{g}_u)
   =\dim_\mathbb{R}\frak{g}_u-\dim_\mathbb{R}\frak{l}_u
   =\dim_\mathbb{R}G_u/L_u,
\end{split}
\]
where we remark that $\frak{g}_u=\frak{l}_u\oplus\operatorname{ad}T(\frak{g}_u)$ and $\dim_\mathbb{R}\operatorname{ad}T(\frak{g}_u)=\dim_\mathbb{R}\frak{u}^+$.
   Thus $(d\iota)_{o_u}:T_{o_u}(G_u/L_u)\to T_{\iota(o_u)}(G_\mathbb{C}/Q^+)$ is linear isomorphic.\par
   
   (vi). 
   $Z(G_\mathbb{C})\subset C_{G_\mathbb{C}}(T)=L_\mathbb{C}\subset Q^+$ by (iii).\par
   
   (vii). 
   Lemma \ref{lem-8.0.1}-(2$''$) yields $L_\mathbb{C}\subset Q^-\cap Q^+$. 
   Now, let $q\in Q^-\cap Q^+$. 
   Then, there exist elements $(l_\pm,u_\pm)\in L_\mathbb{C}\times U^\pm$ satisfying $l_-u_-=q=l_+u_+$ due to (iii).
   From $u_-=l_-^{-1}l_+u_+$, (iii) and (iv), we deduce that $u_-=e$, $l_-^{-1}l_+=e$, $u_+=e$.
   Therefore $q=l_+\in L_\mathbb{C}$, and hence $Q^-\cap Q^+\subset L_\mathbb{C}$.
\end{proof}

\begin{remark}\label{rem-8.2.2}
   By Proposition \ref{prop-8.2.1}-(iii), $\operatorname{Lie}(Q^s)=\frak{q}^s$ is a complex parabolic subalgebra of $\frak{g}_\mathbb{C}$ whose Levi factor and unipotent radical are $\frak{l}_\mathbb{C}$ and $\frak{u}^s$, respectively ($s=\pm$).
\end{remark}

   Propositions \ref{prop-7.3.4} and \ref{prop-8.2.1}-(v) and Lemma \ref{lem-8.0.1}-(5$''$) lead to 
\begin{corollary}\label{cor-8.2.3}
\begin{enumerate}
\item[]
\item[{\rm (1)}]
   Both $G_\mathbb{C}/Q^+$ and $G_\mathbb{C}/Q^-$ are simply connected, compact complex homogeneous spaces.
\item[{\rm (2)}]
   $G_\mathbb{C}/Q^+$ is $G_\mathbb{C}$-equivariant anti-biholomorphic to $G_\mathbb{C}/Q^-$ via the mapping $G_\mathbb{C}/Q^+\ni gQ^+\mapsto\overline{\theta}(g)Q^-\in G_\mathbb{C}/Q^-$.
\end{enumerate}   
\end{corollary} 

\subsection{A complex Grassmann manifold and an invariant K\"{a}hler metric on $G_\mathbb{C}/Q^s$}\label{subsec-8.2.2}
   Now, let $M_{N,K}(\mathbb{C})$ be the set of all $K$-dimensional complex vector subspaces of $\frak{g}_\mathbb{C}$, where $N=\dim_\mathbb{C}\frak{g}_\mathbb{C}$, $K=\dim_\mathbb{C}\frak{q}^s$ and $s=+$ or $-$.
   The special linear group $SL(\frak{g}_\mathbb{C})=SL(N,\mathbb{C})$ acts transitively on $M_{N,K}(\mathbb{C})$, so we put
\[
   M_{N,K}(\mathbb{C})=SL(\frak{g}_\mathbb{C})/P^s,
\]   
where $P^s:=\{\phi\in SL(\frak{g}_\mathbb{C}) \,|\, \phi(\frak{q}^s)\subset\frak{q}^s\}$.
   In this way, $M_{N,K}(\mathbb{C})$ is a complex homogeneous space, which is a complex Grassmann manifold.
   Since $\operatorname{Ad}G_\mathbb{C}\subset SL(\frak{g}_\mathbb{C})$ and $G_\mathbb{C}$ acts on $M_{N,K}(\mathbb{C})$ as a holomorphic transformation group, 
\[
   G_\mathbb{C}\times M_{N,K}(\mathbb{C})\ni(g,\frak{m})\mapsto\operatorname{Ad}g(\frak{m})\in M_{N,K}(\mathbb{C}),
\]
one can consider the orbit $\operatorname{Ad}G_\mathbb{C}(\frak{q}^s)$ of $G_\mathbb{C}$ through the point $\frak{q}^s\in M_{N,K}(\mathbb{C})$, and equip $\operatorname{Ad}G_\mathbb{C}(\frak{q}^s)\subset M_{N,K}(\mathbb{C})$ with the relative topology. 
   Then, it follows from $Q^s=N_{G_\mathbb{C}}(\frak{q}^s)$ that 
\begin{equation}\label{eq-8.2.4}
   \mbox{$f_s:G_\mathbb{C}/Q^s\to\operatorname{Ad}G_\mathbb{C}(\frak{q}^s)$, $gQ^s\mapsto\operatorname{Ad}g(\frak{q}^s)$},
\end{equation}
is a bijective continuous mapping; moreover, it is homeomorphic by Corollary \ref{cor-8.2.3}. 
   Providing $\operatorname{Ad}G_\mathbb{C}(\frak{q}^s)$ with the holomorphic structure so that $f_s:G_\mathbb{C}/Q^s\to\operatorname{Ad}G_\mathbb{C}(\frak{q}^s)$ is biholomorphic, we deduce that the orbit $\operatorname{Ad}G_\mathbb{C}(\frak{q}^s)$ is a simply connected, compact, regular complex submanifold of $M_{N,K}(\mathbb{C})$. 
   Here we remark that the mapping $G_\mathbb{C}/Q^s\ni gQ^s\mapsto(\operatorname{Ad}g)P^s\in SL(\frak{g}_\mathbb{C})/P^s$ is a $G_\mathbb{C}$-equivariant holomorphic embedding.

\begin{remark}\label{rem-8.2.5}
   Via the Pl\"{u}cker embedding $p_{\ddot{u}}:M_{N,K}(\mathbb{C})\to CP(\wedge^K\mathbb{C}^N)$, $\operatorname{span}_\mathbb{C}\{v_1,v_2,\dots,v_K\}\mapsto[v_1\wedge v_2\wedge\cdots\wedge v_K]$, the orbit $\operatorname{Ad}G_\mathbb{C}(\frak{q}^s)$ can be holomorphically embedded into the complex projective space $CP(\wedge^K\mathbb{C}^N)$ of dimension $N!/(K!(N-K)!)-1$. 
   Since $p_{\ddot{u}}\bigl(\operatorname{Ad}G_\mathbb{C}(\frak{q}^s)\bigr)=p_{\ddot{u}}\bigl(f_s(G_\mathbb{C}/Q^s)\bigr)$ is a connected, compact, regular complex submanifold of $CP(\wedge^K\mathbb{C}^N)$, it is a projective algebraic variety by Theorem V in  Chow \cite[p.910]{Ch}.
\end{remark}

   The complex homogeneous space $G_\mathbb{C}/Q^s$ is a K\"{a}hler manifold. 
   Indeed, 
\begin{proposition}\label{prop-8.2.6}
  The simply connected, compact complex homogeneous space $G_\mathbb{C}/Q^s=(G_\mathbb{C}/Q^s,J_s)$ admits a $G_u$-invariant K\"{a}hler metric ${\sf g}_s$ with respect to $J_s$ $(s=\pm)$.
  Here we refer to Remark {\rm \ref{rem-1.2.3}} for the $G_\mathbb{C}$-invariant complex structure $J_s$ on $G_\mathbb{C}/Q^s$.\footnote{Remark.\ $J_s$ is $G_\mathbb{C}$-invariant, but, in contrast, ${\sf g}_s$ is $G_u$-invariant.}
\end{proposition} 
\begin{proof}
   The special unitary group $SU(\frak{g}_u)=SU(N)$ acts transitively on the complex Grassmann manifold $M_{N,K}(\mathbb{C})$,\footnote{Remark.\ $M_{N,K}(\mathbb{C})$ may be represented as $SU(N)/S(U(K)\times U(N-K))$.} and the complex homogeneous space $M_{N,K}(\mathbb{C})=SL(\frak{g}_\mathbb{C})/P^s$ admits a unique $SU(\frak{g}_u)$-invariant K\"{a}hler metric $\tilde{{\sf g}}_s$ up to a positive multiplicative constant. 
   Accordingly one can induce a $G_u$-invariant K\"{a}hler metric ${\sf g}_s$ on $G_\mathbb{C}/Q^s$ by use of the $G_\mathbb{C}$-equivariant holomorphic embedding $G_\mathbb{C}/Q^s\ni gQ^s\mapsto(\operatorname{Ad}g)P^s\in SL(\frak{g}_\mathbb{C})/P^s$, where we remark that $\operatorname{Ad}G_u\subset SU(\frak{g}_u)$.
\end{proof}

\subsection{A complex projective space, an irreducible representation and $G_\mathbb{C}/Q^+$}\label{subsec-8.2.3}
   Our goal in this subsection is to prove that $G_\mathbb{C}/Q^+$ can be holomorphically embedded into a complex projective space $CP({\sf V})$.
   We will construct arguments by obeying the setting of Section \ref{sec-8.1}, in particular, Subsection \ref{subsec-8.1.4}.\par
   
   Let $\hat{G}_\mathbb{C}$ be the quotient group of the Lie group $G_\mathbb{C}$ modulo the center $Z(G_\mathbb{C})$, and set $\hat{Q}^+:=N_{\hat{G}_\mathbb{C}}(\frak{q}^+)$, where we assume 
\[
   \frak{g}_\mathbb{C}=\hat{\frak{g}}_\mathbb{C}.
\] 
   First, let us confirm  
\begin{lemma}\label{lem-8.2.7}
   The mapping $G_\mathbb{C}/Q^+\ni gQ^+\mapsto\pi(g)\hat{Q}^+\in\hat{G}_\mathbb{C}/\hat{Q}^+$ is a $G_\mathbb{C}$-equivariant biholomorphism. 
   Here $\pi$ is the projection of $G_\mathbb{C}$ onto $\hat{G}_\mathbb{C}=G_\mathbb{C}/Z(G_\mathbb{C})$.
\end{lemma}
\begin{proof}
   $\pi:G_\mathbb{C}\to\hat{G}_\mathbb{C}$, $g\mapsto gZ(G_\mathbb{C})$, is a surjective holomorphic homomorphism. 
   So, we conclude this lemma from Proposition \ref{prop-8.2.1}-(vi). 
\end{proof} 

   Next, let us construct a complex projective space from an irreducible representation. 
   Denote by $\{\varpi_a\}_{a=1}^\ell$ the set of the fundamental dominant weights relative to $\Pi_\triangle=\{\alpha_a\}_{a=1}^\ell$.
   For $-iT=\sum_{a=1}^\ell\lambda_aZ_a$ ($\lambda_a\geq 0$) one can separate the set $\{\lambda_a\}_{a=1}^\ell$ into two pieces $\{\lambda_{a_j}\}_{j=1}^r$ and $\{\lambda_{a_k}\}_{k=r+1}^\ell$ so that $\lambda_{a_j}>0$ for all $1\leq j\leq r$ and $\lambda_{a_k}=0$ for all $r+1\leq k\leq\ell$. 
   Namely, 
\[
\begin{split}
   -iT
   &=\mbox{$\lambda_1Z_1+\lambda_2Z_2+\cdots+\lambda_\ell Z_\ell$ with $\lambda_1,\lambda_2,\dots,\lambda_\ell\geq 0$}\\
   &=\mbox{$\lambda_{a_1}Z_{a_1}+\lambda_{a_2}Z_{a_2}+\cdots+\lambda_{a_r}Z_{a_r}$ with $\lambda_{a_1},\lambda_{a_2},\dots,\lambda_{a_r}>0$}.
\end{split}   
\]
   Remark here that $\Pi_\triangle\cap\blacktriangle=\{\alpha_{a_k}\}_{k=r+1}^\ell$. 
   Taking that into account, we define a dominant integral form $\varpi$ on $\frak{h}_\mathbb{C}$ as follows:  
\begin{equation}\label{eq-8.2.8}
   \varpi:=\varpi_{a_1}+\varpi_{a_2}+\cdots+\varpi_{a_r}.
\end{equation} 
   Here $\varpi\neq0$ comes from $T\neq 0$.
   The Cartan-Weyl theorem enables us to obtain an irreducible representation $\rho_*$ of the complex Lie algebra $\frak{g}_\mathbb{C}$ on a finite-dimensional complex vector space ${\sf V}$ which has the above $\varpi$ as its highest weight. 
   Since the complex Lie group $\hat{G}_\mathbb{C}$ is isomorphic to the adjoint group of $\frak{g}_\mathbb{C}$, one can take a holomorphic homomorphism $\rho:\hat{G}_\mathbb{C}\to GL({\sf V})$, $\hat{g}\mapsto\rho(\hat{g})$, whose differential homomorphism accords with $\rho_*:\frak{g}_\mathbb{C}\to\frak{gl}({\sf V})$.
   Now, let ${\sf u}_0\in{\sf V}$ be a maximal vector of weight $\varpi$, let $[{\sf v}]:=\operatorname{span}_\mathbb{C}\{{\sf v}\}$ for $0\neq{\sf v}\in{\sf V}$, and let $CP({\sf V})=\{[{\sf v}]: 0\neq {\sf v}\in{\sf V}\}$ denote the complex projective space of dimension $d-1$, where $d:=\dim_\mathbb{C}{\sf V}$.
   The special linear group $SL({\sf V})=SL(d,\mathbb{C})$ acts transitively on $CP({\sf V})$, so we put 
\[
   CP({\sf V})=SL({\sf V})/P,
\] 
where $P:=\{\varphi\in SL({\sf V}) : [\varphi({\sf u}_0)]=[{\sf u}_0]\}$. 
   In this setting, we verify

\begin{lemma}\label{lem-8.2.9}
   The mapping $\hat{G}_\mathbb{C}/\hat{Q}^+\ni\hat{g}\hat{Q}^+\mapsto\rho(\hat{g})P\in SL({\sf V})/P$ is a $\hat{G}_\mathbb{C}$-equivariant holomorphic embedding.
\end{lemma}
\begin{proof}
   Remark that $\rho(\hat{G}_\mathbb{C})\subset SL({\sf V})$ follows from $\rho(\hat{G}_\mathbb{C})\subset GL({\sf V})$ and $\hat{G}_\mathbb{C}$ being connected semisimple.\par
   
   In this proof we temporarily denote by $\hat{f}$ the mapping $\hat{G}_\mathbb{C}/\hat{Q}^+\ni\hat{g}\hat{Q}^+\mapsto\rho(\hat{g})P\in SL({\sf V})/P$.\par 
   
   (well-defined). 
   It is necessary to confirm that the $\hat{f}$ is well-defined. 
   For this reason, we aim to demonstrate $\rho(\hat{Q}^+)\subset P$. 
   From \eqref{eq-8.1.7} and $\blacktriangle^+=\blacktriangle\cap\triangle^+$ one has 
\[
   \mbox{$\frak{q}^+=\frak{l}_\mathbb{C}\oplus\frak{u}^+=\frak{h}_\mathbb{C}\oplus\frak{n}^+\oplus\bigoplus_{\gamma\in\blacktriangle^+}\frak{g}_{-\gamma}$}.
\]   
   So, for a given $X\in\operatorname{Lie}(\hat{Q}^+)=\frak{q}^+$ there exists a unique $(X_h,X_n,X_l)\in\frak{h}_\mathbb{C}\times\frak{n}^+\times\bigoplus_{\gamma\in\blacktriangle^+}\frak{g}_{-\gamma}$ such that 
\[
   X=X_h+X_n+X_l.
\] 
   By a direct computation we obtain
\begin{equation}\label{eq-1}\tag*{\textcircled{1}}
\begin{array}{ll}
   \rho_*(X_h){\sf u}_0=\varpi(X_h){\sf u}_0\in\operatorname{span}_\mathbb{C}\{{\sf u}_0\},
   & \rho_*(X_n){\sf u}_0=0\in\operatorname{span}_\mathbb{C}\{{\sf u}_0\}
\end{array}
\end{equation}
because ${\sf u}_0$ is a maximal vector of weight $\varpi$ and $X_h\in\frak{h}_\mathbb{C}$, $X_n\in\frak{n}^+$. 
   We want to show that 
\begin{equation}\label{eq-2}\tag*{\textcircled{2}}
   \rho_*(X_l){\sf u}_0\in\operatorname{span}_\mathbb{C}\{{\sf u}_0\},
\end{equation}
which is a consequence of that $\rho_*(E_{-\gamma}){\sf u}_0=0$ for all $\gamma\in\blacktriangle^+$. 
   Therefore, let us show that $\rho_*(E_{-\gamma}){\sf u}_0=0$ for all $\gamma\in\blacktriangle^+$. 
   Take an arbitrary $\gamma\in\blacktriangle^+$. 
   From $\Pi_\triangle\cap\blacktriangle=\{\alpha_{a_k}\}_{k=r+1}^\ell$ we deduce $\blacktriangle^+=\operatorname{span}_{\mathbb{Z}_{\geq 0}}\{\alpha_{a_k}\}_{k=r+1}^\ell$. 
   This, together with \eqref{eq-8.2.8} and $\gamma\in\blacktriangle^+$, gives  
\begin{equation}\label{eq-a}\tag{a} 
   \varpi(H_\gamma^*)=0.
\end{equation} 
   We are going to construct a complex vector subspace of ${\sf V}$ from the $\frak{s}_\gamma=\operatorname{span}_\mathbb{C}\{H_\gamma^*,E_\gamma,E_{-\gamma}\}$. 
   Set 
\begin{equation}\label{eq-b}\tag{b} 
   {\sf u}_n:=\rho_*(E_{-\gamma})^n{\sf u}_0
\end{equation} 
for $n\in\mathbb{N}$. 
   Since $\dim_\mathbb{C}{\sf V}<\infty$ and vectors ${\sf u}_0,{\sf u}_1,{\sf u}_2,\dots$ are linearly independent\footnote{(because ${\sf u}_0,{\sf u}_1,{\sf u}_2,\dots$ are eigenvectors of $\rho_*(H_\gamma^*)$ for distinct eigenvalues)}, there exists a unique $m\in\mathbb{N}$ such that ${\sf u}_p\neq 0$ ($0\leq p\leq m-1$) and ${\sf u}_m={\sf u}_{m+1}=\cdots=0$.
   Here, it follows from \eqref{eq-a} and \eqref{eq-b} that
\[
\left\{
\begin{array}{@{\,}l}
   \mbox{$\rho_*(H_\gamma^*){\sf u}_p=-2p{\sf u}_p$ for all $0\leq p\leq m-1$},\\
   \mbox{$\rho_*(E_\gamma){\sf u}_0=0$, $\rho_*(E_\gamma){\sf u}_q=-q(q-1){\sf u}_{q-1}$ for all $1\leq q\leq m-1$},\\
   \mbox{$\rho_*(E_{-\gamma}){\sf u}_{m-1}=0$, $\rho_*(E_{-\gamma}){\sf u}_j={\sf u}_{j+1}$ for all $0\leq j\leq m-2$}.
\end{array}\right.
\]
   Then ${\sf W}:=\operatorname{span}_\mathbb{C}\{{\sf u}_p\}_{p=0}^{m-1}$ is an $m$-dimensional, $\rho_*(\frak{s}_\gamma)$-invariant complex vector subspace of ${\sf V}$, and moreover, $0={\sf u}_m=\rho_*(E_{-\gamma})^m{\sf u}_0$ and \eqref{eq-b} yield $0=\rho_*(E_\gamma)\bigl(\rho_*(E_{-\gamma})^m{\sf u}_0\bigr)=-m(m-1){\sf u}_{m-1}$; therefore $m=1$. 
   This implies that ${\sf u}_1=0$, and thus $\rho_*(E_{-\gamma}){\sf u}_0={\sf u}_1=0$.
   Accordingly we conclude \ref{eq-2}.
   By virtue of \ref{eq-1}, \ref{eq-2}, one can assert that 
\begin{equation}\label{eq-3}\tag*{\textcircled{3}}
   \mbox{$\rho_*(X){\sf u}_0\in\operatorname{span}_\mathbb{C}\{{\sf u}_0\}$ for all $X\in\operatorname{Lie}(\hat{Q}^+)=\frak{q}^+$}.
\end{equation}
   Therefore, for every $X\in\operatorname{Lie}(\hat{Q}^+)$ there exists a $w\in\mathbb{C}$ satisfying $\rho_*(X){\sf u}_0=w{\sf u}_0$, and then $\rho(\exp X){\sf u}_0=e^w{\sf u}_0\in[{\sf u}_0]$. 
   This assures $\rho(\hat{Q}^+)\subset P$ because the Lie group $\hat{Q}^+$ is connected.
   From $\rho(\hat{Q}^+)\subset P$ we conclude that $\hat{f}:\hat{G}_\mathbb{C}/\hat{Q}^+\to SL({\sf V})/P$, $\hat{g}\hat{Q}^+\mapsto\rho(\hat{g})P$, is well-defined.\par
    
   (injective). 
   Our aim is to prove that $\hat{f}$ is injective. 
   In order to accomplish the aim, we first prepare for a moment. 
   Fix any $\alpha\in\triangle^+-\blacktriangle$. 
   On the one hand; one has $\varpi(H_\alpha^*)>0$ since \eqref{eq-8.2.8}. 
   On the other hand; it follows from $H_\alpha^*=[E_\alpha,E_{-\alpha}]$ and $\rho_*(E_\alpha){\sf u}_0=0$ that 
\[
   \varpi(H_\alpha^*){\sf u}_0
   =\rho_*([E_\alpha,E_{-\alpha}]){\sf u}_0
   =\rho_*(E_\alpha)\bigl(\rho_*(E_{-\alpha}){\sf u}_0\bigr)-\rho_*(E_{-\alpha})\bigl(\rho_*(E_\alpha){\sf u}_0\bigr)
   =\rho_*(E_\alpha)\bigl(\rho_*(E_{-\alpha}){\sf u}_0\bigr).
\]
   Consequently we can assert that  
\begin{equation}\label{eq-4}\tag*{\textcircled{4}}
   \mbox{(i) $\rho_*(E_{-\alpha}){\sf u}_0\neq 0$, and (ii) $\varpi-\alpha$ is a weight of the representation $\rho_*$ (relative to $\frak{h}_\mathbb{C}$) for each $\alpha\in\triangle^+-\blacktriangle$}.
\end{equation}
   Next, let us confirm that 
\begin{equation}\label{eq-5}\tag*{\textcircled{5}} 
   \mbox{$Y\in\frak{g}_\mathbb{C}$ and $\rho_*(Y){\sf u}_0\in\operatorname{span}_\mathbb{C}\{{\sf u}_0\}$ imply $Y\in\frak{q}^+$}.  
\end{equation}
   For $Y\in\frak{g}_\mathbb{C}$ suppose that $\rho_*(Y){\sf u}_0\in\operatorname{span}_\mathbb{C}\{{\sf u}_0\}$. 
   By \eqref{eq-8.1.7} and $\frak{g}_\mathbb{C}=\frak{q}^+\oplus\frak{u}^-$ we have $\frak{g}_\mathbb{C}=\frak{q}^+\oplus\bigoplus_{\alpha\in\triangle^+-\blacktriangle}\frak{g}_{-\alpha}$.
   So, there exist a $Y_q\in\frak{q}^+$ and $w_{-\alpha}\in\mathbb{C}$ such that $Y=Y_q+\sum_{\alpha\in\triangle^+-\blacktriangle}w_{-\alpha}E_{-\alpha}$. 
   Then the supposition, \ref{eq-3} and \ref{eq-4} yield 
\[
   \mbox{${\sf V}_\varpi
   =\operatorname{span}_\mathbb{C}\{{\sf u}_0\}
   \ni\rho_*(Y){\sf u}_0
   =\rho_*(Y_q){\sf u}_0+\sum_{\alpha\in\triangle^+-\blacktriangle}w_{-\alpha}\rho_*(E_{-\alpha}){\sf u}_0
   \in{\sf V}_\varpi\oplus\bigoplus_{\alpha\in\triangle^+-\blacktriangle}{\sf V}_{\varpi-\alpha}$},
\]
and therefore $w_{-\alpha}=0$ for all $\alpha\in\triangle^+-\blacktriangle$, where we denote by ${\sf V}_{\varpi-\alpha}$ the wight subspace of ${\sf V}$ for $\varpi-\alpha$. 
   Consequently it turns out that $Y=Y_q\in\frak{q}^+$, and so \ref{eq-5} holds.\par
   
   Now, we are in a position to accomplish the aim. 
   For the aim, it is enough to prove that
\begin{equation}\label{eq-6}\tag*{\textcircled{6}}
   \rho^{-1}(P)\subset\hat{Q}^+.   
\end{equation}
   For $\hat{g}\in\hat{G}_\mathbb{C}$ we suppose $\rho(\hat{g})\in P$. 
   Then $[\rho(\hat{g}){\sf u}_0]=[{\sf u}_0]$ holds, and for any $X\in\frak{q}^+$ one has 
\[
   \rho_*\bigl(\operatorname{Ad}\hat{g}(X)\bigr){\sf u}_0
   =\rho(\hat{g})\rho_*(X)\rho(\hat{g})^{-1}{\sf u}_0\in\operatorname{span}_\mathbb{C}\{{\sf u}_0\}
\]
by \ref{eq-3}. 
   Accordingly it follows from \ref{eq-5} that $\operatorname{Ad}\hat{g}(X)\in\frak{q}^+$ for all $X\in\frak{q}^+$. 
   Thus $\hat{g}\in N_{\hat{G}_\mathbb{C}}(\frak{q}^+)=\hat{Q}^+$, and one concludes \ref{eq-6}. 
   This \ref{eq-6} implies that $\hat{f}:\hat{G}_\mathbb{C}/\hat{Q}^+\to SL({\sf V})/P$, $\hat{g}\hat{Q}^+\mapsto\rho(\hat{g})P$, is injective.\par
   
   (holomorphic). 
   Since $\rho:\hat{G}_\mathbb{C}\to SL({\sf V})$, $\hat{g}\mapsto\rho(\hat{g})$, is a holomorphic homomorphism, it is now obvious that $\hat{f}:\hat{G}_\mathbb{C}/\hat{Q}^+\to SL({\sf V})/P$, $\hat{g}\hat{Q}^+\mapsto\rho(\hat{g})P$, is a $\hat{G}_\mathbb{C}$-equivariant holomorphic mapping. 
   Moreover, we have already shown that $\hat{f}$ is injective, which also assures that its differential $(d\hat{f})_p$ is injective at each point $p\in\hat{G}_\mathbb{C}/\hat{Q}^+$. 
   Hence, the mapping $\hat{f}$ is a $\hat{G}_\mathbb{C}$-equivariant holomorphic embedding.
\end{proof}

   By Lemmas \ref{lem-8.2.7} and \ref{lem-8.2.9} one establishes 
\begin{theorem}\label{thm-8.2.10}
   $G_\mathbb{C}/Q^+$ is able to be $G_\mathbb{C}$-equivariant holomorphically embedded into the complex projective space $CP({\sf V})=SL({\sf V})/P$, where ${\sf V}$ is a representation space of the irreducible representation of $\frak{g}_\mathbb{C}$ with highest weight $\varpi$ in \eqref{eq-8.2.8}. 
\end{theorem}

\section{Bruhat decompositions}\label{sec-8.3}
   In this section we generalize Bruhat decompositions of $G_\mathbb{C}$ by following Kostant's method \cite{Ko1,Ko2}. 
   The setting of Section \ref{sec-8.1} remains valid in this section.
   Recalling that $\mathcal{W}=N_{G_u}(i\frak{h}_\mathbb{R})/C_{G_u}(i\frak{h}_\mathbb{R})$ and $\blacktriangle=\{\gamma\in\triangle \,|\, \gamma(T)=0\}$, we set
\begin{equation}\label{eq-8.3.1}
\left\{
\begin{array}{@{\,}l}
   \mbox{$\Phi_{[w]}:=\{\beta\in\triangle^+ \,|\, \zeta([w])^{-1}\beta\in\triangle^-\}$ for $[w]\in\mathcal{W}$},\\
   \begin{array}{@{}ll} 
   \mathcal{W}^1:=\{[\sigma]\in\mathcal{W} \,|\, \Phi_{[\sigma]}\subset\triangle^+-\blacktriangle\}, &
   \mathcal{W}_1:=N_{L_u}(i\frak{h}_\mathbb{R})/C_{L_u}(i\frak{h}_\mathbb{R}). 
   \end{array}
\end{array}\right.
\end{equation}   
   Remark here that $\mathcal{W}_1$ is a Weyl group of $L_\mathbb{C}$.
   Hereafter, we assume $\mathcal{W}_1$ to be a subgroup of the Weyl group $\mathcal{W}$ via the mapping $N_{L_u}(i\frak{h}_\mathbb{R})/C_{L_u}(i\frak{h}_\mathbb{R})\ni\tau C_{L_u}(i\frak{h}_\mathbb{R})\mapsto\tau C_{G_u}(i\frak{h}_\mathbb{R})\in N_{G_u}(i\frak{h}_\mathbb{R})/C_{G_u}(i\frak{h}_\mathbb{R})$.
   In addition, we utilize the following notation:
\begin{itemize}
\item
   $[\kappa]$ : the unique element of $\mathcal{W}$ such that $\zeta([\kappa])(\triangle^-)=\triangle^+$ and $\zeta([\kappa])=\zeta([\kappa])^{-1}$, 
\item 
   $n_{[\sigma]}$ : the cardinal number of the set $\Phi_{[\sigma]}$ for $[\sigma]\in\mathcal{W}^1$.   
\end{itemize}

\subsection{A proposition on the root system}\label{subsec-8.3.1}
   We will verify Theorem \ref{thm-8.3.7} in the next subsection. 
   For this reason we need  

\begin{proposition}[{cf.\ Kostant \cite[pp.359--361]{Ko1}, \cite[p.121]{Ko2}}]\label{prop-8.3.2}
\begin{enumerate}
\item[]
\item[{\rm (i)}] 
   $\Phi_{[w]}$ is a closed subset of $\triangle$ for any $[w]\in\mathcal{W}$ $($i.e., $\beta_1,\beta_2\in\Phi_{[w]}$ and $\beta_1+\beta_2\in\triangle$ imply $\beta_1+\beta_2\in\Phi_{[w]})$. 
\item[{\rm (ii)}]
   $\triangle^+=\Phi_{[w]}\amalg\Phi_{[w\kappa]}$ $($disjoint union$)$ for all $[w]\in\mathcal{W}$. 
\item[{\rm (iii)}]
   If $[\sigma]\in\mathcal{W}^1$, then $\zeta([\sigma])^{-1}(\blacktriangle^+)\subset\triangle^+$ and $\zeta([\sigma])^{-1}(\blacktriangle^-)\subset\triangle^-$. 
\item[{\rm (iv)}]
   For each $[w]\in\mathcal{W}$, there exists a unique $([\tau],[\sigma])\in\mathcal{W}_1\times\mathcal{W}^1$ such that $[w]=[\tau\sigma]$. 
\item[{\rm (v)}]
   For a given $[\sigma]\in\mathcal{W}^1$, the following items {\rm (v.1)} and {\rm (v.2)} hold$:$
   \begin{enumerate}
   \item[{\rm (v.1)}]
   $n_{[\sigma]}=0$ if and only if $[e]=[\sigma]$. 
   \item[{\rm (v.2)}]
   $n_{[\sigma]}=1$ if and only if there exists a $\beta\in\Pi_\triangle-\blacktriangle$ satisfying $[w_\beta]=[\sigma]$.
   \end{enumerate}
\end{enumerate}
   Here $\blacktriangle^\pm=\blacktriangle\cap\triangle^\pm$, $e$ is the unit element of $G_\mathbb{C}$, and we refer to \eqref{eq-8.1.4} for $w_\beta$.
\end{proposition}

   The main purpose of this subsection is to prove Proposition \ref{prop-8.3.2}. 
   First of all, we are going to prepare three lemmas for proving it. 
   The first lemma is 

\begin{lemma}\label{lem-8.3.3}
\begin{enumerate}
\item[]
\item[{\rm (1)}] 
   $\Phi_{[w]}$ is a closed subset of $\triangle$ for any $[w]\in\mathcal{W}$. 
\item[{\rm (2)}]
   $\triangle^+=\Phi_{[w]}\amalg\Phi_{[w\kappa]}$ for all $[w]\in\mathcal{W}$. 
\item[{\rm (3)}]
   If $[\sigma]\in\mathcal{W}^1$, then $\zeta([\sigma])^{-1}(\blacktriangle^+)\subset\triangle^+$ and $\zeta([\sigma])^{-1}(\blacktriangle^-)\subset\triangle^-$. 
\end{enumerate}
\end{lemma}   
\begin{proof}
   We only prove (3), since (1) and (2) are clear from the definition \eqref{eq-8.3.1} of $\Phi_{[w]}$ and $\zeta([\kappa])(\triangle^-)=\triangle^+$.\par
   
   (3). 
   For each $\gamma\in\blacktriangle^+$, either the case $\zeta([\sigma])^{-1}\gamma\in\triangle^-$ or $\zeta([\sigma])^{-1}\gamma\in\triangle^+$ has to occur.
   If $\zeta([\sigma])^{-1}\gamma\in\triangle^-$, then it follows from \eqref{eq-8.3.1} and $[\sigma]\in\mathcal{W}^1$ that $\gamma\in\Phi_{[\sigma]}\subset\triangle^+-\blacktriangle$, which contradicts $\gamma\in\blacktriangle^+$. 
   Thus the remaining case $\zeta([\sigma])^{-1}\gamma\in\triangle^+$ occurs, and $\zeta([\sigma])^{-1}(\blacktriangle^+)\subset\triangle^+$ holds.\par

   The $\zeta([\sigma])^{-1}(\blacktriangle^-)\subset\triangle^-$ comes from $\zeta([\sigma])^{-1}(\blacktriangle^+)\subset\triangle^+$, $\triangle^-=-\triangle^+$ and $\blacktriangle^-=-\blacktriangle^+$. 
\end{proof}

   The second lemma is 
\begin{lemma}\label{lem-8.3.4}
   Set $\triangle_{[w]}:=\Phi_{[w]}\amalg(-\Phi_{[w\kappa]})$ for $[w]\in\mathcal{W}$. 
   Then 
\begin{enumerate}
\item[{\rm (1)}] 
   $\triangle_{[w]}$ is a closed subset of $\triangle$ for any $[w]\in\mathcal{W}$. 
\item[{\rm (2)}]
   $\triangle=\triangle_{[w]}\amalg(-\triangle_{[w]})$ for all $[w]\in\mathcal{W}$.
\item[{\rm (3)}]
   $\zeta([w])^{-1}(\triangle_{[w]})=\triangle^-$ for all $[w]\in\mathcal{W}$. 
\end{enumerate}
\end{lemma}   
\begin{proof}
   (1). 
   For $\alpha,\beta\in\triangle_{[w]}$ we suppose that $\alpha+\beta\in\triangle$. 
   If $\alpha,\beta\in\Phi_{[w]}$ (resp.\ $-\Phi_{[w\kappa]}$), then $\alpha+\beta\in\Phi_{[w]}$ (resp.\ $-\Phi_{[w\kappa]}$) due to Lemma \ref{lem-8.3.3}-(1). 
   Now, let us investigate the case where $\alpha\in\Phi_{[w]}$ and $\beta\in-\Phi_{[w\kappa]}$. 
   Then one has $\zeta([w])^{-1}(\alpha+\beta)\in\triangle^-$ in case of $\alpha+\beta\in\triangle^+$, and $\zeta([w\kappa])^{-1}(\alpha+\beta)\in\triangle^+$ in case of $\alpha+\beta\in\triangle^-$.
   Accordingly $\alpha+\beta\in\Phi_{[w]}$ in case of $\alpha+\beta\in\triangle^+$, and $\alpha+\beta\in-\Phi_{[w\kappa]}$ in case of $\alpha+\beta\in\triangle^-$. 
   In any cases we obtain $\alpha+\beta\in\bigl(\Phi_{[w]}\cup(-\Phi_{[w\kappa]})\bigr)\subset\triangle_{[w]}$. 
   Consequently $\triangle_{[w]}$ is a closed subset of $\triangle$.\par
   
   (2) follows from $\triangle=\triangle^+\amalg(-\triangle^+)$, Lemma \ref{lem-8.3.3}-(2) and $\triangle_{[w]}=\Phi_{[w]}\amalg(-\Phi_{[w\kappa]})$.\par
   
   (3). 
   On the one hand; by a direct computation with \eqref{eq-8.3.1} we deduce $\zeta([w])^{-1}(\triangle_{[w]})\subset\triangle^-$. 
   On the other hand; the above (2) implies that the cardinal number of $\triangle^-$ is equal to that of $\triangle_{[w]}$, which is equal to that of $\zeta([w])^{-1}(\triangle_{[w]})$.  
   Therefore one concludes $\zeta([w])^{-1}(\triangle_{[w]})=\triangle^-$. 
\end{proof}

   Lemmas \ref{lem-8.3.3} and \ref{lem-8.3.4} lead to 
\begin{corollary}\label{cor-8.3.5}
   $\Phi_{[w_1]}=\Phi_{[w_2]}$ with $[w_1],[w_2]\in\mathcal{W}$ implies $[w_1]=[w_2]$.
\end{corollary}   
\begin{proof}
   Suppose that $\Phi_{[w_1]}=\Phi_{[w_2]}$ for $[w_1],[w_2]\in\mathcal{W}$. 
   From $\Phi_{[w_1]}=\Phi_{[w_2]}$ and Lemma \ref{lem-8.3.3}-(2) we see that $\Phi_{[w_1\kappa]}=\Phi_{[w_2\kappa]}$. 
   Hence it turns out that $\triangle_{[w_1]}=\bigl(\Phi_{[w_1]}\amalg(-\Phi_{[w_1\kappa]})\bigr)=\triangle_{[w_2]}$, so that 
\[
    \zeta([w_2^{-1}w_1])(\triangle^-)
   =\zeta([w_2])^{-1}\bigl(\zeta([w_1])(\triangle^-)\bigr)
   =\zeta([w_2])^{-1}(\triangle_{[w_1]})
   =\zeta([w_2])^{-1}(\triangle_{[w_2]})
   =\triangle^-
\]
by Lemma \ref{lem-8.3.4}-(3).
   This and \eqref{eq-8.1.3} assure that $\operatorname{Ad}(w_2^{-1}w_1)=\operatorname{id}$ on $\frak{h}_\mathbb{C}$, and hence $[w_1]=[w_2]$. 
\end{proof}

   The last lemma is 
\begin{lemma}\label{lem-8.3.6} 
   Set $\Psi_{[w]}:=\bigl(\zeta([w])(\triangle^-)\bigr)\cap\blacktriangle^+$, $\Psi_{[w]}^c:=\blacktriangle^+-\Psi_{[w]}$, and $\blacktriangle_{[w]}:=\Psi_{[w]}\amalg(-\Psi_{[w]}^c)$ for $[w]\in\mathcal{W}$.
   Then 
\begin{enumerate}
\item[{\rm (1)}] 
   $\blacktriangle^+=\Psi_{[w]}\amalg\Psi_{[w]}^c$ for all $[w]\in\mathcal{W}$.
\item[{\rm (2)}] 
   $\Psi_{[w]}^c=\bigl(\zeta([w])(\triangle^+)\bigr)\cap\blacktriangle^+$ for all $[w]\in\mathcal{W}$.
\item[{\rm (3)}] 
   Both $\Psi_{[w]}$ and $\Psi_{[w]}^c$ are closed subsets of $\blacktriangle$ for each $[w]\in\mathcal{W}$.
\item[{\rm (4)}]
   $\blacktriangle_{[w]}$ is a closed subset of $\blacktriangle$ for any $[w]\in\mathcal{W}$.    
\item[{\rm (5)}]
   $\blacktriangle=\blacktriangle_{[w]}\amalg(-\blacktriangle_{[w]})$ for all $[w]\in\mathcal{W}$.
\item[{\rm (6)}]
   For a given $[w]\in\mathcal{W}$, there exists a unique $[\tau]\in\mathcal{W}_1$ such that $\zeta([\tau])^{-1}(\blacktriangle_{[w]})=\blacktriangle^-$.
\end{enumerate}   
\end{lemma}   
\begin{proof}
   (1) is obvious.\par

   (2). 
   Since $\Psi_{[w]}^c=\blacktriangle^+-\Psi_{[w]}$ and $\Psi_{[w]}=\bigl(\zeta([w])(\triangle^-)\bigr)\cap\blacktriangle^+$, the following six conditions (i) through (vi) are equivalent for $\gamma\in\triangle$:
\begin{center}
\begin{tabular}{lll}
     (i) $\gamma\in\Psi_{[w]}^c$,
   & (ii) $\gamma\in\blacktriangle^+$ and $\gamma\not\in\Psi_{[w]}$,
   & (iii) $\gamma\in\blacktriangle^+$ and $\gamma\not\in\zeta([w])(\triangle^-)$,\\
     (iv) $\gamma\in\blacktriangle^+$ and $\zeta([w])^{-1}\gamma\not\in\triangle^-$,
   & (v) $\gamma\in\blacktriangle^+$ and $\zeta([w])^{-1}\gamma\in\triangle^+$, 
   & (vi) $\gamma\in\bigl(\zeta([w])(\triangle^+)\bigr)\cap\blacktriangle^+$.
\end{tabular}
\end{center} 
   Hence $\Psi_{[w]}^c=\bigl(\zeta([w])(\triangle^+)\bigr)\cap\blacktriangle^+$ follows.\par

   (3). 
   By $\Psi_{[w]}=\bigl(\zeta([w])(\triangle^-)\bigr)\cap\blacktriangle^+$ and $\Psi_{[w]}^c=\bigl(\zeta([w])(\triangle^+)\bigr)\cap\blacktriangle^+$  we conclude that $\Psi_{[w]}$ and $\Psi_{[w]}^c$ are closed subsets of $\blacktriangle$, respectively.\par
   
   (4).
   We conclude (4) by arguments similar to those in the proof of Lemma \ref{lem-8.3.4}-(1) together with the above (3), (2).\par
   
   (5) follows from $\blacktriangle=\blacktriangle^+\amalg(-\blacktriangle^+)$, the above (1) and $\blacktriangle_{[w]}=\Psi_{[w]}\amalg(-\Psi_{[w]}^c)$.\par
   
   (6) is a consequence of the above (4), (5).
\end{proof}

   Now, let us start with proving the proposition. 
\begin{proof}[Proof of Proposition {\rm \ref{prop-8.3.2}}]
   By virtue of Lemma \ref{lem-8.3.3} it suffices to confirm the items (iv) and (v) only.\par
   
   (iv). 
   First, let us verify the uniqueness of $([\tau],[\sigma])\in\mathcal{W}_1\times\mathcal{W}^1$.\par

   (Uniqueness). 
   For $[\sigma_1],[\sigma_2]\in\mathcal{W}^1$ we suppose that $[\tau_1]:=[\sigma_1\sigma_2^{-1}]$ belongs to $\mathcal{W}_1$.
   We are going to prove $[\sigma_1]=[\sigma_2]$.
   In terms of $[\tau_1]\in\mathcal{W}_1$, $[\frak{l}_\mathbb{C},\frak{u}^-]\subset\frak{u}^-$ and $\frak{u}^-=\bigoplus_{\alpha\in\triangle^--\blacktriangle}\frak{g}_\alpha$ one has 
\begin{equation}\label{eq-1}\tag*{\textcircled{1}}
   \zeta([\tau_1])\bigl(\triangle^--\blacktriangle)\subset\triangle^--\blacktriangle. 
\end{equation} 
   Let us show that 
\begin{equation}\label{eq-2}\tag*{\textcircled{2}}
   \Phi_{[\sigma_2^{-1}]}\subset\Phi_{[\sigma_1^{-1}]}. 
\end{equation} 
   For any $\beta\in\Phi_{[\sigma_2^{-1}]}$ we obtain $\beta\in\triangle^+$, $\zeta([\sigma_2])\beta\in\triangle^-$ from \eqref{eq-8.3.1}. 
   Hence it turns out that 
\[
   \zeta([\sigma_1])\beta=\zeta([\tau_1\sigma_2])\beta\in\zeta([\tau_1])(\triangle^-).
\]
   So, either the case $\zeta([\sigma_1])\beta\in\zeta([\tau_1])(\blacktriangle^-)$ or $\zeta([\sigma_1])\beta\in\zeta([\tau_1])(\triangle^--\blacktriangle)$ has to occur.
   If $\zeta([\sigma_1])\beta\in\zeta([\tau_1])(\blacktriangle^-)$, then we conclude  $\beta\in\zeta([\sigma_1^{-1}\tau_1])(\blacktriangle^-)=\zeta([\sigma_2])^{-1}(\blacktriangle^-)\subset\triangle^-$ by Lemma \ref{lem-8.3.3}-(3) and $[\sigma_2]\in\mathcal{W}^1$.
   However, this $\beta\in\triangle^-$ contradicts $\beta\in\triangle^+$. 
   For this reason, the remaining case $\zeta([\sigma_1])\beta\in\zeta([\tau_1])(\triangle^--\blacktriangle)$ occurs. 
   Accordingly \ref{eq-1} implies $\zeta([\sigma_1])\beta\in\triangle^-$, and moreover \eqref{eq-8.3.1} yields $\beta\in\Phi_{[\sigma_1^{-1}]}$. 
   Hence \ref{eq-2} $\Phi_{[\sigma_2^{-1}]}\subset\Phi_{[\sigma_1^{-1}]}$ holds.   
   We have deduced $\Phi_{[\sigma_2^{-1}]}\subset\Phi_{[\sigma_1^{-1}]}$ from $[\sigma_1\sigma_2^{-1}]\in\mathcal{W}_1$. 
   Thus one can show $\Phi_{[\sigma_1^{-1}]}\subset\Phi_{[\sigma_2^{-1}]}$ from $[\sigma_2\sigma_1^{-1}]=[\sigma_1\sigma_2^{-1}]^{-1}\in\mathcal{W}_1$. 
   Consequently $\Phi_{[\sigma_1^{-1}]}=\Phi_{[\sigma_2^{-1}]}$. 
   This and Corollary \ref{cor-8.3.5} allow us to have $[\sigma_1]=[\sigma_2]$.\par
   
   Next, we are going to confirm the existence of $([\tau],[\sigma])\in\mathcal{W}_1\times\mathcal{W}^1$.\par
   (Existence). 
   Take an arbitrary $[w]\in\mathcal{W}$. 
   By Lemma \ref{lem-8.3.6}-(6) there exists a unique $[\tau]\in\mathcal{W}_1$ such that 
\[
   \zeta([\tau])^{-1}(\blacktriangle_{[w]})=\blacktriangle^-.
\] 
   Then, it is enough to confirm that $[\sigma]:=[\tau^{-1}w]$ belongs to $\mathcal{W}^1$. 
   In order to show $[\sigma]\in\mathcal{W}^1$, we first prove 
\begin{equation}\label{eq-3}\tag*{\textcircled{3}}
   \bigl(\zeta([\sigma])(\triangle^-)\bigr)\cap\blacktriangle^+=\emptyset.
\end{equation} 
   Let us use proof by contradiction. 
   Suppose that there exists a $\gamma\in\bigl(\zeta([\sigma])(\triangle^-)\bigr)\cap\blacktriangle^+$. 
   Then $\gamma\in\blacktriangle^+$ and $\zeta([\tau])\gamma\in\zeta([w])(\triangle^-)$. 
   From $\gamma\in\blacktriangle^+$ and $[\tau]\in\mathcal{W}_1$ we deduce $\zeta([\tau])\gamma\in\blacktriangle$. 
   So, $\zeta([\tau])\gamma\in\blacktriangle^+$ or $\zeta([\tau])\gamma\in\blacktriangle^-$.
\begin{enumerate}
\item 
   If $\zeta([\tau])\gamma\in\blacktriangle^+$, then $\zeta([\tau])\gamma\in\bigl(\zeta([w])(\triangle^-)\bigr)\cap\blacktriangle^+=\Psi_{[w]}\subset\blacktriangle_{[w]}=\zeta([\tau])(\blacktriangle^-)$.
\item 
   If $\zeta([\tau])\gamma\in\blacktriangle^-$, then Lemma \ref{lem-8.3.6}-(2) tells us that $\zeta([\tau])\gamma\in\bigl(\zeta([w])(\triangle^-)\bigr)\cap\blacktriangle^-=-\Psi_{[w]}^c\subset\blacktriangle_{[w]}=\zeta([\tau])(\blacktriangle^-)$.
\end{enumerate} 
   These contradict $\gamma\in\blacktriangle^+$. 
   Hence \ref{eq-3} holds.
   Now, let us show $[\sigma]\in\mathcal{W}^1$.
   We need to demonstrate that $\Phi_{[\sigma]}\subset\triangle^+-\blacktriangle$.
   For any $\alpha\in\Phi_{[\sigma]}$, it follows from \eqref{eq-8.3.1} that $\alpha\in\triangle^+$ and $\alpha\in\zeta([\sigma])(\triangle^-)$. 
   This and \ref{eq-3} give $\alpha\in\triangle^+-\blacktriangle$, and $\Phi_{[\sigma]}\subset\triangle^+-\blacktriangle$.
   For this reason we assert $[\sigma]\in\mathcal{W}^1$.\par

   From now on, we are going to prove (v).\par

   (v.1). 
   In view of \eqref{eq-8.3.1} we see that $n_{[\sigma]}=0$ if and only if $\zeta([\sigma])(\triangle^+)=\triangle^+$ if and only if $\operatorname{Ad}\sigma=\operatorname{id}$ on $\frak{h}_\mathbb{C}$  if and only if $[\sigma]=[e]$.
   Hence one has (v.1).\par
   
   (v.2). 
   Suppose that a $\beta\in\Pi_\triangle-\blacktriangle$ satisfies $[w_\beta]=[\sigma]$.
   Since $\beta\in\Pi_\triangle$ one knows that $\{\beta\}=\{\alpha\in\triangle^+ \,|\, \zeta([w_\beta])\alpha\in\triangle^-\}=\Phi_{[w_\beta^{-1}]}=\Phi_{[w_\beta]}=\Phi_{[\sigma]}$.
   Therefore $n_{[\sigma]}=1$.\par

   Conversely, suppose that $n_{[\sigma]}=1$.  
   By Lemma \ref{lem-8.3.3}-(2), $\Pi_\triangle\subset\triangle^+=\Phi_{[\sigma]}\amalg\Phi_{[\sigma\kappa]}$. 
   If $\Pi_\triangle\subset\Phi_{[\sigma\kappa]}$, then we conclude that $\Phi_{[\sigma\kappa]}=\triangle^+$ (because $\Phi_{[\sigma\kappa]}$ is a closed subset of $\triangle$), which contradicts $\Phi_{[\sigma]}\neq\emptyset$. 
   Therefore there exists a $\gamma\in\Pi_\triangle$ such that $\gamma\not\in\Phi_{[\sigma\kappa]}$. 
   Then we deduce $\gamma\in\Phi_{[\sigma]}$, and hence the supposition assures
\[
   \Phi_{[\sigma]}=\{\gamma\}.
\] 
   Here $\gamma\in\Pi_\triangle-\blacktriangle$ follows from $\Phi_{[\sigma]}\subset\triangle^+-\blacktriangle$. 
   Since $\gamma\in\Pi_\triangle$ one knows that $\Phi_{[w_\gamma]}=\{\gamma\}=\Phi_{[\sigma]}$. 
   This and Corollary \ref{cor-8.3.5} provide $[w_\gamma]=[\sigma]$. 
   Consequently (v.2) holds.
\end{proof}

\subsection{The generalized Bruhat decomposition by Kostant}\label{subsec-8.3.2}
   Proposition \ref{prop-8.3.2} enables us to establish the following theorem which is a result of Kostant \cite[p.123, Proposition 6.1]{Ko2} with some slight modifications:
\begin{theorem}\label{thm-8.3.7}
   Let $r=\dim_\mathbb{C}\frak{u}^+$. 
\begin{enumerate}
\item[{\rm (1)}]
   For each $[\sigma]\in\mathcal{W}^1$ we set
\begin{equation}\label{eq-8.3.8}
\begin{array}{lll}
     \Gamma_{[\sigma]}:=\{\gamma\in\Phi_{[\sigma^{-1}\kappa]} \,|\, \zeta([\sigma])\gamma\in\triangle^+-\blacktriangle\}, 
   & \frak{u}^+_{[\sigma]}:=\bigoplus_{\gamma\in\Gamma_{[\sigma]}}\frak{g}_{\zeta([\sigma])\gamma},  
   & U^+_{[\sigma]}:=\exp\frak{u}^+_{[\sigma]}. 
\end{array}
\end{equation}
   Then, $U^+_{[\sigma]}$ is a simply connected closed complex nilpotent subgroup of $U^+$ and it is biholomorphic to the $(r-n_{[\sigma]})$-dimensional complex Euclidean space $\frak{u}^+_{[\sigma]}$ $(\subset\frak{u}^+)$ via the exponential mapping $\exp:\frak{u}^+_{[\sigma]}\to U^+_{[\sigma]}$.
   Furthermore, 
\[
   N^+\sigma^{-1}Q^-=\sigma^{-1}U^+_{[\sigma]}Q^-.
\]
\item[{\rm (2)}]
   For a given $[\sigma]\in\mathcal{W}^1$, the following items {\rm (2.i)} and {\rm (2.ii)} hold$:$
   \begin{enumerate}
   \item[{\rm (2.i)}]
   $\dim_\mathbb{C}U^+_{[\sigma]}=r=\dim_\mathbb{C}U^+$ if and only if $[e]=[\sigma]$. 
   \item[{\rm (2.ii)}]
   $\dim_\mathbb{C}U^+_{[\sigma]}=r-1$ if and only if there exists a $\beta\in\Pi_\triangle-\blacktriangle$ satisfying $[w_\beta]=[\sigma]$.
   \end{enumerate}

\item[{\rm (3)}]
   $G_\mathbb{C}=\coprod_{[\sigma]\in\mathcal{W}^1}N^+\sigma^{-1}Q^-=\coprod_{[\sigma]\in\mathcal{W}^1}\sigma^{-1}U^+_{[\sigma]}Q^-$.
\end{enumerate}   
\end{theorem}
\begin{proof}
   (1). 
   Since both $\Phi_{[\sigma^{-1}\kappa]}$ and $\triangle^+-\blacktriangle$ are closed subsets of $\triangle$, it follows from \eqref{eq-8.3.8} that $\Gamma_{[\sigma]}$ is a closed subset of $\triangle$. 
   Therefore we see that $\frak{u}^+_{[\sigma]}=\bigoplus_{\gamma\in\Gamma_{[\sigma]}}\frak{g}_{\zeta([\sigma])\gamma}$ is a complex subalgebra of the nilpotent Lie algebra $\frak{u}^+=\bigoplus_{\alpha\in\triangle^+-\blacktriangle}\frak{g}_\alpha$.
   Consequently $U^+_{[\sigma]}=\exp\frak{u}^+_{[\sigma]}$ is a simply connected closed complex nilpotent subgroup of $U^+=\exp\frak{u}^+$ and is biholomorphic to $\frak{u}^+_{[\sigma]}$ via $\exp$, and $\dim_\mathbb{C}\frak{u}^+_{[\sigma]}$ accords with the cardinal number $|\Gamma_{[\sigma]}|$.
   Moreover, 
\[
\begin{split}
   \zeta([\sigma])(\Gamma_{[\sigma]})
  &\stackrel{\eqref{eq-8.3.8}}{=}\{\zeta([\sigma])\gamma\in\triangle^+-\blacktriangle \,|\, \gamma\in\Phi_{[\sigma^{-1}\kappa]}\} 
   \stackrel{\eqref{eq-8.3.1}}{=}\{\zeta([\sigma])\gamma\in\triangle^+-\blacktriangle \,|\, \gamma\in\triangle^+,\, \zeta([\sigma^{-1}\kappa])^{-1}\gamma\in\triangle^-\}\\
  &=\{\zeta([\sigma])\gamma\in\triangle^+-\blacktriangle \,|\, \gamma\in\triangle^+\} \quad\mbox{($\because$ $\zeta([\kappa])(\triangle^-)=\triangle^+$, $\triangle^+-\blacktriangle\subset\triangle^+$)}\\
  &=\{\zeta([\sigma])\gamma\in\triangle^+-\blacktriangle \,|\, \zeta([\sigma])^{-1}\bigl(\zeta([\sigma])\gamma\bigr)\in\triangle^+\}
   \stackrel{\eqref{eq-8.3.1}}{=}(\triangle^+-\blacktriangle)-\Phi_{[\sigma]}.
\end{split} 
\] 
   This implies that the number $|\Gamma_{[\sigma]}|$ is equal to $r-n_{[\sigma]}$ because of $|\triangle^+-\blacktriangle|=\dim_\mathbb{C}\frak{u}^+=r$ and $\Phi_{[\sigma]}\subset\triangle^+-\blacktriangle$. 
   Hence one has $\dim_\mathbb{C}\frak{u}^+_{[\sigma]}=r-n_{[\sigma]}$. 
   Now, the rest of proof is to confirm that $N^+\sigma^{-1}Q^-=\sigma^{-1}U^+_{[\sigma]}Q^-$. 
   Proposition \ref{prop-8.3.2}-(i), (ii) implies that $N^+\stackrel{\eqref{eq-8.1.6}}{=}\exp\bigl(\bigoplus_{\alpha\in\triangle^+}\frak{g}_\alpha\bigr)=\exp\bigl(\bigoplus_{\gamma\in\Phi_{[\sigma^{-1}\kappa]}}\frak{g}_\gamma\bigr)\exp\bigl(\bigoplus_{\beta\in\Phi_{[\sigma^{-1}]}}\frak{g}_\beta\bigr)$, so that 
\begin{equation}\label{eq-8.3.9}
\begin{split}
   N^+\sigma^{-1}Q^-
   &=\mbox{$\sigma^{-1}\exp\bigl(\bigoplus_{\gamma\in\Phi_{[\sigma^{-1}\kappa]}}\frak{g}_{\zeta([\sigma])\gamma}\bigr)\exp\bigl(\bigoplus_{\beta\in\Phi_{[\sigma^{-1}]}}\frak{g}_{\zeta([\sigma])\beta}\bigr)Q^-$}\\
  &=\mbox{$\sigma^{-1}\exp\bigl(\bigoplus_{\gamma\in\Phi_{[\sigma^{-1}\kappa]}}\frak{g}_{\zeta([\sigma])\gamma}\bigr)Q^-$}\\
  &=\mbox{$\sigma^{-1}\exp\bigl(\bigoplus_{\gamma_1\in\Gamma_{[\sigma]}}\frak{g}_{\zeta([\sigma])\gamma_1}\bigr)\exp\bigl(\bigoplus_{\gamma_2\in\{\gamma\in\Phi_{[\sigma^{-1}\kappa]} \,|\, \zeta([\sigma])\gamma\in\blacktriangle^+\} }\frak{g}_{\zeta([\sigma])\gamma_2}\bigr)Q^-$}\\
  &=\mbox{$\sigma^{-1}\exp\bigl(\bigoplus_{\gamma_1\in\Gamma_{[\sigma]}}\frak{g}_{\zeta([\sigma])\gamma_1}\bigr)Q^-$}
   \stackrel{\eqref{eq-8.3.8}}{=}\sigma^{-1}U^+_{[\sigma]}Q^-,
\end{split}
\end{equation}
where we note that $\bigoplus_{\beta\in\Phi_{[\sigma^{-1}]}}\frak{g}_{\zeta([\sigma])\beta}\subset\frak{n}^-\subset\frak{q}^-$ and $\Phi_{[\sigma^{-1}\kappa]}=\Gamma_{[\sigma]}\amalg\{\gamma\in\Phi_{[\sigma^{-1}\kappa]} \,|\, \zeta([\sigma])\gamma\in\blacktriangle^+\}$.\par

   (2) is immediate from (1) and Proposition \ref{prop-8.3.2}-(v).\par

   (3). 
   By (1), it is enough to prove $G_\mathbb{C}=\coprod_{[\sigma]\in\mathcal{W}^1}N^+\sigma^{-1}Q^-$. 
   In terms of $B^+=N_{G_\mathbb{C}}(\frak{b}^+)$ we fix a Bruhat decomposition $G_\mathbb{C}=\coprod_{[w]\in\mathcal{W}}N^+w^{-1}B^+$. 
   Then, $\zeta([\kappa])(\triangle^-)=\triangle^+$ yields $G_\mathbb{C}=\kappa^{-1}G_\mathbb{C}=\coprod_{[w]\in\mathcal{W}}N^-(w\kappa)^{-1}B^+=\coprod_{[w]\in\mathcal{W}}N^-w^{-1}B^+$, namely
\begin{equation}\label{eq-8.3.10}
   G_\mathbb{C}=\coprod_{[w]\in\mathcal{W}}N^-w^{-1}B^+.
\end{equation}
   In a similar way, one can obtain 
\[
   L_\mathbb{C}=\coprod_{[\tau]\in\mathcal{W}_1}N^-_1\tau^{-1}B^+_1 
\]
from \eqref{eq-8.1.10} and $B^+_1:=N_{L_\mathbb{C}}(\frak{h}_\mathbb{C}\oplus\frak{n}^+_1)$. 
   This, together with $Q^+=L_\mathbb{C}U^+$ and $B^+=B^+_1U^+$, assures that for any $[\sigma]\in\mathcal{W}^1$,
\begin{equation}\label{eq-8.3.11}
   N^-\sigma^{-1}Q^+
   =N^-\sigma^{-1}L_\mathbb{C}U^+
   =\bigcup_{[\tau]\in\mathcal{W}_1}N^-\sigma^{-1}(N^-_1\tau^{-1}B^+_1)U^+
   =\bigcup_{[\tau]\in\mathcal{W}_1}N^-\sigma^{-1}N^-_1\tau^{-1}B^+
   =\coprod_{[\tau]\in\mathcal{W}_1}N^-(\tau\sigma)^{-1}B^+,
\end{equation}
where $\sigma^{-1}N^-_1\subset N^-\sigma^{-1}$ follows from $[\sigma]\in\mathcal{W}^1$ and Proposition \ref{prop-8.3.2}-(iii). 
   Consequently, \eqref{eq-8.3.10} and Proposition \ref{prop-8.3.2}-(iv) allow us to assert that
\[
   G_\mathbb{C}=\coprod_{[\sigma]\in\mathcal{W}^1}N^-\sigma^{-1}Q^+.
\]
   Thus $G_\mathbb{C}=\coprod_{[\sigma]\in\mathcal{W}^1}N^+\sigma^{-1}Q^-$ because of $\overline{\theta}(G_\mathbb{C})=G_\mathbb{C}$, $\overline{\theta}(N^-)=N^+$, $\overline{\theta}(\sigma)=\sigma$ and $\overline{\theta}(Q^+)=Q^-$.
\end{proof}

\begin{remark}\label{rem-8.3.12}
\begin{enumerate}
\item[]
\item[(1)]
   In the proof of Theorem \ref{thm-8.3.7}-(3) we gave Bruhat decompositions $G_\mathbb{C}=\coprod_{[w]\in\mathcal{W}}N^+w^{-1}B^+$ and $G_\mathbb{C}=\coprod_{[w]\in\mathcal{W}}N^-w^{-1}B^+$, and generalized Bruhat decompositions $G_\mathbb{C}=\coprod_{[\sigma]\in\mathcal{W}^1}N^-\sigma^{-1}Q^+$ and $G_\mathbb{C}=\coprod_{[\sigma]\in\mathcal{W}^1}N^+\sigma^{-1}Q^-$.
\item[(2)]
   Theorem \ref{thm-8.3.7}-(1) and Proposition \ref{prop-8.2.1}-(iii), (iv) imply that $N^+\sigma^{-1}Q^-=\sigma^{-1}U^+_{[\sigma]}Q^-$ is a connected regular complex submanifold of $G_\mathbb{C}$ whose dimension is $\dim_\mathbb{C}U^+-n_{[\sigma]}+\dim_\mathbb{C}Q^-$ ($=\dim_\mathbb{C}G_\mathbb{C}-n_{[\sigma]}$) for each $[\sigma]\in\mathcal{W}^1$. 
\end{enumerate}   
\end{remark}

    Taking the proof of Theorem \ref{thm-8.3.7} into consideration, we prove 
\begin{lemma}\label{lem-8.3.13}
   For every $[\sigma]\in\mathcal{W}^1$, the following three items hold$:$ 
\begin{enumerate}
\item[{\rm (1)}]
   $N^+\sigma^{-1}Q^-=\coprod_{[\tau]\in\mathcal{W}_1}N^+(\tau\sigma)^{-1}B^-$, where $B^-=N_{G_\mathbb{C}}(\frak{b}^-)$.
\item[{\rm (2)}] 
   $\zeta([\sigma])^{-1}(\blacktriangle^+)=\{\gamma\in\Phi_{[\sigma^{-1}\kappa]} \,|\, \zeta([\sigma])\gamma\in\blacktriangle^+\}$.
\item[{\rm (3)}]
   $\dim_\mathbb{C}N^+\sigma^{-1}Q^-\geq\dim_\mathbb{C}N^+(\tau\sigma)^{-1}B^-$ for all $[\tau]\in\mathcal{W}_1$, and $\dim_\mathbb{C}N^+\sigma^{-1}Q^-=\dim_\mathbb{C}N^+\sigma^{-1}B^-$.
\end{enumerate}
\end{lemma}
\begin{proof}
   (1) comes from \eqref{eq-8.3.11}, $\overline{\theta}(N^-)=N^+$, $\overline{\theta}(Q^+)=Q^-$, $\overline{\theta}(B^+)=B^-$ and $\overline{\theta}(w)=w$ ($w\in N_{G_u}(i\frak{h}_\mathbb{R})$).\par
   
   (2). 
   On the one hand; by Proposition \ref{prop-8.3.2}-(iii) and $[\sigma]\in\mathcal{W}^1$ we have $\zeta([\sigma])^{-1}(\blacktriangle^+)\subset\triangle^+$. 
   Besides, a direct computation yields $\zeta([\sigma^{-1}\kappa])^{-1}\bigl(\zeta([\sigma])^{-1}(\blacktriangle^+)\bigr)\subset\zeta([\kappa])^{-1}(\blacktriangle^+)\subset\triangle^-$. 
   Hence we obtain $\zeta([\sigma])^{-1}(\blacktriangle^+)\subset\Phi_{[\sigma^{-1}\kappa]}$ from \eqref{eq-8.3.1}.
   Therefore 
\[
   \zeta([\sigma])^{-1}(\blacktriangle^+)\subset\{\gamma\in\Phi_{[\sigma^{-1}\kappa]} \,|\, \zeta([\sigma])\gamma\in\blacktriangle^+\}.
\] 
   On the other hand; since $\zeta([\sigma]):\triangle\to\triangle$ is bijective, the cardinal number $\big|\{\gamma\in\Phi_{[\sigma^{-1}\kappa]} \,|\, \zeta([\sigma])\gamma\in\blacktriangle^+\}\big|$ is less than or equal to $|\blacktriangle^+|$.
   Consequently (2) follows from $|\blacktriangle^+|=\big|\zeta([\sigma])^{-1}(\blacktriangle^+)\big|$.\par
   
   (3). 
   For any $[\tau]\in\mathcal{W}_1$, both $N^+(\tau\sigma)^{-1}B^-$ and $N^+\sigma^{-1}Q^-$ are regular submanifolds of $G_\mathbb{C}$, and moreover, $N^+(\tau\sigma)^{-1}B^-\subset N^+\sigma^{-1}Q^-$ due to (1).
   Hence we conclude that $\dim_\mathbb{C}N^+(\tau\sigma)^{-1}B^-\leq\dim_\mathbb{C}N^+\sigma^{-1}Q^-$ for all $[\tau]\in\mathcal{W}_1$.
   At this stage, the rest of proof is to deduce 
\[
   \dim_\mathbb{C}N^+\sigma^{-1}B^-=\dim_\mathbb{C}N^+\sigma^{-1}Q^-.
\]   
   In a similar way to \eqref{eq-8.3.9} one has 
\allowdisplaybreaks{
\begin{align*}
   N^+\sigma^{-1}B^-
   &=\mbox{$\sigma^{-1}\exp\bigl(\bigoplus_{\gamma\in\Phi_{[\sigma^{-1}\kappa]}}\frak{g}_{\zeta([\sigma])\gamma}\bigr)\exp\bigl(\bigoplus_{\beta\in\Phi_{[\sigma^{-1}]}}\frak{g}_{\zeta([\sigma])\beta}\bigr)B^-$}\\
  &=\mbox{$\sigma^{-1}\exp\bigl(\bigoplus_{\gamma\in\Phi_{[\sigma^{-1}\kappa]}}\frak{g}_{\zeta([\sigma])\gamma}\bigr)B^-$}\\
  &=\mbox{$\sigma^{-1}\exp\bigl(\bigoplus_{\gamma_1\in\Gamma_{[\sigma]}}\frak{g}_{\zeta([\sigma])\gamma_1}\bigr)\exp\bigl(\bigoplus_{\gamma_2\in\{\gamma\in\Phi_{[\sigma^{-1}\kappa]} \,|\, \zeta([\sigma])\gamma\in\blacktriangle^+\} }\frak{g}_{\zeta([\sigma])\gamma_2}\bigr)B^-$}\\
  &\stackrel{(2)}{=}\mbox{$\sigma^{-1}\exp\bigl(\bigoplus_{\gamma_1\in\Gamma_{[\sigma]}}\frak{g}_{\zeta([\sigma])\gamma_1}\bigr)\exp\bigl(\bigoplus_{\alpha\in\blacktriangle^+}\frak{g}_{\alpha}\bigr)B^-$}
  \stackrel{\eqref{eq-8.3.8}}{=}\sigma^{-1}U^+_{[\sigma]}N^+_1B^-,
\end{align*}}where $N^+_1=\exp\frak{n}^+_1$. 
   Therefore it follows from $(U^+_{[\sigma]}\cap N^+_1B^-)\subset(U^+\cap Q^-)=\{e\}$ and $(N^+_1\cap B^-)\subset(N^+\cap B^-)=\{e\}$ that 
\[
\begin{split}
   \dim_\mathbb{C}N^+\sigma^{-1}B^-
   =\dim_\mathbb{C}U^+_{[\sigma]}N^+_1B^-
  &=\dim_\mathbb{C}U^+_{[\sigma]}+(\dim_\mathbb{C}N^+_1+\dim_\mathbb{C}B^-)\\
  &=(\dim_\mathbb{C}U^+-n_{[\sigma]})+\dim_\mathbb{C}Q^-
   =\dim_\mathbb{C}N^+\sigma^{-1}Q^-.
\end{split}
\]   
   cf.\ Remark \ref{rem-8.3.12}-(2).
\end{proof} 

   Lemma \ref{lem-8.3.13} provides us with 
\begin{proposition}\label{prop-8.3.14}
   Let $[\sigma],[\theta]\in\mathcal{W}^1$. 
   If $N^+\theta^{-1}Q^-\subset\overline{N^+\sigma^{-1}Q^-}-N^+\sigma^{-1}Q^-$, then $\dim_\mathbb{C}N^+\theta^{-1}Q^-<\dim_\mathbb{C}N^+\sigma^{-1}Q^-$ and $n_{[\theta]}>n_{[\sigma]}$.
   Here we denote by $\overline{N^+\sigma^{-1}Q^-}$ the closure of $N^+\sigma^{-1}Q^-$ in $G_\mathbb{C}$.
\end{proposition} 
\begin{proof}
   By Lemma \ref{lem-8.3.13}-(1) and $[\sigma],[\theta]\in\mathcal{W}^1$ we have 
\begin{equation}\label{eq-1}\tag*{\textcircled{1}}
\begin{split}
   N^+\theta^{-1}B^-
  &\subset N^+\theta^{-1}Q^-
   \subset\overline{N^+\sigma^{-1}Q^-}-N^+\sigma^{-1}Q^-
   =\overline{\bigcup_{[\tau]\in\mathcal{W}_1}N^+(\tau\sigma)^{-1}B^-}-\bigcup_{[\tau_1]\in\mathcal{W}_1}N^+(\tau_1\sigma)^{-1}B^-\\
  &=\bigcup_{[\tau]\in\mathcal{W}_1}\overline{N^+(\tau\sigma)^{-1}B^-}-\bigcup_{[\tau_1]\in\mathcal{W}_1}N^+(\tau_1\sigma)^{-1}B^-
  \subset\bigcup_{[\tau]\in\mathcal{W}_1}\bigl(\overline{N^+(\tau\sigma)^{-1}B^-}-N^+(\tau\sigma)^{-1}B^-\bigr)
\end{split} 
\end{equation}
because $\mathcal{W}_1$ is a finite set.
   Here each $N^+w^{-1}B^-$ is a Bruhat cell in $G_\mathbb{C}$ ($[w]\in\mathcal{W}$), so one knows that\footnote{Remark.\ One can assert these statements without the supposition that the group $G_\mathbb{C}$ is algebraic or simply connected.} 
\begin{enumerate}
\item 
   for any $[w]\in\mathcal{W}$, $\overline{N^+w^{-1}B^-}-N^+w^{-1}B^-$ is a disjoint union of Bruhat cells of strictly lower dimension,
\item
   for $[w_1],[w_2]\in\mathcal{W}$, $(N^+w_1^{-1}B^-\cap N^+w_2^{-1}B^-)\neq\emptyset$ if and only if $N^+w_1^{-1}B^-=N^+w_2^{-1}B^-$  
\end{enumerate} 
(e.g.\ Theorem 27.4 in Bump \cite[p.252]{Bu}).
   These, together with \ref{eq-1}, enable us to see that 
\[
   \overline{N^+(\tau\sigma)^{-1}B^-}-N^+(\tau\sigma)^{-1}B^-
   \subset\coprod_{\mbox{\scriptsize{$[w_{[\tau]}]\in\mathcal{W}$ with $\dim_\mathbb{C}N^+w_{[\tau]}^{-1}B^-<\dim_\mathbb{C}N^+(\tau\sigma)^{-1}B^-$}}}N^+w_{[\tau]}^{-1}B^-
\]
for all $[\tau]\in\mathcal{W}_1$; besides, there exist $[\tau_2]\in\mathcal{W}_1$ and $[w_{[\tau_2]}]\in\mathcal{W}$ satisfying 
\[
\begin{array}{ll}
   N^+\theta^{-1}B^-=N^+w_{[\tau_2]}^{-1}B^-, & \dim_\mathbb{C}N^+w_{[\tau_2]}^{-1}B^-<\dim_\mathbb{C}N^+(\tau_2\sigma)^{-1}B^-.
\end{array}
\]
   Consequently Lemma \ref{lem-8.3.13}-(3) and $[\tau_2]\in\mathcal{W}_1$ yield 
\[
   \dim_\mathbb{C}N^+\theta^{-1}Q^-
   =\dim_\mathbb{C}N^+\theta^{-1}B^-
   <\dim_\mathbb{C}N^+(\tau_2\sigma)^{-1}B^-
   \leq\dim_\mathbb{C}N^+\sigma^{-1}Q^-.
\]
   From $\dim_\mathbb{C}N^+\theta^{-1}Q^-<\dim_\mathbb{C}N^+\sigma^{-1}Q^-$ we obtain $n_{[\theta]}>n_{[\sigma]}$. 
   cf.\ Remark \ref{rem-8.3.12}-(2).    
\end{proof}

   The direct product group $N^+\times Q^-$ acts on $G_\mathbb{C}$ by
\[
   (N^+\times Q^-)\times G_\mathbb{C}\ni((n,q),x)\mapsto nxq^{-1}\in G_\mathbb{C}.
\]
   Theorem \ref{thm-8.3.7}-(3) tells us that this orbit space coincides with $\{N^+\sigma^{-1}Q^- : [\sigma]\in\mathcal{W}^1\}$. 
   In addition, since the action is continuous and $\mathcal{W}^1$ is finite, one can show that for any $[\sigma]\in\mathcal{W}^1$ there exist finite elements $[\theta_1],[\theta_2],\dots,[\theta_k]\in\mathcal{W}^1$ such that 
\[
   \overline{N^+\sigma^{-1}Q^-}=N^+\sigma^{-1}Q^-\amalg N^+\theta_1^{-1}Q^-\amalg N^+\theta_2^{-1}Q^-\amalg\cdots\amalg N^+\theta_k^{-1}Q^-.
\] 
   Furthermore, Proposition \ref{prop-8.3.14} leads to 
\begin{corollary}\label{cor-8.3.15}
   For any $[\sigma]\in\mathcal{W}^1$ there exist finite elements $[\theta_1],[\theta_2],\dots,[\theta_k]\in\mathcal{W}^1$ such that $\overline{N^+\sigma^{-1}Q^-}-N^+\sigma^{-1}Q^-=N^+\theta_1^{-1}Q^-\amalg N^+\theta_2^{-1}Q^-\amalg\cdots\amalg N^+\theta_k^{-1}Q^-$ and $n_{[\theta_i]}>n_{[\sigma]}$ for all $1\leq i\leq k$.
\end{corollary}

   Theorem \ref{thm-8.3.7} leads to the following corollary which is an improvement of Proposition \ref{prop-8.2.1}-(iv):
\begin{corollary}\label{cor-8.3.16}
\begin{enumerate}
\item[]
\item[{\rm (i)}]
   The product mapping $U^+\times Q^-\ni(u,q)\mapsto uq\in G_\mathbb{C}$ is a biholomorphism of $U^+\times Q^-$ onto a dense, domain in $G_\mathbb{C}$.
\item[{\rm (ii)}]
   $N^+Q^-$ is a dense, domain in $G_\mathbb{C}$.
\end{enumerate}   
\end{corollary}
\begin{proof}
   (i).
   We only verify that the image $U^+Q^-$ is dense in $G_\mathbb{C}$. 
   For every $[\sigma]\in\mathcal{W}^1-\{[e]\}$, it follows from Theorem \ref{thm-8.3.7} that $\sigma^{-1}U^+_{[\sigma]}Q^-$ is a submanifold of $G_\mathbb{C}$ whose dimension is strictly lower than $\dim_\mathbb{C}G_\mathbb{C}$, so that it has measure $0$ in the manifold $G_\mathbb{C}$. 
   Hence, the finite union $\coprod_{[\sigma]\in\mathcal{W}^1-\{[e]\}}\sigma^{-1}U^+_{[\sigma]}Q^-$ has measure $0$ in $G_\mathbb{C}$ also, and therefore its complement $G_\mathbb{C}-\coprod_{[\sigma]\in\mathcal{W}^1-\{[e]\}}\sigma^{-1}U^+_{[\sigma]}Q^-=e^{-1}U^+_{[e]}Q^-=U^+Q^-$ is dense in $G_\mathbb{C}$.\par

   (ii).
   By Theorem \ref{thm-8.3.7}-(1) one has $N^+Q^-=U^+Q^-$. 
   Hence (ii) comes from (i).  
\end{proof}

\subsection{An analytic continuation related to the Bruhat decomposition}\label{subsec-8.3.3}
   Our aim in this subsection to demonstrate 
\begin{theorem}\label{thm-8.3.17}
   Let
\begin{equation}\label{eq-8.3.18}
   \mbox{$O:=\coprod_{\mbox{\scriptsize{$[\sigma]\in\mathcal{W}^1$ {\rm with} $n_{[\sigma]}\leq 1$}}}N^+\sigma^{-1}Q^-$}.
\end{equation}
   Then, it follows that 
\begin{enumerate}
\item[{\rm (i)}]
   $O$ is a dense, domain in $G_\mathbb{C}$,
\item[{\rm (ii)}]
   any holomorphic function $f$ on $O$ can be continued analytically to the whole $G_\mathbb{C}$.    
\end{enumerate}   
\end{theorem} 

   For the aim we first prepare some lemmas and a proposition. 
   We will conclude Theorem \ref{thm-8.3.17} by Hartogs's continuation theorem and $G_\mathbb{C}-O$ being of complex codimension 2 or more. 
\begin{lemma}\label{lem-8.3.19}
   Let $M$ be a topological manifold, let $A$ be a subset of $M$, and let $D$ be a dense domain in $M$. 
   Suppose that the subset $A\cup D$ of $M$ is open. 
   Then, $A\cup D$ is a dense domain in $M$.
\end{lemma}
\begin{proof}
   We only confirm that $A\cup D$ is arcwise connected. 
   Fix a point $p_0\in D$, and take any $p\in A\cup D$. 
   By the supposition there exists an arcwise connected, open subset $U\subset M$ such that $p\in U\subset A\cup D$. 
   Since $D\subset M$ is dense, there exists a $d\in U\cap D$. 
   Then, $p_0$ and $d$ (resp.\ $d$ and $p$) can be joined by an arc in $D$ (resp.\ $U$), and therefore $p_0$ and $p$ can be joined by an arc in $A\cup D$. 
\end{proof}

\begin{lemma}[Hartogs's continuation theorem]\label{lem-8.3.20}
   Let $P$ be an open subset of $\mathbb{C}^N$ defined by $|z^1|<R, |z^2|<R,\dots,|z^N|<R$ for some $R>0$, and set
\[
   A:=\{(z^1,\dots,z^k,z^{k+1},\dots,z^N)\in P \,|\, z^1=z^2=\cdots=z^k=0\},
\] 
where $2\leq k\leq N$.
   Then, for an arbitrary holomorphic function $f:(P-A)\to\mathbb{C}$, there exists a unique holomorphic function $\tilde{f}:P\to\mathbb{C}$ such that $f=\tilde{f}$ on $P-A$.
\end{lemma}
\begin{proof}
   (Uniqueness). 
   The uniqueness of $\tilde{f}$ comes from $P$ being a domain, $P-A$ being a non-empty open subset of $P$ and the theorem of identity.\par
   
   (Existence).
   Let us confirm the existence of $\tilde{f}$.
   Take an arbitrary $0<r_1<R$, and fix a point $(z^1,z^2,\dots,z^N)\in P$ with $|z^1|<r_1$. 
   Then, for any $|z^1|<|w|<R$ it turns out that $(w,z^2,\dots,z^N)\in P -A$, and hence the definition
\[
   \mbox{$g(w):=\dfrac{f(w,z^2,\dots,z^N)}{w-z^1}$ for $w\in\{w\in\mathbb{C}:|z^1|<|w|<R\}$}
\] 
is well-defined.
   Furthermore, $g(w)$ is holomorphic on the annular domain $|z^1|<|w|<R$ which includes the circle $C_0:|w|=r_1$, and consequently  
\[
   \int_{C_0}g(w)dw
   =\int_{|w|=r_1}\dfrac{f(w,z^2,\dots,z^N)}{w-z^1}dw
\]
exists in $\mathbb{C}$ for every $(z^1,z^2,\dots,z^N)\in D_{r_1}\times D_R\times\cdots\times D_R$, where $D_r:=\{z\in\mathbb{C} : |z|<r\}$. 
   Now, let us prove that 
\begin{equation}\label{eq-1}\tag*{\textcircled{1}} 
   \mbox{the function $\displaystyle{\int_{|w|=r_1}\dfrac{f(w,z^2,\dots,z^N)}{w-z^1}dw}$ is holomorphic on $D_{r_1}\times D_R\times\cdots\times D_R$}.
\end{equation}
   Taking Hartogs's theorem of holomorphy into account, we will only conclude that the function is holomorphic with respect to each variable $z^j$ ($1\leq j\leq N$).
   Let us demonstrate that the function is holomorphic with respect to  $z^1$.
   For any $w\in\mathbb{C}$ with $|w|=r_1$, we see that 
\begin{equation}\label{eq-a}\tag{a}
   \mbox{$F_w(z^1,z^2,\dots,z^N):=f(w,z^2,\dots,z^N)/(w-z^1)$ is holomorphic on $D_{r_1}\times D_R\times\cdots\times D_R$}.
\end{equation}
   Hence for a given piecewise differentiable closed curve $C=\sum_{n=1}^mC_n$, $C_n:z^1=z^1_n(s)$ ($a_n\leq s\leq b_n$) of class $C^1$ which is contained in $D_{r_1}$, Cauchy's integral theorem enables us to deduce that for any $w\in\mathbb{C}$ with $|w|=r_1$ and any $(z^2,\dots,z^N)\in D_R\times\cdots\times D_R$, 
\begin{equation}\label{eq-b}\tag{b}
   \int_CF_w(z^1,z^2,\dots,z^N)dz^1=0
\end{equation}
because $D_{r_1}$ is a star region.
   Therefore it follows from $f(w,z^2,\dots,z^N)/(w-z^1)=F_w(z^1,z^2,\dots,z^N)$ and $C=\sum_{n=1}^mC_n$, $C_n:z^1=z^1_n(s)$ ($a_n\leq s\leq b_n$) that 
\begin{multline*}
  \int_C\Big(\int_{|w|=r_1}\dfrac{f(w,z^2,\dots,z^N)}{w-z^1}dw\Big)dz^1
   =\int_C\Big(\int_{|w|=r_1}F_w(z^1,z^2,\dots,z^N)dw\Big)dz^1\\
  =\sum_{n=1}^m\int_{a_n}^{b_n}\Big(\int_0^{2\pi}F_{r_1e^{it}}(z^1_n(s),z^2,\dots,z^N)\dfrac{dr_1e^{it}}{dt}dt\Big)\dfrac{dz^1_n(s)}{ds}ds
   =\int_{|w|=r_1}\Big(\int_CF_w(z^1,z^2,\dots,z^N)dz^1\Big)dw
   \stackrel{\eqref{eq-b}}{=}0.
\end{multline*}
   Here we applied Fubini's theorem to the continuous function $[0,2\pi]\times[a_n,b_n]\ni(t,s)\mapsto\dfrac{f(r_1e^{it},z^2,\dots,z^N)}{r_1e^{it}-z^1_n(s)}\dfrac{dr_1e^{it}}{dt}\dfrac{dz^1_n(s)}{ds}\in\mathbb{C}$. 
   The above and Morera's theorem allow us to assert that $\int_{|w|=r_1}\bigl(f(w,z^2,\dots,z^N)/(w-z^1)\bigr)dw$ is a holomorphic function with respect to the variable $z^1\in D_{r_1}$.
   In a similar way, one can assert that with respect to the other variables (because of \eqref{eq-a}). 
   Hence \ref{eq-1} holds.
   So, one can define a holomorphic function $\tilde{f}:D_{r_1}\times D_R\times\cdots\times D_R\to\mathbb{C}$ by
\begin{equation}\label{eq-2}\tag*{\textcircled{2}} 
   \mbox{$\displaystyle{\tilde{f}(z^1,z^2,\dots,z^N):=\dfrac{1}{2\pi i}\int_{|w|=r_1}\dfrac{f(w,z^2,\dots,z^N)}{w-z^1}dw}$ for $(z^1,z^2,\dots,z^N)\in D_{r_1}\times D_R\times\cdots\times D_R$}.
\end{equation}
   The function $f$ coincides with this $\tilde{f}$ on $(D_{r_1}\times D_R\times\cdots\times D_R)-A$.
   Indeed, since $D_R\times(D_R-\{0\})\times D_R\times\cdots\times D_R\subset P-A$, $f(z^1,z^2,z^3,\dots,z^N)$ is holomorphic on $D_R\times(D_R-\{0\})\times D_R\times\cdots\times D_R$. 
   From Cauchy's integral formula for $z^1\in D_R$ we obtain
\[
   \mbox{$\displaystyle{f(z^1,z^2,z^3,\dots,z^N)=\dfrac{1}{2\pi i}\int_{|w|=r_1}\dfrac{f(w,z^2,z^3,\dots,z^N)}{w-z^1}dw}$ on $D_{r_1}\times(D_R-\{0\})\times D_R\times\cdots\times D_R$}.  
\]   
   This and \ref{eq-2} assure that $f=\tilde{f}$ on $D_{r_1}\times(D_R-\{0\})\times D_R\times\cdots\times D_R$. 
   Therefore the theorem of identity enables us to show 
\[
    \mbox{$f=\tilde{f}$ on $(D_{r_1}\times D_R\times D_R\times\cdots\times D_R)-A$}
\] 
because $D_{r_1}\times(D_R-\{0\})\times D_R\times\cdots\times D_R$ is a non-empty open subset of the domain $(D_{r_1}\times D_R\times D_R\times\cdots\times D_R)-A$.
   Letting $r_1\nearrow R$ one can get the conclusion. 
   Here we remark that $P=D_R\times D_R\times\cdots\times D_R$.
\end{proof}

\begin{lemma}\label{lem-8.3.21}
   $O=\coprod_{\mbox{\scriptsize{$[\sigma]\in\mathcal{W}^1$ {\rm with} $n_{[\sigma]}\leq 1$}}}N^+\sigma^{-1}Q^-$ is a dense, domain in $G_\mathbb{C}$.
\end{lemma} 
\begin{proof}
   Corollary \ref{cor-8.3.16}-(ii) tells us that $N^+e^{-1}Q^-$ is a dense domain in $G_\mathbb{C}$.
   By that, Proposition \ref{prop-8.3.2}-(v.1) and Lemma \ref{lem-8.3.19}, it suffices to conclude that $O\subset G_\mathbb{C}$ is open, which is equivalent to that $G_\mathbb{C}-O$ is closed in $G_\mathbb{C}$. 
   Since $\mathcal{W}^1$ is a finite set, we deduce 
\[
   \mbox{$\overline{\bigcup_{\mbox{\scriptsize{$[\theta]\in\mathcal{W}^1$ {\rm with} $2\leq n_{[\theta]}$}}}N^+\theta^{-1}Q^-}
   =\bigcup_{\mbox{\scriptsize{$[\theta]\in\mathcal{W}^1$ {\rm with} $2\leq n_{[\theta]}$}}}\overline{N^+\theta^{-1}Q^-}
   =\coprod_{\mbox{\scriptsize{$[\vartheta]\in\mathcal{W}^1$ {\rm with} $2\leq n_{[\vartheta]}$}}}N^+\vartheta^{-1}Q^-$}
   =G_\mathbb{C}-O
\]
by Corollary \ref{cor-8.3.15} and Theorem \ref{thm-8.3.7}-(3).
   Accordingly $G_\mathbb{C}-O$ is closed in $G_\mathbb{C}$.
\end{proof}

\begin{lemma}\label{lem-8.3.22}
   The following two items hold for a given $[\sigma]\in\mathcal{W}^1:$
\begin{enumerate}
\item[{\rm (1)}]
   $\sigma^{-1}U^+Q^-$ is a dense, domain in $G_\mathbb{C}$.
\item[{\rm (2)}]
   $\sigma^{-1}U^+_{[\sigma]}Q^-$ is an analytic subset of $\sigma^{-1}U^+Q^-$ having complex codimension $n_{[\sigma]}$, that is to say, there exist holomorphic functions $f_1,f_2,\dots,f_{n_{[\sigma]}}:\sigma^{-1}U^+Q^-\to\mathbb{C}$ such that {\rm (2.i)} $df_1\wedge df_2\wedge\cdots\wedge df_{n_{[\sigma]}}\neq 0$ on $\sigma^{-1}U^+Q^-$ and {\rm (2.ii)} $\sigma^{-1}U^+_{[\sigma]}Q^-=\{x\in\sigma^{-1}U^+Q^- \,|\, f_1(x)=f_2(x)=\cdots=f_{n_{[\sigma]}}(x)=0\}$. 
\item[{\rm (3)}]
   $\sigma^{-1}N^+Q^-$ is a dense domain in $G_\mathbb{C}$, and $N^+\sigma^{-1}Q^-$ is an analytic subset of $\sigma^{-1}N^+Q^-$ having complex codimension $n_{[\sigma]}$.
\end{enumerate}
   Here we refer to \eqref{eq-8.3.8} for $U^+_{[\sigma]}$.
\end{lemma} 
\begin{proof}
   (1). 
   Since the left translation $L_{\sigma^{-1}}:G_\mathbb{C}\to G_\mathbb{C}$ is homeomorphic, we conclude (1) by Corollary \ref{cor-8.3.16}-(i).\par
   
   (2).
   By Theorem \ref{thm-8.3.7}-(1) one can choose a complex basis $\{E_j\}_{j=1}^r$ of $\frak{u}^+$ so that $\frak{u}^+_{[\sigma]}=\operatorname{span}_\mathbb{C}\{E_{n_{[\sigma]}+k}\}_{k=1}^{r-n_{[\sigma]}}$.
   Let us consider the canonical coordinates $z^1,z^2,\dots,z^r$ of the first kind associated with this basis $\{E_j\}_{j=1}^r\subset \frak{u}^+$. 
   Then, 
\begin{equation}\label{eq-1}\tag*{\textcircled{1}}
\begin{array}{ll}
   \mbox{$dz^1\wedge\cdots\wedge dz^{n_{[\sigma]}}\wedge dz^{n_{[\sigma]}+1}\wedge\cdots\wedge dz^r\neq 0$ on $U^+$}, & U^+_{[\sigma]}=\{u\in U^+ \,|\, z^1(u)=z^2(u)=\cdots=z^{n_{[\sigma]}}(u)=0\}.
\end{array}
\end{equation}
   For each $x\in\sigma^{-1}U^+Q^-$, Proposition \ref{prop-8.2.1}-(iv) assures that there exists a unique $(u,q)\in U^+\times Q^-$ satisfying $x=\sigma^{-1}uq$, and then one can get a holomorphic function $f_i:\sigma^{-1}U^+Q^-\to\mathbb{C}$ by setting $f_i(x):=z^i(u)$ for $1\leq i\leq n_{[\sigma]}$.
   These $f_1,f_2,\dots,f_{n_{[\sigma]}}$ are desired functions due to \ref{eq-1}.\par
   
   (3) is immediate from (1), (2), $\sigma^{-1}N^+Q^-=\sigma^{-1}U^+Q^-$ and $N^+\sigma^{-1}Q^-=\sigma^{-1}U^+_{[\sigma]}Q^-$.
\end{proof}

   For $[\sigma]\in\mathcal{W}^1$ we set  
\begin{equation}\label{eq-8.3.23}
   O_{[\sigma]}:=\sigma^{-1}N^+Q^--\bigcup_{\mbox{\scriptsize{$[\tau]\in\mathcal{W}^1$ {\rm with} $[\sigma]\neq[\tau]$ \& $n_{[\sigma]}\leq n_{[\tau]}$}}}\overline{N^+\tau^{-1}Q^-}
\end{equation} 
and demonstrate 

\begin{proposition}\label{prop-8.3.24}
   For any $[\sigma]\in\mathcal{W}^1$, it follows that 
\begin{enumerate}
\item[{\rm (i)}]
   $O_{[\sigma]}$ is an open subset of $G_\mathbb{C}$,
\item[{\rm (ii)}]
   $N^+\sigma^{-1}Q^-\subset O_{[\sigma]}\subset \sigma^{-1}N^+Q^-$,
\item[{\rm (iii)}] 
   $O_{[\sigma]}-N^+\sigma^{-1}Q^-\subset\coprod_{\mbox{\scriptsize{$[\eta]\in\mathcal{W}^1$ {\rm with} $n_{[\eta]}\leq n_{[\sigma]}-1$}}}N^+\eta^{-1}Q^-$.
\end{enumerate}   
\end{proposition}
\begin{proof}
   (i). 
   Since $\sigma^{-1}N^+Q^-\subset G_\mathbb{C}$ is open and $\mathcal{W}^1$ is finite, we have (i) by \eqref{eq-8.3.23}.\par
   
   (ii). 
   It is natural from \eqref{eq-8.3.23} that $O_{[\sigma]}\subset \sigma^{-1}N^+Q^-$.
   So, let us show $N^+\sigma^{-1}Q^-\subset O_{[\sigma]}$, namely 
\[
   N^+\sigma^{-1}Q^-\subset\Big(\sigma^{-1}N^+Q^--\bigcup_{\mbox{\scriptsize{$[\tau]\in\mathcal{W}^1$ {\rm with} $[\sigma]\neq[\tau]$ \& $n_{[\sigma]}\leq n_{[\tau]}$}}}\overline{N^+\tau^{-1}Q^-}\Big). 
\]
   Since $N^+\sigma^{-1}Q^-\subset\sigma^{-1}N^+Q^-$ it is enough to confirm that 
\begin{equation}\label{eq-1}\tag*{\textcircled{1}} 
   \mbox{$x\not\in\bigcup_{\mbox{\scriptsize{$[\tau]\in\mathcal{W}^1$ {\rm with} $[\sigma]\neq[\tau]$ \& $n_{[\sigma]}\leq n_{[\tau]}$}}}\overline{N^+\tau^{-1}Q^-}$ for all $x\in N^+\sigma^{-1}Q^-$}.
\end{equation}
   Let us use proof by contradiction. 
   Suppose a $y\in N^+\sigma^{-1}Q^-$ to satisfy $y\in\bigcup_{\mbox{\scriptsize{$[\tau]\in\mathcal{W}^1$ {\rm with} $[\sigma]\neq[\tau]$ \& $n_{[\sigma]}\leq n_{[\tau]}$}}}\overline{N^+\tau^{-1}Q^-}$.
   Then, there exists a $[\tau]\in\mathcal{W}^1$ such that $[\sigma]\neq[\tau]$, $n_{[\sigma]}\leq n_{[\tau]}$ and $y\in\overline{N^+\tau^{-1}Q^-}$.
   From $y\in N^+\sigma^{-1}Q^-$ and $y\in\overline{N^+\tau^{-1}Q^-}$ one obtains 
\[
   N^+\sigma^{-1}Q^-\subset\overline{N^+\tau^{-1}Q^-}.
\]
   Here we recall that $N^+\sigma^{-1}Q^-$ is an orbit of the group $N^+\times Q^-$ and $N^+\tau^{-1}Q^-$ is also. 
   In case of $N^+\sigma^{-1}Q^-\cap N^+\tau^{-1}Q^-=\emptyset$ Corollary \ref{cor-8.3.15} and $N^+\sigma^{-1}Q^-\subset\overline{N^+\tau^{-1}Q^-}$ cause $n_{[\tau]}<n_{[\sigma]}$, which is a contradiction to $n_{[\sigma]}\leq n_{[\tau]}$. 
   Even if $N^+\sigma^{-1}Q^-\cap N^+\tau^{-1}Q^-\neq\emptyset$, we have $[\sigma]=[\tau]$, which contradicts $[\sigma]\neq[\tau]$. 
   Hence \ref{eq-1} holds.\par

   (iii). 
   By a direct computation we obtain
\allowdisplaybreaks{
\begin{align*}
   O_{[\sigma]}-N^+\sigma^{-1}Q^-
  &\stackrel{\eqref{eq-8.3.23}}{=}
   \Big(\sigma^{-1}N^+Q^--\bigcup_{\mbox{\scriptsize{$[\tau]\in\mathcal{W}^1$ {\rm with} $[\sigma]\neq[\tau]$ \& $n_{[\sigma]}\leq n_{[\tau]}$}}}\overline{N^+\tau^{-1}Q^-}\Big)-N^+\sigma^{-1}Q^-\\
  &\subset\Big(\sigma^{-1}N^+Q^--\bigcup_{\mbox{\scriptsize{$[\tau]\in\mathcal{W}^1$ {\rm with} $[\sigma]\neq[\tau]$ \& $n_{[\sigma]}\leq n_{[\tau]}$}}}N^+\tau^{-1}Q^-\Big)-N^+\sigma^{-1}Q^-\\
  &\subset\sigma^{-1}N^+Q^--\bigcup_{\mbox{\scriptsize{$[\theta]\in\mathcal{W}^1$ {\rm with} $n_{[\sigma]}\leq n_{[\theta]}$}}}N^+\theta^{-1}Q^-\\
  &\subset G_\mathbb{C}-\coprod_{\mbox{\scriptsize{$[\theta]\in\mathcal{W}^1$ {\rm with} $n_{[\sigma]}\leq n_{[\theta]}$}}}N^+\theta^{-1}Q^-.
\end{align*}}This, combined with Theorem \ref{thm-8.3.7}-(3), gives rise to (iii).   
\end{proof}

   Utilizing the notation $O_{[\sigma]}$ in \eqref{eq-8.3.23} we show
\begin{lemma}\label{lem-8.3.25}
   Let $D$ be a dense, domain in $G_\mathbb{C}$. 
   For each $[\sigma]\in\mathcal{W}^1$, the following two items hold$:$
\begin{enumerate}
\item[{\rm (1)}]
   $O_{[\sigma]}\cup D$ be a dense domain in $G_\mathbb{C}$.
\item[{\rm (2)}]
   Suppose that {\rm (s1)} $2\leq n_{[\sigma]}$ and {\rm (s2)} $O_{[\sigma]}-N^+\sigma^{-1}Q^-\subset D$. 
   Then, for a given holomorphic function $f:D\to\mathbb{C}$, there exists a unique holomorphic function $\tilde{f}:O_{[\sigma]}\cup D\to\mathbb{C}$ such that $f=\tilde{f}$ on $D$.
\end{enumerate}   
\end{lemma}
\begin{proof}
   (1) is a consequence of Lemma \ref{lem-8.3.19} and Proposition \ref{prop-8.3.24}-(i).\par
   
   (2). 
   The uniqueness of $\tilde{f}:O_{[\sigma]}\cup D\to\mathbb{C}$ follows by (1) and the theorem of identity.
   We are going to verify its existence. 
   First, let us establish the following: 
\begin{equation}\label{eq-1}\tag*{\textcircled{1}}
\begin{split}
 & \mbox{For each $x\in O_{[\sigma]}\cup D$, there exist an open neighborhood $P_x$ of $x\in O_{[\sigma]}\cup D$ and a holomorphic}\\
 & \mbox{function $\tilde{f}_x:P_x\to\mathbb{C}$ such that $f=\tilde{f}_x$ on $P_x-N^+\sigma^{-1}Q^-$}.
\end{split}
\end{equation}
   Here we note that (s2) assures $P_x-N^+\sigma^{-1}Q^-\subset D$, so $f$ exists on $P_x-N^+\sigma^{-1}Q^-$.
   Now, fix any $x\in O_{[\sigma]}\cup D$. 
   In either of the cases $x\in D$ and $x\in O_{[\sigma]}-N^+\sigma^{-1}Q^-$, it follows from (s2) that $x\in D$, and hence we deduce \ref{eq-1} by putting 
\[
\begin{array}{ll}
   P_x:=D, & \tilde{f}_x:=f.
\end{array}
\] 
   Let us consider the remaining case 
\[
   x\in N^+\sigma^{-1}Q^-
\]
from now on. 
   By Proposition \ref{prop-8.3.24}-(i), (ii) and Lemma \ref{lem-8.3.22}-(3), there exist holomorphic functions $h_1,\dots,h_{n_{[\sigma]}}:O_{[\sigma]}\to\mathbb{C}$ which satisfy 
\[
\begin{array}{ll}
   \mbox{$dh_1\wedge\cdots\wedge dh_{n_{[\sigma]}}\neq 0$ on $O_{[\sigma]}$}, & N^+\sigma^{-1}Q^-=\{y\in O_{[\sigma]} \,|\, h_1(y)=\cdots=h_{n_{[\sigma]}}(y)=0\}.
\end{array}
\] 
   Then, the inverse mapping theorem enables us to take a holomorphic coordinate neighborhood $(P_x,\psi)$ of $x\in O_{[\sigma]}$ such that (i) $z^j\bigl(\psi(x)\bigr)=0$ for all $1\leq j\leq N=\dim_\mathbb{C}O_{[\sigma]}$, (ii) $\psi$ is a homeomorphism of $P_x$ onto an open subset of $\mathbb{C}^N$ defined by $|z^1|<R,\dots,|z^N|<R$ for some $R>0$ and (iii) $z^i\circ\psi=h_i$ for all $1\leq i\leq n_{[\sigma]}$. 
   Consequently Lemma \ref{lem-8.3.20} and (s1) imply that for the holomorphic function $f:(P_x-N^+\sigma^{-1}Q^-)\to\mathbb{C}$, there exists a unique holomorphic function $\tilde{f}_x:P_x\to\mathbb{C}$ such that $f=\tilde{f}_x$ on $P_x-N^+\sigma^{-1}Q^-$. 
   Thus \ref{eq-1} holds.
   One can construct a holomorphic function $\tilde{f}:O_{[\sigma]}\cup D\to\mathbb{C}$ from \ref{eq-1} and
\begin{equation}\label{eq-2}\tag*{\textcircled{2}}
   \mbox{$\tilde{f}|_{P_x}:=\tilde{f}_x$ for $x\in O_{[\sigma]}\cup D$}.
\end{equation}
   Here, it is necessary to confirm that \ref{eq-2} is well-defined. 
   If $P_x\cap P_y\neq\emptyset$ ($x,y\in O_{[\sigma]}\cup D$), then one can show that
\[
   \mbox{$\tilde{f}_x=\tilde{f}_y$ on $P_x\cap P_y$}
\]
because (a) both $\tilde{f}_x,\tilde{f}_y$ are continuous on $P_x\cap P_y$, (b) $\tilde{f}_x=f=\tilde{f}_y$ on $P_x\cap P_y-N^+\sigma^{-1}Q^-$, and (c) $P_x\cap P_y-N^+\sigma^{-1}Q^-$ is dense in $P_x\cap P_y$ ($\because$ $1\leq n_{[\sigma]}$).
   Accordingly \ref{eq-2} is well-defined.
   For the function $\tilde{f}$ in \ref{eq-2}, we conclude that $f=\tilde{f}$ on $D$.
\end{proof}

   We are in a position to accomplish the aim.
\begin{proof}[Proof of Theorem {\rm \ref{thm-8.3.17}}]
   (i). 
   cf.\ Lemma \ref{lem-8.3.21}.\par

   (ii). 
   Take any holomorphic function $f$ on $O$.
   First, let $k_2:=\min\{n_{[\sigma]}\in\mathbb{N} : [\sigma]\in\mathcal{W}^1,\, 2\leq n_{[\sigma]}\}$.
   For every $[\sigma]\in\mathcal{W}^1$ with $n_{[\sigma]}=k_2$, Proposition \ref{prop-8.3.24}-(iii) implies
\[
   \mbox{$O_{[\sigma]}-N^+\sigma^{-1}Q^-
   \subset\coprod_{\mbox{\scriptsize{$[\eta]\in\mathcal{W}^1$ {\rm with} $n_{[\eta]}\leq k_2-1$}}}N^+\eta^{-1}Q^-
   =\coprod_{\mbox{\scriptsize{$[\theta]\in\mathcal{W}^1$ {\rm with} $n_{[\theta]}\leq 1$}}}N^+\theta^{-1}Q^-
   \stackrel{\eqref{eq-8.3.18}}{=}O$}.   
\] 
   Hence for every $[\sigma]\in\mathcal{W}^1$ with $n_{[\sigma]}=k_2$, the function $f$ can be continued analytically from $O$ to the dense domain $O\cup O_{[\sigma]}$ by virtue of (i), $2\leq k_2=n_{[\sigma]}$ and Lemma \ref{lem-8.3.25}.
   Furthermore, the theorem of identity assures that $f$ can be continued analytically from $O\cup O_{[\sigma]}$ to the dense domain $O_2:=O\cup\big(\bigcup_{\mbox{\scriptsize{$[\sigma]\in\mathcal{W}^1$ {\rm with} $n_{[\sigma]}=k_2$}}}O_{[\sigma]}\big)$.
   Here Proposition \ref{prop-8.3.24}-(ii) and \eqref{eq-8.3.18} yield   
\begin{equation}\label{eq-a}\tag{a}
   \mbox{$\coprod_{\mbox{\scriptsize{$[\theta]\in\mathcal{W}^1$ {\rm with} $n_{[\theta]}\leq k_2$}}}N^+\theta^{-1}Q^-\subset O_2$}.
\end{equation}
   Next, let $k_3:=\min\{n_{[\rho]}\in\mathbb{N} : [\rho]\in\mathcal{W}^1,\, k_2+1\leq n_{[\rho]}\}$. 
   For any $[\rho]\in\mathcal{W}^1$ with $n_{[\rho]}=k_3$, we obtain
\[
   \mbox{$O_{[\rho]}-N^+\rho^{-1}Q^-
   \subset\coprod_{\mbox{\scriptsize{$[\eta]\in\mathcal{W}^1$ {\rm with} $n_{[\eta]}\leq k_3-1$}}}N^+\eta^{-1}Q^-
   =\coprod_{\mbox{\scriptsize{$[\theta]\in\mathcal{W}^1$ {\rm with} $n_{[\theta]}\leq k_2$}}}N^+\theta^{-1}Q^-
   \subset O_2$}   
\] 
from Proposition \ref{prop-8.3.24}-(iii) and \eqref{eq-a}.
   Accordingly, we conclude that $f$ can be continued analytically from $O_2$ to the dense domain $O_3:=O_2\cup\big(\bigcup_{\mbox{\scriptsize{$[\rho]\in\mathcal{W}^1$ {\rm with} $n_{[\rho]}=k_3$}}}O_{[\rho]}\big)$ and $\coprod_{\mbox{\scriptsize{$[\theta]\in\mathcal{W}^1$ {\rm with} $n_{[\theta]}\leq k_3$}}}N^+\theta^{-1}Q^-\subset O_3$ by arguments similar to those stated above and \eqref{eq-a}.
   Now, let $k_4:=\min\{n_{[\varsigma]}\in\mathbb{N} : [\varsigma]\in\mathcal{W}^1,\, k_3+1\leq n_{[\varsigma]}\}$. 
   Then, $f$ can be continued analytically from $O_3$ to the dense domain $O_4:=O_3\cup\big(\bigcup_{\mbox{\scriptsize{$[\varsigma]\in\mathcal{W}^1$ {\rm with} $n_{[\varsigma]}=k_4$}}}O_{[\varsigma]}\big)$ and $\coprod_{\mbox{\scriptsize{$[\theta]\in\mathcal{W}^1$ {\rm with} $n_{[\theta]}\leq k_4$}}}N^+\theta^{-1}Q^-\subset O_4$. 
   By inductive arguments we can get the conclusion.   
\end{proof}

\begin{remark}\label{rem-8.3.26}
   The $O$ in Theorem \ref{thm-8.3.17} is also expressed as  
\[
\begin{array}{r@{\,\,}l}
   O &
   =\coprod_{\mbox{\scriptsize{$[\sigma]\in\mathcal{W}^1$ {\rm with} $n_{[\sigma]}\leq 1$}}}N^+\sigma^{-1}Q^-
   =N^+Q^-\amalg\bigl(\coprod_{\mbox{\scriptsize{$\beta\in\Pi_\triangle-\blacktriangle$}}}N^+w_\beta^{-1}Q^-\bigr)\\
     & 
   =\coprod_{\mbox{\scriptsize{$[\sigma]\in\mathcal{W}^1$ {\rm with} $n_{[\sigma]}\leq 1$}}}\sigma^{-1}U^+_{[\sigma]}Q^-
   =U^+Q^-\amalg\bigl(\coprod_{\mbox{\scriptsize{$\beta\in\Pi_\triangle-\blacktriangle$}}}w_\beta^{-1}U^+_{[w_\beta]}Q^-\bigr).  
\end{array}
\]
   Here $\blacktriangle=\{\gamma\in\triangle(\frak{g}_\mathbb{C},\frak{h}_\mathbb{C}) \,|\, \gamma(T)=0\}$.
   cf.\ Theorem \ref{thm-8.3.7}-(1), Proposition \ref{prop-8.3.2}-(v).   
\end{remark}

   The following lemma will be needed in Chapter \ref{ch-11}:
\begin{lemma}\label{lem-8.3.27}
   For each $\beta\in\Pi_\triangle-\blacktriangle$, the following three items hold$:$
\begin{enumerate}
\item[{\rm (1)}]
   $N^+Q^-\cap w_\beta^{-1}N^+Q^-=\bigl(\exp\bigoplus_{\alpha\in\triangle^+-\{\beta\}}\frak{g}_\alpha\bigr)w_\beta^{-1}\exp(\frak{g}_\beta-\{0\})Q^-$.
\item[{\rm (2)}]
   $N^+Q^-\cap w_\beta^{-1}N^+Q^-$ is a dense domain in $G_\mathbb{C}$.
\item[{\rm (3)}]
   $U^+Q^-\cap w_\beta^{-1}U^+Q^-$ is a dense domain in $G_\mathbb{C}$.    
\end{enumerate}   
\end{lemma}
\begin{proof}
   (1). 
   First, let us demonstrate that $\bigl(\exp\bigoplus_{\alpha\in\triangle^+-\{\beta\}}\frak{g}_\alpha\bigr)w_\beta^{-1}\exp(\frak{g}_\beta-\{0\})Q^-\subset N^+Q^-\cap w_\beta^{-1}N^+Q^-$.
   Since $\beta\in\Pi_\triangle$ one has $\zeta([w_\beta])(\triangle^+-\{\beta\})=\triangle^+-\{\beta\}$. 
   Hence 
\[
   \mbox{$\bigl(\exp\bigoplus_{\alpha\in\triangle^+-\{\beta\}}\frak{g}_\alpha\bigr)w_\beta^{-1}
   =w_\beta^{-1}\bigl(\exp\bigoplus_{\alpha\in\triangle^+-\{\beta\}}\frak{g}_{\zeta([w_\beta])\alpha}\bigr)
   =w_\beta^{-1}\bigl(\exp\bigoplus_{\delta\in\triangle^+-\{\beta\}}\frak{g}_\delta\bigr)
   \subset w_\beta^{-1}N^+$}.
\]
   This, together with $\frak{g}_\beta\subset\frak{n}^+$, gives rise to 
\begin{equation}\label{eq-a}\tag{a}
   \mbox{$\bigl(\bigl(\exp\bigoplus_{\alpha\in\triangle^+-\{\beta\}}\frak{g}_\alpha\bigr)w_\beta^{-1}\bigr)\exp(\frak{g}_\beta-\{0\})Q^-\subset(w_\beta^{-1}N^+)N^+Q^-\subset w_\beta^{-1}N^+Q^-$}. 
\end{equation}
   Since $\frak{s}_\beta=\operatorname{span}_\mathbb{C}\{H_\beta^*,E_\beta,E_{-\beta}\}$ is a complex subalgebra of $\frak{g}_\mathbb{C}$ which is isomorphic to $\frak{sl}(2,\mathbb{C})$, the connected Lie subgroup $S_\beta\subset G_\mathbb{C}$ corresponding to $\frak{s}_\beta$ is isomorphic to either $SL(2,\mathbb{C})$ or $SL(2,\mathbb{C})/\mathbb{Z}_2$. 
   Accordingly, for any $z\in\mathbb{C}-\{0\}$ we obtain
\begin{equation}\label{eq-8.3.28}
   w_\beta^{-1}\exp(zE_\beta)
   =\exp\bigl((-1/z)E_\beta\bigr)\exp\bigl(-(\log r+i\theta)H_\beta^*\bigr)\exp\bigl((1/z)E_{-\beta}\bigr)
\end{equation}
from \eqref{eq-8.1.4}, where $z=re^{i\theta}$, $r>0$, $-\pi<\theta\leq\pi$. 
   Then, it turns out that $w_\beta^{-1}\exp(\frak{g}_\beta-\{0\})\subset(\exp\frak{g}_\beta)L_\mathbb{C}U^-\subset N^+Q^-$, and so 
\begin{equation}\label{eq-b}\tag{b}
   \mbox{$\bigl(\exp\bigoplus_{\alpha\in\triangle^+-\{\beta\}}\frak{g}_\alpha\bigr)\bigl(w_\beta^{-1}\exp(\frak{g}_\beta-\{0\})\bigr)Q^-\subset N^+(N^+Q^-)Q^-\subset N^+Q^-$}. 
\end{equation}
   From \eqref{eq-a} and \eqref{eq-b} we conclude that $\bigl(\exp\bigoplus_{\alpha\in\triangle^+-\{\beta\}}\frak{g}_\alpha\bigr)w_\beta^{-1}\exp(\frak{g}_\beta-\{0\})Q^-\subset N^+Q^-\cap w_\beta^{-1}N^+Q^-$. 
   Next, let us show that the converse inclusion also holds. 
   Take an arbitrary $x\in N^+Q^-\cap w_\beta^{-1}N^+Q^-$.
   Proposition \ref{prop-8.3.2}-(i), $\Phi_{[w_\beta]}=\{\beta\}$ and $\Phi_{[w_\beta\kappa]}=\triangle^+-\{\beta\}$ imply that  $N^+=\bigl(\exp\bigoplus_{\delta\in\triangle^+-\{\beta\}}\frak{g}_\delta\bigr)\exp\frak{g}_\beta$, and moreover
\[
   \mbox{$w_\beta^{-1}N^+Q^-=w_\beta^{-1}\bigl(\exp\bigoplus_{\delta\in\triangle^+-\{\beta\}}\frak{g}_\delta\bigr)(\exp\frak{g}_\beta)Q^-=\bigl(\exp\bigoplus_{\alpha\in\triangle^+-\{\beta\}}\frak{g}_\alpha\bigr)w_\beta^{-1}(\exp\frak{g}_\beta)Q^-$}.
\] 
   Hence there exist $n\in\exp\bigoplus_{\alpha\in\triangle^+-\{\beta\}}\frak{g}_\alpha$, $z_1\in\mathbb{C}$ and $q\in Q^-$ satisfying $x=n w_\beta^{-1}\exp(z_1E_\beta)q$.
   Then, we can assert that $x\in\bigl(\exp\bigoplus_{\alpha\in\triangle^+-\{\beta\}}\frak{g}_\alpha\bigr)w_\beta^{-1}\exp(\frak{g}_\beta-\{0\})Q^-$, if 
\begin{equation}\label{eq-c}\tag{c}
   z_1\neq 0.
\end{equation}
   For this reason, the rest of proof is to confirm \eqref{eq-c}. 
   Let us use proof by contradiction.
   Suppose that $z_1=0$. 
   Then, it follows that $x=n w_\beta^{-1}\exp(z_1E_\beta)q=n w_\beta^{-1}q\in N^+w_\beta^{-1}Q^-$, and $x\in N^+Q^-\cap N^+ w_\beta^{-1}Q^-$. 
   This is a contradiction to Theorem \ref{thm-8.3.7}-(3). 
   Therefore \eqref{eq-c} holds.\par
   
   (2). 
   (1) implies that $N^+Q^-\cap w_\beta^{-1}N^+Q^-$ is connected; furthermore, it is a dense domain in $G_\mathbb{C}$ by Corollary \ref{cor-8.3.16}-(ii).\par
   
   (3) is an easy consequence of (2) and $N^+Q^-=U^+Q^-$.
\end{proof}

\chapter{Homogeneous symplectic manifolds}\label{ch-9}
   In this chapter we first study homogeneous symplectic manifolds and afterwards investigate relation between homogeneous symplectic manifolds and adjoint orbits of semisimple Lie groups.

\section[Invariant symplectic forms on homogeneous spaces]{Invariant symplectic forms on homogeneous spaces and skew-symmetric bilinear forms on Lie algebras}\label{sec-9.1}
   Let us establish the following theorem which will play a role in the next section: 
   
\begin{theorem}\label{thm-9.1.1} 
   Let $G$ be a $($real$)$ Lie group which satisfies the second countability axiom, let $H$ be a closed subgroup of $G$, let $\pi$ denote the projection of $G$ onto $G/H$, and let $o:=\pi(e)$. 
   Then, the following two items {\rm (I)} and {\rm (II)} hold$:$
\begin{enumerate}
\item[{\rm (I)}] 
   Suppose the homogeneous space $G/H$ to admit a $G$-invariant symplectic form $\Omega$.  
   Then, there exists a unique skew-symmetric bilinear form $\omega:\frak{g}\times\frak{g}\to\mathbb{R}$ satisfying the following four conditions$:$ 
   \begin{enumerate}
   \item[{\rm (s.1)}] 
      $\omega([X_1,X_2],X_3)+\omega([X_2,X_3],X_1)+\omega([X_3,X_1],X_2)=0$ for all $X_1,X_2,X_3\in\frak{g}$,
   \item[{\rm (s.2)}]
      $\frak{h}=\{Z\in\frak{g} \,|\, \mbox{$\omega(Z,X)=0$ for all $X\in\frak{g}$}\}$,       
   \item[{\rm (s.3)}]
      $\omega\bigl(\operatorname{Ad}z(X),\operatorname{Ad}z(Y)\bigr)=\omega(X,Y)$ for all $z\in H$ and $X,Y\in\frak{g}$,
   \item[{\rm (s.4)}]
      $\omega(X,Y)=\Omega_o\bigl((d\pi)_eX_e,(d\pi)_eY_e\bigr)$ for all $X,Y\in\frak{g}$.
   \end{enumerate}
\item[{\rm (II)}] 
   Suppose that there exists a skew-symmetric bilinear form $\omega:\frak{g}\times\frak{g}\to\mathbb{R}$ satisfying the above three conditions {\rm (s.1)}, {\rm (s.2)} and {\rm (s.3)}. 
   Then, $G/H$ admits a unique $G$-invariant symplectic form $\Omega$ so that $\omega$ is related to $\Omega$ by {\rm (s.4)}.   
\end{enumerate}
   Here $G/H$ is an $n$-dimensional real analytic manifold in view of Theorem {\rm \ref{thm-1.1.2}}, and we identify the real constants with the real-valued constant functions on $G$.   
\end{theorem}
\begin{proof}
   (I). 
   Let $\Omega$ be a $G$-invariant symplectic form on $G/H$.
   Define a skew-symmetric bilinear form $\omega:\frak{g}\times\frak{g}\to\mathbb{R}$ by 
\begin{equation}\label{eq-9.1.2}
   \mbox{$\omega(X,Y):=\Omega_o\bigl((d\pi)_eX_e,(d\pi)_eY_e\bigr)$ for $X,Y\in\frak{g}$}.
\end{equation} 
   Needless to say, (s.4) holds for this $\omega$.
   For any $z\in H$ and $X,Y\in\frak{g}$, we obtain 
\[
\begin{split}
   \omega\bigl(\operatorname{Ad}z(X),\operatorname{Ad}z(Y)\bigr)
  &\stackrel{\eqref{eq-9.1.2}}{=}\Omega_o\bigl((d\pi)_e(\operatorname{Ad}z(X))_e,(d\pi)_e(\operatorname{Ad}z(Y))_e\bigr)
   =\Omega_o\bigl((d\tau_z)_o((d\pi)_eX_e),(d\tau_z)_o((d\pi)_eY_e)\bigr)\\
  &=\Omega_o\bigl((d\pi)_eX_e,(d\pi)_eY_e\bigr)
   \stackrel{\eqref{eq-9.1.2}}{=}\omega(X,Y)
\end{split}   
\] 
because $\Omega$ is $G$-invariant and $\pi(z)=o$ (see Corollary \ref{cor-1.1.7} for $\tau_z$). 
   Hence (s.3) holds for $\omega$. 
   We are going to confirm that the rest of conditions (s.1) and (s.2) holds for $\omega$.\par
   
   (s.1). 
   Set
\[
   \mbox{$\hat{\omega}_g(u,v):=\Omega_{\pi(g)}\bigl((d\pi)_gu,(d\pi)_gv\bigr)$ for $g\in G$ and $u,v\in T_gG$},
\] 
namely $\hat{\omega}$ is the pullback of $\Omega$ by $\pi:G\to G/H$. 
   Then $\hat{\omega}_g(X_g,Y_g)=\hat{\omega}_e(X_e,Y_e)$ for all $g\in G$ and $X,Y\in\frak{g}$, since $\Omega$ is $G$-invariant. 
   Accordingly we see that for each $X,Y\in\frak{g}$,
\begin{equation}\label{eq-1}\tag*{\textcircled{1}}
   \mbox{the mapping $G\ni g\mapsto\hat{\omega}_g(X_g,Y_g)\in\mathbb{R}$ is the constant function with the value $\hat{\omega}_e(X_e,Y_e)$}.
\end{equation} 
   Identifying the real constants with the real-valued constant functions on $G$, one may assume that
\[
   \mbox{$\omega(X,Y)=\hat{\omega}(X,Y)$ for all $X,Y\in\frak{g}$}
\]
by \eqref{eq-9.1.2}. 
   Hence it suffices to confirm that (s.1) holds for the $\hat{\omega}$.
   Moreover, it follows from $d\Omega=0$ and $\hat{\omega}=\pi^*\Omega$ that $d\hat{\omega}=0$, so that for all $X_1,X_2,X_3\in\frak{g}$,
\[
\begin{split}
   0
  &=(d\hat{\omega})(X_1,X_2,X_3)\\
  &=X_1\bigl(\hat{\omega}(X_2,X_3)\bigr)-X_2\bigl(\hat{\omega}(X_1,X_3)\bigr)+X_3\bigl(\hat{\omega}(X_1,X_2)\bigr)
   -\hat{\omega}([X_1,X_2],X_3)+\hat{\omega}([X_1,X_3],X_2)-\hat{\omega}([X_2,X_3],X_1)\\
  &\stackrel{\ref{eq-1}}{=}-\hat{\omega}([X_1,X_2],X_3)-\hat{\omega}([X_3,X_1],X_2)-\hat{\omega}([X_2,X_3],X_1). 
\end{split}
\]
   Thus (s.1) holds.\par
   
   (s.2).
   Let $\frak{h}_\omega:=\{Z\in\frak{g} \,|\, \mbox{$\omega(Z,X)=0$ for all $X\in\frak{g}$}\}$.
   We want to show $\frak{h}=\frak{h}_\omega$. 
   First, let us show $\frak{h}\subset\frak{h}_\omega$.
   Take any $Z\in\frak{h}$. 
   Then, it is immediate from \eqref{eq-9.1.2} and $(d\pi)_eZ_e=0$ that $\omega(Z,X)=\Omega_o\bigl((d\pi)_eZ_e,(d\pi)_eX_e\bigr)=0$ for all $X\in\frak{g}$. 
   Therefore the inclusion $\frak{h}\subset\frak{h}_\omega$ follows. 
   Next, let us prove that the converse inclusion also holds. 
   For $Y\in\frak{g}$ we suppose that $\omega(Y,X)=0$ for all $X\in\frak{g}$. 
   Then, 
\[
   \Omega_o\bigl((d\pi)_eY_e,(d\pi)_eX_e\bigr)
   \stackrel{\eqref{eq-9.1.2}}{=}\omega(Y,X)
   =0
\]
for all $X\in\frak{g}$. 
   This yields $(d\pi)_eY_e=0$ because $\Omega_o$ is non-degenerate on the vector space $T_o(G/H)$ and the mapping $\frak{g}\ni X\mapsto(d\pi)_eX_e\in T_o(G/H)$ is surjective. 
   By $(d\pi)_eY_e=0$ and Lemma \ref{lem-1.1.13} we conclude $Y\in\frak{h}$, and $\frak{h}_\omega\subset\frak{h}$. 
   Therefore $\frak{h}=\frak{h}_\omega$, and (s.2) holds.
   Thus one can conclude (I) since (s.4) assures the uniqueness of $\omega$.\par

   (II). 
   Now, suppose that a skew-symmetric bilinear form $\omega:\frak{g}\times\frak{g}\to\mathbb{R}$ satisfies the following three conditions: 
\begin{enumerate}
\item[{\rm (s.1)}] 
   $\omega([X_1,X_2],X_3)+\omega([X_2,X_3],X_1)+\omega([X_3,X_1],X_2)=0$ for all $X_1,X_2,X_3\in\frak{g}$,
\item[{\rm (s.2)}]
   $\frak{h}=\{Z\in\frak{g} \,|\, \mbox{$\omega(Z,X)=0$ for all $X\in\frak{g}$}\}$,       
\item[{\rm (s.3)}]
   $\omega\bigl(\operatorname{Ad}z(X),\operatorname{Ad}z(Y)\bigr)=\omega(X,Y)$ for all $z\in H$ and $X,Y\in\frak{g}$.
\end{enumerate} 
   Our aim is to construct a $G$-invariant symplectic form $\Omega$ on $G/H$, on this supposition.   
   First, let us construct a closed differential form $\widetilde{\omega}$ of degree $2$ on $G$ from the $\omega$.
   Let $\{E_r\}_{r=1}^N$ be a real basis of $\frak{g}$, and set 
\[
   \mathcal{C}^\infty(G):=\{\tilde{f}:G\to\mathbb{R} \,|\, \mbox{$\tilde{f}$ is of class $C^\infty$}\}.
\]
   Since $\{(E_r)_g\}_{r=1}^N$ is a real basis of the vector space $T_gG$ for each $g\in G$, an arbitrary vector $U\in\frak{X}(G)$ is uniquely expressed as $U=\sum_{r=1}^N\tilde{f}_rE_r$, $\tilde{f}_r\in\mathcal{C}^\infty(G)$. 
   Then, for $V=\sum_{s=1}^N\tilde{h}_sE_s\in\frak{X}(G)$ ($\tilde{h}_s\in\mathcal{C}^\infty(G)$) we put
\begin{equation}\label{eq-a}\tag{a} 
   \mbox{$\widetilde{\omega}_g(U_g,V_g):=\sum_{r,s=1}^N\tilde{f}_r(g)\tilde{h}_s(g)\omega(E_r,E_s)$ for $g\in G$}.
\end{equation}
   This \eqref{eq-a} is independent of the choice of $\{E_r\}_{r=1}^N$ because $\omega$ is $\mathbb{R}$-bilinear.
   So, $\widetilde{\omega}$ is a differential form of degree $2$ on $G$. 
   From now on, we are going to verify that the $\widetilde{\omega}$ is closed.
   One can express $[E_r,E_s]$ as $[E_r,E_s]=\sum_{\ell=1}^Nc_{rs}^\ell E_\ell$, $c_{rs}^\ell\in\mathbb{R}$. 
   For any $X,Y,Z\in\frak{g}$ there uniquely exist $a^r,b^s,c^t\in\mathbb{R}$ satisfying $X=\sum_{r=1}^Na^rE_r,Y=\sum_{s=1}^Nb^sE_s,Z=\sum_{t=1}^Nc^tE_t$; then all $\widetilde{\omega}(Y,Z)$, $\widetilde{\omega}(X,Z)$ and $\widetilde{\omega}(X,Y)$ are constant functions on $G$ due to \eqref{eq-a}, and moreover, 
\allowdisplaybreaks{
\begin{align*}
    (d\widetilde{\omega})(X,Y,Z)
   &=X\bigl(\widetilde{\omega}(Y,Z)\bigr)-Y\bigl(\widetilde{\omega}(X,Z)\bigr)+Z\bigl(\widetilde{\omega}(X,Y)\bigr)-\widetilde{\omega}([X,Y],Z)+\widetilde{\omega}([X,Z],Y)-\widetilde{\omega}([Y,Z],X)\\
   &=-\widetilde{\omega}([X,Y],Z)-\widetilde{\omega}([Z,X],Y)-\widetilde{\omega}([Y,Z],X)\\
   &\!\stackrel{\eqref{eq-a}}{=}-\sum_{r,s,t,\ell=1}^Na^rb^sc^tc_{rs}^\ell\omega(E_\ell,E_t)-\sum_{r,s,t,\ell=1}^Na^rb^sc^tc_{tr}^\ell\omega(E_\ell,E_s)-\sum_{r,s,t,\ell=1}^Na^rb^sc^tc_{st}^\ell\omega(E_\ell,E_r)\\
   &=-\omega([X,Y],Z)-\omega([Z,X],Y)-\omega([Y,Z],X) \quad\mbox{($\because$ $\omega$ is $\mathbb{R}$-bilinear)}\\
   &\!\!\stackrel{{\rm (s.1)}}{=}0,
\end{align*}}namely $(d\widetilde{\omega})(X,Y,Z)=0$ for all $X,Y,Z\in\frak{g}$. 
   This gives rise to 
\[
   \mbox{$(d\widetilde{\omega})(U,V,W)=0$ for all $U,V,W\in\frak{X}(G)$}
\]
because $d\widetilde{\omega}$ is $\mathcal{C}^\infty(G)$-multilinear and $\frak{X}(G)$ is generated by smooth functions $\tilde{f}:G\to\mathbb{R}$ and elements $X\in\frak{g}$.
   Thus $\widetilde{\omega}$ is closed. 
   Consequently $\widetilde{\omega}$ is a closed, differential form of degree $2$ on $G$.
   Identifying the real constants with the real-valued constant functions on $G$, one may assume that 
\begin{equation}\label{eq-b}\tag{b} 
   \mbox{$\omega(X,Y)=\widetilde{\omega}(X,Y)$ for all $X,Y\in\frak{g}$}   
\end{equation}
by \eqref{eq-a}.
   From now on, let us construct a $G$-invariant symplectic form $\Omega$ on $G/H$.
   For given vectors $u,v\in T_o(G/H)$, we choose $X,Y\in\frak{g}$ so that $u=(d\pi)_eX_e,v=(d\pi)_eY_e$, and set 
\begin{equation}\label{eq-c}\tag{c} 
   \Omega_o(u,v)=\Omega_o\bigl((d\pi)_eX_e,(d\pi)_eY_e\bigr):=\omega(X,Y).
\end{equation}
   Lemma \ref{lem-1.1.13} and (s.2) assure that \eqref{eq-c} is independent of the choice of $X$ and $Y$, and that $\Omega_o$ is non-degenerate. 
   Therefore $\Omega_o$ is a symplectic form on the vector space $T_o(G/H)$. 
   Then, one defines a symplectic form $\Omega_{\pi(g)}$ on $T_{\pi(g)}(G/H)$ ($g\in G$) by  
\begin{equation}\label{eq-d}\tag{d} 
   \mbox{$\Omega_{\pi(g)}(w_1,w_2):=\Omega_o\bigl((d\tau_{g^{-1}})_{\pi(g)}w_1,(d\tau_{g^{-1}})_{\pi(g)}w_2\bigr)$ for $w_1,w_2\in T_{\pi(g)}(G/H)$}.
\end{equation}
   Here we remark that \eqref{eq-d} is well-defined by virtue of (s.3) and \eqref{eq-c}. 
   From \eqref{eq-d} it follows that $\Omega$ is $G$-invariant.
   If we show that $\Omega$ is of class $C^\infty$ and $d\Omega=0$, then one can assert that $\Omega$ is a $G$-invariant symplectic form on $G/H$.\par
   
   (class $C^\infty$). 
   We are going to show that $\Omega$ is of class $C^\infty$.
   For any point $p\in G/H$, there exist coordinate neighborhoods $\bigl(U,(y^1,\dots,y^n)\bigr)$ of class $C^\omega$ of $G/H$ and $\bigl(\pi^{-1}(U),(x^1,\dots,x^n,x^{n+1},\dots,x^N)\bigr)$ of class $C^\omega$ of $G$ such that $p\in U$ and $x^i=y^i\circ\pi$ on $\pi^{-1}(U)$ for all $1\leq i\leq n$; moreover, there exists a real analytic mapping $\sigma:U\to G$ such that $\pi\bigl(\sigma(q)\bigr)=q$ for all $q\in U$ (cf.\ Section \ref{sec-1.3}). 
   Therefore, for any $g\in\pi^{-1}(U)$ and $1\leq i,j\leq n$ we obtain 
\[
\begin{split}
   \Omega_{\pi(g)}\Big(\Big(\dfrac{\partial}{\partial y^i}\Big)_{\!\pi(g)},\Big(\dfrac{\partial}{\partial y^j}\Big)_{\!\pi(g)}\Big)
  &\stackrel{\eqref{eq-d}}{=}\Omega_o\Big((d\tau_{g^{-1}})_{\pi(g)}\Big(\dfrac{\partial}{\partial y^i}\Big)_{\!\pi(g)},(d\tau_{g^{-1}})_{\pi(g)}\Big(\dfrac{\partial}{\partial y^j}\Big)_{\!\pi(g)}\Big)\\
  &=\Omega_o\Big((d\tau_{g^{-1}})_{\pi(g)}\Big((d\pi)_g\Big(\dfrac{\partial}{\partial x^i}\Big)_{\!g}\Big),(d\tau_{g^{-1}})_{\pi(g)}\Big((d\pi)_g\Big(\dfrac{\partial}{\partial x^j}\Big)_{\!g}\Big)\Big) \quad\mbox{($\because$ $x^i=y^i\circ\pi$)}\\
  &=\Omega_o\Big((d\pi)_e\Big((dL_{g^{-1}})_g\Big(\dfrac{\partial}{\partial x^i}\Big)_{\!g}\Big),(d\pi)_e\Big((dL_{g^{-1}})_g\Big(\dfrac{\partial}{\partial x^j}\Big)_{\!g}\Big)\Big).
\end{split}
\]   
   Temporarily we express $(dL_{g^{-1}})_g(\partial/\partial x^k)_g\in T_eG$ as $(dL_{g^{-1}})_g(\partial/\partial x^k)_g=X^k_e$ with $X^k\in\frak{g}$, and then the last term is 
\[
\begin{split}
   \Omega_o\Big((d\pi)_e\Big((dL_{g^{-1}})_g\Big(\dfrac{\partial}{\partial x^i}\Big)_{\!g}\Big),(d\pi)_e\Big((dL_{g^{-1}})_g\Big(\dfrac{\partial}{\partial x^j}\Big)_{\!g}\Big)\Big)
  &=\Omega_o\bigl((d\pi)_eX^i_e,(d\pi)_eX^j_e\bigr)
   \stackrel{\eqref{eq-c}}{=}\omega(X^i,X^j)
   \stackrel{\eqref{eq-b}}{=}\widetilde{\omega}_g(X^i_g,X^j_g)\\
  &=\widetilde{\omega}_g\bigl((dL_g)_eX^i_e,(dL_g)_eX^j_e\bigr)
   =\widetilde{\omega}_g\Big(\Big(\dfrac{\partial}{\partial x^i}\Big)_{\!g},\Big(\dfrac{\partial}{\partial x^j}\Big)_{\!g}\Big).
\end{split} 
\]
   Therefore, it turns out that $\Omega_{ij}\circ\pi=\widetilde{\omega}_{ij}$ on $\pi^{-1}(U)$ ($1\leq i,j\leq n$), where $\Omega_{ij}:=\Omega(\partial/\partial y^i,\partial/\partial y^j)$, $\widetilde{\omega}_{ij}:=\widetilde{\omega}(\partial/\partial x^i,\partial/\partial x^j)$.
   Furthermore, $\pi\circ\sigma=\operatorname{id}$ yields 
\begin{equation}\label{eq-e}\tag{e} 
   \mbox{$\Omega_{ij}=\widetilde{\omega}_{ij}\circ\sigma$ on $U$ ($1\leq i,j\leq n$)}.   
\end{equation}
   This \eqref{eq-e} implies that $\Omega$ is of class $C^\infty$ because $\sigma:U\to\pi^{-1}(U)$ is real analytic and $\widetilde{\omega}_{ij}:\pi^{-1}(U)\to\mathbb{R}$ is smooth.\par

   ($d\Omega=0$). 
   It follows from \eqref{eq-e} and $d\widetilde{\omega}=0$ that for all $1\leq k\leq n$
\[
   \dfrac{\partial\Omega_{ij}}{\partial y^k}
   =\dfrac{\partial(\widetilde{\omega}_{ij}\circ\sigma)}{\partial y^k} 
   =\sum_{s=1}^N\dfrac{\partial\widetilde{\omega}_{ij}}{\partial x^s}\dfrac{\partial(x^s\circ\sigma)}{\partial y^k}
   =0,
\]
and thus $d\Omega=0$ holds.\par

   We have proven that the $\Omega$ in \eqref{eq-d} is a $G$-invariant symplectic form on $G/H$. 
   From \eqref{eq-c} it is natural that $\omega$ is related to this $\Omega$ by (s.4). 
   Now, the uniqueness of $\Omega$ follows from (s.4), $G$-invariability and Lemma \ref{lem-1.1.13}. 
   Hence we complete the proof of Theorem \ref{thm-9.1.1}.
\end{proof}

   Let $\frak{g}$ be a real Lie algebra. 
   Suppose that $\omega$ is a skew-symmetric bilinear form $\omega:\frak{g}\times\frak{g}\to\mathbb{R}$ satisfying 
\begin{enumerate}
\item[{\rm (s.1)}] 
   $\omega([X_1,X_2],X_3)+\omega([X_2,X_3],X_1)+\omega([X_3,X_1],X_2)=0$ for all $X_1,X_2,X_3\in\frak{g}$.
\end{enumerate}
   In this setting, one can get a subalgebra $\frak{h}_\omega\subset\frak{g}$ by putting
\begin{equation}\label{eq-9.1.3}
   \frak{h}_\omega:=\{Z\in\frak{g} \,|\, \mbox{$\omega(Z,X)=0$ for all $X\in\frak{g}$}\},
\end{equation}
and deduce the following proposition from the proof of Theorem 2 in Chu \cite[p.149]{Chu}:

\begin{proposition}\label{prop-9.1.4}
   Let $G$ be a simply connected Lie group with the Lie algebra $\frak{g}$, and let $H_\omega$ be the connected Lie subgroup of $G$ corresponding to the subalgebra  $\frak{h}_\omega\subset\frak{g}$ in \eqref{eq-9.1.3}.
   Then, 
\begin{enumerate}
\item[{\rm (i)}]
   $H_\omega$ is a connected, closed subgroup of $G$,
\item[{\rm (ii)}]
   $\omega\bigl(\operatorname{Ad}z(X),\operatorname{Ad}z(Y)\bigr)=\omega(X,Y)$ for all $z\in H_\omega$ and $X,Y\in\frak{g}$,
\item[{\rm (iii)}]
   $G/H_\omega$ is a simply connected homogeneous space, and there exists a unique $G$-invariant symplectic form $\Omega$ on $G/H_\omega$ such that $\omega(X,Y)=\Omega_{\pi(e)}\bigl((d\pi)_eX_e,(d\pi)_eY_e\bigr)$ for all $X,Y\in\frak{g}$.
\end{enumerate} 
   Here $\pi$ is the projection of $G$ onto $G/H_\omega$, and we identify the real constants with the real-valued constant functions on $G$.
\end{proposition} 
\begin{proof}
   (i). 
   It is enough to prove that $H_\omega$ is closed in $G$.
   Since (s.1) holds for the $\omega$, one can define a closed differential form $\widetilde{\omega}$ of degree $2$ on $G$ by a similar way to \eqref{eq-a} in the proof of Theorem \ref{thm-9.1.1}-(II).
   Then one can assert that 
\begin{equation}\label{eq-1}\tag*{\textcircled{1}}
   \mbox{$\omega(X,Y)=\widetilde{\omega}(X,Y)$ for all $X,Y\in\frak{g}$}
\end{equation} 
where we identify the real constants with the real-valued constant functions on $G$. 
   For the real vector space $\frak{g}^*$ of left invariant differential forms of degree $1$ on $G$,\footnote{Remark.\ $\dim_\mathbb{R}\frak{g}^*=\dim_\mathbb{R}\frak{g}<\infty$.} the group of affine transformations of the vector space $\frak{g}^*$ is $GL(\frak{g}^*)\ltimes\frak{g}^*$ (semidirect). 
   Moreover, its Lie algebra is $\frak{gl}(\frak{g}^*)\ltimes\frak{g}^*$ and the exponential mapping $\exp:\frak{gl}(\frak{g}^*)\ltimes\frak{g}^*\to GL(\frak{g}^*)\ltimes\frak{g}^*$ is expressed as 
\begin{equation}\label{eq-2}\tag*{\textcircled{2}}
   \mbox{$\bigl(\exp(B,\eta)\bigr)(\xi)=(\exp B)(\xi)+\bigl(\sum_{n=1}^\infty(1/n!)B^{n-1}\bigr)(\eta)$},
\end{equation} 
where $(B,\eta)\in\frak{gl}(\frak{g}^*)\ltimes\frak{g}^*$ and $\xi\in\frak{g}^*$.
   Besides, the bracket product of Lie algebra $\frak{gl}(\frak{g}^*)\ltimes\frak{g}^*$ is expressed as\footnote{e.g.\ \begin{CJK}{UTF8}{min}命題 5.8.2 in 杉浦\end{CJK} \cite[p.406]{Su2}.} 
\begin{equation}\label{eq-3}\tag*{\textcircled{3}}
   \big[(B_1,\eta_1),(B_2,\eta_2)\big]=\bigl([B_1,B_2],B_1(\eta_2)-B_2(\eta_1)\bigr).
\end{equation} 
   Now, for any $X,Y\in\frak{g}$ it follows from \ref{eq-1} that $\widetilde{\omega}(X,Y)$ is a real-valued constant function on $G$. 
   Accordingly one can define a mapping $\phi_*:\frak{g}\to\frak{gl}(\frak{g}^*)\ltimes\frak{g}^*$ by 
\begin{equation}\label{eq-4}\tag*{\textcircled{4}} 
   \mbox{$\phi_*(X):=(L_X,\imath(X)\widetilde{\omega})$ for $X\in\frak{g}$}, 
\end{equation} 
where $L_X$ and $\imath(X)\widetilde{\omega}$ stand for the Lie derivative with respect to the vector field $X$ and the interior product of $\widetilde{\omega}$ with $X$, respectively.
   Here for any $X,Y\in\frak{g}$ one has
\[
\begin{split}
   L_Y\bigl(\imath(X)\widetilde{\omega}\bigr)
  &=(d\circ\imath(Y)+\imath(Y)\circ d)\bigl(\imath(X)\widetilde{\omega}\bigr)
   =\imath(Y)\bigl(d(\imath(X)\widetilde{\omega})\bigr) \quad\mbox{($\because$ $\imath(Y)\bigl(\imath(X)\widetilde{\omega}\bigr)=\widetilde{\omega}(X,Y)$ is constant)}\\
  &=\imath(Y)\bigl( (L_X-\imath(X)\circ d)\widetilde{\omega} \bigr)
   =\imath(Y)\bigl(L_X\widetilde{\omega}\bigr) \quad\mbox{($\because$ $d\widetilde{\omega}=0$)},
\end{split} 
\] 
since $L_W=d\circ\imath(W)+\imath(W)\circ d$ for all $W\in\frak{X}(G)$.
   This shows 
\begin{equation}\label{eq-5}\tag*{\textcircled{5}} 
   \mbox{$L_Y\bigl(\imath(X)\widetilde{\omega}\bigr)=\imath(Y)\bigl(L_X\widetilde{\omega}\bigr)$ for all $X,Y\in\frak{g}$}. 
\end{equation} 
   From now on, let us confirm that the mapping $\phi_*:\frak{g}\to\frak{gl}(\frak{g}^*)\ltimes\frak{g}^*$ in \ref{eq-4} is a Lie algebra homomorphism. 
   It is obvious that $\phi_*:X\mapsto(L_X,\imath(X)\widetilde{\omega})$ is linear. 
   For any $X,Y\in\frak{g}$, we obtain 
\begin{multline*}
   [\phi_*(X),\phi_*(Y)]
   \stackrel{\ref{eq-4}}{=}\big[(L_X,\imath(X)\widetilde{\omega}),(L_Y,\imath(Y)\widetilde{\omega})\big]
   \stackrel{\ref{eq-3}}{=}\bigl([L_X,L_Y],L_X(\imath(Y)\widetilde{\omega})-L_Y(\imath(X)\widetilde{\omega})\bigr)\\
   =\bigl(L_{[X,Y]},L_X(\imath(Y)\widetilde{\omega})-L_Y(\imath(X)\widetilde{\omega})\bigr)
   \stackrel{\ref{eq-5}}{=}\bigl(L_{[X,Y]},L_X(\imath(Y)\widetilde{\omega})-\imath(Y)(L_X\widetilde{\omega})\bigr)
   =\bigl(L_{[X,Y]},\imath([X,Y])\widetilde{\omega}\bigr)
   \stackrel{\ref{eq-4}}{=}\phi_*([X,Y]).
\end{multline*}   
   Thus $\phi_*:\frak{g}\to\frak{gl}(\frak{g}^*)\ltimes\frak{g}^*$, $X\mapsto(L_X,\imath(X)\widetilde{\omega})$, is a Lie algebra homomorphism. 
   Since $\phi_*:\frak{g}\to\frak{gl}(\frak{g}^*)\ltimes\frak{g}^*$ is a homomorphism and $G$ is a simply connected, there uniquely exists a Lie group homomorphism $\phi$ of $G$ into the identity component of $GL(\frak{g}^*)\ltimes\frak{g}^*$ such that its differential homomorphism accords with $\phi_*$. 
   Then one can take a real analytic action of $G$ on $\frak{g}^*$,  
\[
   G\times\frak{g}^*\ni(g,\eta)\mapsto\phi(g)(\eta)\in\frak{g}^*,
\] 
and take the isotropy subgroup $H$ of $G$ at $0\in\frak{g}^*$ into consideration. 
   Needless to say, $H=\{g\in G \,|\, \phi(g)(0)=0\}$ is a closed subgroup of $G$. 
   If 
\begin{equation}\label{eq-6}\tag*{\textcircled{6}} 
   \frak{h}_\omega=\operatorname{Lie}(H),  
\end{equation} 
then \ref{eq-6} implies that $H_\omega$ coincides with the identity component of $H$, so that $H_\omega$ is closed in $H$; and therefore $H_\omega$ is closed in $G$. 
   For this reason, the rest of proof is to demonstrate \ref{eq-6}. 
   For any $Z\in\frak{h}_\omega$ and $t\in\mathbb{R}$ we see that  
\[ 
   \phi(\exp tZ)(0)
   =\bigl(\exp t\phi_*(Z)\bigr)(0)
   \stackrel{\ref{eq-4}}{=}\bigl(\exp t(L_Z,\imath(Z)\widetilde{\omega})\bigr)(0)
   \stackrel{\ref{eq-2}}{=}(\exp tL_Z)(0)+\sum_{n=1}^\infty\dfrac{1}{n!}(tL_Z)^{n-1}(\imath(tZ)\widetilde{\omega})
   =0
\] 
because $\imath(tZ)\widetilde{\omega}=0$ comes from $tZ\in\frak{h}_\omega$, \eqref{eq-9.1.3} and \ref{eq-1}.
   Hence $Z\in\operatorname{Lie}(H)$, and so $\frak{h}_\omega\subset\operatorname{Lie}(H)$. 
   Let us show that the converse inclusion also holds.
   For any $A\in\operatorname{Lie}(H)$ and $t\in\mathbb{R}$, one has
\[
\begin{split}
   0
  &=\phi(\exp tA)(0) \quad\mbox{($\because$ $tA\in\operatorname{Lie}(H)$)}\\
  &=\bigl(\exp t\phi_*(A)\bigr)(0)
   \stackrel{\ref{eq-4},\, \ref{eq-2}}{=}\sum_{n=1}^\infty\dfrac{1}{n!}(tL_A)^{n-1}(\imath(tA)\widetilde{\omega})
   =\sum_{n=1}^\infty\dfrac{t^n}{n!}(L_A)^{n-1}(\imath(A)\widetilde{\omega}).
\end{split}
\]
   Differentiating this equation at $t=0$ we obtain $0=\imath(A)\widetilde{\omega}$, and therefore $A\in\frak{h}_\omega$ due to \eqref{eq-9.1.3} and \ref{eq-1}. 
   Hence $\operatorname{Lie}(H)\subset\frak{h}_\omega$ holds.
   This completes the proof of \ref{eq-6}.     
   
   (ii).
   Since $\omega$ is skew-symmetric, it follows from (s.1) and \eqref{eq-9.1.3} that $\omega\bigl([Z,X],Y\bigr)+\omega\bigl(X,[Z,Y]\bigr)=0$ for all $Z\in\frak{h}_\omega$, $X,Y\in\frak{g}$.
   Therefore (ii) holds because $H_\omega$ is connected.\par

   (iii). 
   $G/H_\omega$ is a simply connected homogeneous space by (i) and $G$ being simply connected.
   Hence we can conclude (iii) by Theorem \ref{thm-9.1.1}-(II) together with (s.1), \eqref{eq-9.1.3} and (ii).
\end{proof}

\section{Homogeneous symplectic manifolds of semisimple Lie groups}\label{sec-9.2}
   We want to first show
\begin{lemma}\label{lem-9.2.1}
   Let $\frak{g}$ be a real semisimple Lie algebra, and let $\omega$ be a skew-symmetric bilinear form $\omega:\frak{g}\times\frak{g}\to\mathbb{R}$ satisfying 
\begin{enumerate}
\item[{\rm (s.1)}] 
   $\omega([X_1,X_2],X_3)+\omega([X_2,X_3],X_1)+\omega([X_3,X_1],X_2)=0$ for all $X_1,X_2,X_3\in\frak{g}$.
\end{enumerate}
   Then, there exists a unique $S\in\frak{g}$ such that 
\[
\begin{array}{ll}
   \mbox{{\rm (1)} $\omega(X,Y)=B_\frak{g}(S,[X,Y])$ for all $X,Y\in\frak{g}$},
   & \mbox{{\rm (2)} $\frak{c}_\frak{g}(S)=\{Z\in\frak{g} \,|\, \mbox{$\omega(Z,X)=0$ for all $X\in\frak{g}$}\}$}.
\end{array}
\]
   Here $B_\frak{g}$ is the Killing form of $\frak{g}$ and $\frak{c}_\frak{g}(S)=\{Y\in\frak{g}\,|\,\operatorname{ad}S(Y)=0\}$.  
\end{lemma}
\begin{proof}
   (Uniqueness). 
   The uniqueness of $S$ follows by (1), $\frak{g}=[\frak{g},\frak{g}]$ and $B_\frak{g}$ being non-degenerate.\par
   
   (Existence). 
   Let us confirm that there exists an $S\in\frak{g}$ satisfying the conditions (1) and (2).
   Consider the cohomology group $H^k(\frak{g})=Z^k(\frak{g})/B^k(\frak{g})$ for the trivial representation of $\frak{g}$ on the vector space $\mathbb{R}$.
   On the one hand; (s.1) implies that $\omega\in Z^2(\frak{g})$. 
   On the other hand; by the Whitehead lemma one knows $\dim_\mathbb{R}H^1(\frak{g})=\dim_\mathbb{R}H^2(\frak{g})=0$, since $\frak{g}$ is real semisimple.
   Hence there exists a unique linear mapping $\alpha:\frak{g}\to\mathbb{R}$ such that 
\[
   \mbox{$\alpha\bigl([X,Y]\bigr)=\omega(X,Y)$ for all $X,Y\in\frak{g}$}.
\]   
   Furthermore, there exists a unique $S\in\frak{g}$ such that $\alpha(V)=B_\frak{g}(S,V)$ for all $V\in\frak{g}$ because $B_\frak{g}$ is non-degenerate. 
   Then, this $S$ satisfies (1). 
   By (1) we deduce that  
\[
   \omega(Z,X)
   =B_\frak{g}\bigl(S,[Z,X]\bigr)
   =B_\frak{g}\bigl(\operatorname{ad}S(Z),X\bigr)
\]
for all $X,Z\in\frak{g}$. 
   This implies that $\omega(Z,X)=0$ for all $X\in\frak{g}$ if and only if $\operatorname{ad}S(Z)=0$.
   Therefore $S$ satisfies (2) also.  
\end{proof}  

   From Theorem \ref{thm-9.1.1} and Lemma \ref{lem-9.2.1} we conclude  
\begin{proposition}[{cf.\ Matsushima \cite{Ma}\footnote{Remark.\ Th\'{e}or\`{e}me 1 in Matsushima \cite[p.54]{Ma} and its proof enable one to make a more excellent assertion.}}]\label{prop-9.2.2} 
   Let $G$ be a real semisimple Lie group which satisfies the second countability axiom, let $H$ be a closed subgroup of $G$, let $\pi$ denote the projection of $G$ onto $G/H$, and let $o:=\pi(e)$. 
   Suppose that the homogeneous space $G/H$ admits a $G$-invariant symplectic form $\Omega$. 
   Then, there exists a unique $S\in\frak{g}$ such that 
\begin{enumerate}
\item[{\rm (i)}]
   $B_\frak{g}(S,[X,Y])=\Omega_o\bigl((d\pi)_eX_e,(d\pi)_eY_e\bigr)$ for all $X,Y\in\frak{g}$,
\item[{\rm (ii)}] 
   $C_G(S)_0\subset H\subset C_G(S)$.
\end{enumerate}
   Here $G/H$ is a real analytic manifold in view of Theorem {\rm \ref{thm-1.1.2}}, $C_G(S)_0$ is the identity component of $C_G(S)=\{g\in G \,|\, \operatorname{Ad}g(S)=S\}$, and we identify the real constants with the real-valued constant functions on $G$.
\end{proposition}
\begin{proof}
   By virtue of Theorem \ref{thm-9.1.1}-(I) and Lemma \ref{lem-9.2.1}, it suffices to verify that (ii) $C_G(S)_0\subset H\subset C_G(S)$. 
   From Lemma \ref{lem-9.2.1}-(2) and Theorem \ref{thm-9.1.1}-(I)-(s.2) we obtain $\operatorname{Lie}(C_G(S))=\operatorname{Lie}(H)$, and therefore 
\[
   C_G(S)_0=H_0\subset H.
\] 
   Hence, the rest of proof is to confirm $H\subset C_G(S)$. 
   For any $z\in H$ and $X,Y\in\frak{g}$, Lemma \ref{lem-9.2.1}-(1) and Theorem \ref{thm-9.1.1}-(I)-(s.3) imply that  
\[
   B_\frak{g}\bigl(S-\operatorname{Ad}z(S),[X,Y]\bigr)
   =B_\frak{g}(S,[X,Y])-B_\frak{g}\bigl(S,[\operatorname{Ad}z^{-1}(X),\operatorname{Ad}z^{-1}(Y)]\bigr)
   =\omega(X,Y)-\omega\bigl(\operatorname{Ad}z^{-1}(X),\operatorname{Ad}z^{-1}(Y)\bigr)
   =0.
\] 
   Accordingly one has $S-\operatorname{Ad}z(S)=0$ because $\frak{g}=[\frak{g},\frak{g}]$ and $B_\frak{g}$ is non-degenerate. 
   Thus it turns out that $z\in C_G(S)$, and $H\subset C_G(S)$.      
\end{proof}

   Proposition \ref{prop-9.2.2} tells us that homogeneous symplectic manifolds of semisimple Lie groups are essentially adjoint orbits. 
   The converse also holds:
\begin{lemma}\label{lem-9.2.3}
   Let $G$ be a real semisimple Lie group which satisfies the second countability axiom, let $S$ be a given element of $\frak{g}$, and let $H$ be a subgroup of $G$ such that 
\[
   C_G(S)_0\subset H\subset C_G(S).
\] 
   Then, $H$ is a closed subgroup of $G$, and there exists a unique $G$-invariant symplectic form $\Omega$ on $G/H$ such that $B_\frak{g}(S,[X,Y])=\Omega_o\bigl((d\pi)_eX_e,(d\pi)_eY_e\bigr)$ for all $X,Y\in\frak{g}$. 
   Here we identify the real constants with the real-valued constant functions on $G$. 
\end{lemma}
\begin{proof}
   First, we confirm that the subgroup $H$ is a closed subset of $G$.
   Since $C_G(S)_0$ is an open subset of $C_G(S)$ and $H=\bigcup_{h\in H}L_h(C_G(S)_0)$, we see that $H$ is an open subgroup of $C_G(S)$. 
   Hence $H$ is closed in $C_G(S)$; besides, $C_G(S)$ is closed in $G$. 
   So, $H$ is a closed subset of $G$.
   At this stage $\frak{h}=\frak{c}_\frak{g}(S)$ follows from $C_G(S)_0\subset H\subset C_G(S)$.\par
   
   Next, we show the existence of $\Omega$. 
   Define a skew-symmetric bilinear form $\omega:\frak{g}\times\frak{g}\to\mathbb{R}$ by 
\[
   \mbox{$\omega(X,Y):=B_\frak{g}(S,[X,Y])$ for $X,Y\in\frak{g}$}.
\]   
   Then, this $\omega$ satisfies the (s.1), (s.2) and (s.3) in Theorem \ref{thm-9.1.1}, because of the Jacobi identity, $\frak{h}=\frak{c}_\frak{g}(S)$ and $H\subset C_G(S)$.
   Consequently, Theorem \ref{thm-9.1.1}-(II) provides us with a unique $G$-invariant symplectic form $\Omega$ so that $B_\frak{g}(S,[X,Y])=\omega(X,Y)=\Omega_o\bigl((d\pi)_eX_e,(d\pi)_eY_e\bigr)$ for $X,Y\in\frak{g}$.  
\end{proof}

\section{An appendix (an orbit space)}\label{sec-9.3}
   The direct product group $GL(1,\mathbb{R})\times SL(2,\mathbb{R})$ acts on $\frak{sl}(2,\mathbb{R})$ by 
\[
   \bigl(GL(1,\mathbb{R})\times SL(2,\mathbb{R})\bigr)\times\frak{sl}(2,\mathbb{R})\ni\bigl((\lambda,g),X\bigr)\mapsto\lambda\operatorname{Ad}g(X)\in\frak{sl}(2,\mathbb{R}).
\]   
   First, let us calculate this orbit space $\frak{sl}(2,\mathbb{R})/(GL(1,\mathbb{R})\times SL(2,\mathbb{R}))$. 
   For a non-zero, element $X=\begin{pmatrix} a & b\\ c & -a\end{pmatrix}\in\frak{sl}(2,\mathbb{R})$ we investigate the following three cases individually:
\begin{center}
\begin{tabular}{lll}
   (k) $0<\det X$, & (a) $0>\det X$, & (n) $0=\det X$.
\end{tabular}
\end{center}
\begin{enumerate}
\item[] 
   Case (k) $0<\det X=-a^2-bc$. 
   Setting 
\[
   (\lambda,g):=\begin{cases}
   \Big(1/\sqrt{-a^2-bc},\begin{pmatrix} \sqrt{c}/\sqrt[4]{-a^2-bc} & -a/(\sqrt{c}\sqrt[4]{-a^2-bc})\\ 0 & \sqrt[4]{-a^2-bc}/\sqrt{c}\end{pmatrix}\Big)
   & \mbox{if $c>0$},\\
   \Big(-1/\sqrt{-a^2-bc},\begin{pmatrix} \sqrt{-c}/\sqrt[4]{-a^2-bc} & a/(\sqrt{-c}\sqrt[4]{-a^2-bc})\\ 0 & \sqrt[4]{-a^2-bc}/\sqrt{-c}\end{pmatrix}\Big)
   & \mbox{if $c<0$},
   \end{cases}
\] 
we have $(\lambda,g)\in GL(1,\mathbb{R})\times SL(2,\mathbb{R})$ and $\lambda\operatorname{Ad}g(X)=\begin{pmatrix} 0 & -1\\ 1 & 0\end{pmatrix}$.
\item[]
   Case (a) $0>\det X=-a^2-bc$.  
   Setting 
\[
   (\lambda,g):=\Big(-1/\sqrt{a^2+bc},\begin{pmatrix} 1 & -(a+\sqrt{a^2+bc})/c\\ c/(2\sqrt{a^2+bc}) & (\sqrt{a^2+bc}-a)/(2\sqrt{a^2+bc})\end{pmatrix}\Big),
\]
we have $(\lambda,g)\in GL(1,\mathbb{R})\times SL(2,\mathbb{R})$ and $\lambda\operatorname{Ad}g(X)=\begin{pmatrix} 1 & 0\\ 0 & -1\end{pmatrix}$.
\item[]
   Case (n) $0=\det X=-a^2-bc$.
   Setting
\[
   (\lambda,g):=\begin{cases}
   \Big(-1/c,\begin{pmatrix} c & 1-a\\ -1 & a/c\end{pmatrix}\Big)
   & \mbox{if $c\neq0$},\\ 
   \Big(1/b,\begin{pmatrix} 1 & 0\\ 0 & 1\end{pmatrix}\Big)
   & \mbox{if $c=0$},   
   \end{cases}
\]    
we see that $(\lambda,g)\in GL(1,\mathbb{R})\times SL(2,\mathbb{R})$ and $\lambda\operatorname{Ad}g(X)=\begin{pmatrix} 0 & 1\\ 0 & 0\end{pmatrix}$. 
   Remark here that $a=0$ and $b\neq 0$ if $c=0$.
\end{enumerate}
   Consequently the orbit space $\frak{sl}(2,\mathbb{R})/(GL(1,\mathbb{R})\times SL(2,\mathbb{R}))$ is as follows:
\begin{equation}\label{eq-9.3.1}
   \frak{sl}(2,\mathbb{R})/(GL(1,\mathbb{R})\times SL(2,\mathbb{R}))
   =\{[K], [A], [N], [O_2]\},
\end{equation}
where $K:=\begin{pmatrix} 0 & -1\\ 1 & 0\end{pmatrix}$, $A:=\begin{pmatrix} 1 & 0\\ 0 & -1\end{pmatrix}$, $N:=\begin{pmatrix} 0 & 1\\ 0 & 0\end{pmatrix}$ and $O_2:=\begin{pmatrix} 0 & 0\\ 0 & 0\end{pmatrix}$.\par

   Now, the centralizers of the above $K$, $A$, $N$ and $O_2$ in $SL(2,\mathbb{R})$ are 
\[
\begin{array}{llll}
    C_{SL(2,\mathbb{R})}(K)=SO(2), 
  & C_{SL(2,\mathbb{R})}(A)=S(GL(1,\mathbb{R})\times GL(1,\mathbb{R})),
  & C_{SL(2,\mathbb{R})}(N)=\mathbb{R}\times\mathbb{Z}_2,
  & C_{SL(2,\mathbb{R})}(O_2)=SL(2,\mathbb{R}),
\end{array}  
\]   
respectively.
   Accordingly \eqref{eq-9.3.1}, Proposition \ref{prop-9.2.2} and Lemma \ref{lem-9.2.3} ensure that a homogeneous symplectic manifold of $SL(2,\mathbb{R})$ is one of the following:
\begin{enumerate}[(1)]
\item
   $SL(2,\mathbb{R})/SO(2)$ $*$ the open unit disk in $\mathbb{C}$, 
\item
   $SL(2,\mathbb{R})/S(GL(1,\mathbb{R})\times GL(1,\mathbb{R}))$ $*$ a hyperboloid of one sheet, 
\item
   $SL(2,\mathbb{R})/S(GL(1,\mathbb{R})\times GL(1,\mathbb{R}))_0$ $*$ a covering space of (2), 
\item
   $SL(2,\mathbb{R})/(\mathbb{R}\times\mathbb{Z}_2)$ $*$ the light cone in the $3$-dimensional Lorentz-Minkowski space $\mathbb{R}^3_1$,
\item
   $SL(2,\mathbb{R})/\mathbb{R}$ $*$ a covering space of (4), 
\item
   $SL(2,\mathbb{R})/SL(2,\mathbb{R})$ $*$ $0$-dimensional manifold.          
\end{enumerate}

\chapter{Homogeneous pseudo-K\"{a}hler manifolds}\label{ch-10}
   It is known that elliptic (adjoint) orbits can be geometrically characterized as follows:
\begin{quote}
   Any elliptic orbit $G/C_G(T)$ is a homogeneous pseudo-K\"{a}hler manifold of $G$. 
   Conversely, a homogeneous pseudo-K\"{a}hler manifold $M$ of $G$ is an elliptic orbit.
   cf.\ Dorfmeister-Guan \cite{DG1}, \cite{DG2}.
\end{quote}
   In this chapter we confirm this fact.
\begin{remark}\label{rem-10.0.1}
   We consider a K\"{a}hler manifold to be one of the pseudo-K\"{a}hler manifolds.
\end{remark}

\section{Projectable vector fields}\label{sec-10.1}
   The setting of Section \ref{sec-10.1} is as follows:
 \begin{itemize}
\item
   $G$ is a Lie group which satisfies the second countability axiom,
\item 
   $H$ is a closed subgroup of $G$, 
\item 
   $\pi$ is the projection of $G$ onto $G/H$.   
\end{itemize} 
   The homogeneous space $G/H$ is an $n$-dimensional real analytic manifold in view of Theorem \ref{thm-1.1.2}.\par

   In the next section we will prove Theorem \ref{thm-10.2.2}. 
   For this reason we need to know some properties of projectable vector fields. 
   Here, a smooth vector filed $V$ on $G$ is said to be {\it projectable},\index{projectable vector field@projectable vector field\dotfill} if there exists an $A\in\frak{X}(G/H)$ such that 
\[
   \mbox{$(d\pi)_gV_g=A_{\pi(g)}$ for all $g\in G$}
\]   
(i.e., $V$ is $\pi$-related to $A$), where $\frak{X}(G/H)$ stands for the real Lie algebra of smooth vector fields on $G/H$. 
   This $A$ is uniquely determined by $V$ since $\pi:G\to G/H$ is surjective. 
   So, we write $\pi_*V$ for $A$.

\begin{lemma}\label{lem-10.1.1}
\begin{enumerate}
\item[]
\item[{\rm (i)}]
   Let $V,W\in\frak{X}(G)$ be projectable, and let $\lambda,\mu\in\mathbb{R}$.
   Then,
   \begin{enumerate}
   \item[{\rm (i.1)}]
   $\lambda V+\mu W$ is a projectable vector field on $G$, and $\pi_*(\lambda V+\mu W)=\lambda(\pi_*V)+\mu(\pi_*W)$,
   \item[{\rm (i.2)}] 
   $[V,W]$ is  a projectable vector field on $G$, and $\pi_*[V,W]=[\pi_*V,\pi_*W]$.
   \end{enumerate}
\item[{\rm (ii)}] 
   For any $A\in\frak{X}(G/H)$ there exists a projectable vector field $V$ on $G$ satisfying $A=\pi_*V$.
\item[{\rm (iii)}] 
   All right invariant vector fields on $G$ are projectable. 
\end{enumerate}
\end{lemma}
\begin{proof}
   (i) follows from $V$ (resp.\ $W$) being $\pi$-related to $\pi_*V$ (resp.\ $\pi_*W$).\par

   (ii).
   Let $A\in\frak{X}(G/H)$. 
   Our aim is to construct a $V\in\frak{X}(G)$ which is $\pi$-related to $A$. 
   There exist coordinate neighborhoods $\bigl(U_a,(y_a^1,\dots,y_a^n)\bigr)$ of class $C^\omega$ of $G/H$ and $\bigl(\pi^{-1}(U_a),(x_a^1,\dots,x_a^n,x_a^{n+1},\dots,x_a^N)\bigr)$ of class $C^\omega$ of $G$ ($a\in\Lambda$) such that 
\begin{enumerate}
\item[(1)]
   $G/H=\bigcup_{a\in\Lambda}U_a$,
\item[(2)]
   $x_a^i=y_a^i\circ\pi$ on $\pi^{-1}(U_a)$ ($a\in\Lambda$, $1\leq i\leq n$),
\item[(3)]
   $\displaystyle{\dfrac{\partial}{\partial x_b^j}=\sum_{i=1}^n\dfrac{\partial x_a^i}{\partial x_b^j}\dfrac{\partial}{\partial x_a^i}}$ ($1\leq j\leq n$), $\displaystyle{\dfrac{\partial}{\partial x_b^r}=\sum_{s=n+1}^N\dfrac{\partial x_a^s}{\partial x_b^r}\dfrac{\partial}{\partial x_a^s}}$ ($n+1\leq r\leq N$) whenever $\pi^{-1}(U_a)\cap\pi^{-1}(U_b)\neq\emptyset$,
\end{enumerate} 
because $(G,\pi,G/H)$ is a real analytic principal fiber bundle (cf.\ Section \ref{sec-1.3}).
   For $a\in\Lambda$ and $1\leq i\leq n$, we put $A_a^i:=A(y_a^i)$. 
   Then, the vector field $A$ is expressed as 
\[
   A=\sum_{i=1}^nA_a^i\dfrac{\partial}{\partial y_a^i}
\]
on each $U_a$, and so we define a smooth vector field $V_a$ on $\pi^{-1}(U_a)$ by 
\[
   V_a:=\sum_{i=1}^n(A_a^i\circ\pi)\dfrac{\partial}{\partial x_a^i} 
\]
for each $a\in\Lambda$. 
   Now, suppose that $\pi^{-1}(U_a)\cap\pi^{-1}(U_b)\neq\emptyset$ ($a,b\in\Lambda$). 
   Then, one has 
\[
\begin{split}
   (V_a)_g
  &=\sum_{i=1}^nA_a^i(\pi(g))\Big(\dfrac{\partial}{\partial x_a^i}\Big)_g
   =\sum_{i,j=1}^nA_b^j(\pi(g))\dfrac{\partial y_a^i}{\partial y_b^j}(\pi(g))\Big(\dfrac{\partial}{\partial x_a^i}\Big)_g\\
  &=\sum_{i,j=1}^nA_b^j(\pi(g))\dfrac{\partial x_a^i}{\partial x_b^j}(g)\Big(\dfrac{\partial}{\partial x_a^i}\Big)_g
   \stackrel{(3)}{=}\sum_{j=1}^nA_b^j(\pi(g))\Big(\dfrac{\partial}{\partial x_b^j}\Big)_g
   =(V_b)_g
\end{split}
\] 
for all $g\in\pi^{-1}(U_a)\cap\pi^{-1}(U_b)$, since it follows from $\sum_{i=1}^nA_a^i(\partial/\partial y_a^i)=A=\sum_{j=1}^nA_b^j(\partial/\partial y_b^j)$ that $A_a^i=\sum_{j=1}^nA_b^j(\partial y_a^i/\partial y_b^j)$ and it follows from (2) $x_a^i=y_a^i\circ\pi$ that 
\[
   \dfrac{\partial x_a^i}{\partial x_b^j}(g)
   =\dfrac{\partial(y_a^i\circ\pi)}{\partial x_b^j}(g)
   =\sum_{k=1}^n\dfrac{\partial y_a^i}{\partial y_b^k}(\pi(g))\dfrac{\partial(y_b^k\circ\pi)}{\partial x_b^j}(g)
   =\sum_{k=1}^n\dfrac{\partial y_a^i}{\partial y_b^k}(\pi(g))\delta^k_j
   =\dfrac{\partial y_a^i}{\partial y_b^j}(\pi(g)).
\]
   Consequently one can construct a smooth vector field $V$ on the whole $G=\bigcup_{a\in\Lambda}\pi^{-1}(U_a)$ from $V|_{\pi^{-1}(U_a)}:=V_a$ for $a\in\Lambda$. 
   Besides, this $V$ is $\pi$-related to $A$ by virtue of (2).\par
   
   (iii). 
   Denote by $\frak{g}'$ the real Lie algebra of right invariant vector fields on $G$.
   For $X\in\frak{g}$ we define a right invariant vector field $X'$ on $G$ and a smooth vector field $X^*$ on $G/H$ by 
\[
\begin{split}
   & \mbox{$X'_g\tilde{f}:=\dfrac{d}{dt}\Big|_{t=0}\tilde{f}\bigl(\exp(-tX)g\bigr)$ for $g\in G$ and $\tilde{f}\in\mathcal{C}^\infty(G)$},\\
   & \mbox{$X^*_pf:=\dfrac{d}{dt}\Big|_{t=0}f\bigl(\tau_{\exp(-tX)}(p)\bigr)$ for $p\in G/H$ and $f\in\mathcal{C}^\infty(G/H)$},
\end{split}   
\]
respectively (see Corollary \ref{cor-1.1.7} for $\tau_{\exp(-tX)}$).
   Then, the mapping $\frak{g}\ni X\mapsto X'\in\frak{g}'$ is a Lie algebra isomorphism, and the mapping $\frak{g}\ni X\mapsto X^*\in\frak{X}(G/H)$ is a Lie algebra homomorphism. 
   Moreover, $X'$ is $\pi$-related to $X^*$ for every $X\in\frak{g}$. 
\end{proof}

\section[Invariant complex structures on homogeneous spaces]{Invariant complex structures on homogeneous spaces and linear transformations of Lie algebras}\label{sec-10.2}

   We first prove the following lemma, and afterwards demonstrate Theorem \ref{thm-10.2.2}:
\begin{lemma}\label{lem-10.2.1}
   Let $G$ be a Lie group, let $\jmath$ be a linear transformation of $\frak{g}$, and let $\hat{\jmath}$ be a tensor field of type $(1,1)$ on $G$. 
   Suppose that $(\hat{\jmath}X)_g=\hat{\jmath}_gX_g=(\jmath X)_g$ for all $(g,X)\in G\times\frak{g}$. 
   Then, the tensor $\hat{\jmath}$ is of class $C^\infty$.
\end{lemma}
\begin{proof}
   Take a real basis $\{E_k\}_{k=1}^N$ of $\frak{g}$ and express $\jmath E_k\in\frak{g}$ as  $\jmath E_k=\sum_{\ell=1}^Nc_k{}^\ell E_\ell$, $c_k{}^\ell\in\mathbb{R}$. 
   Any vector $U\in\frak{X}(G)$ is expressed as $U=\sum_{k=1}^N\tilde{f}_kE_k$ ($\tilde{f}_k\in\mathcal{C}^\infty(G)$), and then the supposition enables us to show that for all $g\in G$ and $\tilde{h}\in\mathcal{C}^\infty(G)$,
\[
\begin{split}
   \mbox{$\bigl((\hat{\jmath}U)\tilde{h}\bigr)(g)
   =(\hat{\jmath}U)_g\tilde{h}
   =\bigl(\sum_{k=1}^N\tilde{f}_k(g)(\jmath E_k)_g\bigr)\tilde{h}$}
  &=\mbox{$\bigl(\sum_{k,\ell=1}^N\tilde{f}_k(g)c_k{}^\ell(E_\ell)_g\bigr)\tilde{h}$}\\
  &=\mbox{$\sum_{k,\ell=1}^N\tilde{f}_k(g)c_k{}^\ell(E_\ell\tilde{h})(g)
   =\bigl(\sum_{k,\ell=1}^N\tilde{f}_k\cdot c_k{}^\ell\cdot(E_\ell\tilde{h})\bigr)(g)$}.
\end{split}
\] 
   The last term is a smooth function on $G$, so the tensor $\hat{\jmath}$ is of class $C^\infty$.
\end{proof}

   Now, let us demonstrate 
\begin{theorem}[{cf.\ Koszul \cite[Paragraph 2]{Ks}}]\label{thm-10.2.2} 
   Let $G$ be a $($real$)$ Lie group which satisfies the second countability axiom, let $H$ be a closed subgroup of $G$, let $\pi$ denote the projection of $G$ onto $G/H$, and let $o:=\pi(e)$. 
   Then, the following two items {\rm (I)} and {\rm (II)} hold$:$
\begin{enumerate}
\item[{\rm (I)}] 
   Suppose the homogeneous space $G/H$ to admit a $G$-invariant complex structure $J$.
   Then, there exists a linear transformation $\jmath:\frak{g}\to\frak{g}$ satisfying the following five conditions$:$ 
   \begin{enumerate}
   \item[{\rm (c.1)}]
      $\jmath Z=0$ for all $Z\in\frak{h}$,
   \item[{\rm (c.2)}] 
      $\jmath^2 X=-X$ $(\bmod{\frak{h}})$ for all $X\in\frak{g}$,
   \item[{\rm (c.3)}] 
      $\jmath\bigl(\operatorname{Ad}z(X)\bigr)=\operatorname{Ad}z(\jmath X)$ $(\bmod{\frak{h}})$ for all $(z,X)\in H\times\frak{g}$,
   \item[{\rm (c.4)}] 
      $[\jmath X,\jmath Y]-[X,Y]-\jmath[\jmath X,Y]-\jmath[X,\jmath Y]=0$ $(\bmod{\frak{h}})$ for all $X,Y\in\frak{g}$,
   \item[{\rm (c.5)}] 
      $(d\pi)_e(\jmath X)_e=J_o\bigl((d\pi)_eX_e\bigr)$ for all $X\in\frak{g}$. 
   \end{enumerate}
\item[{\rm (II)}] 
   Suppose that there exists a linear transformation $\jmath:\frak{g}\to\frak{g}$ satisfying the above four conditions {\rm (c.1)} through {\rm (c.4)}. 
   Then, $G/H$ admits a unique $G$-invariant complex structure $J$ so that $\jmath$ is related to $J$ by {\rm (c.5)}.   
\end{enumerate}
   Here $G/H$ is an $n$-dimensional real analytic manifold in view of Theorem {\rm \ref{thm-1.1.2}}.   
\end{theorem}
\begin{proof}
   (I). 
   Let $J$ be a $G$-invariant complex structure on $G/H$.
   Take a real vector subspace $\frak{m}\subset\frak{g}$ so that 
\begin{equation}\label{eq-1}\tag*{\textcircled{1}}
   \frak{g}=\frak{m}\oplus\frak{h},
\end{equation} 
and define a surjective linear mapping $F:\frak{g}\to T_o(G/H)$ by 
\begin{equation}\label{eq-2}\tag*{\textcircled{2}}
   \mbox{$F(X):=(d\pi)_eX_e$ for $X\in\frak{g}$}.
\end{equation} 
   Then Lemma \ref{lem-1.1.13} implies that
\begin{equation}\label{eq-3}\tag*{\textcircled{3}}
   \frak{h}=\ker(F).
\end{equation} 
   From \ref{eq-1} and \ref{eq-3} we deduce that the linear mapping $F:\frak{m}\to T_o(G/H)$ is injective, so that 
\[
   \mbox{$F:\frak{m}\to T_o(G/H)$ is a linear isomorphism}
\] 
by virtue of $\dim_\mathbb{R}\frak{m}=\dim_\mathbb{R}T_o(G/H)$.
   For this reason one can define a linear mapping $\jmath:\frak{g}\to\frak{m}$ ($\subset\frak{g}$) as follows: 
\begin{equation}\label{eq-4}\tag*{\textcircled{4}}
   \mbox{$\jmath X:=(F|_\frak{m})^{-1}\bigl(J_o(F(X))\bigr)$ for $X\in\frak{g}$}.
\end{equation} 
   Let us prove that this $\jmath$ satisfies the five conditions (c.1) through (c.5), from now on.\par

   (c.1) is immediate from \ref{eq-4} and \ref{eq-3}.\par

   (c.2). 
   For any $X\in\frak{g}$ we obtain
\[
\begin{split}
   \jmath^2X
  &\stackrel{\ref{eq-4}}{=}\bigl(((F|_\frak{m})^{-1}\circ J_o\circ F)\circ((F|_\frak{m})^{-1}\circ J_o\circ F)\bigr)(X)\\
  &=(F|_\frak{m})^{-1}\bigl(F(-X)\bigr) \quad\mbox{($\because$ $F\circ(F|_\frak{m})^{-1}=\operatorname{id}$, $J_o^2=-\operatorname{id}$ on $T_o(G/H)$)}\\
  &=(F|_\frak{m})^{-1}\bigl(F(-X_m-X_h)\bigr)
   \stackrel{\ref{eq-3}}{=}(F|_\frak{m})^{-1}\bigl(F(-X_m)\bigr)
   =-X_m
   =-X \,\,(\bmod{\frak{h}})
\end{split} 
\] 
by a direct computation.
   Here we have expressed the $X\in\frak{g}=\frak{m}\oplus\frak{h}$ as $X=X_m+X_h$ ($X_m\in\frak{m}$, $X_h\in\frak{h}$).\par
   
   (c.3). 
   For any $(z,u)\in H\times T_o(G/H)$ one has 
\[
\begin{split}
   F\bigl((F|_\frak{m})^{-1}\bigl((d\tau_z)_ou\bigr)-\operatorname{Ad}z\bigl((F|_\frak{m})^{-1}u\bigr)\bigr)
  &=(d\tau_z)_ou-F\bigl( \operatorname{Ad}z\bigl((F|_\frak{m})^{-1}u\bigr)\bigr)
   \stackrel{\ref{eq-2}}{=}(d\tau_z)_ou-(d\pi)_e\bigl(\operatorname{Ad}z\bigl((F|_\frak{m})^{-1}u\bigr)\bigr)_e\\
  &=(d\tau_z)_ou-(d\tau_z)_o\bigl((d\pi)_e\bigl((F|_\frak{m})^{-1}u\bigr)_e\bigr)
   \stackrel{\ref{eq-2}}{=}(d\tau_z)_ou-(d\tau_z)_ou
   =0.
\end{split} 
\] 
   Accordingly it follows from \ref{eq-3} that 
\begin{equation}\label{eq-5}\tag*{\textcircled{5}}
   \mbox{$(F|_\frak{m})^{-1}\bigl((d\tau_z)_ou\bigr)=\operatorname{Ad}z\bigl((F|_\frak{m})^{-1}u\bigr)$ ($\bmod{\frak{h}}$) for all $(z,u)\in H\times T_o(G/H)$}.
\end{equation} 
   Therefore for any $(z,X)\in H\times\frak{g}$ we conclude 
\[
\begin{split}
   \jmath\bigl(\operatorname{Ad}z(X)\bigr)
  &\stackrel{\ref{eq-4}}{=}(F|_\frak{m})^{-1}\bigl(J_o(F(\operatorname{Ad}z(X)))\bigr)
   \stackrel{\ref{eq-2}}{=}(F|_\frak{m})^{-1}\bigl(J_o\bigl((d\pi)_e(\operatorname{Ad}z(X))_e\bigr)\bigr)
   =(F|_\frak{m})^{-1}\bigl(J_o\bigl((d\tau_z)_o((d\pi)_eX_e)\bigr)\bigr)\\
  &=(F|_\frak{m})^{-1}\bigl((d\tau_z)_o\bigl(J_o((d\pi)_eX_e)\bigr)\bigr)
   \stackrel{\ref{eq-5}}{=}\operatorname{Ad}z\bigl((F|_\frak{m})^{-1}\bigl(J_o((d\pi)_eX_e)\bigr)\bigr) \,\,(\bmod{\frak{h}})\\
  &\stackrel{\ref{eq-2}, \ref{eq-4}}{=}\operatorname{Ad}z(\jmath X) \,\,(\bmod{\frak{h}})
\end{split} 
\]
since $J$ is $G$-invariant and $\pi(z)=o$.\par

   (c.5). 
   Let us verify (c.5) before proving (c.4).
   For any $X\in\frak{g}$ one shows that 
\[
\begin{split}
   (d\pi)_e(\jmath X)_e
  &\stackrel{\ref{eq-4}}{=}(d\pi)_e\bigl((F|_\frak{m})^{-1}\bigl(J_o(F(X))\bigr)\bigr)_e
   \stackrel{\ref{eq-2}}{=}F\bigl((F|_\frak{m})^{-1}\bigl(J_o(F(X))\bigr)\bigr)
   =J_o\bigl(F(X)\bigr) \quad\mbox{($\because$ $F\circ(F|_\frak{m})^{-1}=\operatorname{id}$ on $T_o(G/H)$)}\\
  &\stackrel{\ref{eq-2}}{=}J_o\bigl((d\pi)_eX_e\bigr).
\end{split}
\]   
   Hence (c.5) holds.\par
   
   (c.4).   
   First, let us construct a smooth tensor field $\hat{\jmath}$ of type $(1,1)$ on $G$. 
   Define a linear isomorphism $\alpha:\frak{g}\to T_eG$ by 
\[
   \mbox{$\alpha(X):=X_e$ for $X\in\frak{g}$}.
\]   
   Using this $\alpha$ and the $\jmath$ in \ref{eq-4}, we define a linear mapping $\hat{\jmath}_g:T_gG\to T_gG$ ($g\in G$) by
\[
   \mbox{$\hat{\jmath}_gu:=\bigl((dL_g)_e\circ\alpha\circ\jmath\circ\alpha^{-1}\circ (dL_{g^{-1}})_g\bigr)(u)$ for $u\in T_gG$}.
\]   
   From this $\hat{\jmath}_g$ we construct a tensor field $\hat{\jmath}$ of type $(1,1)$ on $G$ as follows:
\[
   \mbox{$(\hat{\jmath}U)_g:=\hat{\jmath}_gU_g$ for $g\in G$ and $U\in\frak{X}(G)$}.
\]
   Then, it turns out that 
\begin{equation}\label{eq-6}\tag*{\textcircled{6}}
   \mbox{$(\hat{\jmath}X)_g=\hat{\jmath}_gX_g=(\jmath X)_g$ for all $(g,X)\in G\times\frak{g}$},
\end{equation} 
so that the tensor $\hat{\jmath}$ is of class $C^\infty$ in terms of Lemma \ref{lem-10.2.1}.
   Next, let us clarify a property of this $\hat{\jmath}$.
   For any $g\in G$ and $U\in\frak{X}(G)$, there exists a unique $X\in\frak{g}$ such that $U_g=X_g$, and then we have 
\[
\begin{split}
   (d\pi)_g(\hat{\jmath}U)_g
  &=(d\pi)_g\bigl(\hat{\jmath}_gX_g\bigr)
   \stackrel{\ref{eq-6}}{=}(d\pi)_g\bigl((dL_g)_e(\jmath X)_e\bigr)
   =(d\tau_g)_o\bigl((d\pi)_e(\jmath X)_e\bigr)
   \stackrel{{\rm (c.5)}}{=}(d\tau_g)_o\bigl(J_o\bigl((d\pi)_eX_e\bigr)\bigr)\\
  &=J_{\pi(g)}\bigl((d\pi)_g\bigl( (dL_g)_eX_e \bigr)\bigr)
   =J_{\pi(g)}\bigl((d\pi)_gU_g\bigr)
\end{split} 
\] 
(because $\jmath X\in\frak{g}$ and $J$ is $G$-invariant). 
   That is to say, 
\[
   \mbox{$(d\pi)_g(\hat{\jmath}U)_g=J_{\pi(g)}\bigl((d\pi)_gU_g\bigr)$ for all $(g,U)\in G\times\frak{X}(G)$}. 
\]
   Accordingly, for an arbitrary $V\in\frak{P}(G)$ it follows that  
\begin{equation}\label{eq-7}\tag*{\textcircled{7}}
   \mbox{$\hat{\jmath}V$ is a projectable vector field on $G$, and $\pi_*(\hat{\jmath}V)=J(\pi_*V)$},
\end{equation} 
where $\frak{P}(G)$ denotes the Lie subalgebra of $\frak{X}(G)$ generated by projectable vector fields on $G$.
   Now, let us define a skew-symmetric, smooth tensor field $S$ of type $(1,2)$ on $G/H$ and a skew-symmetric (real) bilinear mapping $\hat{\frak{s}}:\frak{X}(G)\times\frak{X}(G)\to\frak{X}(G)$ by 
\[
\left\{
\begin{split}
   & \mbox{$S(A_1,A_2):=[JA_1,JA_2]-[A_1,A_2]-J[JA_1,A_2]-J[A_1,JA_2]$ for $A_1,A_2\in\frak{X}(G/H)$},\\
   & \mbox{$\hat{\frak{s}}(U_1,U_2):=[\hat{\jmath}U_1,\hat{\jmath}U_2]-[U_1,U_2]-\hat{\jmath}[\hat{\jmath}U_1,U_2]-\hat{\jmath}[U_1,\hat{\jmath}U_2]$ for $U_1,U_2\in\frak{X}(G)$},
\end{split}\right.
\]
respectively.
   Then, \ref{eq-7} and Lemma \ref{lem-10.1.1}-(i) enable us to assert that   
\[
   \mbox{$S(\pi_*V,\pi_*W)=\pi_*\bigl(\hat{\frak{s}}(V,W)\bigr)$ for all $V,W\in\frak{P}(G)$}.
\] 
   Furthermore, since the Nijenhuis tensor $S$ of $J$ vanishes we have $\pi_*\bigl(\hat{\frak{s}}(V,W)\bigr)=0$, and then 
\begin{equation}\label{eq-8}\tag*{\textcircled{8}}
   \hat{\frak{s}}(V,W)\in\mathcal{C}^\infty(G)\frak{h}
\end{equation} 
for all $V,W\in\frak{P}(G)$.\footnote{Let us show \ref{eq-8} for the sake of completeness. 
   Take real bases $\{E_i\}_{i=1}^n$ of $\frak{m}$ and $\{E_s\}_{s=n+1}^N$ of $\frak{h}$. 
   Then, $\{(E_k)_g\}_{k=1}^N$ is a real basis of $T_gG$ for each $g\in G$; and any vector $U\in\frak{X}(G)$ is expressed as $U=\sum_{k=1}^N\tilde{f}_kE_k$ ($\tilde{f}_k\in\mathcal{C}^\infty(G)$). 
   If it is projectable and $\pi_*U=0$, then it follows that $0=(d\pi)_gU_g=\sum_{k=1}^N\tilde{f}_k(g)((d\pi)_g(E_k)_g)=\sum_{i=1}^n\tilde{f}_i(g)((d\pi)_g(E_i)_g)$ for all $g\in G$, so that $\tilde{f}_1=\cdots=\tilde{f}_n=0$ because $(d\pi)_g(E_1)_g,\dots,(d\pi)_g(E_n)_g$ is linearly independent for each $g\in G$.
   Hence $U=\sum_{s=k+1}^N\tilde{f}_sE_s\in\mathcal{C}^\infty(G)\frak{h}$ if it is projectable and $\pi_*U=0$.}
   Here $\mathcal{C}^\infty(G)\frak{h}$ stands for the submodule of $\frak{X}(G)$ generated by smooth functions $\tilde{f}:G\to\mathbb{R}$ and vectors $Y\in\frak{h}$. 
   From \ref{eq-8} we are going to conclude that $\hat{\frak{s}}(U_1,U_2)\in\mathcal{C}^\infty(G)\frak{h}$ for all $U_1,U_2\in\frak{X}(G)$. 
   Denote by $\frak{g}'$ the Lie algebra of right invariant vector fields on $G$, and recall that $\frak{g}'\subset\frak{P}(G)$ (cf.\ Lemma \ref{lem-10.1.1}-(iii)). 
   On the one hand; \ref{eq-8} yields 
\[
   \mbox{$\hat{\frak{s}}(X_1',X_2')\in\mathcal{C}^\infty(G)\frak{h}$ for all $X_1',X_2'\in\frak{g}'$}.
\]   
   On the other hand; for any $X'\in\frak{g}'$ it follows from \ref{eq-7} and $J^2=-\operatorname{id}$ that $\pi_*(\hat{\jmath}^2X')=\pi_*(-X')$, so that $\hat{\jmath}^2X'+X'\in\mathcal{C}^\infty(G)\frak{h}$. 
   Thus for given $X_1',X_2'\in\frak{g}'$ and smooth function $\tilde{f}:G\to\mathbb{R}$ we have 
\[
   \hat{\frak{s}}(\tilde{f}X_1',X_2')=\tilde{f}\hat{\frak{s}}(X_1',X_2')+(X_2'\tilde{f})(X_1'+\hat{\jmath}^2X_1')=\tilde{f}\hat{\frak{s}}(X_1',X_2') \,\,(\bmod{\,}\mathcal{C}^\infty(G)\frak{h}).
\]
   Consequently we show that 
\begin{equation}\label{eq-9}\tag*{\textcircled{9}}
   \mbox{$\hat{\frak{s}}(U_1,U_2)\in\mathcal{C}^\infty(G)\frak{h}$ for all $U_1,U_2\in\frak{X}(G)$} 
\end{equation} 
because $\hat{\frak{s}}:\frak{X}(G)\times\frak{X}(G)\to\frak{X}(G)$ is skew-symmetric bilinear and $\frak{X}(G)$ is generated by smooth functions $\tilde{f}:G\to\mathbb{R}$ and elements $X'\in\frak{g}'$.  
   For any $X,Y\in\frak{g}$ ($\subset\frak{X}(G)$), in view of \ref{eq-9} one sees that 
\[
   \mathcal{C}^\infty(G)\frak{h}
   \ni\hat{\frak{s}}(X,Y)
   =[\hat{\jmath}X,\hat{\jmath}Y]-[X,Y]-\hat{\jmath}[\hat{\jmath}X,Y]-\hat{\jmath}[X,\hat{\jmath}Y]
   \stackrel{\ref{eq-6}}{=}[\jmath X,\jmath Y]-[X,Y]-\jmath[\jmath X,Y]-\jmath[X,\jmath Y]
   \in\frak{g},
\]   
and therefore $[\jmath X,\jmath Y]-[X,Y]-\jmath[\jmath X,Y]-\jmath[X,\jmath Y]\in\bigl(\frak{g}\cap\mathcal{C}^\infty(G)\frak{h}\bigr)\subset\frak{h}$.
   Hence (c.4) holds. 
   This completes the proof of (I).\par

   (II). 
   Now, let us suppose that a linear transformation $\jmath:\frak{g}\to\frak{g}$ satisfies the following four conditions: 
\begin{enumerate}
\item[(c.1)]
   $\jmath Z=0$ for all $Z\in\frak{h}$,
\item[(c.2)] 
   $\jmath^2 X=-X$ $(\bmod{\frak{h}})$ for all $X\in\frak{g}$,
\item[(c.3)] 
   $\jmath\bigl(\operatorname{Ad}z(X)\bigr)=\operatorname{Ad}z(\jmath X)$ $(\bmod{\frak{h}})$ for all $(z,X)\in H\times\frak{g}$,
\item[(c.4)] 
   $[\jmath X,\jmath Y]-[X,Y]-\jmath[\jmath X,Y]-\jmath[X,\jmath Y]=0$ $(\bmod{\frak{h}})$ for all $X,Y\in\frak{g}$.
\end{enumerate}
   We want to construct a $G$-invariant complex structure $J$ on $G/H$ from this $\jmath$.
   For a vector $u\in T_o(G/H)$, we choose an $X\in\frak{g}$ so that $u=(d\pi)_eX_e$, and put
\begin{equation}\label{eq-a}\tag{a} 
   J_ou=J_o\bigl((d\pi)_eX_e\bigr):=(d\pi)_e(\jmath X)_e.   
\end{equation} 
   Lemma \ref{lem-1.1.13} and (c.1) assure that this \eqref{eq-a} is independent of the choice of $X$ because $\jmath:\frak{g}\to\frak{g}$ is linear. 
   Thus $J_o$ is a linear transformation of the vector space $T_o(G/H)$. 
   Moreover, (c.2) and Lemma \ref{lem-1.1.13} imply that $(J_o)^2=-\operatorname{id}$ on $T_o(G/H)$. 
   Using this $J_o$ we define a complex structure $J_{\pi(g)}$ on $T_{\pi(g)}(G/H)$ ($g\in G$) by
\begin{equation}\label{eq-b}\tag{b} 
   \mbox{$J_{\pi(g)}w:=(d\tau_g)_o\bigl(J_o((d\tau_{g^{-1}})_{\pi(g)}w)\bigr)$ for $w\in T_{\pi(g)}(G/H)$}.
\end{equation} 
   This is well-defined in terms of \eqref{eq-a}, (c.3) and Lemma \ref{lem-1.1.13}; and besides, it is immediate from \eqref{eq-b} that $J$ is $G$-invariant. 
   Therefore one can assert that the $J$ is a $G$-invariant complex structure on $G/H$, if $J$ is of class $C^\infty$ and its Nijenhuis tensor $S$ is vanishes.\par

   (class $C^\infty$). 
   Let us prove that the tensor $J$ is of class $C^\infty$. 
   The arguments below will be similar to the arguments in the latter half of the proof of (I). 
   Define a linear mapping $\tilde{\jmath}_g:T_gG\to T_gG$ ($g\in G$) by
\[
   \mbox{$\tilde{\jmath}_gu:=\bigl((dL_g)_e\circ\alpha\circ\jmath\circ\alpha^{-1}\circ(dL_{g^{-1}})_g\bigr)(u)$ for $u\in T_gG$},
\]
and define a tensor field $\tilde{\jmath}$ of type $(1,1)$ on $G$ by
\[
   \mbox{$(\tilde{\jmath}U)_g:=\tilde{\jmath}_gU_g$ for $g\in G$ and $U\in\frak{X}(G)$}.
\]
   Then, it turns out that 
\begin{equation}\label{eq-c}\tag{c} 
   \mbox{$(\tilde{\jmath}X)_g=\tilde{\jmath}_gX_g=(\jmath X)_g$ for all $(g,X)\in G\times\frak{g}$},
\end{equation} 
so that the tensor $\tilde{\jmath}$ is of class $C^\infty$ due to Lemma \ref{lem-10.2.1}. 
   Moreover, it follows from \eqref{eq-c} and (c.2) that $\tilde{\jmath}^2X=-X$ ($\bmod\frak{h}$) for all $X\in\frak{g}$, and hence 
\begin{equation}\label{eq-d}\tag{d} 
   \mbox{$\tilde{\jmath}^2U=-U$ ($\bmod{\,}\mathcal{C}^\infty(G)\frak{h}$) for all $U\in\frak{X}(G)$}
\end{equation} 
because $\frak{X}(G)$ is generated by smooth functions $\tilde{f}:G\to\mathbb{R}$ and vectors $X\in\frak{g}$. 
   For any $g\in G$ and $U\in\frak{X}(G)$, there exists a unique $X\in\frak{g}$ such that $U_g=X_g$, and then 
\[
\begin{split}
   (d\pi)_g(\tilde{\jmath}U)_g
  &=(d\pi)_g\bigl(\tilde{\jmath}_gX_g\bigr)
   \stackrel{\eqref{eq-c}}{=}(d\pi)_g\bigl((dL_g)_e(\jmath X)_e\bigr)
   =(d\tau_g)_o\bigl((d\pi)_e(\jmath X)_e\bigr)
   \stackrel{\eqref{eq-a}}{=}(d\tau_g)_o\bigl(J_o\bigl((d\pi)_eX_e\bigr)\bigr)\\
  &=(d\tau_g)_o\bigl(J_o\bigl((d\tau_{g^{-1}})_{\pi(g)}((d\pi)_gX_g)\bigr)\bigr)
   \stackrel{\eqref{eq-b}}{=}J_{\pi(g)}\bigl((d\pi)_gX_g\bigr)
   =J_{\pi(g)}\bigl((d\pi)_gU_g\bigr)
\end{split}
\] 
because $\jmath X\in\frak{g}$. 
   Accordingly, for an arbitrary $V\in\frak{P}(G)$ we assert that  
\begin{equation}\label{eq-e}\tag{e}
   \mbox{$\tilde{\jmath}V$ is a projectable vector field on $G$, and $\pi_*(\tilde{\jmath}V)=J(\pi_*V)$}.
\end{equation} 
   Now, for each point $p\in G/H$, one can find a coordinate neighborhood $\bigl(U,(y^1,\dots,y^n)\bigr)$ of class $C^\omega$ of $G/H$ and a coordinate neighborhood $\bigl(\pi^{-1}(U),(x^1,\dots,x^n,x^{n+1},\dots,x^N)\bigr)$ of class $C^\omega$ of $G$ such that $p\in U$ and $x^i=y^i\circ\pi$ on $\pi^{-1}(U)$ for all $1\leq i\leq n$; moreover, there exists a real analytic mapping $\sigma:U\to G$ such that $\pi\bigl(\sigma(q)\bigr)=q$ for all $q\in U$. 
   Then, $x^i=y^i\circ\pi$ and \eqref{eq-e} yield  
\[
   J\Big(\dfrac{\partial}{\partial y^i}\Big)
   =J\Big(\pi_*\Big(\dfrac{\partial}{\partial x^i}\Big)\Big)
   =\pi_*\Big(\tilde{\jmath}\Big(\dfrac{\partial}{\partial x^i}\Big)\Big)
\] 
for all $1\leq i\leq n$.
   This and $x^i=y^i\circ\pi$ imply that $J_i{}^j\circ\pi=\tilde{\jmath}_i{}^j$ on $\pi^{-1}(U)$ for all $1\leq i,j\leq n$, where $J(\partial/\partial y^i)=\sum_{j=1}^nJ_i{}^j(\partial/\partial y^j)$ and $\tilde{\jmath}(\partial/\partial x^i)=\sum_{k=1}^N\tilde{\jmath}_i{}^k(\partial/\partial x^k)$.
   Furthermore, it follows from $\pi\circ\sigma=\operatorname{id}$ that  
\[
   \mbox{$J_i{}^j=\tilde{\jmath}_i{}^j\circ\sigma$ on $U$ ($1\leq i,j\leq n$)}.
\] 
   Consequently the tensor $J$ is of class $C^\infty$, since $\sigma:U\to\pi^{-1}(U)$ is real analytic and $\tilde{\jmath}_i{}^j:\pi^{-1}(U)\to\mathbb{R}$ is smooth.\par

   ($S=0$). 
   Let us show that the Nijenhuis tensor $S$ of $J$ vanishes.
   For any $X_1,X_2\in\frak{g}$ we obtain 
\[
   [\tilde{\jmath}X_1,\tilde{\jmath}X_2]-[X_1,X_2]-\tilde{\jmath}[\tilde{\jmath}X_1,X_2]-\tilde{\jmath}[X_1,\tilde{\jmath}X_2]
   \stackrel{\eqref{eq-c}}{=}[\jmath X_1,\jmath X_2]-[X_1,X_2]-\jmath[\jmath X_1,X_2]-\jmath[X_1,\jmath X_2]
   \in\frak{h}
\]
from (c.4). 
   Accordingly \eqref{eq-d} implies that 
\begin{equation}\label{eq-f}\tag{f}
   \mbox{$[\tilde{\jmath}U_1,\tilde{\jmath}U_2]-[U_1,U_2]-\tilde{\jmath}[\tilde{\jmath}U_1,U_2]-\tilde{\jmath}[U_1,\tilde{\jmath}U_2]\in\mathcal{C}^\infty(G)\frak{h}$ for all $U_1,U_2\in\frak{X}(G)$}
\end{equation} 
because $\frak{X}(G)$ is generated by smooth functions $\tilde{f}:G\to\mathbb{R}$ and vectors $X\in\frak{g}$.
   For given $A,B\in\frak{X}(G/H)$, Lemma \ref{lem-10.1.1}-(ii) enables us to find $V,W\in\frak{P}(G)$ satisfying $A=\pi_*V,B=\pi_*W$, respectively.
   Then Lemma \ref{lem-10.1.1}-(i), combined with \eqref{eq-e} and \eqref{eq-f}, yields 
\[
   S(A,B)
   =[JA,JB]-[A,B]-J[JA,B]-J[A,JB]
   =\pi_*\bigl([\tilde{\jmath}V,\tilde{\jmath}W]-[V,W]-\tilde{\jmath}[\tilde{\jmath}V,W]-\tilde{\jmath}[V,\tilde{\jmath}W]\big)
   =0.
\] 
   Consequently the $J$ in \eqref{eq-b} is a $G$-invariant complex structure on $G/H$. 
   Besides, $\jmath$ is related to $J$ by (c.5); indeed \eqref{eq-a} assures that $(d\pi)_e(\jmath X)_e=J_o\bigl((d\pi)_eX_e\bigr)$ for all $X\in\frak{g}$. 
   The uniqueness of $J$ follows from (c.5), $G$-invariability and Lemma \ref{lem-1.1.13}. 
   This completes the proof of Theorem \ref{thm-10.2.2}.             
\end{proof}

\begin{remark}\label{rem-10.2.3}
   Theorem \ref{thm-10.2.2}-(II) assures the uniqueness of $J$ for each $\jmath$; but in contrast, (I) does not assure the uniqueness of $\jmath$ for any $J$.   
\end{remark}

   Modifying Theorem \ref{thm-10.2.2} slightly, one can assure the uniqueness of $\jmath$ in Theorem \ref{thm-10.2.2}-(I).
\begin{proposition}\label{prop-10.2.4}
   In the setting of Theorem {\rm \ref{thm-10.2.2}}$;$ let  $\frak{m}$ be a real vector subspace of $\frak{g}$ so that 
\[
   \frak{g}=\frak{m}\oplus\frak{h}.
\]
   Then, the following two items {\rm (I)} and {\rm (II)} hold$:$
\begin{enumerate}
\item[{\rm (I)}] 
   Suppose the homogeneous space $G/H$ to admit a $G$-invariant complex structure $J$.
   Then, there exists a unique linear mapping $\jmath:\frak{g}\to\frak{m}$ satisfying the following five conditions$:$ 
   \begin{enumerate}
   \item[{\rm (c.1)}]
      $\jmath Z=0$ for all $Z\in\frak{h}$,
   \item[{\rm (c.2)}] 
      $\jmath^2 X=-X$ $(\bmod{\frak{h}})$ for all $X\in\frak{g}$,
   \item[{\rm (c.3)}] 
      $\jmath\bigl(\operatorname{Ad}z(X)\bigr)=\operatorname{Ad}z(\jmath X)$ $(\bmod{\frak{h}})$ for all $(z,X)\in H\times\frak{g}$,
   \item[{\rm (c.4)}] 
      $[\jmath X,\jmath Y]-[X,Y]-\jmath[\jmath X,Y]-\jmath[X,\jmath Y]=0$ $(\bmod{\frak{h}})$ for all $X,Y\in\frak{g}$,
   \item[{\rm (c.5)}] 
      $(d\pi)_e(\jmath X)_e=J_o\bigl((d\pi)_eX_e\bigr)$ for all $X\in\frak{g}$. 
   \end{enumerate}
\item[{\rm (II)}] 
   Suppose that there exists a linear mapping $\jmath:\frak{g}\to\frak{m}$ satisfying the above four conditions {\rm (c.1)} through {\rm (c.4)}. 
   Then, $G/H$ admits a unique $G$-invariant complex structure $J$ so that $\jmath$ is related to $J$ by {\rm (c.5)}.   
\end{enumerate}
\end{proposition}   
\begin{proof}
   cf.\ the proof of Theorem \ref{thm-10.2.2}.
\end{proof}

\section{Invariant pseudo-K\"{a}hlerian structures on homogeneous spaces}\label{sec-10.3}
   By Theorems \ref{thm-10.2.2} and \ref{thm-9.1.1} we conclude 
\begin{theorem}[{cf.\ Dorfmeister-Guan \cite{DG3}}]\label{thm-10.3.1} 
   Let $G$ be a $($real$)$ Lie group which satisfies the second countability axiom, let $H$ be a closed subgroup of $G$, let $\pi$ denote the projection of $G$ onto $G/H$, and let $o:=\pi(e)$. 
   Then, the following two items {\rm (I)} and {\rm (II)} hold$:$
\begin{enumerate}
\item[{\rm (I)}] 
   Suppose the homogeneous space $G/H$ to admit a $G$-invariant complex structure $J$ and a $G$-invariant symplectic form $\Omega$ such that  
\begin{equation}\label{eq-10.3.2}
   \mbox{$\Omega(JA,JB)=\Omega(A,B)$ for all $A,B\in\frak{X}(G/H)$}.
\end{equation}
   Then, there exist a linear transformation $\jmath:\frak{g}\to\frak{g}$ and a unique skew-symmetric bilinear form $\omega:\frak{g}\times\frak{g}\to\mathbb{R}$ satisfying the following ten conditions$:$ 
   \begin{enumerate}
   \item[{\rm (c.1)}]
      $\jmath Z=0$ for all $Z\in\frak{h}$,
   \item[{\rm (c.2)}] 
      $\jmath^2 X=-X$ $(\bmod{\frak{h}})$ for all $X\in\frak{g}$,
   \item[{\rm (c.3)}] 
      $\jmath\bigl(\operatorname{Ad}z(X)\bigr)=\operatorname{Ad}z(\jmath X)$ $(\bmod{\frak{h}})$ for all $(z,X)\in H\times\frak{g}$,
   \item[{\rm (c.4)}] 
      $[\jmath X,\jmath Y]-[X,Y]-\jmath[\jmath X,Y]-\jmath[X,\jmath Y]=0$ $(\bmod{\frak{h}})$ for all $X,Y\in\frak{g}$,
   \item[{\rm (c.5)}] 
      $(d\pi)_e(\jmath X)_e=J_o\bigl((d\pi)_eX_e\bigr)$ for all $X\in\frak{g};$
   \item[{\rm (s.1)}] 
      $\omega([X_1,X_2],X_3)+\omega([X_2,X_3],X_1)+\omega([X_3,X_1],X_2)=0$ for all $X_1,X_2,X_3\in\frak{g}$,
   \item[{\rm (s.2)}]
      $\frak{h}=\{Z\in\frak{g} \,|\, \mbox{$\omega(Z,X)=0$ for all $X\in\frak{g}$}\}$,       
   \item[{\rm (s.3)}]
      $\omega\bigl(\operatorname{Ad}z(X),\operatorname{Ad}z(Y)\bigr)=\omega(X,Y)$ for all $z\in H$ and $X,Y\in\frak{g}$,
   \item[{\rm (s.4)}]
      $\omega(X,Y)=\Omega_o\bigl((d\pi)_eX_e,(d\pi)_eY_e\bigr)$ for all $X,Y\in\frak{g};$
   \item[{\rm (c.s)}]
      $\omega(\jmath X,\jmath Y)=\omega(X,Y)$ for all $X,Y\in\frak{g}$.   
   \end{enumerate}
\item[{\rm (II)}] 
   Suppose that there exist a linear transformation $\jmath:\frak{g}\to\frak{g}$ and a skew-symmetric bilinear form $\omega:\frak{g}\times\frak{g}\to\mathbb{R}$ satisfying the above eight conditions {\rm (c.1)} through {\rm (c.4)}, {\rm (s.1)} through {\rm (s.3)}, and {\rm (c.s)}. 
   Then, $G/H$ admits a unique $G$-invariant complex structure $J$ and a unique $G$-invariant symplectic form $\Omega$ so that \eqref{eq-10.3.2} holds, $\jmath$ is related to $J$ by {\rm (c.5)}, and $\omega$ is related to $\Omega$ by {\rm (s.4)}.   
\end{enumerate}
   Here $G/H$ is a real analytic manifold in view of Theorem {\rm \ref{thm-1.1.2}}, and we identify the real constants with the real-valued constant functions on $G$.   
\end{theorem}

   Here are comments on Theorem \ref{thm-10.3.1}.
\begin{remark}\label{rem-10.3.3}
\begin{enumerate}
\item[]
\item 
   By virtue of \eqref{eq-10.3.2} one can construct a $G$-invariant pseudo-K\"{a}hler metric ${\sf g}$ on $G/H$ from 
\[
   \mbox{${\sf g}(A,B):=\Omega(A,JB)$ for $A,B\in\frak{X}(G/H)$}.
\] 
\item
   We refer to Dorfmeister-Guan \cite[Section 1.2]{DG3} for Theorem \ref{thm-10.3.1}. 
   Remark that the paper \cite{DG3} has been created earlier than the paper \cite{DG1}, but \cite{DG3} is published later than \cite{DG1}.
\end{enumerate}
\end{remark}

\section[Elliptic orbits and homogeneous pseudo-K\"{a}hler manifolds]{Elliptic orbits and homogeneous pseudo-K\"{a}hler manifolds of semisimple Lie groups}\label{sec-10.4}
   In this section we will confirm that there is no essential difference between elliptic orbits and homogeneous pseudo-K\"{a}hler manifolds of semisimple Lie groups. 
   The setting of Section \ref{sec-10.4} is as follows:
\begin{itemize}
\item
   $G$ is a connected, real semisimple Lie group,
\item 
   $\frak{g}_\mathbb{C}$ is the complexification of the (real) Lie algebra $\frak{g}=\operatorname{Lie}(G)$,
\item
   $\overline{\sigma}$ is the conjugation of $\frak{g}_\mathbb{C}$ with respect to $\frak{g}$.   
\end{itemize}

\subsection{A pseudo-K\"{a}hlerian structure on an elliptic adjoint orbit}\label{subsec-10.4.1}
   The main purpose of this subsection is to prove 
\begin{proposition}\label{prop-10.4.1}
   Let $T$ be any elliptic element of $\frak{g}$, and let $L:=C_G(T)$.
   Then, the homogeneous space $G/L$ admits a $G$-invariant complex structure $J$ and a $G$-invariant symplectic form $\Omega$ such that 
\[
   \mbox{$\Omega(JA,JB)=\Omega(A,B)$ for all $A,B\in\frak{X}(G/L)$}.
\]
   Therefore $G/L$ is a simply connected, homogeneous pseudo-K\"{a}hler manifold of $G$. 
   Here $G/L$ is a real analytic manifold in view of Theorem {\rm \ref{thm-1.1.2}}, and we identify the real constants with the real-valued constant functions on $G$. 
\end{proposition}
\begin{proof}
   By Proposition \ref{prop-7.3.4} and Theorem \ref{thm-10.3.1}-(II), it is enough to show that there exist a linear transformation $\jmath:\frak{g}\to\frak{g}$ and a skew-symmetric bilinear form $\omega:\frak{g}\times\frak{g}\to\mathbb{R}$ satisfying the eight conditions (c.1) through (c.4), (s.1) through (s.3), and (c.s) in Theorem \ref{thm-10.3.1}.
   Taking the Killing form $B_\frak{g}$ of $\frak{g}$ we define a skew-symmetric bilinear form $\omega:\frak{g}\times\frak{g}\to\mathbb{R}$ by
\begin{equation}\label{eq-1}\tag*{\textcircled{1}}
   \mbox{$\omega(X,Y):=B_\frak{g}(T,[X,Y])$ for $X,Y\in\frak{g}$}.
\end{equation}
   Then, one knows that this $\omega$ satisfies the conditions (s.1) through (s.3) by the proof of Lemma \ref{lem-9.2.3}. 
   For this reason, the rest of proof is to construct a linear transformation $\jmath:\frak{g}\to\frak{g}$ satisfying the (c.1) through (c.4) and (c.s).
   We quote the notation $\frak{l}_\mathbb{C}$, $\frak{u}^\pm$ and $\frak{l}$ from Lemma \ref{lem-7.2.8}; and first define a complex linear transformation $\jmath_\mathbb{C}$ of $\frak{g}_\mathbb{C}=\frak{u}^+\oplus\frak{l}_\mathbb{C}\oplus\frak{u}^-$ by
\[
   \mbox{$\jmath_\mathbb{C}(V^++Z+V^-):=iV^++(-i)V^-$ for $V^\pm\in\frak{u}^\pm$, $Z\in\frak{l}_\mathbb{C}$},
\]
where $i=\sqrt{-1}$.
   Then, we deduce 
\begin{equation}\label{eq-2}\tag*{\textcircled{2}}
   \overline{\sigma}\circ\jmath_\mathbb{C}=\jmath_\mathbb{C}\circ\overline{\sigma}
\end{equation}
by Lemma \ref{lem-7.2.8}-(5$'$).
   Moreover,  
\begin{enumerate}
\item[(c.1)$'$]
   $\jmath_\mathbb{C}(Z)=0$ for all $Z\in\frak{l}_\mathbb{C}$.
\item[(c.2)$'$]
   For any $V^\pm\in\frak{u}^\pm$ and $Z\in\frak{l}_\mathbb{C}$ we see that $\jmath_\mathbb{C}^2(V^++Z+V^-)=-(V^++V^-)=-(V^++Z+V^-)$ ($\bmod\frak{l}_\mathbb{C}$), and thus $\jmath_\mathbb{C}^2(W)=-W$ ($\bmod\frak{l}_\mathbb{C}$) for all $W\in\frak{g}_\mathbb{C}$.
\item[(c.3)$'$] 
   Lemma \ref{lem-7.2.8}-(2$'$) implies that for every $z\in L$, $V^\pm\in\frak{u}^\pm$ and $Z\in\frak{l}_\mathbb{C}$,
\[
   \jmath_\mathbb{C}\bigl(\operatorname{Ad}z(V^++Z+V^-)\bigr)
   =i\operatorname{Ad}z(V^+)-i\operatorname{Ad}z(V^-)
   =\operatorname{Ad}z(iV^+-iV^-)
   =\operatorname{Ad}z\bigl(\jmath_\mathbb{C}(V^++Z+V^-)\bigr);
\]
and $\jmath_\mathbb{C}\bigl(\operatorname{Ad}z(W)\bigr)=\operatorname{Ad}z\bigl(\jmath_\mathbb{C}(W)\bigr)$ for all $(z,W)\in L\times\frak{g}_\mathbb{C}$.
\item[(c.4)$'$]
   For given $V_a^\pm\in\frak{u}^\pm$ and $Z_a\in\frak{l}_\mathbb{C}$ ($a=1,2$), we obtain 
\[
\begin{split}
  &[\jmath_\mathbb{C}(V_1^++Z_1+V_1^-),\jmath_\mathbb{C}(V_2^++Z_2+V_2^-)]
   -[V_1^++Z_1+V_1^-,V_2^++Z_2+V_2^-]\\
  &\qquad
   -\jmath_\mathbb{C}\bigl([\jmath_\mathbb{C}(V_1^++Z_1+V_1^-),V_2^++Z_2+V_2^-]\bigr)
   -\jmath_\mathbb{C}\bigl([V_1^++Z_1+V_1^-,\jmath_\mathbb{C}(V_2^++Z_2+V_2^-)]\bigr)\\
 =&-[Z_1,Z_2]\in\frak{l}_\mathbb{C} 
\end{split}
\] 
by a direct computation with Lemma \ref{lem-7.2.8}-(3$'$).
   So, we conclude that $[\jmath_\mathbb{C}(W_1),\jmath_\mathbb{C}(W_2)]-[W_1,W_2]-\jmath_\mathbb{C}\bigl([\jmath_\mathbb{C}(W_1),W_2]\bigr)-\jmath_\mathbb{C}\bigl([W_1,\jmath_\mathbb{C}(W_2)]\bigr)\in\frak{l}_\mathbb{C}$ for all $W_1,W_2\in\frak{g}_\mathbb{C}$.
\item[(c.s)$'$]
   Fix any $V_a^\pm\in\frak{u}^\pm$ and $Z_a\in\frak{l}_\mathbb{C}$ ($a=1,2$). 
   On the one hand; Lemma \ref{lem-7.2.8}-(3$'$), (4$'$) and $T\in\frak{l}_\mathbb{C}$ allow us to have
\begin{multline*}
   B_{\frak{g}_\mathbb{C}}\bigl(T,[\jmath_\mathbb{C}(V_1^++Z_1+V_1^-),\jmath_\mathbb{C}(V_2^++Z_2+V_2^-)]\bigr)
   =B_{\frak{g}_\mathbb{C}}\bigl(T,[iV_1^+-iV_1^-,iV_2^+-iV_2^-]\bigr)\\
   =B_{\frak{g}_\mathbb{C}}\bigl(T,[V_1^+,V_2^-]+[V_1^-,V_2^+]\bigr)
   =B_{\frak{g}_\mathbb{C}}\bigl(T,[V_1^++V_1^-,V_2^++V_2^-]\bigr).
\end{multline*}On the other hand; $[T,Z_a]=0$ and $B_{\frak{g}_\mathbb{C}}([P,Q],R)=-B_{\frak{g}_\mathbb{C}}(Q,[P,R])$ yield 
\[
\begin{split}
  &B_{\frak{g}_\mathbb{C}}\bigl(T,[V_1^++Z_1+V_1^-,V_2^++Z_2+V_2^-]\bigr)
   =-B_{\frak{g}_\mathbb{C}}\bigl([V_1^++Z_1+V_1^-,T],V_2^++Z_2+V_2^-\bigr)\\
  &=-B_{\frak{g}_\mathbb{C}}\bigl([V_1^++V_1^-,T],V_2^++Z_2+V_2^-\bigr)
   =B_{\frak{g}_\mathbb{C}}\bigl(T,[V_1^++V_1^-,V_2^++Z_2+V_2^-]\bigr)\\
  &=B_{\frak{g}_\mathbb{C}}\bigl([V_2^++Z_2+V_2^-,T],V_1^++V_1^-\bigr) 
   =B_{\frak{g}_\mathbb{C}}\bigl([V_2^++V_2^-,T],V_1^++V_1^-\bigr) 
   =B_{\frak{g}_\mathbb{C}}\bigl(T,[V_1^++V_1^-,V_2^++V_2^-]\bigr).
\end{split} 
\] 
   Consequently $B_{\frak{g}_\mathbb{C}}\bigl(T,[\jmath_\mathbb{C}(V_1^++Z_1+V_1^-),\jmath_\mathbb{C}(V_2^++Z_2+V_2^-)]\bigr)=B_{\frak{g}_\mathbb{C}}\bigl(T,[V_1^++Z_1+V_1^-,V_2^++Z_2+V_2^-]\bigr)$; and it follows that $B_{\frak{g}_\mathbb{C}}\bigl(T,[\jmath_\mathbb{C}(W_1),\jmath_\mathbb{C}(W_2)]\bigr)=B_{\frak{g}_\mathbb{C}}\bigl(T,[W_1,W_2]\bigr)$ for all $W_1,W_2\in\frak{g}_\mathbb{C}$.
\end{enumerate} 
   Accordingly $\jmath:=\jmath_\mathbb{C}|_\frak{g}$ is a real linear transformation of $\frak{g}$ and satisfies the conditions (c.1) through (c.4) and (c.s), because of $\frak{g}=\{X\in\frak{g}_\mathbb{C} \,|\, \overline{\sigma}(X)=X\}$ and $\frak{l}=\{Y\in\frak{l}_\mathbb{C} \,|\, \overline{\sigma}(Y)=Y\}$.
\end{proof} 

\begin{remark}\label{rem-10.4.2}
   Here are comments on the proof of Proposition \ref{prop-10.4.1}. 
   One can realize the linear transformation $\jmath=\jmath_\mathbb{C}|_\frak{g}$ of $\frak{g}=\frak{l}\oplus\frak{u}$ by setting
\[
   \mbox{$\jmath\bigl(Y+V+\overline{\sigma}(V)\bigr):=iV-i\overline{\sigma}(V)$ for $Y\in\frak{l}$ and $V\in\frak{u}^+$}.
\]
   cf.\ Lemma \ref{lem-7.2.8}-(i), (iii).   
\end{remark}

\subsection{A realization of homogeneous pseudo-K\"{a}hler manifolds as elliptic adjoint orbits}\label{subsec-10.4.2}
   We are going to inductively prove that any homogeneous pseudo-K\"{a}hler manifold of $G$ is an elliptic orbit of $G$ (see Theorem \ref{thm-10.4.7}).\par

   Let $H$ be a closed subgroup of the connected real semisimple Lie group $G$. 
   Suppose that the homogeneous space $G/H$ admits a $G$-invariant complex structure $J$ and a $G$-invariant symplectic form $\Omega$ such that 
\[
   \mbox{$\Omega(JA,JB)=\Omega(A,B)$ for all $A,B\in\frak{X}(G/H)$}.
\] 
   Then, there exist a linear transformation $\jmath:\frak{g}\to\frak{g}$ and a unique skew-symmetric bilinear form $\omega:\frak{g}\times\frak{g}\to\mathbb{R}$ satisfying the ten conditions in Theorem \ref{thm-10.3.1}-(I).
   Moreover, there exists a unique $S\in\frak{g}$ such that 
\begin{enumerate}
\item[(i)]
   $\omega(X,Y)=B_\frak{g}(S,[X,Y])$ for all $X,Y\in\frak{g}$,
\item[(ii)] 
   $C_G(S)_0\subset H\subset C_G(S)$
\end{enumerate}
by Lemma \ref{lem-9.2.1} and Proposition \ref{prop-9.2.2}. 
   Here $B_\frak{g}$ is the Killing form of $\frak{g}$.
   Let us remark $\frak{h}=\frak{c}_\frak{g}(S)$, denote by $\jmath_\mathbb{C}$ the complex linear extension of $\jmath$ to $\frak{g}_\mathbb{C}$, and prove  
\begin{lemma}\label{lem-10.4.3}
   Let $\frak{h}_\mathbb{C}:=\frak{c}_{\frak{g}_\mathbb{C}}(S)$, $\frak{q}^+:=\{V\in\frak{g}_\mathbb{C} \,|\, \mbox{$\jmath_\mathbb{C}(V)=iV$ $(\bmod{\frak{h}_\mathbb{C}})$}\}$ and $\frak{q}^-:=\{V\in\frak{g}_\mathbb{C} \,|\, \mbox{$\jmath_\mathbb{C}(V)=-iV$ $(\bmod{\frak{h}_\mathbb{C}})$}\}$.
   Then, it follows that for each $s=\pm$,
\begin{enumerate}
\item[{\rm (1)}]
   $[\frak{h}_\mathbb{C},\frak{q}^s]\subset\frak{q}^s;$ $\operatorname{Ad}z(\frak{q}^s)\subset\frak{q}^s$ for all $z\in H$,
\item[{\rm (2)}]
   $\frak{q}^s$ is a complex subalgebra of $\frak{g}_\mathbb{C}$, 
\item[{\rm (3)}] 
   $\frak{q}^+\cap\frak{q}^-=\frak{h}_\mathbb{C}$,
\item[{\rm (4)}]
   $\overline{\sigma}(\frak{h}_\mathbb{C})\subset\frak{h}_\mathbb{C}$, $\overline{\sigma}(\frak{q}^+)\subset\frak{q}^-$ and $\overline{\sigma}(\frak{q}^-)\subset\frak{q}^+$, 
\item[{\rm (5)}]
   $\frak{q}^++\frak{q}^-=\frak{g}_\mathbb{C}$,
\item[{\rm (6)}] 
   $\dim_\mathbb{C}\frak{q}^s-\dim_\mathbb{C}\frak{h}_\mathbb{C}=\dim_\mathbb{C}\frak{g}_\mathbb{C}-\dim_\mathbb{C}\frak{q}^s$,
\item[{\rm (7)}]
   $\frak{h}_\mathbb{C}=\{Z\in\frak{g}_\mathbb{C} \,|\, \mbox{$B_{\frak{g}_\mathbb{C}}(S,[Z,W])=0$ for all $W\in\frak{g}_\mathbb{C}$}\}$,
\item[{\rm (8)}]
   $B_{\frak{g}_\mathbb{C}}\bigl(S,[\frak{q}^s,\frak{q}^s]\bigr)=\{0\}$.
\end{enumerate}   
\end{lemma}
\begin{proof}
   Since $\jmath:\frak{g}\to\frak{g}$ satisfies the conditions in Theorem \ref{thm-10.3.1}-(I), we conclude that the complex linear transformation $\jmath_\mathbb{C}:\frak{g}_\mathbb{C}\to\frak{g}_\mathbb{C}$ satisfies
\begin{enumerate}
   \item[(c.1)$'$]
      $\jmath_\mathbb{C}(Z)=0$ for all $Z\in\frak{h}_\mathbb{C}$,
   \item[(c.2)$'$] 
      $\jmath_\mathbb{C}^2(W)=-W$ $(\bmod{\frak{h}_\mathbb{C}})$ for all $W\in\frak{g}_\mathbb{C}$,
   \item[(c.3)$'$]
      $\jmath_\mathbb{C}\bigl(\operatorname{ad}Z(W)\bigr)=\operatorname{ad}Z\bigl(\jmath_\mathbb{C}(W)\bigr)$ $(\bmod{\frak{h}_\mathbb{C}})$ for all $(Z,W)\in\frak{h}_\mathbb{C}\times\frak{g}_\mathbb{C}$,\par
      $\jmath_\mathbb{C}\bigl(\operatorname{Ad}z(W)\bigr)=\operatorname{Ad}z\bigl(\jmath_\mathbb{C}(W)\bigr)$ $(\bmod{\frak{h}_\mathbb{C}})$ for all $(z,W)\in H\times\frak{g}_\mathbb{C}$,
   \item[(c.4)$'$] 
      $[\jmath_\mathbb{C}(W_1),\jmath_\mathbb{C}(W_2)]-[W_1,W_2]-\jmath_\mathbb{C}\bigl([\jmath_\mathbb{C}(W_1),W_2]\bigr)-\jmath_\mathbb{C}\bigl([W_1,\jmath_\mathbb{C}(W_2)]\bigr)\in\frak{h}_\mathbb{C}$ for all $W_1,W_2\in\frak{g}_\mathbb{C}$,
   \item[(c.s)$'$]
      $B_{\frak{g}_\mathbb{C}}\bigl(S,[\jmath_\mathbb{C}(W_1),\jmath_\mathbb{C}(W_2)]\bigr)=B_{\frak{g}_\mathbb{C}}\bigl(S,[W_1,W_2]\bigr)$ for all $W_1,W_2\in\frak{g}_\mathbb{C}$.
\end{enumerate}
   Note that $\frak{h}_\mathbb{C}$ is a complex subalgebra of $\frak{g}_\mathbb{C}$ and $\overline{\sigma}\circ\jmath_\mathbb{C}=\jmath_\mathbb{C}\circ\overline{\sigma}$.\par
   
   (1) is a consequence of (c.3)$'$.\par
   
   (2). 
   It is clear that $\frak{q}^s$ is a complex vector subspace of $\frak{g}_\mathbb{C}$. 
   From (1) and (c.4)$'$ we obtain $[\frak{q}^s,\frak{q}^s]\subset\frak{q}^s$. 
   Thus $\frak{q}^s$ is a complex subalgebra of $\frak{g}_\mathbb{C}$.\par
   
   (3).
   For each $V\in\frak{q}^+\cap\frak{q}^-$ there exist $Z_+,Z_-\in\frak{h}_\mathbb{C}$ such that $iV+Z_+=\jmath_\mathbb{C}(V)=-iV+Z_-$.
   Therefore one shows $V=(i/2)(Z_+-Z_-)\in\frak{h}_\mathbb{C}$, and $\frak{q}^+\cap\frak{q}^-\subset\frak{h}_\mathbb{C}$.
   The converse inclusion $\frak{h}_\mathbb{C}\subset\frak{q}^+\cap\frak{q}^-$ follows from (c.1)$'$.\par
   
   (4). 
   From $\frak{h}_\mathbb{C}=\frak{c}_{\frak{g}_\mathbb{C}}(S)$ and $\overline{\sigma}(S)=S$ we deduce that $\overline{\sigma}(\frak{h}_\mathbb{C})\subset\frak{h}_\mathbb{C}$, which leads to $\overline{\sigma}(\frak{q}^s)\subset\frak{q}^{-s}$ since $\overline{\sigma}\circ\jmath_\mathbb{C}=\jmath_\mathbb{C}\circ\overline{\sigma}$ and $\frak{q}^s=\{V\in\frak{g}_\mathbb{C} \,|\, \mbox{$\jmath_\mathbb{C}(V)=siV$ $(\bmod{\frak{h}_\mathbb{C}})$}\}$.\par
   
   (5).
   For an arbitrary $W\in\frak{g}_\mathbb{C}$, it follows from (c.2)$'$ that $W=(1/2)\bigl((W-i\jmath_\mathbb{C}(W))+(W+i\jmath_\mathbb{C}(W))\bigr)\in\frak{q}^++\frak{q}^-$, so that $\frak{g}_\mathbb{C}\subset\frak{q}^++\frak{q}^-$.
   Hence $\frak{q}^++\frak{q}^-=\frak{g}_\mathbb{C}$.\par
   
   (6). 
   A direct computation yields $\dim_\mathbb{C}\frak{g}_\mathbb{C}\stackrel{(5)}{=}\dim_\mathbb{C}(\frak{q}^++\frak{q}^-)=\dim_\mathbb{C}\frak{q}^++\dim_\mathbb{C}\frak{q}^--\dim_\mathbb{C}\frak{q}^+\cap\frak{q}^-\stackrel{(3)}{=}\dim_\mathbb{C}\frak{q}^++\dim_\mathbb{C}\frak{q}^--\dim_\mathbb{C}\frak{h}_\mathbb{C}\stackrel{(4)}{=}2\dim_\mathbb{C}\frak{q}^s-\dim_\mathbb{C}\frak{h}_\mathbb{C}$.\par
   
   (7). 
   For a given $Z\in\frak{g}_\mathbb{C}$, $Z\in\frak{h}_\mathbb{C}=\frak{c}_{\frak{g}_\mathbb{C}}(S)$ if and only if $\operatorname{ad}S(Z)=0$ if and only if $0=B_{\frak{g}_\mathbb{C}}([S,Z],W])=B_{\frak{g}_\mathbb{C}}(S,[Z,W])$ for all $W\in\frak{g}_\mathbb{C}$ (because $B_{\frak{g}_\mathbb{C}}$ is non-degenerate). 
   Thus (7) holds.\par

   (8). 
   For any $V_1,V_2\in\frak{q}^s$ ($s=\pm$), there exist $Z_1,Z_2\in\frak{h}_\mathbb{C}$ such that $\jmath_\mathbb{C}(V_1)=siV_1+Z_1$, $\jmath_\mathbb{C}(V_2)=siV_2+Z_2$. 
   Then
\[
\begin{split}
   B_{\frak{g}_\mathbb{C}}\bigl(S,[V_1,V_2]\bigr)
  &\stackrel{\mbox{\small{(c.s)$'$}}}{=}B_{\frak{g}_\mathbb{C}}\bigl(S,[\jmath_\mathbb{C}(V_1),\jmath_\mathbb{C}(V_2)]\bigr)
   =B_{\frak{g}_\mathbb{C}}\bigl(S,-s^2[V_1,V_2]+si[V_1,Z_2]+si[Z_1,V_2]+[Z_1,Z_2]\bigr)\\
  &\stackrel{(7)}{=}B_{\frak{g}_\mathbb{C}}\bigl(S,-s^2[V_1,V_2]\bigr)
   =-B_{\frak{g}_\mathbb{C}}\bigl(S,[V_1,V_2]\bigr).
\end{split} 
\]
   This implies that $B_{\frak{g}_\mathbb{C}}\bigl(S,[V_1,V_2]\bigr)=0$, and so $B_{\frak{g}_\mathbb{C}}\bigl(S,[\frak{q}^s,\frak{q}^s]\bigr)=\{0\}$.
\end{proof}

   Set $\frak{h}_\mathbb{C}$, $\frak{q}^\pm$ as in Lemma \ref{lem-10.4.3}.
   In view of Lemma \ref{lem-10.4.3}-(2), (6), (7), (8) we see that $\frak{q}^+$ is a complex subalgebra of the complex semisimple Lie algebra $\frak{g}_\mathbb{C}$ and is a weak polarization of $S$. 
   Thus, 
\begin{enumerate}
\item
   $\frak{q}^+$ is a complex parabolic subalgebra of $\frak{g}_\mathbb{C}$ and includes a complex Borel subalgebra $\frak{b}'$ of $\frak{g}_\mathbb{C}$,
\item
   $[S,\frak{q}^+]$ is an ideal of $\frak{q}^+$.
\end{enumerate} 
   cf.\ Theorem 2.2 in Ozeki-Wakimoto \cite[p.447]{OW}.
   In general, the intersection of two complex Borel subalgebras is not empty and includes a Cartan subalgebra in a complex semisimple Lie algebra. 
   Hence there exists a Cartan subalgebra $\frak{c}_\mathbb{C}'$ of $\frak{g}_\mathbb{C}$ such that $\frak{c}_\mathbb{C}'\subset\frak{b}'\cap\overline{\sigma}(\frak{b}')$, and then
\begin{equation}\label{eq-10.4.4}
   \frak{c}_\mathbb{C}'
   \subset\bigl(\frak{b}'\cap\overline{\sigma}(\frak{b}')\bigr)
   \subset(\frak{q}^+\cap\frak{q}^-)
   =\frak{h}_\mathbb{C}
   =\frak{c}_{\frak{g}_\mathbb{C}}(S) 
\end{equation}
by Lemma \ref{lem-10.4.3}-(4), (3). 
   This enables us to assert that $S\in\frak{c}_\mathbb{C}'$ and $S$ is a semisimple element of $\frak{g}$. 
   Moreover,
   
\begin{proposition}\label{prop-10.4.5}
   There exists an elliptic element $T\in\frak{g}$ such that $\frak{h}=\frak{c}_\frak{g}(T)$ and
\[
\begin{array}{lllll}
   \frak{h}_\mathbb{C}=\frak{g}^0, &
   [S,\frak{q}^+]=\bigoplus_{\lambda>0}\frak{g}^\lambda, & 
   [S,\frak{q}^-]=\bigoplus_{\lambda>0}\frak{g}^{-\lambda}, &
   \frak{q}^+=\bigoplus_{\nu\geq0}\frak{g}^\nu, &    
   \frak{q}^-=\bigoplus_{\nu\geq0}\frak{g}^{-\nu},   
\end{array}
\]
where $\frak{g}^\lambda:=\{W\in\frak{g}_\mathbb{C} \,|\, \operatorname{ad}T(W)=i\lambda W\}$ for $\lambda\in\mathbb{R}$.
   Here we refer to Lemma {\rm \ref{lem-10.4.3}} for $\frak{h}_\mathbb{C}$, $\frak{q}^\pm$.
\end{proposition}
\begin{proof}
   We are going to prepare some notation first.
   Since $S\in\frak{g}$ is semisimple, there exists a (real) Cartan subalgebra $\frak{c}\subset\frak{g}$ containing $S$. 
   Denote by $\frak{c}_\mathbb{C}$ the complex vector subspace of $\frak{g}_\mathbb{C}$ generated by $\frak{c}$, by $\triangle=\triangle(\frak{g}_\mathbb{C},\frak{c}_\mathbb{C})$ the root system of $\frak{g}_\mathbb{C}$ relative to $\frak{c}_\mathbb{C}$, by $\frak{g}_\alpha$ the root subspace of $\frak{g}_\mathbb{C}$ for $\alpha\in\triangle$, and by $H_\alpha$ ($\alpha\in\triangle$) the unique element of $\frak{c}_\mathbb{C}$ such that $\alpha(X)=B_{\frak{g}_\mathbb{C}}(H_\alpha,X)$ for all $X\in\frak{c}_\mathbb{C}$. 
   Taking $S\in\frak{c}_\mathbb{C}$ and $\overline{\sigma}(\frak{c}_\mathbb{C})=\frak{c}_\mathbb{C}$ into account, we define a symmetric closed subset $\blacktriangle\subset\triangle$ and an involutive transformation $\overline{\sigma}^*:\triangle\to\triangle$ by 
\begin{equation}\label{eq-0}\tag*{\textcircled{0}}
\begin{array}{ll}
   \blacktriangle:=\{\gamma\in\triangle \,|\, \gamma(S)=0\}, 
 & \mbox{$(\overline{\sigma}^*\alpha)(X):=\overline{\alpha(\overline{\sigma}(X))}$ for $\alpha\in\triangle$ and $X\in\frak{c}_\mathbb{C}$},
\end{array}
\end{equation}
respectively.
   Then it turns out that $\frak{g}_\mathbb{C}=\frak{c}_\mathbb{C}\oplus\bigoplus_{\alpha\in\triangle}\frak{g}_\alpha$, $\frak{c}_{\frak{g}_\mathbb{C}}(S)=\frak{c}_\mathbb{C}\oplus\bigoplus_{\gamma\in\blacktriangle}\frak{g}_\gamma$, and $\frak{h}_\mathbb{C}=\frak{c}_{\frak{g}_\mathbb{C}}(S)$ is a complex reductive Lie algebra including $\frak{c}_\mathbb{C}$.
   In addition, since $S$ is semisimple and $[S,\frak{q}^\pm]\subset\frak{q}^\pm$ we have $\frak{q}^\pm=\frak{c}_{\frak{q}^\pm}(S)\oplus\operatorname{ad}S(\frak{q}^\pm)=\frak{h}_\mathbb{C}\oplus[S,\frak{q}^\pm]$, and there exist subsets $\triangle_q^+,\triangle_q^-\subset\triangle$ such that
\begin{equation}\label{eq-1}\tag*{\textcircled{1}}
\begin{array}{lll}
   \triangle=\triangle_q^+\amalg\blacktriangle\amalg\triangle_q^-, 
   & [S,\frak{q}^+]=\bigoplus_{\delta\in\triangle_q^+}\frak{g}_\delta, 
   & [S,\frak{q}^-]=\bigoplus_{\beta\in\triangle_q^-}\frak{g}_\beta
\end{array}
\end{equation}
by $[\frak{c}_\mathbb{C},\frak{q}^\pm]\subset\frak{q}^\pm$ and Lemma \ref{lem-10.4.3}-(5), (3).
   Remark that the cardinal number $|\triangle_q^+|$ is equal to $|\triangle_q^-|$ (because $\bigoplus_{\delta\in\triangle_q^+}\frak{g}_\delta=[S,\frak{q}^+]=\overline{\sigma}([S,\frak{q}^-])=\bigoplus_{\beta\in\triangle_q^-}\frak{g}_{\overline{\sigma}^*\beta}$ follows from \ref{eq-1}, $\overline{\sigma}(S)=S$ and Lemma \ref{lem-10.4.3}-(4)), and that $\overline{\sigma}^*(\triangle_q^-)=\triangle_q^+$, $\overline{\sigma}^*(\triangle_q^+)=\triangle_q^-$ and $\overline{\sigma}^*(\blacktriangle)=\blacktriangle$.
\par

   Now, $\frak{c}_\mathbb{C}',\frak{c}_\mathbb{C}$ are Cartan subalgebras of $\frak{h}_\mathbb{C}$ and $\frak{h}_\mathbb{C}$ is a complex reductive Lie algebra. 
   cf.\ \eqref{eq-10.4.4}.
   For this reason there exists an inner automorphism $\psi$ of $\frak{h}_\mathbb{C}$ such that $\psi(\frak{c}_\mathbb{C}')=\frak{c}_\mathbb{C}$. 
   One can regard this $\psi$ as an inner automorphism of $\frak{g}_\mathbb{C}$, so $\frak{b}:=\psi(\frak{b}')$ is a complex Borel subalgebra of $\frak{g}_\mathbb{C}$. 
   Besides, it follows from $[\frak{h}_\mathbb{C},\frak{b}']\subset[\frak{h}_\mathbb{C},\frak{q}^+]\subset\frak{q}^+$ that $\psi(\frak{b}')\subset\frak{q}^+$, so that 
\begin{equation}\label{eq-2}\tag*{\textcircled{2}}
   \frak{c}_\mathbb{C}\subset\frak{b}\subset\frak{q}^+
\end{equation}
by $\frak{c}_\mathbb{C}=\psi(\frak{c}_\mathbb{C}')\subset\psi(\frak{b}')$.
   Relative to this Borel subalgebra $\frak{b}\subset\frak{g}_\mathbb{C}$, we fix the set $\triangle^+$ ($\subset\triangle(\frak{g}_\mathbb{C},\frak{c}_\mathbb{C})$) of positive roots and put $\triangle^-:=-\triangle^+$.   
   Then $\frak{b}\subset\frak{q}^+=\frak{h}_\mathbb{C}\oplus[S,\frak{q}^+]$ and \ref{eq-1} yield $\triangle^+-\blacktriangle\subset\triangle_q^+$. 
   Accordingly we conclude 
\[
\begin{array}{ll}
   \triangle^+-\blacktriangle=\triangle_q^+, & \triangle^--\blacktriangle=\triangle_q^-
\end{array}   
\] 
from $\triangle=(\triangle^+-\blacktriangle)\amalg\blacktriangle\amalg(\triangle^--\blacktriangle)$, \ref{eq-1}, $|\triangle^+-\blacktriangle|=|\triangle^--\blacktriangle|$ and $|\triangle_q^+|=|\triangle_q^-|$. 
   Summarizing the statements above we show that 
\begin{equation}\label{eq-3}\tag*{\textcircled{3}}
\left\{\begin{array}{@{}l}
   \begin{array}{lll}
     \frak{h}_\mathbb{C}=\frak{c}_\mathbb{C}\oplus\bigoplus_{\gamma\in\blacktriangle}\frak{g}_\gamma, 
   & [S,\frak{q}^+]=\bigoplus_{\delta\in\triangle^+-\blacktriangle}\frak{g}_\delta, 
   & [S,\frak{q}^-]=\bigoplus_{\beta\in\triangle^--\blacktriangle}\frak{g}_\beta,
   \end{array}\\
   \begin{array}{lllll}
   \frak{g}_\mathbb{C}=\frak{c}_{\frak{g}_\mathbb{C}}(S)\oplus\operatorname{ad}S(\frak{g}_\mathbb{C}), &\frak{c}_{\frak{g}_\mathbb{C}}(S)=\frak{h}_\mathbb{C}, & \operatorname{ad}S(\frak{g}_\mathbb{C})=[S,\frak{q}^+]\oplus[S,\frak{q}^-], & \frak{q}^+=\frak{h}_\mathbb{C}\oplus[S,\frak{q}^+], & \frak{q}^-=\frak{h}_\mathbb{C}\oplus[S,\frak{q}^-],
   \end{array}\\
   \begin{array}{lll}
   \overline{\sigma}^*(\blacktriangle)=\blacktriangle, & \overline{\sigma}^*(\triangle^+-\blacktriangle)=\triangle^--\blacktriangle, & \overline{\sigma}^*(\triangle^--\blacktriangle)=\triangle^+-\blacktriangle.
   \end{array}
   \end{array}\right.
\end{equation}
   Let us put $Z:=\sum_{\delta\in\triangle^+-\blacktriangle}H_\delta$. 
   Then $Z$ belongs to $\frak{c}_\mathbb{C}$ and it follows form $\big[\frak{h}_\mathbb{C},[S,\frak{q}^+]\big]\subset[S,\frak{q}^+]$, $\big[[S,\frak{q}^+],[S,\frak{q}^+]\big]\subset[S,\frak{q}^+]$ that for each $\alpha\in\triangle$,
\[
   \mbox{$\alpha(Z)$ is }
   \begin{cases}
    \mbox{the zero} & \mbox{if $\alpha\in\blacktriangle$},\\
    \mbox{a positive real number} & \mbox{if $\alpha\in\triangle^+-\blacktriangle$},\\
    \mbox{a negative real number} & \mbox{if $\alpha\in\triangle^--\blacktriangle$},\\    
   \end{cases}
\]
   cf.\ Corollary 5.101 in Knapp \cite[p.330]{Kn}.\footnote{
   Indeed; let $\langle\alpha,\beta\rangle:=B_{\frak{g}_\mathbb{C}}(H_\alpha,H_\beta)$ for $\alpha,\beta\in\triangle$. 
   Define $\zeta$ (resp.\ $w_\alpha$) in a similar way to \eqref{eq-8.1.3} (resp.\ \eqref{eq-8.1.4}).\par
   $\bullet$ In case of $\alpha\in\blacktriangle$, it follows from $E_\alpha-E_{-\alpha}\in\frak{h}_\mathbb{C}$ and $\big[\frak{h}_\mathbb{C},[S,\frak{q}^+]\big]\subset[S,\frak{q}^+]$ that $\operatorname{Ad}w_\alpha(\frak{g}_\delta)=\frak{g}_{\zeta([w_\alpha])\delta}$, $\zeta([w_\alpha])\delta\in\triangle^+-\blacktriangle$ for all $\delta\in\triangle^+-\blacktriangle$. 
   Therefore $\alpha(Z)=\alpha\bigl(\sum_{\delta\in\triangle^+-\blacktriangle}H_\delta\bigr)=\sum_{\delta\in\triangle^+-\blacktriangle}\langle\alpha,\delta\rangle=\sum_{\delta\in\triangle^+-\blacktriangle}\big\langle\zeta([w_\alpha])\alpha,\zeta([w_\alpha])\delta\big\rangle=-\sum_{\delta\in\triangle^+-\blacktriangle}\big\langle\alpha,\zeta([w_\alpha])\delta\big\rangle=-\alpha(Z)$. 
   This implies that $\alpha(Z)=0$.\par
   
   $\bullet$ Suppose that $\alpha\in\triangle^+-\blacktriangle$.
   If $\delta'\in\triangle^+-\blacktriangle$ and $\langle\delta',\alpha\rangle<0$, then one has $2\langle\delta',\alpha\rangle/\langle\alpha,\alpha\rangle=-1$, $-2$ or $-3$ (because of $\delta'\neq-\alpha$) and accordingly $\zeta([w_\alpha])\delta'=\delta'+\alpha$, $\delta'+2\alpha$ or $\delta'+3\alpha$.
   At any rate $\zeta([w_\alpha])\delta'$ belongs to $\triangle^+-\blacktriangle$ since $E_\alpha,E_{\delta'}\in[S,\frak{q}^+]$ and $\big[[S,\frak{q}^+],[S,\frak{q}^+]\big]\subset[S,\frak{q}^+]$.
   Besides, one shows $\big\langle\zeta([w_\alpha])\delta',\alpha\big\rangle=-\langle\delta',\alpha\rangle>0$.
   Consequently it turns out that
\[
\begin{split}
   \alpha(Z)
   &=\mbox{$\sum_{\delta'\in\triangle^+-\blacktriangle\mbox{ with }\langle\delta',\alpha\rangle<0}\langle\delta',\alpha\rangle
    +\sum_{\delta_0\in\triangle^+-\blacktriangle\mbox{ with }\langle\delta_0,\alpha\rangle=0}\langle\delta_0,\alpha\rangle
    +\sum_{\delta_a\in\triangle^+-\blacktriangle\mbox{ with }\langle\delta_a,\alpha\rangle>0}\langle\delta_a,\alpha\rangle$}\\
   &=\mbox{$\sum_{\delta'\in\triangle^+-\blacktriangle\mbox{ with }\langle\delta',\alpha\rangle<0}\big\langle\delta'+\zeta([w_\alpha])\delta',\alpha\big\rangle
    +\sum_{\delta\in\triangle^+-\blacktriangle\mbox{ with }\langle\delta,\alpha\rangle>0,\, \zeta([w_\alpha])\delta\not\in\triangle^+-\blacktriangle}\langle\delta,\alpha\rangle$}\\
   &=\mbox{$\sum_{\delta\in\triangle^+-\blacktriangle\mbox{ with }\langle\delta,\alpha\rangle>0,\, \zeta([w_\alpha])\delta\not\in\triangle^+-\blacktriangle}\langle\delta,\alpha\rangle$}
   >0.
\end{split}
\]
   Here we remark that $\alpha\in\triangle^+-\blacktriangle$, $\langle\alpha,\alpha\rangle>0$, $\zeta([w_\alpha])\alpha\not\in\triangle^+-\blacktriangle$.\par

   $\bullet$ If $\alpha\in\triangle^--\blacktriangle$, then we conclude $\alpha(Z)<0$ from $-\alpha\in\triangle^+-\blacktriangle$.   
} 
   Therefore \ref{eq-3} and \ref{eq-0} assure that for each $\alpha\in\triangle$, $\overline{(\overline{\sigma}^*\alpha)(Z)}=(\overline{\sigma}^*\alpha)(Z)$ and  
\begin{equation}\label{eq-4}\tag*{\textcircled{4}}
   \mbox{$\alpha\bigl(Z-\overline{\sigma}(Z)\bigr)$ is }
   \begin{cases}
    =0 & \mbox{if $\alpha\in\blacktriangle$},\\
    >0 & \mbox{if $\alpha\in\triangle^+-\blacktriangle$},\\
    <0 & \mbox{if $\alpha\in\triangle^--\blacktriangle$}.\\    
   \end{cases}
\end{equation}
   Setting $T:=i\bigl(Z-\overline{\sigma}(Z)\bigr)\in\frak{c}_\mathbb{C}$, we demonstrate that $T=iZ+\overline{\sigma}(iZ)$ is an element of $\frak{g}=\{W\in\frak{g}_\mathbb{C} \,|\, \overline{\sigma}(W)=W\}$; moreover, $T$ is elliptic, $\frak{h}_\mathbb{C}=\frak{g}^0$, $[S,\frak{q}^\pm]=\bigoplus_{\lambda>0}\frak{g}^{\pm\lambda}$ and $\frak{q}^\pm=\bigoplus_{\nu\geq0}\frak{g}^{\pm\nu}$ by \ref{eq-4}, \ref{eq-3}. 
   In addition, it follows from  $\frak{c}_{\frak{g}_\mathbb{C}}(S)=\frak{h}_\mathbb{C}=\frak{g}^0=\frak{c}_{\frak{g}_\mathbb{C}}(T)$ that $\frak{h}=\frak{c}_\frak{g}(S)=\bigl(\frak{g}\cap\frak{c}_{\frak{g}_\mathbb{C}}(S)\bigr)=\bigl(\frak{g}\cap\frak{c}_{\frak{g}_\mathbb{C}}(T)\bigr)=\frak{c}_\frak{g}(T)$. 
\end{proof}

   We will state Theorem \ref{thm-10.4.7} after proving 
\begin{lemma}\label{lem-10.4.6}
   Let $T$ have the properties in Proposition {\rm \ref{prop-10.4.5}}.
   Then, $H$ coincides with $C_G(T)$. 
\end{lemma} 
\begin{proof}
   At the beginning of this subsection one has known $C_G(S)_0\subset H\subset C_G(S)$.
   Therefore Proposition \ref{prop-10.4.5} and Lemma \ref{lem-7.3.3} give rise to 
\[
   C_G(T)=C_G(S)_0\subset H\subset C_G(S).
\]
   Hence, the rest of proof is to confirm that $H\subset C_G(T)$. 
   Let us denote by $\hat{G}$ the adjoint group of $\frak{g}$, set $\hat{H}:=\operatorname{Ad}H$, and identify $\frak{g}$ with $\hat{\frak{g}}$ via $\frak{g}\ni X\mapsto\operatorname{ad}X\in\hat{\frak{g}}$.  
   Our first aim is to prove 
\begin{equation}\label{eq-1}\tag*{\textcircled{1}}
   \hat{H}\subset C_{\hat{G}}(T). 
\end{equation}
   Set $\hat{G}_\mathbb{C}$ as the adjoint group of $\frak{g}_\mathbb{C}$.  
   In view of Lemma \ref{lem-10.4.3}-(1) and Proposition \ref{prop-10.4.5} we see that 
\[
   \mbox{$\hat{H}\subset\bigl(N_{\hat{G}_\mathbb{C}}(\bigoplus_{\nu\geq0}\frak{g}^\nu)\cap N_{\hat{G}_\mathbb{C}}(\bigoplus_{\nu\geq0}\frak{g}^{-\nu})\bigr)$},
\]
where we identify $\frak{g}_\mathbb{C}=\hat{\frak{g}}_\mathbb{C}$ in a similar way.
   That, together with Proposition \ref{prop-8.2.1}-(vii), yields $\hat{H}\subset C_{\hat{G}_\mathbb{C}}(T)$, and hence $\hat{H}\subset(\hat{G}\cap C_{\hat{G}_\mathbb{C}}(T))=C_{\hat{G}}(T)$.
   Thus \ref{eq-1} holds.    
   From now on, let us confirm that $H\subset C_G(T)$. 
   For any $h\in H$, it follows that $\operatorname{Ad}h\in\hat{H}$, and \ref{eq-1} implies $\operatorname{Ad}h\circ\operatorname{ad}T\circ(\operatorname{Ad}h)^{-1}=\operatorname{ad}T$.
   Hence we have $\operatorname{Ad}h(T)=T$, and $H\subset C_G(T)$.  
\end{proof}
   
    By summarizing the statements above and by Proposition \ref{prop-7.3.4} we conclude 
\begin{theorem}[{cf.\ Dorfmeister-Guan \cite[p.335]{DG1}}]\label{thm-10.4.7}
   Let $G$ be a connected real semisimple Lie group, and let $H$ be a closed subgroup of $G$. 
   Suppose the homogeneous space $G/H$ to admit a $G$-invariant complex structure $J$ and a $G$-invariant symplectic form $\Omega$ such that  
\[
   \mbox{$\Omega(JA,JB)=\Omega(A,B)$ for all $A,B\in\frak{X}(G/H)$}.
\] 
   Then, there exists an elliptic element $T\in\frak{g}$ satisfying 
\[
   H=C_G(T).
\]   
   Therefore any homogeneous pseudo-K\"{a}hler manifold of $G$ is an elliptic adjoint orbit, and it is always simply connected. 
\end{theorem} 

\section{Invariant complex structures on an elliptic orbit}\label{sec-10.5}
   
   It is known that there are several kinds of invariant complex structures on an elliptic adjoint orbit.
   One can understand that from the following example:
\begin{example}[$G/C_G(T)=SU(2,1)/S(U(1)\times U(1)\times U(1))$]\label{ex-10.5.1}
   Let $G:=SU(2,1)=\{X\in SL(3,\mathbb{C}) \,|\, {}^t\!XI_{2,1}\overline{X}=I_{2,1}\}$ and
\[
   T:=
   \left(\begin{array}{cc|c}
      i & 0 & 0\\
      0 & 0 & 0\\ \hline
      0 & 0 & -i
   \end{array}\right)\!,
\]
where $I_{2,1}=\begin{pmatrix} -1 & 0 & 0\\ 0 & -1 & 0\\ 0 & 0 & 1\end{pmatrix}$. 
   Then it turns out that
\[
   \frak{g}=\frak{su}(2,1)=
   \left\{\begin{array}{@{}c|l@{}}
   \left(\begin{array}{cc|c}
    ia_1 &  b+ic & ix-y\\
   -b+ic &  ia_2 & iz-w\\ \hline
   -ix-y & -iz-w & ia_3
   \end{array}\right)
   & \begin{array}{@{}r@{}} a_1,a_2,a_3,b,c,x,y,z,w\in\mathbb{R},\\ a_1+a_2+a_3=0\end{array}
   \end{array}\right\}  
\] 
and $T$ is an elliptic element of $\frak{g}$; besides, 
\[
\begin{array}{l}
   C_G(T)
   =\left\{
   \left(\begin{array}{cc|c}
    \epsilon_1 & 0   & 0\\
    0   & \epsilon_2 & 0\\ \hline
    0   & 0   & \epsilon_3
   \end{array}\right)\!\in G\right\}
   =S(U(1)\times U(1)\times U(1)),\\ 
   \operatorname{ad}T(\frak{g})=
   \left\{\begin{array}{@{}c|l@{}}
   \left(\begin{array}{cc|c}
    0    &  b+ic & ix-y\\
   -b+ic &  0    & iz-w\\ \hline
   -ix-y & -iz-w & 0
   \end{array}\right)
   & b,c,x,y,z,w\in\mathbb{R}
   \end{array}\right\}    
\end{array}
\]
and $\frak{g}=\frak{c}_\frak{g}(T)\oplus\operatorname{ad}T(\frak{g})$. 
   Now, let us define linear mappings $\jmath_1,\jmath_2,\jmath_3,\jmath_4,\jmath_5,\jmath_6:\frak{g}\to\operatorname{ad}T(\frak{g})$ by   
\allowdisplaybreaks{
\begin{align*}
&  \jmath_1
   \left(\begin{array}{cc|c}
    ia_1 &  b+ic & ix-y\\
   -b+ic &  ia_2 & iz-w\\ \hline
   -ix-y & -iz-w & ia_3
   \end{array}\right)
   :=
   \left(\begin{array}{cc|c}
    0    &  ib-c  & -x-iy\\
   ib+c  &  0     & -z-iw\\ \hline
   -x+iy &  -z+iw & 0
   \end{array}\right)\!,
   \qquad \jmath_2:=-\jmath_1,\\
&  \jmath_3
   \left(\begin{array}{cc|c}
    ia_1 &  b+ic & ix-y\\
   -b+ic &  ia_2 & iz-w\\ \hline
   -ix-y & -iz-w & ia_3
   \end{array}\right)
   :=
   \left(\begin{array}{cc|c}
    0    &  ib-c  & -x-iy\\
   ib+c  &  0     & z+iw\\ \hline
   -x+iy &  z-iw  & 0
   \end{array}\right)\!,
   \qquad \jmath_4:=-\jmath_3,\\  
&  \jmath_5
   \left(\begin{array}{cc|c}
    ia_1 &  b+ic & ix-y\\
   -b+ic &  ia_2 & iz-w\\ \hline
   -ix-y & -iz-w & ia_3
   \end{array}\right)
   :=
   \left(\begin{array}{cc|c}
    0    &  ib-c  & x+iy\\
   ib+c  &  0     & z+iw\\ \hline
    x-iy &  z-iw  & 0
   \end{array}\right)\!,
   \qquad \jmath_6:=-\jmath_5,
\end{align*}}respectively. 
   By a direct computation we see that all the mappings $\jmath_a$ ($1\leq a\leq 6$) satisfy the conditions (c.1) through (c.4) in Proposition \ref{prop-10.2.4}, and therefore the elliptic orbit $G/C_G(T)=SU(2,1)/S(U(1)\times U(1)\times U(1))$ admits six $G$-invariant complex structures $J_a$. 
   Incidentally, if $\Omega$ is a $G$-invariant symplectic form on $G/C_G(T)$ constructed from $\omega(X,Y):=B_\frak{g}(T,[X,Y])$ ($X,Y\in\frak{g}$), then all $(J_a,\Omega)$ are $G$-invariant pseudo-K\"{a}hlerian structures on $G/C_G(T)$ and the signatures of pseudo-K\"{a}hler metrics ${\sf g}_a(A,B):=\Omega(A,J_aB)$ ($A,B\in\frak{X}(G/C_G(T))$) are as follows:
\begin{center}
\begin{tabular}{|c|c|c|c|c|} \hline
   Signature & $(+,+,+,+,+,+)$ & $(-,-,+,+,+,+)$ & $(-,-,-,-,+,+)$ & $(-,-,-,-,-,-)$ \\\hline 
   & ${\sf g}_6$ & ${\sf g}_1$,  ${\sf g}_4$ & ${\sf g}_2$,  ${\sf g}_3$ & ${\sf g}_5$
   \\\hline
\end{tabular}
\end{center}
   This implies that ${\sf g}_6$ is a $G$-invariant K\"{a}hler metric on $G/C_G(T)$.    
\end{example}

\chapter{Homogeneous holomorphic vector bundles over elliptic orbits}\label{ch-11}
   In this chapter we deal with continuous representations of real semisimple Lie groups concerning homogeneous holomorphic vector bundles over elliptic orbits. 
   Here the definition of continuous representation is as follows:
\begin{definition}\label{def-11.0.1}
   Let $G$ be a Lie group, $\mathcal{V}$ a Fr\'{e}chet space over $\mathbb{C}$, and $\varrho:G\to GL(\mathcal{V})$, $g\mapsto\varrho(g)$, a homomorphism, where $GL(\mathcal{V})$ is the general linear group on $\mathcal{V}$ and it does not matter whether $\varrho$ is continuous here. 
   Then, $\varrho$ is called a {\it continuous representation}\index{continuous representation@continuous representation\dotfill} of $G$ on $\mathcal{V}$, if the mapping $\pi_\varrho:G\times\mathcal{V}\to\mathcal{V}$, $(g,\xi)\mapsto\varrho(g)\xi$, is continuous.
\end{definition}   

\section{A realization of elliptic orbits as domains in complex flag manifolds}\label{sec-11.1}
   In this section we realize elliptic (adjoint) orbits as domains in complex flag manifolds.\par
 
   Let $G_\mathbb{C}$ be a connected complex semisimple Lie group, let $G$ be a connected closed subgroup of $G_\mathbb{C}$ such that $\frak{g}$ is a real form of $\frak{g}_\mathbb{C}$, and let $T$ be a non-zero elliptic element of $\frak{g}$. 
   Let us define closed subgroups $L\subset G$ and $L_\mathbb{C}\subset G_\mathbb{C}$ by $L:=C_G(T)$ and $L_\mathbb{C}:=C_{G_\mathbb{C}}(T)$, respectively, and set
\begin{equation}\label{eq-11.1.1}
   \begin{array}{llll}
   \mbox{$\frak{g}^\lambda:=\{X\in\frak{g}_\mathbb{C} \,|\, \operatorname{ad}T(X)=i\lambda X\}$ for $\lambda\in\mathbb{R}$}, &
   \frak{u}^\pm:=\bigoplus_{\lambda>0}\frak{g}^{\pm\lambda}, & U^\pm:=\exp\frak{u}^\pm, & Q^\pm:=N_{G_\mathbb{C}}(\bigoplus_{\nu\geq 0}\frak{g}^{\pm\nu}),
   \end{array}
\end{equation}
where $\exp:\frak{g}_\mathbb{C}\to G_\mathbb{C}$ is the exponential mapping.
   Then $\operatorname{Ad}G(T)=G/L$ is an elliptic orbit, and $G_\mathbb{C}/Q^\pm$ are complex flag manifolds due to Proposition \ref{prop-8.2.1}-(iii). 
   By use of the mapping $\iota$ in Lemma \ref{lem-11.1.2}-(2) below, we realize $G/L$ as a simply connected domain in $G_\mathbb{C}/Q^\pm$.  

\begin{lemma}\label{lem-11.1.2}
   Let $s=+$ or $-$. 
\begin{enumerate}
\item[{\rm (1)}]
   $L$ coincides with $G\cap Q^s$. 
\item[{\rm (2)}]
   $\iota:G/L\to G_\mathbb{C}/Q^s$, $gL\mapsto gQ^s$, is a $G$-equivariant real analytic diffeomorphism of $G/L$ onto a simply connected domain in $G_\mathbb{C}/Q^s$.   
\item[{\rm (3)}]
   $GQ^s$ is a domain in $G_\mathbb{C}$.
\end{enumerate}   
\end{lemma}
\begin{proof}
   (1). 
   By Lemma \ref{lem-7.2.8}-(2) and $L=C_G(T)$ one has $L\subset N_{G_\mathbb{C}}(\bigoplus_{\nu\geq 0}\frak{g}^{s\nu})=Q^s$; thus $L\subset G\cap Q^s$. 
   Let us confirm that the converse inclusion also holds. 
   Take any $x\in G\cap Q^s$. 
   Proposition \ref{prop-8.2.1}-(iii), (i) and $x\in Q^s$ imply that $Q^s=N_{G_\mathbb{C}}(\frak{l}_\mathbb{C}\oplus\frak{u}^s)$ and there exists a unique $(z,Y)\in L_\mathbb{C}\times\frak{u}^s$ satisfying  
\[
   x=z\exp Y.
\] 
   We want to show that $\exp Y=e$ (the unit element of $G_\mathbb{C}$). 
   Let $\overline{\sigma}$ denote the conjugation of $\frak{g}_\mathbb{C}$ with respect to $\frak{g}$. 
   On the one hand, $x\in G$, $T\in\frak{l}=\frak{l}_\mathbb{C}\cap\frak{g}$, $L_\mathbb{C}=C_{G_\mathbb{C}}(T)$ and $Q^s=N_{G_\mathbb{C}}(\frak{l}_\mathbb{C}\oplus\frak{u}^s)$ yield
\[
   \frak{g}
   \ni\operatorname{Ad}x^{-1}(T)
   =(\operatorname{Ad}\exp(-Y)z^{-1})T
   =\operatorname{Ad}\exp(-Y)T
   \in\frak{l}_\mathbb{C}\oplus\frak{u}^s.
\] 
   On the other hand, $\operatorname{Ad}x^{-1}(T)\in\frak{g}$ implies that $\operatorname{Ad}x^{-1}(T)=\overline{\sigma}\bigl(\operatorname{Ad}x^{-1}(T)\bigr)$, so that $\operatorname{Ad}\exp(-Y)T=\overline{\sigma}\bigl(\operatorname{Ad}\exp(-Y)T\bigr)\in\overline{\sigma}(\frak{l}_\mathbb{C}\oplus\frak{u}^s)\subset\frak{l}_\mathbb{C}\oplus\frak{u}^{-s}$ by Lemma \ref{lem-7.2.8}-(5$'$). 
   Consequently we assert that  
\[
   \operatorname{Ad}\exp(-Y)T\in(\frak{l}_\mathbb{C}\oplus\frak{u}^s)\cap(\frak{l}_\mathbb{C}\oplus\frak{u}^{-s})=\frak{l}_\mathbb{C}.
\]
   Therefore $\frak{l}_\mathbb{C}\ni-T+\operatorname{Ad}\exp(-Y)T=\sum_{n=1}^\infty(1/n!)(-\operatorname{ad}Y)^nT\in\frak{u}^s$, and hence  
\[
   \operatorname{Ad}\exp(-Y)T=T.
\]   
   This yields $\exp Y\in(L_\mathbb{C}\cap U^s)=\{e\}$. 
   From $\exp Y=e$ we obtain $x=z\exp Y=z\in(L_\mathbb{C}\cap G)=L$, and $G\cap Q^s\subset L$. 
   For this reason $L=G\cap Q^s$ holds.\par
   
   (2). 
   We conclude (2) from (1), $\dim_\mathbb{R}G/L=\dim_\mathbb{R}\frak{u}^+=\dim_\mathbb{R}G_\mathbb{C}/Q^s$ and Proposition \ref{prop-7.3.4}.\par
 
   (3).
   Denote by $\pi_\mathbb{C}$ the projection of $G_\mathbb{C}$ onto $G_\mathbb{C}/Q^s$. 
   It is immediate from (2) that $GQ^s=\pi_\mathbb{C}^{-1}\bigl(\iota(G/L)\bigr)$ is an open subset of $G_\mathbb{C}$. 
   Moreover, $GQ^s$ is connected because the product mapping $G\times Q^s\ni(g,q)\mapsto gq\in GQ^s$ is surjective continuous and both $G$ and $Q^s$ are connected.  
\end{proof}

\begin{remark}\label{rem-11.1.3}
\begin{enumerate}
\item[]
\item[(i)]
   Henceforth, we assume that the elliptic orbit $G/L$ is a simply connected domain in the complex flag manifold $G_\mathbb{C}/Q^-$ via the $G$-equivariant mapping $\iota:G/L\to G_\mathbb{C}/Q^-$, $gL\mapsto gQ^-$. 
   By inducing a $G$-invariant complex structure $J$ on $G/L=\iota(G/L)$ from $G_\mathbb{C}/Q^-=(G_\mathbb{C}/Q^-,J_-)$, we consider $G/L$ as a homogeneous complex manifold of $G$. 
   Here we refer to Remark \ref{rem-1.2.3} for the $G_\mathbb{C}$-invariant complex structure $J_-$ on $G_\mathbb{C}/Q^-$.
\item[(ii)]
   In general, there are several kinds of invariant complex structures on the elliptic orbit $G/L$ (e.g.\ Example \ref{ex-10.5.1}).
   In this chapter we deal with the complex structure $J$ on $G/L$ induced by $\iota:G/L\to G_\mathbb{C}/Q^-$, $gL\mapsto gQ^-$.
\end{enumerate}
\end{remark}

\section{Homogeneous holomorphic vector bundles over elliptic orbits}\label{sec-11.2}
   The setting of Section \ref{sec-11.1} remains valid in this section.\par
   
   Let ${\sf V}$ be a finite-dimensional complex vector space, and let $\rho:Q^-\to GL({\sf V})$, $q\mapsto\rho(q)$, be a holomorphic homomorphism. 
   Then, one can take the homogeneous holomorphic vector bundle $G_\mathbb{C}\times_\rho{\sf V}$ over the complex flag manifold $G_\mathbb{C}/Q^-$ associated with $\rho$, and its restriction $\iota^\sharp(G_\mathbb{C}\times_\rho{\sf V})$ to the domain $G/L\subset G_\mathbb{C}/Q^-$. 
   Moreover, one may assume that 
\begin{equation}\label{eq-11.2.1}
\begin{split}
&  \mathcal{V}_{G_\mathbb{C}/Q^-}
   =\left\{\begin{array}{@{}c|c@{}}
   h:G_\mathbb{C}\to{\sf V} 
   & \begin{array}{@{}l@{}} \mbox{(i) $h$ is holomorphic},\\ \mbox{(ii) $h(aq)=\rho(q)^{-1}\bigl(h(a)\bigr)$ for all $(a,q)\in G_\mathbb{C}\times Q^-$}\end{array}
   \end{array}\right\},\\
&  \mathcal{V}_{G/L}
   =\left\{\begin{array}{@{}c|c@{}}
   \psi:GQ^-\to{\sf V} 
   & \begin{array}{@{}l@{}} \mbox{(i) $\psi$ is holomorphic},\\ \mbox{(ii) $\psi(xq)=\rho(q)^{-1}\bigl(\psi(x)\bigr)$ for all $(x,q)\in GQ^-\times Q^-$}\end{array}
   \end{array}\right\}
\end{split}
\end{equation}   
are the complex vector spaces of holomorphic cross-sections of the bundles $G_\mathbb{C}\times_\rho{\sf V}$ and $\iota^\sharp(G_\mathbb{C}\times_\rho{\sf V})$, respectively (cf.\ Chapter \ref{ch-3}).
   Let us define a homomorphism $\varrho:G\to GL(\mathcal{V}_{G/L})$, $g\mapsto\varrho(g)$, as follows:
\begin{equation}\label{eq-11.2.2}
   \mbox{$\bigl(\varrho(g)\psi\bigr)(x):=\psi(g^{-1}x)$ for $\psi\in\mathcal{V}_{G/L}$ and $x\in GQ^-$}.
\end{equation}   
   In this section, we first prove that this $\varrho$ is a continuous representation of the Lie group $G$ on $\mathcal{V}_{G/L}$, next show that every $K$-finite vector $\varphi\in\mathcal{V}_{G/L}$ (for the continuous representation $\varrho$) can be continued analytically from $U^+\cap GQ^-$ to $U^+$, and finally provide a sufficient condition for the vector space $\mathcal{V}_{G/L}$ to be finite-dimensional.

\begin{remark}\label{rem-11.2.3}
\begin{enumerate}
\item[]
\item[(1)]
   For the sake of simplicity, we write $\iota^\sharp(G_\mathbb{C}\times_\rho{\sf V})$, $\mathcal{V}_{G_\mathbb{C}/Q^-}$, and $\mathcal{V}_{G/L}$ for $(G_\mathbb{C}\times_\rho{\sf V})_{G/L}$, $\mathcal{V}(G_\mathbb{C}\times_\rho{\sf V})$, and $\mathcal{V}(G_\mathbb{C}\times_\rho{\sf V})_{G/L}$, respectively. 
   cf.\ \eqref{eq-2.5.1}, \eqref{eq-3.2.3}, \eqref{eq-3.2.6}.  
\item[(2)] 
   Corollary \ref{cor-8.2.3}-(1) implies that $G_\mathbb{C}/Q^-$ is a connected compact complex manifold. 
   Thus, one knows $\dim_\mathbb{C}\mathcal{V}_{G_\mathbb{C}/Q^-}<\infty$. 
   e.g.\ Kodaira \cite[p.161]{Kod}.
\end{enumerate}
\end{remark}

\subsection{A continuous representation of a Lie group}\label{subsec-11.2.1}
   We equip the complex vector space $\mathcal{V}_{G/L}$ with the Fr\'{e}chet metric $d$ in \eqref{eq-4.1.3}, and hereafter consider $\mathcal{V}_{G/L}$ as a Fr\'{e}chet space over $\mathbb{C}$ (cf.\ Proposition \ref{prop-4.3.1}).
   Our goal in this subsection is to prove that $\pi_\varrho:G\times\mathcal{V}_{G/L}\to\mathcal{V}_{G/L}$, $(g,\psi)\mapsto\varrho(g)\psi$, is continuous (see Proposition \ref{prop-11.2.8}).
   We are going to verify three lemmas first, and then obtain the goal.
\begin{lemma}\label{lem-11.2.4}
   The following two items hold for a given $\psi\in\mathcal{V}_{G/L}:$ 
\begin{enumerate}
\item[{\rm (1)}]
   For any $\epsilon>0$ and any non-empty compact subset $E\subset GQ^-$, there exists an open neighborhood $U$ of the unit $e\in G$ such that $g\in U$ implies $d_E\bigl(\varrho(g)\psi,\psi\bigr)<\epsilon$. 
   Here we refer to \eqref{eq-4.1.2} for $d_E$. 
\item[{\rm (2)}]
   The mapping $G\ni g\mapsto\varrho(g)\psi\in\mathcal{V}_{G/L}$ is continuous at $e\in G$, namely, for every $\epsilon>0$ there exists an open neighborhood $U$ of $e\in G$ such that $g\in U$ implies $d\bigl(\varrho(g)\psi,\psi\bigr)<\epsilon$.  
\end{enumerate} 
\end{lemma}
\begin{proof}
   (1).
   The mapping $G\times GQ^-\ni(g,q)\mapsto g^{-1}q\in GQ^-$ is continuous, $\psi:GQ^-\to{\sf V}$ is continuous, and so the mapping $f:G\times GQ^-\to{\sf V}$, $(g,q)\mapsto\psi(g^{-1}q)$, is continuous. 
   Therefore, for each $y\in E$ there exist an open neighborhood $U_y$ of $e\in G$ and an open neighborhood $O_y'$ of $y\in GQ^-$ such that $(g,z')\in U_y\times O_y'$ implies
\begin{equation}\label{eq-1}\tag*{\textcircled{1}}
   \|\psi(g^{-1}z')-\psi(y)\|=\|f(g,z')-f(e,y)\|<\epsilon/4
\end{equation}
because $f$ is continuous at $(e,y)$.
   Here $\|\cdot\|$ is a norm on the vector space ${\sf V}$. 
   Since $O_y'$ is an open neighborhood of $y\in GQ^-$ and $\psi:GQ^-\to{\sf V}$ is continuous at the $y$, one can choose an open neighborhood $O_y$ of $y\in O_y'$ so that $z\in O_y$ implies
\begin{equation}\label{eq-2}\tag*{\textcircled{2}}
   \|\psi(y)-\psi(z)\|<\epsilon/4.
\end{equation}
   Since $E\subset\bigcup_{y\in E}O_y$ and $E\subset GQ^-$ is compact, there exist finite elements $y_1,y_2,\dots,y_k\in E$ satisfying $E\subset\bigcup_{j=1}^kO_{y_j}$.
   In this setting, we put $U:=\bigcap_{j=1}^kU_{y_j}$, and see that $U$ is an open neighborhood of $e\in G$. 
   Furthermore, for an arbitrary $(g,w)\in U\times E$, there exists a $1\leq i\leq k$ such that $w\in O_{y_i}\subset O_{y_i}'$, and it follows from $g\in(\bigcap_{j=1}^kU_{y_j})\subset U_{y_i}$, \ref{eq-1} and \ref{eq-2} that 
\[
   \big\|\bigl(\varrho(g)\psi\bigr)(w)-\psi(w)\big\|
   \stackrel{\eqref{eq-11.2.2}}{=}
   \|\psi(g^{-1}w)-\psi(w)\|
   \leq\|\psi(g^{-1}w)-\psi(y_i)\|+\|\psi(y_i)-\psi(w)\|
   <\epsilon/4+\epsilon/4=\epsilon/2.
\] 
   This and \eqref{eq-4.1.2} assure that $d_E\bigl(\varrho(g)\psi,\psi\bigr)\leq\epsilon/2<\epsilon$ for all $g\in U$.
   Hence (1) holds.\par
   
   (2) follows by (1) and Proposition \ref{prop-4.3.1}-(3).
\end{proof}
   
   Lemma \ref{lem-11.2.4}-(2) leads to
\begin{corollary}\label{cor-11.2.5}
   For each $\psi\in\mathcal{V}_{G/L}$, the mapping $G\ni g\mapsto\varrho(g)\psi\in\mathcal{V}_{G/L}$ is continuous.
\end{corollary}
\begin{proof}
   Fix any $g_0\in G$ and $\epsilon>0$. 
   Since $\psi_0:=\varrho(g_0)\psi$ is an element of $\mathcal{V}_{G/L}$, Lemma \ref{lem-11.2.4}-(2) enables us to obtain an open neighborhood $U'$ of $e\in G$ such that $g'\in U'$ implies $d\bigl(\varrho(g')\psi_0,\psi_0\bigr)<\epsilon$. 
   Setting $U:=R_{g_0}(U')$ one can assert that $U$ is an open neighborhood of $g_0\in G$; moreover, $g\in U$ implies 
\[
   d\bigl(\varrho(g)\psi,\varrho(g_0)\psi\bigr)
   =d\bigl(\varrho(gg_0^{-1})\psi_0,\psi_0\bigr)
   <\epsilon
\]
because of $gg_0^{-1}\in U'$.
   Thus the mapping $G\ni g\mapsto\varrho(g)\psi\in\mathcal{V}_{G/L}$ is continuous at the point $g_0$.
\end{proof}

   The following lemma, together with Corollary \ref{cor-11.2.5}, tells us that $\pi_\varrho:(g,\psi)\mapsto\varrho(g)\psi$ is a separately continuous linear action of $G$ on $\mathcal{V}_{G/L}$:
\begin{lemma}\label{lem-11.2.6}
   For each $g\in G$, the linear mapping $\varrho(g):\mathcal{V}_{G/L}\to\mathcal{V}_{G/L}$, $\psi\mapsto\varrho(g)\psi$, is uniformly continuous.   
\end{lemma}
\begin{proof}
   By virtue of Proposition \ref{prop-4.3.1}-(3) it suffices to prove that $\mathcal{V}_{G/L}\ni\psi\mapsto\varrho(g)\psi\in\mathcal{V}_{G/L}$ is uniformly continuous in the topology of uniform convergence on compact sets.
   For any $\epsilon>0$ and any non-empty compact subset $E\subset GQ^-$, we set $E':=g^{-1}E$ and $\delta:=\epsilon$. 
   Then, $E'$ is a non-empty compact subset of $GQ^-$ and $\delta>0$. 
   In addition, it follows from \eqref{eq-4.1.2} and \eqref{eq-11.2.2} that $d_{E'}(\psi_1,\psi_2)<\delta$ and $\psi_1,\psi_2\in\mathcal{V}_{G/L}$ imply
\[
\begin{split}
   d_E\bigl(\varrho(g)\psi_1,\varrho(g)\psi_2\bigr)
  &=\sup\big\{\|\psi_1(g^{-1}y)-\psi_2(g^{-1}y)\|:y\in E\big\}\\
  &=\sup\big\{\|\psi_1(z)-\psi_2(z)\|:z\in g^{-1}E\big\}
   =d_{E'}(\psi_1,\psi_2)<\delta=\epsilon.
\end{split}  
\]   
   Hence, the mapping $\mathcal{V}_{G/L}\ni\psi\mapsto\varrho(g)\psi\in\mathcal{V}_{G/L}$ is uniformly continuous.
\end{proof}

   We will show Proposition \ref{prop-11.2.8} after proving 
\begin{lemma}\label{lem-11.2.7}
   For any non-empty compact subset $C$ of $G$ and any open neighborhood $\mathcal{B}$ of $0\in\mathcal{V}_{G/L}$, there exists an open neighborhood $\mathcal{A}$ of $0\in\mathcal{V}_{G/L}$ such that $\varrho(g)\psi\in\mathcal{B}$ for all $(g,\psi)\in C\times\mathcal{A}$.
\end{lemma}
\begin{proof}
   For $\epsilon>0$ we set an open neighborhood $\mathcal{B}_\epsilon$ of $0\in\mathcal{V}_{G/L}$ as $\mathcal{B}_\epsilon:=\{\psi\in\mathcal{V}_{G/L} \,|\, d(0,\psi)<\epsilon\}$, and put $\mathcal{D}_\epsilon:=\overline{\mathcal{B}_\epsilon}$ (the closure of $\mathcal{B}_\epsilon$ in $\mathcal{V}_{G/L}$).\par
   
   Since $\mathcal{B}$ is an open neighborhood of $0\in\mathcal{V}_{G/L}$ and the addition $\mathcal{V}_{G/L}\times\mathcal{V}_{G/L}\ni(\psi_1,\psi_2)\mapsto\psi_1+\psi_2\in\mathcal{V}_{G/L}$ is continuous at $(0,0)$, there exists an $r>0$ such that 
\begin{equation}\label{eq-a}\tag{a}
   \mathcal{D}_r+\mathcal{D}_r\subset\mathcal{B}.
\end{equation}
   Lemma \ref{lem-4.1.4}-(3) assures that
\begin{equation}\label{eq-b}\tag{b}
   \mbox{$t\mathcal{B}_r\subset\mathcal{B}_r$, $t\mathcal{D}_r\subset\mathcal{D}_r$ for all $-1\leq t\leq 1$},
\end{equation}
and it follows from \eqref{eq-b} that  
\begin{equation}\label{eq-c}\tag{c}
   \mathcal{B}_r\subset 2\mathcal{B}_r\subset\cdots\subset n\mathcal{B}_r\subset(n+1)\mathcal{B}_r\subset\cdots.
\end{equation}
   Here $\lambda\mathcal{B}_r$ means $\{\lambda\psi \,|\, \psi\in\mathcal{B}_r\}$ for $\lambda\in\mathbb{R}$.
   Furthermore, one can show 
\begin{equation}\label{eq-d}\tag{d}
   \mbox{$\mathcal{V}_{G/L}=\bigcup_{n=1}^\infty n\mathcal{B}_r$}.
\end{equation}
   Indeed; for any $\psi_0\in\mathcal{V}_{G/L}$, the mapping $\mathbb{C}\ni\alpha\mapsto\alpha\psi_0\in \mathcal{V}_{G/L}$ is continuous at $0\in\mathbb{C}$, and therefore there exists an $m\in\mathbb{N}$ such that $(1/m)\psi_0\in\mathcal{B}_r$, since $\mathcal{B}_r$ is an open neighborhood of $0\in\mathcal{V}_{G/L}$ and $\displaystyle{\lim_{n\to\infty}(1/n)=0}$. 
   Hence $\psi_0=m\bigl((1/m)\psi_0\bigr)\in m\mathcal{B}_r$. 
   This yields $\mathcal{V}_{G/L}\subset\bigcup_{n=1}^\infty n\mathcal{B}_r$, and \eqref{eq-d} follows.\par
   
   Now, let us define 
\begin{equation}\label{eq-1}\tag*{\textcircled{1}}
   \mbox{$\mathcal{F}_n:=\bigcap_{g\in C}\{\psi\in\mathcal{V}_{G/L} \,|\, \varrho(g)\psi\in n\mathcal{D}_r\}$}
\end{equation}
for $n\in\mathbb{N}$.
   For each $g\in C$, Lemma \ref{lem-11.2.6} ensures that $\{\psi\in\mathcal{V}_{G/L} \,|\, \varrho(g)\psi\in n\mathcal{D}_r\}=\varrho(g)^{-1}(n\mathcal{D}_r)$ is a closed subset of $\mathcal{V}_{G/L}$ because $n\mathcal{D}_r\subset\mathcal{V}_{G/L}$ is closed. 
   Thus it follows from \ref{eq-1} that 
\begin{equation}\label{eq-2}\tag*{\textcircled{2}}
   \mbox{$\mathcal{F}_n$ is a closed subset of $\mathcal{V}_{G/L}$ for each $n\in\mathbb{N}$}.
\end{equation}
   We want to show $\mathcal{V}_{G/L}=\bigcup_{n=1}^\infty\mathcal{F}_n$.
   For an arbitrary $\psi_0\in\mathcal{V}_{G/L}$, the mapping $G\ni g\mapsto\varrho(g)\psi_0\in\mathcal{V}_{G/L}$ is continuous by Corollary \ref{cor-11.2.5}. 
   Accordingly $\{\varrho(g)\psi_0 \,|\, g\in C\}$ is a compact subset of $\mathcal{V}_{G/L}$. 
   This, combined with \eqref{eq-d} and \eqref{eq-c}, enables us to find a $k\in\mathbb{N}$ such that $\{\varrho(g)\psi_0 \,|\, g\in C\}\subset k\mathcal{B}_r$.
   Then, $\mathcal{B}_r\subset\mathcal{D}_r$ and \ref{eq-1} give rise to $\psi_0\in\mathcal{F}_k\subset\bigcup_{n=1}^\infty\mathcal{F}_n$. 
   For this reason we conclude  
\begin{equation}\label{eq-3}\tag*{\textcircled{3}}
   \mbox{$\mathcal{V}_{G/L}=\bigcup_{n=1}^\infty\mathcal{F}_n$}.
\end{equation}
   By \ref{eq-2}, \ref{eq-3} and Proposition \ref{prop-4.4.1}, there exist an $N\in\mathbb{N}$, a $\psi_N\in\mathcal{F}_N$ and an open subset $\mathcal{O}_N\subset\mathcal{V}_{G/L}$ which satisfy 
\begin{equation}\label{eq-4}\tag*{\textcircled{4}}
   \psi_N\in\mathcal{O}_N\subset\mathcal{F}_N.
\end{equation}
   Setting $\mathcal{A}':=\mathcal{O}_N-\psi_N$, we see that $\mathcal{A}'$ is an open neighborhood of $0\in\mathcal{V}_{G/L}$. 
   Moreover, for any $(g,\psi')\in C\times\mathcal{A}'$, it follows from \ref{eq-4}, \ref{eq-1} and \eqref{eq-a} that 
\[
   \varrho(g)\psi'
   =\varrho(g)\bigl(\psi'+\psi_N)+\varrho(g)(-\psi_N)
   \in\varrho(g)(\mathcal{F}_N)+\varrho(g)(-\psi_N)
   \subset N\mathcal{D}_r+N\mathcal{D}_r
   \subset N\mathcal{B},
\]
where we note that $\varrho(g)(-\psi_N)=-\varrho(g)(\psi_N)\in -\varrho(g)(\mathcal{F}_N)\subset -N\mathcal{D}_r=N(-\mathcal{D}_r)\subset N\mathcal{D}_r$ due to \eqref{eq-b}.
   Hence we can deduce the conclusion from $\mathcal{A}:=(1/N)\mathcal{A}'$. 
\end{proof}

   Let us show 
\begin{proposition}\label{prop-11.2.8}
   The $\varrho$ in \eqref{eq-11.2.2} is a continuous representation of the Lie group $G$ on the Fr\'{e}chet space $\mathcal{V}_{G/L}$.
   Here we refer to \eqref{eq-11.2.1} for $\mathcal{V}_{G/L}$, and equip $\mathcal{V}_{G/L}$ with the Fr\'{e}chet metric $d$ in \eqref{eq-4.1.3}.
\end{proposition}
\begin{proof}
   Let us prove that $\pi_\varrho:G\times\mathcal{V}_{G/L}\to\mathcal{V}_{G/L}$, $(g,\psi)\mapsto\varrho(g)\psi$, is continuous.
   Take any element $(g_0,\psi_0)\in G\times\mathcal{V}_{G/L}$ and any open neighborhood $\mathcal{O}$ of $\pi_\varrho(g_0,\psi_0)=\varrho(g_0)\psi_0\in\mathcal{V}_{G/L}$.  
   Since the addition $\mathcal{V}_{G/L}\times\mathcal{V}_{G/L}\ni(\psi_1,\psi_2)\mapsto\psi_1+\psi_2\in\mathcal{V}_{G/L}$ is continuous at $\bigl(0,\varrho(g_0)\psi_0\bigr)$, there exist open neighborhoods $\mathcal{B}$ of $0\in\mathcal{V}_{G/L}$ and $\mathcal{U}$ of $\varrho(g_0)\psi_0\in\mathcal{V}_{G/L}$ such that
\begin{equation}\label{eq-a}\tag{a}
   \mathcal{B}+\mathcal{U}\subset\mathcal{O}.
\end{equation}
   Corollary \ref{cor-11.2.5} assures that $G\ni g\mapsto\varrho(g)\psi_0\in\mathcal{V}_{G/L}$ is continuous at $g_0$, so there exists an open neighborhood $U'$ of $g_0\in G$ such that 
\begin{equation}\label{eq-b}\tag{b}
   \mbox{$\varrho(g')\psi_0\in\mathcal{U}$ for all $g'\in U'$}.
\end{equation}
   Besides, since $G$ is a locally compact Hausdorff space, there exists an open neighborhood $U$ of $g_0\in G$ so that 
\begin{equation}\label{eq-c}\tag{c}
   \mbox{$\overline{U}\subset U'$ and the closure $\overline{U}$ is a compact subset of $G$}.
\end{equation}
   By \eqref{eq-c} and Lemma \ref{lem-11.2.7}, there exists an open neighborhood $\mathcal{A}$ of $0\in\mathcal{V}_{G/L}$ such that    
\begin{equation}\label{eq-d}\tag{d}
   \mbox{$\varrho(g)\psi\in\mathcal{B}$ for all $(g,\psi)\in\overline{U}\times\mathcal{A}$}.
\end{equation}
   Putting $\mathcal{V}:=\mathcal{A}+\psi_0$, we assert that $\mathcal{V}$ is an open neighborhood of $\psi_0\in\mathcal{V}_{G/L}$.
   In addition, for any $(g,\psi)\in U\times\mathcal{V}$ we obtain 
\[
   \pi_\varrho(g,\psi)
   =\varrho(g)\psi
   =\varrho(g)(\psi-\psi_0)+\varrho(g)\psi_0
   \in\varrho(U)(\mathcal{A})+\varrho(U)\psi_0
   \subset\mathcal{B}+\mathcal{U}
   \subset\mathcal{O} 
\]
from \eqref{eq-d}, \eqref{eq-c}, \eqref{eq-b} and \eqref{eq-a}. 
   Consequently, $\pi_\varrho$ is continuous at $(g_0,\psi_0)$.
\end{proof}

\subsection{$K$-finite vectors}\label{subsec-11.2.2}
   Since the element $T\in\frak{g}$ is elliptic, Lemma \ref{lem-7.2.4} enables us to have a Cartan decomposition $\frak{g}=\frak{k}\oplus\frak{p}$ such that
\[
   T\in\frak{k},
\]
where $\frak{k}$ is a maximal compact subalgebra of $\frak{g}$. 
   Noting that the center $Z(G)$ of $G$ is finite due to $Z(G)\subset Z(G_\mathbb{C})$ and that $\frak{g}_u:=\frak{k}\oplus i\frak{p}$ is a compact real form of  $\frak{g}_\mathbb{C}$, we denote by $K$ and $G_u$ the maximal compact subgroups of $G$ and $G_\mathbb{C}$ corresponding to the subalgebras $\frak{k}\subset\frak{g}$ and $\frak{g}_u\subset\frak{g}_\mathbb{C}$, respectively. 
   In addition, let $\overline{\theta}$ be the (anti-holomorphic) Cartan involution of $G_\mathbb{C}$ so that 
\[
   G_u=\{g_u\in G_\mathbb{C} \,|\, \overline{\theta}(g_u)=g_u\}. 
\]      
   Fix a maximal torus $i\frak{h}_\mathbb{R}$ of the compact semisimple Lie algebra $\frak{g}_u$ containing the $T$, and take the (non-zero) root system $\triangle$ of $\frak{g}_\mathbb{C}$ relative to $\frak{h}_\mathbb{C}$, where $\frak{h}_\mathbb{C}$ is the complex vector subspace of $\frak{g}_\mathbb{C}$ generated by $i\frak{h}_\mathbb{R}$.
   For each $\alpha\in\triangle$, we denote by $\frak{g}_\alpha$ the root subspace of $\frak{g}_\mathbb{C}$, and suppose vectors $E_{\pm\alpha}\in\frak{g}_{\pm\alpha}$ to satisfy \eqref{eq-8.1.1}.
   Letting $\blacktriangle=\{\gamma\in\triangle \,|\, \gamma(T)=0\}$, we are going to demonstrate three lemmas and two propositions.

\begin{lemma}\label{lem-11.2.9}
   Let $\frak{k}_\mathbb{C}$ be the complex subalgebra of $\frak{g}_\mathbb{C}$ generated by $\frak{k}$.
   For a root $\beta\in\triangle-\blacktriangle$, the following {\rm (a)}, {\rm (b)} and {\rm (c)} are equivalent$:$
\[
\begin{array}{lll}   
   \mbox{{\rm (a)} $\frak{g}_\beta\subset\frak{k}_\mathbb{C}$}, & \mbox{{\rm (b)} $E_\beta\in\frak{k}_\mathbb{C}$}, & \mbox{{\rm (c)} $(E_\beta-E_{-\beta})\in\frak{k}$}.
\end{array}
\]
   Therefore, $w_\beta=\exp(\pi/2)(E_\beta-E_{-\beta})$ belongs to $K\cap N_{G_u}(i\frak{h}_\mathbb{R})$ whenever one of the conditions {\rm (a)}, {\rm (b)} and {\rm (c)} holds.
\end{lemma}
\begin{proof}
   Since (a)$\Leftrightarrow$(b) is obvious, we only confirm (b)$\Leftrightarrow$(c). 
   cf.\ Subsection \ref{subsec-8.1.3} for $w_\beta\in N_{G_u}(i\frak{h}_\mathbb{R})$.\par
   
   (b)$\Rightarrow$(c). 
   This follows by \eqref{eq-8.1.2}, $\overline{\theta}_*(\frak{k}_\mathbb{C})\subset\frak{k}_\mathbb{C}$ and $\frak{k}=\{X\in\frak{k}_\mathbb{C} \,|\, \overline{\theta}_*(X)=X\}$.
   Here we remark that \eqref{eq-8.1.2} always holds for any vector $E_\alpha$ with \eqref{eq-8.1.1}.\par
   
   (c)$\Rightarrow$(b).
   Suppose that $(E_\beta-E_{-\beta})\in\frak{k}$.
   Then, from $T\in\frak{k}$ one obtains
\[
   \beta(T)(E_\beta+E_{-\beta})
   =[T,E_\beta-E_{-\beta}]\in[\frak{k},\frak{k}]\subset\frak{k};
\]
and so $0\neq\beta(T)\in i\mathbb{R}$ yields $(E_\beta+E_{-\beta})\in i\frak{k}$.
   Hence $E_\beta=(1/2)(E_\beta-E_{-\beta}+E_\beta+E_{-\beta})\in\frak{k}+i\frak{k}\subset\frak{k}_\mathbb{C}$. 
\end{proof}  

   Let $\Pi_\triangle$ be a fundamental root system of $\triangle$ satisfying \eqref{eq-8.1.5}, and let $\triangle^+$ be the set of positive roots relative to $\Pi_\triangle$.
   Let us suppose that $\triangle^+-\blacktriangle$ consists of $r$-roots $\beta_1,\beta_2,\dots,\beta_r$ ($r=\dim_\mathbb{C}\frak{u}^+$). 
   Then, it turns out that $\{E_{\beta_j}\}_{j=1}^r$ is a complex basis of $\frak{u}^+=\bigoplus_{\alpha\in\triangle^+-\blacktriangle}\frak{g}_\alpha=\bigoplus_{j=1}^r\frak{g}_{\beta_j}$, and Proposition \ref{prop-8.2.1}-(i) allows us to identify $U^+$ with $\mathbb{C}^r$ via 
\begin{equation}\label{eq-11.2.10} 
   U^+\ni\exp(z^1E_{\beta_1}+z^2E_{\beta_2}+\cdots+z^rE_{\beta_r})\leftrightarrow(z^1,z^2,\dots,z^r)\in\mathbb{C}^r.
\end{equation}
   Remark that $z^1,z^2,\dots,z^r$ is the canonical coordinates of the first kind associated with $\{E_{\beta_j}\}_{j=1}^r\subset\frak{u}^+$.   
   Setting $\omega_j:=\beta_j(-iT)$ for $1\leq j\leq r$, one has $\beta_j(T)=i\omega_j$, $\omega_j>0$ and 
\begin{equation}\label{eq-11.2.11}
   \mbox{$\operatorname{Ad}(\exp tT)E_{\beta_j}=e^{i\omega_jt}E_{\beta_j}$ ($1\leq j\leq r$)}
\end{equation}
for all $t\in\mathbb{R}$.
   About these $\omega_1,\omega_2,\dots,\omega_r>0$ we assert

\begin{lemma}\label{lem-11.2.12}
   For a given $\vartheta\in\mathbb{R}$, the number of non-negative integer solutions $(n_1,n_2,\dots,n_r)$ to the equation
\[
   \vartheta=\omega_1n_1+\omega_2n_2+\cdots+\omega_rn_r
\] 
is only finite or zero. 
\end{lemma}
\begin{proof}
   If $(n_1,n_2,\dots,n_r)$ is a non-negative integer solution to the equation, then it follows from $\omega_j>0$ ($1\leq j\leq r$) that $\vartheta-\omega_kn_k=\omega_1n_1+\cdots+\omega_{k-1}n_{k-1}+\omega_{k+1}n_{k+1}+\cdots+\omega_rn_r\geq 0$, so that $0\leq n_k\leq\vartheta/\omega_k$, $n_k\in\mathbb{Z}$ for all $1\leq k\leq r$.
\end{proof}

   Now, let $(\mathcal{V}_{G/L})_K$ be the set of $K$-finite vectors in $\mathcal{V}_{G/L}$ for the continuous representation $\varrho$ of $G$ on $\mathcal{V}_{G/L}$, that is, 
\begin{equation}\label{eq-11.2.13}
   (\mathcal{V}_{G/L})_K:=\{\varphi\in\mathcal{V}_{G/L} \,|\, \dim_\mathbb{C}\operatorname{span}_\mathbb{C}\{\varrho(k)\varphi : k\in K\}<\infty\}.
\end{equation}
   Note that $(\mathcal{V}_{G/L})_K$ is a $\varrho(K)$-invariant complex vector subspace of $\mathcal{V}_{G/L}$. 
   With this notation \eqref{eq-11.2.13} we show

\begin{lemma}\label{lem-11.2.14}
\begin{enumerate}
\item[]
\item[{\rm (1)}]
    For each $\varphi\in(\mathcal{V}_{G/L})_K$ we set a $\varrho(K)$-invariant complex vector subspace $\mathcal{V}_\varphi\subset\mathcal{V}_{G/L}$ as 
\[
   \mathcal{V}_\varphi:=\operatorname{span}_\mathbb{C}\{\varrho(k)\varphi : k\in K\}.
\] 
   Then, there exist a complex basis $\{\varphi_a\}_{a=1}^{k_\varphi}$ of $\mathcal{V}_\varphi$ and $\mu_1,\mu_2,\dots,\mu_{k_\varphi}\in\mathbb{R}$ such that 
\[
   \varrho(\exp tT)\varphi_a=e^{i\mu_at}\varphi_a
\]
for all $1\leq a\leq k_\varphi=\dim_\mathbb{C}\mathcal{V}_\varphi$ and $t\in\mathbb{R}$.
\item[{\rm (2)}]
   There exist a complex basis $\{{\sf v}_b\}_{b=1}^m$ of ${\sf V}$ and $\theta_1,\theta_2,\dots,\theta_m\in\mathbb{R}$ such that 
\[
   \rho(\exp tT){\sf v}_b=e^{i\theta_bt}{\sf v}_b
\]
for all $1\leq b\leq m=\dim_\mathbb{C}{\sf V}$ and $t\in\mathbb{R}$.
\end{enumerate}
\end{lemma}
\begin{proof}
   Since the center $Z(G)$ is finite and $T\neq 0$, Lemma \ref{lem-7.2.1} implies that $S^1=\{\exp tT : t\in\mathbb{R}\}$ is a $1$-dimensional torus.\par

   (1). 
   It follows from $T\in\frak{k}$ that $S^1\subset K$.
   Therefore, since $\mathcal{V}_\varphi$ is $\varrho(K)$-invariant and $k_\varphi=\dim_\mathbb{C}\mathcal{V}_\varphi<\infty$, one can decompose $\mathcal{V}_\varphi$ into a direct sum of $1$-dimensional $\varrho(S^1)$-invariant complex vector subspaces: $\mathcal{V}_\varphi=\mathcal{V}_1\oplus\mathcal{V}_2\oplus\cdots\oplus\mathcal{V}_{k_\varphi}$. 
   Hence there exist a complex basis $\{\varphi_a\}_{a=1}^{k_\varphi}$ of $\mathcal{V}_\varphi$ and $\mu_1,\mu_2,\dots,\mu_{k_\varphi}\in\mathbb{R}$ such that $\varphi_a\in\mathcal{V}_a$ and
\[
   \varrho(\exp tT)\varphi_a=e^{i\mu_at}\varphi_a
\]
for all $1\leq a\leq k_\varphi=\dim_\mathbb{C}\mathcal{V}_\varphi$ and $t\in\mathbb{R}$.\par

   (2). 
   One can conclude (2) by arguments similar to those above, $S^1\subset L\subset Q^-$, ${\sf V}$ being $\rho(Q^-)$-invariant and $m=\dim_\mathbb{C}{\sf V}<\infty$.
\end{proof}

   We are in a position to demonstrate
\begin{proposition}\label{prop-11.2.15}
   Let $\varphi\in(\mathcal{V}_{G/L})_K$ and $\mathcal{V}_\varphi=\operatorname{span}_\mathbb{C}\{\varrho(k)\varphi : k\in K\}$.
\begin{enumerate}
\item[{\rm (i)}]
   Let $\{\varphi_a\}_{a=1}^{k_\varphi}$ and $\{{\sf v}_b\}_{b=1}^m$ be the bases of $\mathcal{V}_\varphi$ and ${\sf V}$ in Lemma {\rm \ref{lem-11.2.14}}, respectively. 
   For $x\in GQ^-$ we express $\varphi_a(x)\in{\sf V}$ as
\[
    \varphi_a(x)=\varphi_a^1(x){\sf v}_1+\varphi_a^2(x){\sf v}_2+\cdots+\varphi_a^m(x){\sf v}_m.
\]
   Then, for each $1\leq a\leq k_\varphi$ and $1\leq b\leq m$, there exists a unique polynomial $($holomorphic$)$ function $\varphi_a^b{}'=\varphi_a^b{}'(z^1,\dots,z^r)$ on $U^+=\mathbb{C}^r$ of finite degree such that 
\[
   \mbox{$\varphi_a^b=\varphi_a^b{}'$ on $U^+\cap GQ^-$}.
\]
   Here $U^+$ is identified with $\mathbb{C}^r$ via \eqref{eq-11.2.10}, and $z^1,\dots,z^r$ is the canonical coordinates of the first kind associated with the basis $\{E_{\beta_j}\}_{j=1}^r\subset\frak{u}^+$.
\item[{\rm (ii)}]
   For a given $\phi\in\mathcal{V}_\varphi$ there exists a unique holomorphic mapping $\phi':U^+\to{\sf V}$ such that $\phi=\phi'$ on $U^+\cap GQ^-$.
\end{enumerate}
\end{proposition}
\begin{proof}
   (i). 
   By Lemma \ref{lem-11.1.2}-(3), $U^+\cap GQ^-$ is an open neighborhood of $e\in U^+$. 
   Hence the theorem of identity assures the uniqueness of $\varphi_a^b{}'$, where we remark that the restriction $\varphi_a^b|_{U^+\cap GQ^-}$ is holomorphic since $U^+$ is a regular complex submanifold of $G_\mathbb{C}$. 
   From now on, let us confirm the existence of $\varphi_a^b{}'$.
   Since $\varphi_a^b:U^+\cap GQ^-\to\mathbb{C}$ is holomorphic, we can find an $R>0$ so that the following (a1) and (a2) hold for $P:=\{u\in U^+ : |z^j(u)|<R, \, 1\leq j\leq r\}:$
\begin{enumerate}
\item[(a1)]
   $P$ is an open subset of $U^+\cap GQ^-$ containing $e$, and
\item[(a2)] 
   on $P$ we can express $\varphi_a^b|_{U^+\cap GQ^-}$ as 
\[
   \varphi_a^b(z^1,z^2,\dots,z^r)
   =\sum_{n_1,n_2,\dots,n_r\geq 0}\alpha_{n_1n_2\cdots n_r}(z^1)^{n_1}(z^2)^{n_2}\cdots(z^r)^{n_r}
\]
(the Taylor expansion of $\varphi_a^b|_{U^+\cap GQ^-}$ at $e=(0,0,\dots,0)$).   
\end{enumerate}
   Remark, it follows from \eqref{eq-11.2.11} that $sPs^{-1}\subset P$ for all $s\in S^1=\{\exp tT : t\in\mathbb{R}\}$.
   For any $t\in\mathbb{R}$ and $u\in P$ we obtain 
\allowdisplaybreaks{
\begin{align*}
   \mbox{$\sum_{b=1}^me^{i\theta_bt}\varphi_a^b(u){\sf v}_b$}
  &=\mbox{$\rho(\exp tT)\bigl(\sum_{b=1}^m\varphi_a^b(u){\sf v}_b\bigr)$ \quad ($\because$ Lemma \ref{lem-11.2.14}-(2))}\\
  &=\rho(\exp tT)\bigl(\varphi_a(u)\bigr)
   =\varphi_a\bigl(u\exp(-tT)\bigr) \quad\mbox{($\because$ $\varphi_a\in\mathcal{V}_{G/L}$, \eqref{eq-11.2.1}-(ii))}\\
  &\!\!\!\!\stackrel{\eqref{eq-11.2.2}}{=}\bigl(\varrho(\exp tT)\varphi_a\bigr)\bigl((\exp tT)u\exp(-tT)\bigr)
  =(e^{i\mu_at}\varphi_a)\bigl((\exp tT)u\exp(-tT)\bigr) \quad\mbox{($\because$ Lemma \ref{lem-11.2.14}-(1))}\\
 &=\mbox{$\sum_{b=1}^me^{i\mu_at}\varphi_a^b\bigl((\exp tT)u\exp(-tT)\bigr){\sf v}_b$}.
\end{align*}}This provides us with
\begin{equation}\label{eq-1}\tag*{\textcircled{1}}
   e^{i(\theta_b-\mu_a)t}\varphi_a^b(u)=\varphi_a^b\bigl((\exp tT)u\exp(-tT)\bigr).
\end{equation}
   If $u=\exp(z^1E_{\beta_1}+z^2E_{\beta_2}+\cdots+z^rE_{\beta_r})$, then it follows from (a2), \ref{eq-1} and \eqref{eq-11.2.11} that   
\begin{multline*}
   \sum_{n_1,n_2,\dots,n_r\geq 0}e^{i(\theta_b-\mu_a)t}\alpha_{n_1n_2\cdots n_r}(z^1)^{n_1}(z^2)^{n_2}\cdots(z^r)^{n_r}
   =e^{i(\theta_b-\mu_a)t}\varphi_a^b(z^1,z^2,\dots,z^r)\\
   =e^{i(\theta_b-\mu_a)t}\varphi_a^b(u)   
   =\varphi_a^b\bigl((\exp tT)u\exp(-tT)\bigr)
   =\varphi_a^b(e^{i\omega_1t}z^1,e^{i\omega_2t}z^2,\dots,e^{i\omega_rt}z^r)\\
   =\sum_{n_1,n_2,\dots,n_r\geq 0}e^{i(\omega_1n_1+\omega_2n_2+\cdots+\omega_rn_r)t}\alpha_{n_1n_2\cdots n_r}(z^1)^{n_1}(z^2)^{n_2}\cdots(z^r)^{n_r}.
\end{multline*}
   Therefore we see that 
\[
   e^{i(\theta_b-\mu_a)t}\alpha_{n_1n_2\cdots n_r}=e^{i(\omega_1n_1+\omega_2n_2+\cdots+\omega_rn_r)t}\alpha_{n_1n_2\cdots n_r}
\]
for all $t\in\mathbb{R}$ and $n_1,n_2,\dots,n_r\geq 0$.   
   Differentiating this equation at $t=0$ we deduce 
\begin{equation}\label{eq-2}\tag*{\textcircled{2}}
   (\theta_b-\mu_a)\alpha_{n_1n_2\cdots n_r}=(\omega_1n_1+\omega_2n_2+\cdots+\omega_rn_r)\alpha_{n_1n_2\cdots n_r}.
\end{equation}
   Here Lemma \ref{lem-11.2.12} implies that the number of non-negative integer solutions $(n_1,n_2,\dots,n_r)$ to the equation 
\[
   \theta_b-\mu_a=\omega_1n_1+\omega_2n_2+\cdots+\omega_rn_r
\]
is only finite or zero, so that the number of the non-zero coefficients $\alpha_{n_1n_2\cdots n_r}$ is only finite. 
   Consequently $\varphi_a^b(z^1,z^2,\dots,z^r)=\sum_{n_1,n_2,\dots,n_r\geq 0}\alpha_{n_1n_2\cdots n_r}(z^1)^{n_1}(z^2)^{n_2}\cdots(z^r)^{n_r}$ must be a polynomial function on the open subset $P\subset U^+$ of finite degree.
   Moreover, one can extend it as a polynomial function on $U^+$ of finite degree, since $z^1,z^2,\dots,z^r$ is a global coordinate system in $U^+$.\par
   
   (ii). 
   For any $\phi\in\mathcal{V}_\varphi$ there exist $\alpha_1,\dots,\alpha_{k_\varphi}\in\mathbb{C}$ such that $\phi=\sum_{a=1}^{k_\varphi}\alpha_a\varphi_a$. 
   Hence (ii) follows from (i).
\end{proof}

   Proposition \ref{prop-11.2.15}-(ii) leads to
\begin{corollary}\label{cor-11.2.16}
   For any $\varphi\in(\mathcal{V}_{G/L})_K$, there exists a unique holomorphic mapping $\varphi':U^+\to{\sf V}$ such that $\varphi=\varphi'$ on $U^+\cap GQ^-$.
   Here we refer to \eqref{eq-11.2.13} for $(\mathcal{V}_{G/L})_K$.  
\end{corollary} 

   Recalling that $\blacktriangle=\{\gamma\in\triangle \,|\, \gamma(T)=0\}$, we establish the following proposition which will play a role in the next subsection:
\begin{proposition}\label{prop-11.2.17}
   Suppose that the fundamental root system $\Pi_\triangle$ satisfies not only \eqref{eq-8.1.5} but also  
\[
   \mbox{$\frak{g}_\beta\subset\frak{k}_\mathbb{C}$ for all $\beta\in\Pi_\triangle-\blacktriangle$}.
\]
   Then, for each $\varphi\in(\mathcal{V}_{G/L})_K$ there exists a unique $h\in\mathcal{V}_{G_\mathbb{C}/Q^-}$ such that $\varphi=h$ on $GQ^-$.
\end{proposition}   
\begin{proof}
   The uniqueness of $h$ comes from the theorem of identity, Lemma \ref{lem-11.1.2}-(3) and $G_\mathbb{C}$ being connected.
   So, let us prove the existence of $h$. 
   Fix an arbitrary $\varphi\in(\mathcal{V}_{G/L})_K$.
   By Corollary \ref{cor-11.2.16} there exists a unique holomorphic mapping $\varphi':U^+\to{\sf V}$ so that $\varphi=\varphi'$ on $U^+\cap GQ^-$. 
   Then, Proposition \ref{prop-8.2.1}-(iv) enables us to construct a holomorphic mapping $\varphi'':U^+Q^-\to{\sf V}$ from 
\[
   \mbox{$\varphi''(uq):=\rho(q)^{-1}\bigl(\varphi'(u)\bigr)$ for $(u,q)\in U^+\times Q^-$}.
\]
   Here it follows from $\varphi\in\mathcal{V}_{G/L}$, \eqref{eq-11.2.1}-(ii) and $(U^+Q^-\cap GQ^-)=(U^+\cap GQ^-)Q^-$ that 
\begin{equation}\label{eq-1}\tag*{\textcircled{1}}
   \mbox{$\varphi=\varphi''$ on $U^+Q^-\cap GQ^-$}.
\end{equation}
   For every $\beta\in\Pi_\triangle-\blacktriangle$, the supposition and Lemma \ref{lem-11.2.9} assure $w_\beta\in K$, and thus $\varrho(w_\beta)\varphi$ belongs to $(\mathcal{V}_{G/L})_K$ since $(\mathcal{V}_{G/L})_K$ is $\varrho(K)$-invariant. 
   Consequently, for each $\beta\in\Pi_\triangle-\blacktriangle$ there exists a unique holomorphic mapping $(\varrho(w_\beta)\varphi)'':U^+Q^-\to{\sf V}$ such that
\begin{equation}\label{eq-2}\tag*{\textcircled{2}}
   \mbox{$\varrho(w_\beta)\varphi=(\varrho(w_\beta)\varphi)''$ on $U^+Q^-\cap GQ^-$},
\end{equation}
where we remark that $U^+Q^-$ is connected (cf.\ Proposition \ref{prop-8.2.1}-(iv)).
   Taking these $\varphi'',(\varrho(w_\beta)\varphi)'':U^+Q^-\to{\sf V}$ ($\beta\in\Pi_\triangle-\blacktriangle$) into account, we define a holomorphic mapping $\hat{\varphi}$ of $D:=U^+Q^-\cup(\bigcup_{\beta\in\Pi_\triangle-\blacktriangle}w_\beta^{-1}U^+Q^-)$ into ${\sf V}$ as follows:
\begin{equation}\label{eq-3}\tag*{\textcircled{3}}
   \hat{\varphi}(x)
   :=\begin{cases} \varphi''(x) & \mbox{if $x\in U^+Q^-$},\\ 
                   (\varrho(w_\beta)\varphi)''(w_\beta x) & \mbox{if $x\in w_\beta^{-1}U^+Q^-$ ($\beta\in\Pi_\triangle-\blacktriangle$)}.
     \end{cases}
\end{equation}
   Here $D=U^+Q^-\cup(\bigcup_{\beta\in\Pi_\triangle-\blacktriangle}w_\beta^{-1}U^+Q^-)$ is a dense domain in $G_\mathbb{C}$ by Corollary \ref{cor-8.3.16}-(i) and Lemma \ref{lem-8.3.19}. 
   We need to confirm that \ref{eq-3} is well-defined. 
   For any $y\in GQ^-\cap U^+Q^-\cap(\bigcap_{\beta\in\Pi_\triangle-\blacktriangle}w_\beta^{-1}U^+Q^-)$ one has $w_\beta y\in U^+Q^-$, $w_\beta y\in KGQ^-\subset GQ^-$, and 
\[
   (\varrho(w_\beta)\varphi)''(w_\beta y)
   \stackrel{\ref{eq-2}}{=}(\varrho(w_\beta)\varphi)(w_\beta y)
   \stackrel{\eqref{eq-11.2.2}}{=}\varphi(y)
   \stackrel{\ref{eq-1}}{=}\varphi''(y).
\]
   Thus \ref{eq-3} is well-defined by the theorem of identity and Lemma \ref{lem-8.3.27}-(3). 
   In addition, from the above computation we deduce 
\begin{equation}\label{eq-4}\tag*{\textcircled{4}}
   \mbox{$\varphi=\hat{\varphi}$ on $D\cap GQ^-$}.
\end{equation}
   Now, Lemma \ref{lem-8.3.22}-(2) and Remark \ref{rem-8.3.26} imply that the domain $D$ of $\hat{\varphi}$ includes the $O$ in Theorem \ref{thm-8.3.17}. 
   Therefore there exists a unique holomorphic mapping $h:G_\mathbb{C}\to{\sf V}$ such that 
\[
   \mbox{$\hat{\varphi}=h$ on $D$}
\]
by Theorem \ref{thm-8.3.17}-(ii). 
   This $h$ satisfies 
\begin{equation}\label{eq-5}\tag*{\textcircled{5}}
\begin{array}{ll}
   \mbox{$\varphi=h$ on $GQ^-$}, & \mbox{$h(aq)=\rho(q)^{-1}\bigl(h(a)\bigr)$ for all $(a,q)\in G_\mathbb{C}\times Q^-$}.
\end{array}
\end{equation}
   Indeed; since $GQ^-$ is connected, it follows from $\hat{\varphi}=h|_D$, \ref{eq-4} and the theorem of identity that $\varphi=h$ on $GQ^-$. 
   Furthermore, it follows from $\varphi=h|_{GQ^-}$ that for all $(x,q)\in GQ^-\times Q^-$ 
\[
\begin{split}
   h(xq)
   =\varphi(xq)
  &=\rho(q)^{-1}\bigl(\varphi(x)\bigr) \quad\mbox{($\because$ $\varphi\in\mathcal{V}_{G/L}$, \eqref{eq-11.2.1}-(ii))}\\
  &=\rho(q)^{-1}\bigl(h(x)\bigr).
\end{split}
\] 
   This and the theorem of identity imply that $h(aq)=\rho(q)^{-1}\bigl(h(a)\bigr)$ for all $(a,q)\in G_\mathbb{C}\times Q^-$, since $GQ^-\subset G_\mathbb{C}$ is open. 
   Accordingly \ref{eq-5} holds.
   By \eqref{eq-11.2.1} and \ref{eq-5} we conclude $h\in\mathcal{V}_{G_\mathbb{C}/Q^-}$ and this proposition.
\end{proof}

\subsection{A sufficient condition for $\mathcal{V}_{G/L}$ to be finite-dimensional}\label{subsec-11.2.3}
   In order to state Theorem \ref{thm-11.2.18}, let us fix its setting.
\begin{itemize}
\item 
   $G_\mathbb{C}$ is a connected complex semisimple Lie group,
\item
   $G$ is a connected closed subgroup of $G_\mathbb{C}$ such that $\frak{g}$ is a real form of $\frak{g}_\mathbb{C}$,
\item
   $T$ is a non-zero elliptic element of $\frak{g}$, 
\item
   $\frak{g}=\frak{k}\oplus\frak{p}$ is a Cartan decomposition of $\frak{g}$ with $T\in\frak{k}$, 
\item 
   $i\frak{h}_\mathbb{R}$ is a maximal torus of $\frak{g}_u:=\frak{k}\oplus i\frak{p}$ containing $T$, 
\item
   $\triangle=\triangle(\frak{g}_\mathbb{C},\frak{h}_\mathbb{C})$ is the root system of $\frak{g}_\mathbb{C}$ relative to $\frak{h}_\mathbb{C}$, where $\frak{h}_\mathbb{C}$ is the complex vector subspace of $\frak{g}_\mathbb{C}$ generated by $i\frak{h}_\mathbb{R}$,
\item
   $\frak{g}_\alpha$ is the root subspace of $\frak{g}_\mathbb{C}$ for $\alpha\in\triangle$,
\item
   $L=C_G(T)$,
\item
   $Q^-=N_{G_\mathbb{C}}(\bigoplus_{\nu\geq 0}\frak{g}^{-\nu})$, where $\frak{g}^\lambda=\{X\in\frak{g}_\mathbb{C} \,|\, \operatorname{ad}T(X)=i\lambda X\}$ for $\lambda\in\mathbb{R}$,
\item 
   $\frak{k}_\mathbb{C}$ is the complex subalgebra of $\frak{g}_\mathbb{C}$ generated by $\frak{k}$,
\item
   ${\sf V}$ is a finite-dimensional complex vector space, 
\item
   $\rho:Q^-\to GL({\sf V})$, $q\mapsto\rho(q)$, is a holomorphic homomorphism,
\item
   $\mathcal{V}_{G_\mathbb{C}/Q^-}$ and $\mathcal{V}_{G/L}$ are the complex vector spaces defined by 
\[
\begin{split}
&  \mathcal{V}_{G_\mathbb{C}/Q^-}
   :=\left\{\begin{array}{@{}c|c@{}}
   h:G_\mathbb{C}\to{\sf V} 
   & \begin{array}{@{}l@{}} \mbox{(i) $h$ is holomorphic},\\ \mbox{(ii) $h(aq)=\rho(q)^{-1}\bigl(h(a)\bigr)$ for all $(a,q)\in G_\mathbb{C}\times Q^-$}\end{array}
   \end{array}\right\},\\
&  \mathcal{V}_{G/L}
   :=\left\{\begin{array}{@{}c|c@{}}
   \psi:GQ^-\to{\sf V} 
   & \begin{array}{@{}l@{}} \mbox{(i) $\psi$ is holomorphic},\\ \mbox{(ii) $\psi(xq)=\rho(q)^{-1}\bigl(\psi(x)\bigr)$ for all $(x,q)\in GQ^-\times Q^-$}\end{array}
   \end{array}\right\},
\end{split}
\]
respectively. 
\end{itemize}
   In the setting above we establish 

\begin{theorem}[{cf.\ \cite{B}\footnote{We improve the proof of Theorem 3.1 in \cite{B}.}}]\label{thm-11.2.18}
   Suppose that {\rm (S)} there exists a fundamental root system $\Pi_\triangle$ of $\triangle$  satisfying 
\begin{enumerate}
\item[{\rm (s1)}]
   $\alpha(-iT)\geq 0$ for all $\alpha\in\Pi_\triangle$, and 
\item[{\rm (s2)}]
   $\frak{g}_\beta\subset\frak{k}_\mathbb{C}$ for every $\beta\in\Pi_\triangle$ with $\beta(T)\neq 0$.  
\end{enumerate} 
   Then, the complex vector space $\mathcal{V}_{G_\mathbb{C}/Q^-}$ is linear isomorphic to $\mathcal{V}_{G/L}$ via 
\[
   \mbox{$F:\mathcal{V}_{G_\mathbb{C}/Q^-}\to\mathcal{V}_{G/L}$, $h\mapsto h|_{GQ^-}$};
\]
and therefore $\dim_\mathbb{C}\mathcal{V}_{G/L}=\dim_\mathbb{C}\mathcal{V}_{G_\mathbb{C}/Q^-}<\infty$.
   Here $h|_{GQ^-}$ stands for the restriction of $h$ to $GQ^-$ $(\subset G_\mathbb{C})$.
\end{theorem}
\begin{proof}
   Needless to say, the mapping $F:\mathcal{V}_{G_\mathbb{C}/Q^-}\to\mathcal{V}_{G/L}$, $h\mapsto h|_{GQ^-}$, is complex linear. 
   Lemma \ref{lem-11.1.2}-(3) and the theorem of identity imply that $F$ is injective because $G_\mathbb{C}$ is connected. 
   Consequently, the rest of proof is to demonstrate that $F$ is surjective, cf.\ Remark \ref{rem-11.2.3}-(2).
   Fix an arbitrary $\psi\in\mathcal{V}_{G/L}$. 
   By Propositions \ref{prop-11.2.8} and \ref{prop-6.2.1}, and by \eqref{eq-11.2.13} we deduce that $(\mathcal{V}_{G/L})_K$ is a dense subset of $\mathcal{V}_{G/L}=(\mathcal{V}_{G/L},d)$. 
   So, there exists a sequence $\{\varphi_n\}_{n=1}^\infty\subset(\mathcal{V}_{G/L})_K$ satisfying  
\[
   \lim_{n\to\infty}d(\psi,\varphi_n)=0.
\]
   On the one hand; the supposition (S) and Proposition \ref{prop-11.2.17} assure that $(\mathcal{V}_{G/L})_K\subset F(\mathcal{V}_{G_\mathbb{C}/Q^-})$, and thus 
\[
   \{\varphi_n\}_{n=1}^\infty\subset F(\mathcal{V}_{G_\mathbb{C}/Q^-}).
\] 
   On the other hand; since $F:\mathcal{V}_{G_\mathbb{C}/Q^-}\to\mathcal{V}_{G/L}$ is injective linear, Proposition \ref{prop-4.1.8}-(2) and $\dim_\mathbb{C}\mathcal{V}_{G_\mathbb{C}/Q^-}<\infty$ enable us to see that  
\[
   \mbox{$F(\mathcal{V}_{G_\mathbb{C}/Q^-})$ is a closed subset of $\mathcal{V}_{G/L}$}.
\] 
   Therefore $\displaystyle{\psi=\lim_{n\to\infty}\varphi_n\in F(\mathcal{V}_{G_\mathbb{C}/Q^-})}$, and hence $F$ is surjective.
\end{proof}

\begin{remark}\label{rem-11.2.19}
   If the supposition (S) in Theorem \ref{thm-11.2.18} holds for the elliptic orbit $G/L$, then one can clarify several properties of $G/L$---for example,
\begin{enumerate}
\item 
   any holomorphic function on $G/L$ is constant,
\item 
   the group ${\rm Hol}(G/L)$ of holomorphic automorphisms of $G/L$ is a (finite-dimensional) Lie group,   
\end{enumerate}
and so on.
\end{remark}

   Let us give examples which satisfy the supposition (S) in Theorem \ref{thm-11.2.18}, and give examples which do not so.\par
   
   The first example is 
\begin{example}[$G/L=G_{2(2)}/(SL(2,\mathbb{R})\cdot T^1)$]\label{ex-11.2.20}
   Let $\frak{g}_\mathbb{C}$ be the exceptional complex simple Lie algebra $(\frak{g}_2)_\mathbb{C}$ of the type $G_2$. 
   Assume that the Dynkin diagram of $\triangle=\triangle(\frak{g}_\mathbb{C},\frak{h}_\mathbb{C})$ is as follows (cf.\ Bourbaki \cite[p.289]{BrL}):
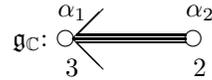
\begin{center}
\unitlength=1mm
\begin{picture}(25,9)
\put(6,8){$\alpha_1$}
\put(7,5){\circle{2}}
\put(7,0){$3$}
\put(8,5){\line(1,1){4}}
\put(8,5){\line(1,-1){4}}
\put(8.5,5.5){\line(1,0){14.5}}
\put(8.5,4.5){\line(1,0){14.5}}
\put(8,5){\line(1,0){15}}
\put(23,8){$\alpha_2$}
\put(24,5){\circle{2}}
\put(24,0){$2$}
\put(0,4){$\frak{g}_\mathbb{C}$:}
\end{picture}
\end{center} 
   Then $\Pi_\triangle=\{\alpha_1,\alpha_2\}$, and the set $\triangle^+$ of positive roots is 
\begin{equation}\label{eq-1}\tag*{\textcircled{1}}
   \triangle^+
   =\{3\alpha_1+2\alpha_2, 3\alpha_1+\alpha_2, 2\alpha_1+\alpha_2, \alpha_1+\alpha_2, \alpha_1, \alpha_2\}.
\end{equation}
   Let us fix a non-compact real form $\frak{g}\subset\frak{g}_\mathbb{C}$.
   Taking Chevalley's canonical basis $\{H_{\alpha_1}^*,H_{\alpha_2}^*\}\amalg\{E_\alpha\,|\,\alpha\in\triangle\}$ of $\frak{g}_\mathbb{C}$ we first construct a compact real form $\frak{g}_u\subset\frak{g}_\mathbb{C}$ from 
\[
\begin{array}{ll}
   \frak{h}_\mathbb{R}:=\operatorname{span}_\mathbb{R}\{H_{\alpha_1}^*,H_{\alpha_2}^*\}, & 
   \frak{g}_u:=i\frak{h}_\mathbb{R}\oplus\bigoplus_{\alpha\in\triangle}\operatorname{span}_\mathbb{R}\{E_{\alpha}-E_{-\alpha}\}\oplus\operatorname{span}_\mathbb{R}\{i(E_{\alpha}+E_{-\alpha})\},
\end{array}
\]
and denote by $\{Z_1,Z_2\}$ ($\subset\frak{h}_\mathbb{R}$) the dual basis of $\Pi_\triangle=\{\alpha_1,\alpha_2\}$. 
   By use of this $Z_2$ we next construct an involutive automorphism $\theta$ of the complex Lie algebra $\frak{g}_\mathbb{C}$ from
\begin{equation}\label{eq-2}\tag*{\textcircled{2}}
   \theta:=\exp\pi\operatorname{ad}(iZ_2).
\end{equation}
   Since $\theta(\frak{g}_u)\subset\frak{g}_u$ one can get a non-compact real form $\frak{g}\subset\frak{g}_\mathbb{C}$ by setting 
\[
\begin{array}{lll}
  \frak{k}:=\{X\in\frak{g}_u \,|\, \theta(X)=X\}, & i\frak{p}:=\{Y\in\frak{g}_u \,|\, \theta(Y)=-Y\}, & \frak{g}:=\frak{k}\oplus\frak{p}.
\end{array}
\] 
   Here we remark that $\frak{g}_u=\frak{k}\oplus i\frak{p}$, $\frak{k}=\frak{sp}(1)\oplus\frak{sp}(1)$, $\frak{g}=\frak{g}_{2(2)}$ and 
\begin{equation}\label{eq-3}\tag*{\textcircled{3}}
   \frak{k}_\mathbb{C}=\{V\in\frak{g}_\mathbb{C} \,|\, \theta(V)=V\},
\end{equation}
where $\frak{k}_\mathbb{C}$ stands for the complex subalgebra of $\frak{g}_\mathbb{C}$ generated by $\frak{k}$.
\begin{center}
\unitlength=1mm
\begin{picture}(35,6)
\put(6,5){$\alpha_1$}
\put(7,2){\circle{2}}
\put(18,5){$-3\alpha_1-2\alpha_2$}
\put(24,2){\circle{2}}
\put(0,1){$\frak{k}_\mathbb{C}$:}
\end{picture}
\end{center} 
   In this setting, each $T\in i\frak{h}_\mathbb{R}$ is an elliptic element of $\frak{g}$ and we know that for $\frak{l}:=\frak{c}_\frak{g}(T)$,
\begin{center}
\begin{tabular}{ll}
     (A) $\frak{l}=\frak{sl}(2,\mathbb{R})\oplus\frak{t}^1$ in case of $T=i(Z_1-2Z_2)$, 
   & (B) $\frak{l}=\frak{sl}(2,\mathbb{R})\oplus\frak{t}^1$ in case of $T=i(Z_1-3Z_2)$. 
\end{tabular}
\end{center} 
   cf.\ Proposition 5.5 in \cite[p.1157]{Bn}.\par
   
   Case (A). 
   Let $T:=i(Z_1-2Z_2)$ and $\Pi_A:=\{2\alpha_1+\alpha_2,-3\alpha_1-2\alpha_2\}$. 
   Then $\Pi_A$ is a fundamental root system of $\triangle$ by \ref{eq-1}.
\begin{center}
\unitlength=1mm
\begin{picture}(45,9)
\put(6,8){$2\alpha_1+\alpha_2$}
\put(17,5){\circle{2}}
\put(18,5){\line(1,1){4}}
\put(18,5){\line(1,-1){4}}
\put(18.5,5.5){\line(1,0){14.5}}
\put(18.5,4.5){\line(1,0){14.5}}
\put(18,5){\line(1,0){15}}
\put(28,8){$-3\alpha_1-2\alpha_2$}
\put(34,5){\circle{2}}
\put(0,4){$\Pi_A$:}
\end{picture}
\end{center}
   From a direct computation with $\alpha_k(Z_j)=\delta_{k,j}$ we obtain  
\[
\begin{array}{ll}
   (2\alpha_1+\alpha_2)(-iT)=0, & (-3\alpha_1-2\alpha_2)(-iT)=1\geq0.
\end{array}
\]
   This assures that $\Pi_A$ satisfies the condition (s1) in Theorem \ref{thm-11.2.18}. 
   Moreover, \ref{eq-2} yields $\theta(E_{-3\alpha_1-2\alpha_2})=E_{-3\alpha_1-2\alpha_2}$, and so \ref{eq-3} yields $\frak{g}_{-3\alpha_1-2\alpha_2}\subset\frak{k}_\mathbb{C}$.
   Therefore $\Pi_A$ satisfies the condition (s2) also. 
   Hence, the supposition (S) in Theorem \ref{thm-11.2.18} holds in this case.\par
   
   Case (B). 
   Let $T:=i(Z_1-3Z_2)$ and $\Pi_B:=\{\alpha_1,-3\alpha_1-\alpha_2\}$. 
   Then, one can conclude that the supposition (S) in Theorem \ref{thm-11.2.18} holds, by arguments similar to those above, 
\begin{center}
\unitlength=1mm
\begin{picture}(45,9)
\put(15,8){$\alpha_1$}
\put(17,5){\circle{2}}
\put(18,5){\line(1,1){4}}
\put(18,5){\line(1,-1){4}}
\put(18.5,5.5){\line(1,0){14.5}}
\put(18.5,4.5){\line(1,0){14.5}}
\put(18,5){\line(1,0){15}}
\put(28,8){$-3\alpha_1-\alpha_2$}
\put(34,5){\circle{2}}
\put(0,4){$\Pi_B$:}
\end{picture}
\end{center}
and $\alpha_1(-iT)=1$, $(-3\alpha_1-\alpha_2)(-iT)=0$. 
\end{example}

\begin{remark}\label{rem-11.2.21}
   In case of Example \ref{ex-11.2.20}-(A), $\Pi':=\{-2\alpha_1-\alpha_2,3\alpha_1+\alpha_2\}$ is another fundamental root system of $\triangle$, and 
\[
\begin{array}{ll}
   (-2\alpha_1-\alpha_2)(-iT)=0, & (3\alpha_1+\alpha_2)(-iT)=1.
\end{array}
\] 
   Thus $\Pi'$ satisfies the condition (s1) in Theorem \ref{thm-11.2.18}. 
   However, it follows from \ref{eq-2} that $\theta(E_{3\alpha_1+\alpha_2})=-E_{3\alpha_1+\alpha_2}$, so that the condition (s2) cannot hold for this $\Pi'$.
\end{remark}

   The second example is 
\begin{example}[$G/L=SU(2,1)/S(U(1)\times U(1,1))$]\label{ex-11.2.22}
   Let 
\[
\begin{array}{ll}
   G_\mathbb{C}:=SL(3,\mathbb{C})=\{g\in GL(3,\mathbb{C}) \,|\, \det g=1\}, & G:=SU(2,1)=\{X\in G_\mathbb{C} \,|\, {}^t\!XI_{2,1}\overline{X}=I_{2,1}\},
\end{array}
\]
where $I_{2,1}=\begin{pmatrix} -1 & 0 & 0\\ 0 & -1 & 0\\ 0 & 0 & 1\end{pmatrix}$. 
   Then one has 
\[
   \frak{g}=
   \left\{\begin{array}{@{}c|l@{}}
   \left(\begin{array}{cc|c}
    ia_1 &  b+ic & ix-y\\
   -b+ic &  ia_2 & iz-w\\ \hline
   -ix-y & -iz-w & ia_3
   \end{array}\right)
   & \begin{array}{@{}r@{}} a_1,a_2,a_3,b,c,x,y,z,w\in\mathbb{R},\\ a_1+a_2+a_3=0\end{array}
   \end{array}\right\}   
\]
and obtains a Cartan decomposition $\frak{g}=\frak{k}\oplus\frak{p}$, 
\[
\begin{array}{ll}
   \frak{k}=
   \left\{
   \left(\begin{array}{cc|c}
    ia_1 &  b+ic & 0\\
   -b+ic &  ia_2 & 0\\ \hline
       0 &     0 & ia_3
   \end{array}\right)\!\in\frak{g}
   \right\}
   \!=\frak{s}\bigl(\frak{u}(2)\oplus\frak{u}(1)\bigr),
&
   \frak{p}=
   \left\{
   \left(\begin{array}{cc|c}
       0 &     0 & ix-y\\
       0 &     0 & iz-w\\ \hline
   -ix-y & -iz-w & 0
   \end{array}\right)\!\in\frak{g}
   \right\}\!.
\end{array}   
\]
   Setting $\frak{g}_u:=\frak{k}\oplus i\frak{p}$ and 
\[
   \frak{h}_\mathbb{R}
   :=
   \left\{\begin{array}{@{}c|l@{}}
   \left(\begin{array}{cc|c}
    a_1 &    0 & 0\\
      0 &  a_2 & 0\\ \hline
      0 &    0 & a_3
   \end{array}\right)
   & \begin{array}{@{}r@{}} a_1,a_2,a_3\in\mathbb{R},\\ a_1+a_2+a_3=0\end{array}
   \end{array}\right\}\!,
\]
we assert that $\frak{g}_u$ is a compact real form of $\frak{g}_\mathbb{C}=\frak{sl}(3,\mathbb{C})$ and 
\[
   \frak{g}_u
   =
   \left\{\begin{array}{@{}c|l@{}}
   \left(\begin{array}{cc|c}
    ia_1 &  b+ic & x+iy\\
   -b+ic &  ia_2 & z+iw\\ \hline
   -x+iy & -z+iw & ia_3
   \end{array}\right)
   & \begin{array}{@{}r@{}} a_1,a_2,a_3,b,c,x,y,z,w\in\mathbb{R},\\ a_1+a_2+a_3=0\end{array}
   \end{array}\right\}
   =\frak{su}(3);
\]   
besides, $i\frak{h}_\mathbb{R}$ is a maximal torus of $\frak{g}_u$.
   Remark that each $T\in i\frak{h}_\mathbb{R}$ is an elliptic element of $\frak{g}=\frak{su}(2,1)$ due to $i\frak{h}_\mathbb{R}\subset\frak{k}$.\par

   Now, let us define complex linear mappings $\alpha_1,\alpha_2:\frak{h}_\mathbb{C}\to\mathbb{C}$ by 
\[
\begin{array}{ll}
   \alpha_1\!\left(
   \left(\begin{array}{cc|c}
    \epsilon_1 &    0 & 0\\
      0 &  \epsilon_2 & 0\\ \hline
      0 &    0 & \epsilon_3
   \end{array}\right) \right)\!:=\epsilon_1-\epsilon_2,   
&
   \alpha_2\!\left(
   \left(\begin{array}{cc|c}
    \epsilon_1 &    0 & 0\\
      0 &  \epsilon_2 & 0\\ \hline
      0 &    0 & \epsilon_3
   \end{array}\right) \right)\!:=\epsilon_2-\epsilon_3,
\end{array}   
\] 
respectively.
   Then $\Pi_\triangle:=\{\alpha_1,\alpha_2\}$ is a fundamental root system of $\triangle(\frak{g}_\mathbb{C},\frak{h}_\mathbb{C})$.
\begin{center}
\unitlength=1mm
\begin{picture}(25,9)
\put(6,8){$\alpha_1$}
\put(7,5){\circle{2}}
\put(7,0){$1$}
\put(8,5){\line(1,0){15}}
\put(23,8){$\alpha_2$}
\put(24,5){\circle{2}}
\put(24,0){$1$}
\put(0,4){$\frak{g}_\mathbb{C}$:}
\end{picture}
\end{center} 
   By setting  
\[
\begin{array}{ll}
   Z_1:=
   \dfrac{1}{\,3\,}\!
   \left(\begin{array}{cc|c}
      2 &  0 & 0\\
      0 & -1 & 0\\ \hline
      0 &  0 & -1
   \end{array}\right)\!,
&
   Z_2:=
   \dfrac{1}{\,3\,}\!
   \left(\begin{array}{cc|c}
      1 &  0 & 0\\
      0 &  1 & 0\\ \hline
      0 &  0 & -2
   \end{array}\right)\!,
\end{array}
\]
we have $\frak{h}_\mathbb{R}=\operatorname{span}_\mathbb{R}\{Z_1,Z_2\}$ and $\alpha_k(Z_j)=\delta_{k,j}$ ($k,j=1,2$).\par

   $\bullet$ Case $T=iZ_1$. 
   Let $T:=iZ_1$. 
   Then it follows from $T\in i\frak{h}_\mathbb{R}$ that $T$ is an elliptic element of $\frak{g}=\frak{su}(2,1)$. 
   From a direct computation with $\alpha_k(Z_j)=\delta_{k,j}$ we obtain 
\[
\begin{array}{ll}
   \alpha_1(-iT)=1, & \alpha_2(-iT)=0.
\end{array}
\] 
   Hence $\Pi_\triangle=\{\alpha_1,\alpha_2\}$ satisfies the condition (s1) in Theorem \ref{thm-11.2.18}. 
   Since
\[
   \frak{g}_{\alpha_1}
   =
   \left\{\begin{array}{@{}c|l@{}}
   \left(\begin{array}{cc|c}
    0 & \epsilon & 0\\
    0 &        0 & 0\\ \hline
    0 &        0 & 0
   \end{array}\right)
   & \epsilon\in\mathbb{C}
   \end{array}\right\}
   \subset\frak{k}_\mathbb{C},
\]
it satisfies the condition (s2) also. 
   For this reason the supposition (S) in Theorem \ref{thm-11.2.18} holds for the $T=iZ_1$.
   Incidentally, $L:=C_G(T)=S(U(1)\times U(1,1))$ and $G/L=SU(2,1)/S(U(1)\times U(1,1))$.\par
   
   $\bullet$ Case $T=iZ_2$. 
   Let $T:=iZ_2$. 
   Then $T$ is an elliptic element of $\frak{g}$, and one has 
\[
\begin{array}{ll}
   \alpha_1(-iT)=0, & \alpha_2(-iT)=1,
\end{array}
\] 
so $\Pi_\triangle=\{\alpha_1,\alpha_2\}$ satisfies the condition (s1) in Theorem \ref{thm-11.2.18}. 
   However, the condition (s2) cannot hold for this $T=iZ_2$ because 
\[
   \frak{g}_{\alpha_2}
   =
   \left\{\begin{array}{@{}c|l@{}}
   \left(\begin{array}{cc|c}
    0 & 0 & 0\\
    0 & 0 & \epsilon\\ \hline
    0 & 0 & 0
   \end{array}\right)
   & \epsilon\in\mathbb{C}
   \end{array}\right\}
   \subset\frak{p}_\mathbb{C}.
\]   
   Incidentally, $G/L=SU(2,1)/S(U(2)\times U(1))$ and is a symmetric bounded domain in $\mathbb{C}^2$.   
\end{example} 

   The third example is 
\begin{example}\label{ex-11.2.23}
   The supposition (S) in Theorem \ref{thm-11.2.18} cannot hold for any symmetric bounded domain $D$ in $\mathbb{C}^n$ at all.\par
   
   Let us explain the reason why. 
   In order to do so, we take an elliptic orbit $G/L=G/C_G(T)$ in the setting of Theorem \ref{thm-11.2.18}, and put $\frak{u}:=\operatorname{ad}T(\frak{g})$.
   Since $\operatorname{ad}T:\frak{g}\to\frak{g}$ is semisimple, $\frak{g}$ is decomposed into $\frak{g}=\frak{l}\oplus\frak{u}$; and furthermore, it is decomposed into
\[
   \frak{g}=(\frak{k}\cap\frak{l})\oplus(\frak{p}\cap\frak{l})\oplus(\frak{k}\cap\frak{u})\oplus(\frak{p}\cap\frak{u})
\]
because of $T\in\frak{k}$.
   Then, Lemma \ref{lem-11.2.9} implies that 
\[
   \frak{k}\cap\frak{u}\neq\{0\}
\]
is a necessary condition for the (s2) to hold.
   However, if $G/L$ is a symmetric bounded domain in $\mathbb{C}^n$ (where $G$ is the identity component of ${\rm Hol}(G/L)$), then it follows that 
\[
\begin{array}{llll}
   (\frak{k}\cap\frak{l})=\frak{k}, & (\frak{p}\cap\frak{l})=\{0\}, & (\frak{k}\cap\frak{u})=\{0\}, & (\frak{p}\cap\frak{u})=\frak{p}.
\end{array}
\]
   For this reason the supposition (S) cannot hold. 
\end{example} 

   We end this chapter with stating 
\begin{remark}\label{rem-11.2.24}
   For each complex flag manifold $G_\mathbb{C}/Q^-$ one can determine the complex Lie algebra $\mathcal{O}(T^{1,0}(G_\mathbb{C}/Q^-))$ of holomorphic vector fields on $G_\mathbb{C}/Q^-$ by Theorem 7.1 in Onishchik \cite[pp.52--53]{On}. 
   Accordingly we deduce that 
\begin{enumerate}
\item 
  $G/L=G_{2(2)}/(SL(2,\mathbb{R})\cdot T^1)$ and $\mathcal{O}(T^{1,0}(G/L))=(\frak{g}_2)_\mathbb{C}$ in case of Example \ref{ex-11.2.20}-(A),
\item 
  $G/L=G_{2(2)}/(SL(2,\mathbb{R})\cdot T^1)$ and $\mathcal{O}(T^{1,0}(G/L))=\frak{so}(7,\mathbb{C})$ in case of Example \ref{ex-11.2.20}-(B),
\item 
  $G/L=SU(2,1)/S(U(1)\times U(1,1))$ and $\mathcal{O}(T^{1,0}(G/L))=\frak{sl}(3,\mathbb{C})$ in case of Example \ref{ex-11.2.22} with $T=iZ_1$
\end{enumerate}
from Theorem \ref{thm-11.2.18}.
\end{remark} 

\printindex
   
\end{document}